\documentclass[11pt,reqno]{amsart}
\usepackage{bm}
\usepackage{mathrsfs}
\usepackage{amsmath, amsfonts, amsthm, amssymb,graphicx, color, cite,enumitem}

 \usepackage{hyperref}
\hypersetup{hidelinks}

\usepackage[margin=1.05in]{geometry}

\newtheorem{theorem}{Theorem}[section]
\newtheorem{lemma}{Lemma}[section]
\newtheorem{proposition}{Proposition}[section]
\newtheorem{corollary}{Corollary}[section]

\theoremstyle{definition}
\newtheorem{definition}{Definition}[section]

\theoremstyle{remark}
\newtheorem{remark}{Remark}[section]

\numberwithin{equation}{section}
\allowdisplaybreaks[4]

\def\p{\partial}

\newcommand{\R}{{\mathbb R}}
\newcommand{\pe}{\Pi^{\perp}}

\def\dd{{\rm d}}

\def\be{\begin{equation}}
\def\en{\end{equation}}
\def\bs{\begin{split}}
\def\es{\end{split}}
\def\ba{\begin{align}}
\def\ea{\end{align}}

\newcommand{\vertiii}[1]{{\left\vert\kern-0.25ex\left\vert\kern-0.25ex\left\vert #1
    \right\vert\kern-0.25ex\right\vert\kern-0.25ex\right\vert}}

\newcommand{\supp}{{\rm supp}}

\author[Robin Ming Chen]{Robin Ming Chen}
\address{Department of Mathematics, University of Pittsburgh, Pittsburgh, PA 15260, USA.}
\email{mingchen@pitt.edu}

\author[Feimin Huang]{Feimin Huang}
\address{State Key Laboratory of Mathematical Sciences, Academy of Mathematics and Systems Science, Chinese Academy of Sciences, Beijing 100190, China.}
\email{fhuang@amt.ac.cn}

\author[Dehua Wang]{Dehua Wang}
\address{Department of Mathematics, University of Pittsburgh, Pittsburgh, PA 15260, USA.}
\email{dwang@math.pitt.edu}

\author[Difan Yuan]{Difan Yuan}
\address{School of Mathematical Sciences, Beijing Normal University and
 Laboratory of Mathematics, and Complex Systems, Ministry of Education, Beijing 100875, China.}
\email{yuandf@amss.ac.cn}

\title[Stability of vortex sheet in elastodynamics]
{Nonlinear stability of compressible vortex sheets in three-dimensional elastodynamics}

\date{}

\begin{document}
\begin{abstract}
We investigate the nonlinear stability and existence of compressible vortex sheet solutions for three-dimensional isentropic elastic flows. This problem involves a nonlinear hyperbolic system with a characteristic free boundary. Compared to the two-dimensional case, the additional spatial dimension introduces intricate frequency interactions between elasticity and velocity, significantly complicating the stability analysis. Building upon previous results on the weakly linear stability of elastic vortex sheets \cite{RChen2021}, we perform a detailed study of the roots of the Lopatinski$\breve{\mathrm{i}}$ determinant and identify a geometric stability condition associated with the deformation gradient.

To address the challenges of the variable-coefficient linearized problem, we employ an upper triangularization technique that isolates the outgoing modes into a closed system, where they appear only at the leading order. This enables us to derive energy estimates despite derivative loss. The major novelty of our approach includes the following two key aspects:
(1) For the three-dimensional compressible Euler vortex sheets, the front symbol exhibits degenerate ellipticity in certain frequency directions, which makes it challenging to ensure the front's regularity using standard energy estimates. Our analysis reveals that the non-parallel structure of the deformation gradient tensor plays a crucial role in recovering ellipticity in the front symbol, thereby enhancing the regularity of the free interface. (2) Another significant challenge in three dimensions arises from the strong degeneracy caused by the collision of repeated roots and poles. Unlike in two dimensions, where such interactions are absent, we encounter a co-dimension one set in frequency space where a double root coincides with a double pole. To resolve this, we refine Coulombel's diagonalization framework \cite{Coulombel2004shocks} and construct a suitable transformation that reduces the degeneracy order of the Lopatinski$\breve{\mathrm{i}}$ matrix, enabling the use of localized G${\rm\mathring a}$rding-type estimates to control the characteristic components. Finally, we employ a Nash-Moser iteration scheme to establish the local existence and nonlinear stability of vortex sheets under small initial perturbations, demonstrating stability within a subsonic regime.
\end{abstract}
\maketitle

\setcounter{tocdepth}{1}
\tableofcontents

\section{Introduction}

Vortex sheets are interfaces between two inviscid incompressible or compressible flows, characterized by a contact discontinuity in the fluid velocity.  Across these interfaces, the tangential velocity field exhibits a jump discontinuity, while the normal component of the flow velocity remains continuous. Vortex sheets arise from various physical phenomena in fluid mechanics, including oceanography, plasma physics, astrophysics, elastodynamics and aerodynamics. In compressible flows, they are one of the fundamental wave types, along with shock waves and rarefaction waves, in multi-dimensional (M-D) hyperbolic systems of conservation laws. Studying the existence and stability of compressible vortex sheets can give a better understanding of M-D Riemann problems and the behavior of entropy solutions; see Chen-Feldman \cite{ChenFeldman} and Dafermos \cite{Dafermos}.

In this paper, we focus on studying the vortex sheets in 3D compressible inviscid flows in elastodynamics: ({\it cf.} \cite{Gurtin,Joseph} for the physical background):
\begin{equation}\label{elastic2}
\begin{cases}
\rho_t+\mathrm{div}(\rho \mathbf{\textbf{v}})=0,\\
(\rho \mathbf{\textbf{v}})_t+\mathrm{div}(\rho \mathbf{\textbf{v}}\otimes \mathbf{\textbf{v}})+\nabla p=\mathrm{div}(\rho \mathbf{F}\mathbf{F}^{\top}),\\
(\rho\mathcal{F}_j)_t+ \mathrm{div}(\rho \mathcal{F}_j\otimes\textbf{v}-\textbf{v}\otimes\rho \mathcal{F}_j)=\mathbf{0},\\
\end{cases}
\end{equation}
where $\rho$ denotes the density, $\mathbf{v}=(v_1,v_2,v_3)^{\top}\in \R^3$ the velocity, ${\mathcal{F}}_j$ is the $j$th column of deformation gradient $\mathrm{\mathbf{F}}=(F_{ij})\in \mathbf{M}^{3\times3},i,j=1,2,3$ and $p$ for the pressure with $p=p(\rho)$ a smooth strictly increasing function on $(0,\infty).$  We also introduce Mach number $M=\frac{|\mathbf{\textbf{v}}|}{c},$ where
\begin{align}\label{c.def}
c:=c(\rho)=\sqrt{p'(\rho)}, \text{ for } \rho>0.
 \end{align}
 The fluids occupy in 3D space and the interfaces are 2D embedded inside the fluid.
 \smallskip

Note that by taking divergence of the third equations in \eqref{elastic2}, we end up with
 \begin{equation}\nonumber
 \partial_t(\mathrm{div}(\rho \mathcal{F}_j))=0,\text{ for } j=1,2,3.
\end{equation}
In column-wise components, we can write the {\it intrinsic property} (involution condition for the elastic flow, refer to \cite{Dafermos}) as follows:
\begin{equation}\label{intrinsic}
\mathrm{div}(\rho \mathcal{F}_j)=0, \text{ for } j=1,2,3.
\end{equation}
This intrinsic property holds at any time throughout the flow if it is initially satisfied.

\subsection{History Review}
The study of compressible vortex sheets has a long and rich history, originating from the seminal works of Miles \cite{Miles1957,Miles1958} and Fejer-Miles \cite{Fejer1963}. These early studies established that in the two-dimensional compressible Euler flow, vortex sheets exhibit violent instability when the Mach number $M<\sqrt{2}$, an effect analogous to the Kelvin--Helmholtz instability in incompressible fluids. Later, Artola--Majda \cite{Artola,Artola1987,Artola1989} investigated the interaction between vortex sheets and highly oscillatory waves, demonstrating that global-in-time nonlinear instability persists for $M > \sqrt{2}$, making global existence results challenging in the multidimensional settings.

A major breakthrough in the mathematical analysis of compressible vortex sheets came from Coulombel and Secchi \cite{Coulombel2004,CS08MR2423311}, who employed microlocal analysis and the Nash-Moser iteration technique to establish the local-in-time nonlinear stability of 2D compressible vortex sheets under small perturbations. Their results were restricted to the supersonic regime $M>\sqrt{2}$, relying on stability conditions akin to those used for shock waves by Majda \cite{Majda1983,Majda19832} and Coulombel \cite{Coulombel2002,Coulombel2004shocks}. Extensions of these results to non-isentropic Euler flows \cite{CoulombelMorando2004,Morando2008,MTW2018} further demonstrated how entropy variations influence stability. More recently, research has been carried out on steady three-dimensional compressible vortex sheets \cite{WangYJDE2013,WangYSIAM2015,WangY2015} and relativistic vortex sheets \cite{CSW19MR3925528}, providing insights into broader applications and mathematical structures.

In three-dimensional flows, the dynamics become substantially more intricate. Miles \cite{Miles1958} observed that disturbances propagating at large angles to the undisturbed flow amplify instability, and Serre \cite{Serre} later demonstrated through normal mode analysis that 3D compressible vortex sheets remain unstable for all Mach numbers, mirroring the Kelvin--Helmholtz instability in incompressible fluids. This suggests that additional physical mechanisms -- such as external forces, surface tension, or viscosity -- are required to stabilize the interface.

Magnetohydrodynamics (MHD) provides one such stabilizing effect. Chen--Wang \cite{ChenG2008,chen2012characteristic} and Trakhinin \cite{Trakhinin2009} independently proved that non-parallel magnetic fields stabilize compressible current-vortex sheets, a result that was also obtained for 2D MHD current-vortex sheets \cite{WangYARMA2013,Morando2023}. Thermoelasticity is proved to be able to stabilize contact discontinuities \cite{chen2020stability}. Another stabilizing mechanism arises from viscoelastic effects; numerical simulations and theoretical studies \cite{Azaiez,Huilgol1981} suggest that viscoelasticity counteracts vortex sheet instabilities. Huilgol \cite{Huilgol1981,Huilgol2015} studied vortex sheet formation in viscoelastic fluids, showing that unsteady shearing motions can induce vortex sheet structures, while Hu--Wang \cite{Hu2012} investigated singularity formation in viscoelastic flows. The stabilization by surface tension was confirmed in the work of Stevens \cite{stevens}, where the local existence and structural stability for 3D compressible Euler vortex sheets were obtained. For one-phase compressible elastic fluids, stability can be inferred from the classical Rayleigh--Taylor sign condition or non-collinearity of the elastic tensor; see Trakhinin \cite{Trakhinin2018}.

Significant effort has been devoted to understanding the influence of elasticity on vortex sheet stability. Linear and nonlinear stability for 2D compressible elastic vortex sheets has been rigorously established in \cite{RChen2017,RChen2018,RChen2020,RChen20212,RChen2023CAMC}. More recently, Chen--Huang--Wang--Yuan \cite{RChen2021} investigated the weakly linear stability of 3D compressible elastic vortex sheets, deriving necessary and sufficient conditions for stability through spectral analysis and {\it a priori} estimates. We would also like to mention the structural stability of 2D shock waves in compressible elastodynamics \cite{morando2020structural,trakhinin2022weak}.

\subsection{Classical Challenges and Resolutions in 2D}\label{subsec classical challenges}
One of the fundamental difficulties in analyzing vortex sheets stems from the characteristic nature of the free boundary, which limits control over the trace of characteristic components \cite{Coulombel2004,Lax,Majda1975,Gavage2007}. Specifically, the failure of the uniform Kreiss--Lopatinski$\breve{\mathrm{i}}$ (UKL) condition leads to a loss of tangential derivatives in estimating the solutions in terms of source terms in the linearized problem \cite{Coulombel2004}, making standard approaches insufficient for proving stability results. Additionally, the presence of elasticity complicates the root structure of the Lopatinski$\breve{\mathrm{i}}$ determinant, making it difficult to directly apply classical Kreiss symmetrization techniques.

A recent development in overcoming these difficulties was introduced in \cite{RChen2017} in the study of 2D rectilinear compressible elastic vortex sheets, where the authors proposed an upper triangularization technique to isolate outgoing modes from the system at all points in the frequency space. This method effectively separates the system into a closed form where the outgoing modes can be proved to be zero, simplifying the analysis to estimating only the incoming modes, which can be derived directly from the Lopatinski$\breve{\mathrm{i}}$ determinant. As a result, the linear stability was achieved. This approach was further extended by Chen--Huang--Wang--Yuan \cite{RChen2021} to analyze 3D linear stability for rectilinear compressible elastic fluids, providing a crucial step toward the nonlinear dynamics.

Linearizing around a non-constant background state introduces spatially dependent coefficients in the system, leading to additional complications in controlling the behavior of solutions. One convenient way to derive energy-type estimates is to use paradifferential calculus of Bony \cite{bony1981}. A particularly difficult issue in the paralinearization approach is that the Lopatinski$\breve{\mathrm{i}}$ determinant may vanish at certain frequencies (called {\it roots}). Coulombel \cite{Coulombel2002,Coulombel2004shocks} and Coulombel--Secchi \cite{Coulombel2004} developed a bicharacteristic extension method to construct weight functions that mitigate this degeneracy. However, this method relies on the assumption that the leading order symbol matrix for the paralinearized system of the non-characteristic form remains diagonalizable along bicharacteristic curves. This assumption fails when the roots coincide with the points where the system cannot be reduced to a non-characteristic form -- these points are referred to as {\it poles}. This breakdown has been a significant obstacle in applying classical energy methods.

To overcome this, a new approach was designed in \cite{RChen2018}, based on a refined upper triangularization of the para-linearized system. Instead of relying on bicharacteristic curves, this method constructs weight functions that depend solely on the background state variables. This avoids discrepancies between bicharacteristic extensions and pole distributions, allowing for a more robust stability analysis. A key advantage of this approach is that it provides a framework to ensure improved regularity of the outgoing mode even in the variable-coefficient case, which plays a crucial role in compensating the loss of higher regularity for the characteristic components near the poles. This is particularly important when extending stability results to nonlinear settings, where controlling derivative loss is essential for proving local well-posedness \cite{RChen2020}.

\subsection{New Challenges from Dimension Increase and Resolutions in 3D}
In addition to the above difficulties, the transition from two-dimensional to fully three-dimensional vortex sheets introduces new analytical and structural challenges. The additional spatial dimension not only increases the degrees of freedom in frequency space but also creates more potential instability directions, making frequency interactions significantly more intricate. Unlike in 2D, where instabilities are constrained to a plane, 3D vortex sheets exhibit a broader range of possible resonance mechanisms. However, in the case of elastic vortex sheets, the additional elasticity components in 3D play a stabilizing role by restricting the growth of unstable perturbations, provided that certain geometric conditions on the deformation gradient are satisfied.

\subsubsection{Enhanced ellipticity from elasticity}
As is mentioned in Section \ref{subsec classical challenges}, the energy estimates suffer a loss of tangential derivative due to the failure of the UKL condition. Since the wave front appears only in the boundary conditions of the vortex sheet system, one key strategy is to ensure that instabilities can only arise from the traces of solutions to the interior dynamical system rather than from the front symbol. In other words, the boundary conditions for the front need to satisfy certain {\it ellipticity} condition.

For 2D compressible Euler vortex sheets, this ellipticity is achieved because the front symbol is homogeneous and does not vanish on the closed hemisphere in the frequency space; see \cite[Lemma 4.1]{Coulombel2004}. This property allows for the recovery of one derivative in the regularity estimates of the vortex sheet front. Furthermore, it enables the elimination of the front from the system, reducing the problem to a standard boundary value problem with a symbolic boundary condition, similar to the case of shock waves \cite{Majda1983}. The introduction of elasticity preserves the essential algebraic structure, ensuring that the ellipticity condition and subsequent reduction of the system remain valid \cite{RChen2017}.

However, increasing the spatial dimension introduces additional tangential components, significantly complicating frequency interactions and resonance effects. In particular, for the 3D compressible Euler vortex sheets, the front symbol exhibits {\it degenerate ellipticity} along certain frequency directions, which makes it more difficult to control the regularity of the front using standard energy estimates. A key discovery in our analysis is that elasticity provides an {\it ellipticity enhancement} mechanism that counteracts this degeneracy. More precisely, we show that if the deformation gradient on the free boundary satisfies the geometric condition
\begin{equation}\label{nonpara}
\mathrm{F}_1\times \mathrm{F}_2\neq\mathbf{0}, \qquad \text{or equivalently,} \qquad  \mathrm{F}_1\nparallel\mathrm{F}_2,
\end{equation}
where $\mathrm{F}_1$ and $\mathrm{F}_2$ denote the first two rows of the deformation gradient (see \eqref{row vecs}), then boundary ellipticity is restored, allowing for the recovery of one  derivative for the free interface regularity; see \eqref{bfront}. Another fundamental difference between the 2D and 3D elastic vortex sheet problems lies in the elimination of the front from the boundary conditions. In two dimensions, there exists a natural linear transformation  that removes the front-related terms from the boundary conditions, leaving the remaining system non-singular and ensuring that the normal component of the unknown function can be controlled by its non-characteristic part. In contrast, in 3D, no obvious transformation structure is available due to the increased complexity of frequency interactions. Nevertheless, by exploiting the non-parallel structural property of elasticity in \eqref{nonpara}, we construct a suitable matrix that facilitates the elimination of the front, thus preserving the essential structure needed for stability estimates; see \eqref{PI}.

\subsubsection{Resolving higher-order singularities in the Lopatinski$\breve{\mathrm{i}}$ condition}\label{subsec deg}
As is explained in Section \ref{subsec classical challenges}, the presence of roots of the Lopatinski$\breve{\mathrm{i}}$ determinant leads to a loss of derivatives in the energy estimates, while at each pole, the para-linearized system cannot be reduced to a non-characteristic form. In the case of 2D elastic vortex sheets, the method developed in \cite{RChen2018} effectively treats the scenario where a simple root coincides with a simple pole.

In 3D, the situation is considerably more complex. Specifically, there exists a co-dimension one set in frequency space where a double pole, arising from the left and right states of the two-phase system, collides with a double root. Since the system cannot be directly transformed into a non-characteristic form, we employ the refined upper triangularization technique of \cite{RChen2018} to separate the outgoing modes into a closed form, allowing us to derive improved regularity estimates. These estimates are then used to analyze the coupling between the characteristic part and the outgoing mode, ultimately enabling control of the characteristic components. The final step is to estimate the (boundary trace of) incoming modes in terms of the outgoing modes and source terms.

A crucial factor in determining whether the boundary estimates can be closed is the behavior of the restriction of the boundary symbol $\beta$ to the stable subspace $\text{span}\{ E^r, E^l \}$ of the linearized system, which corresponds to the Lopatinski$\breve{\mathrm{i}}$ matrix $L := \beta (E^r \ E^l)$. As noted earlier, $L$ is not invertible, and direct computation reveals that $L$ has a one-dimensional kernel at the roots. This singularity leads to the failure of UKL, resulting in derivative loss. The important work of Coulombel \cite{Coulombel2004shocks} on weakly stable shock waves developed a framework to handle cases where the boundary symbol $\beta$ vanishes at first order  at the roots. This technique has been successfully applied to the study of 2D compressible Euler vortex sheets \cite{Coulombel2004}, 2D compressible elastic vortex sheets \cite{RChen2018,RChen2020}, and 2D relativistic vortex sheets \cite{CSW19MR3925528}. However, for 3D elastic vortex sheets, a double root may appear, leading to a higher-order degeneracy where $\beta$ vanishes at second order at the double root. To resolve this issue, we extend Coulombel's argument by constructing two invertible mappings $P_1$ and $P_2$ near the double root, which are symbols of type $\Gamma^0_2$ (degree 0 and regularity 2; see Definition \ref{definition para}). Under these transformations, the Lopatinski$\breve{\mathrm{i}}$ matrix $L$ is transformed to
\[
\beta_{\rm{in}}:=P_1LP_2=\left[\begin{matrix}
1 & 0\\
0 & \Lambda^{-2}(\gamma+i\sigma_1)(\gamma+i\sigma_2)\end{matrix}\right],
\]
where $\Lambda, \sigma_1, \sigma_2 \in \Gamma^1_2$ are real-valued scalar symbols. The second-order vanishing of $\beta$ is captured by the symbol $(\gamma+i\sigma_1)(\gamma+i\sigma_2)$. We show that this construction ensures $\beta_{\rm{in}} \in \Gamma^0_2$, which brings the problem into the framework of \cite{Coulombel2004shocks}, allowing us to use localized G${\rm\mathring a}$rding's inequality to compensate for the loss of derivatives; see Section \ref{subsec case1}. We would like to comment that this diagonalization of $L$ into $\beta_{\rm{in}}$ and the associated reduction in the degree of the double roots of the Lopatinski$\breve{\mathrm{i}}$ determinant are expected to be useful for other free-boundary models with similar algebraic structures.

\medskip

The rest of the paper is organized as follows. In Section \ref{vortex}, we formulate 3D nonlinear problem of vortex sheets, fix the free boundary, linearize the system around a given constant solution, introduce the function spaces and useful lemmas, and state our main result, Theorem \ref{thm}.   In Section \ref{variablecoe}, we introduce the effective linear problem and its formulation with variable coefficients. In Section \ref{sec.well-posed}, we prove a well-posedness result of the effective linear problem in the usual Sobolev space $H^m$ with $m$ large enough. In Section \ref{sec.compa}, we transform the original
nonlinear problem into the case with zero initial data. We construct approximate solutions to incorporate the initial data into the
interior equations. The necessary compatibility conditions are imposed on the initial data
for the construction of smooth approximate solutions. Finally, we show the existence and stability results of solutions to the reduced problem and conclude the main result in Section \ref{sec.Nash} by using Nash-Moser iteration.
\bigskip

\section{Formulation, Notations and Main Result}\label{vortex}

In this section, we will derive the governing dynamics of vortex sheets from the elastic equation \eqref{elastic2}, linearize them around a planar vortex sheet, and state our main result.

\subsection{Statement for the Vortex Sheet  Problem}\label{govern}

Recall the definition of vortex sheet solutions for \eqref{elastic2}. Let
$U(t,x_1,x_2,x_3)=(\rho,\textbf{v},\mathbf{F})(t,x_1,x_2,x_3)$ be a solution that is  piecewise  smooth  on both sides  of a smooth hypersurface $$\Gamma=\{x_3=\varphi(t,x_1,x_2)\}.$$ Denote $\partial_i=\partial_{x_i},i=1,2,3,$ for the partial derivatives, normal vector $\nu=(-\partial_1\varphi,-\partial_2\varphi,1)$ on $\Gamma$ and
\begin{eqnarray}\nonumber
U(t,x_1,x_2,x_3)=
\begin{cases}
U^+(t,x_1,x_2,x_3)& \text{ if } x_3>\varphi(t,x_1,x_2),\\
U^-(t,x_1,x_2,x_3)& \text{ if } x_3<\varphi(t,x_1,x_2),\\
\end{cases}
\end{eqnarray}
where $U^{\pm}=(\rho^{\pm},\textbf{v}^{\pm},\mathbf{F}^{\pm})(t,x_1,x_2,x_3).$  The solution $U$ satisfies the Rankine-Hugoniot jump relations at each point on $\Gamma:$
\begin{equation}\label{RH}
\begin{split}
&\partial_t\varphi[\rho]-[\rho\textbf{v}\cdot\nu]=0,\\
&\partial_t\varphi[\rho\textbf{v}]-[(\rho\textbf{v}\cdot\nu)\textbf{v}]-[p]\nu+[\rho\mathbf{F}\mathbf{F}^{\top}\nu]=\mathbf{0},\\
&\partial_t\varphi[\rho\mathcal{F}_j]-[(\textbf{v}\cdot\nu)\rho\mathcal{F}_j]+[(\rho\mathcal{F}_j\cdot \nu)\textbf{v}]=\mathbf{0},\\
\end{split}
\end{equation}
where we write $[f]$ as the jump of the quantity $f$ across the hypersurface $\Gamma.$ For a vortex sheet (contact discontinuity), we require
\begin{equation}\label{cd}
[\textbf{v}\cdot\nu]=0, \; [\textbf{v}]\neq\mathbf{0}, \text{ and } \varphi_t=\textbf{v}^{\pm}\cdot \nu\Big|_{\Gamma}.
\end{equation}
Therefore the jump conditions reduce to
\begin{equation}\label{RH3}
\rho^+=\rho^-, \quad \varphi_t=\textbf{v}^+\cdot\nu=\textbf{v}^-\cdot\nu, \quad \mathcal{F}^{+}_j\cdot\nu=\mathcal{F}^{-}_j\cdot\nu.
\end{equation}

It is crucial that in the derivation of the Rankine-Hugoniot condition, we need to regard $$\rho^{\pm} \mathcal{F}^{\pm}_{j}\cdot\nu=0$$ as an {\it intrinsic property}.
Therefore, we also have
\begin{align}\label{F}
\mathcal{F}^{\pm}_{j}\cdot\nu=0, \text{ for } j=1,2,3\quad \text{on } \Gamma(t).
  \end{align}

To flatten and fix the free boundary $\Gamma,$ we need to introduce the function $\varPhi(t,x_1,x_2,x_3)$ to set the variable transformation $\varPhi^{\pm}(t,x_1,x_2,x_3)$ as follows. We first consider the class of functions $\varPhi(t,x_1,x_2,x_3)$ such that $\inf\{\partial_3\varPhi\}>0.$ Then we define
\begin{align}\label{transform}
U^{\pm}_{\sharp}=(\rho^{\pm}_{\sharp},\textbf{v}^{\pm}_{\sharp},\mathbf{F}^{\pm}_{\sharp})(t,x_1,x_2,x_3) :=(\rho,\textbf{v},\mathbf{F})(t,x_1,x_2,\varPhi(t,x_1,x_2,\pm x_3)),
\end{align}
for $x_3\geq0.$  In the following argument, we drop the index $\sharp$ for notation simplicity. Define $\varPhi^{\pm}(t,x_1,x_2,x_3) := \varPhi(t,x_1,x_2,\pm x_3).$
Inspired by \cite{Coulombel2004,Francheteau2000}, it is natural to require $\varPhi^{\pm}$ satisfying the eikonal equation
\begin{equation} \label{Phi.eq}
\begin{split}
 \partial_t\varPhi^{\pm}+v_1^{\pm}\partial_1\varPhi^{\pm}+v_2^{\pm}\partial_2\varPhi^{\pm}-v_3^{\pm}=0, \quad
 \pm\partial_3\varPhi^{\pm}\geq \kappa>0,
\end{split}
\end{equation}
when $x_3\geq 0$, and
\begin{equation}\label{Phi.eq.c}
\varPhi^{+}=\varPhi^{-}=\varphi, \quad \mathrm{if}\ x_3= 0,
\end{equation}
{for some constant $\kappa>0$.}

Through this variable transformation, equations \eqref{elastic2} become
\begin{equation}\label{U}
\begin{split}
&\partial_tU^{\pm}+A_1(U^{\pm})\partial_1U^{\pm}+A_2(U^{\pm})\partial_2U^{\pm}\\
&\quad+\frac{1}{\partial_3\varPhi^{\pm}}\big(A_3(U^{\pm})-\partial_t\varPhi^{\pm}I-\partial_1\varPhi^{\pm}A_1(U^{\pm})-\partial_2\varPhi^{\pm}A_2(U^{\pm})\big)\partial_3U^{\pm}=\mathbf{0},
\end{split}
\end{equation}
for $x_3>0$ with free boundary $x_3=0,$ where
\begin{equation}\label{A}
\begin{split}
 &A_1(U) :=\begin{bmatrix}\begin{smallmatrix}        v_1 & \rho & 0 & 0 & 0 & 0 & 0 & 0 & 0 & 0 & 0 & 0 & 0\\
    \frac{p'}{\rho} & v_1 & 0 & 0 & -F_{11} & 0 & 0 & -F_{12} & 0 & 0 & -F_{13} & 0 & 0\\
    0 & 0 & v_1 & 0 & 0 & -F_{11} & 0 & 0 & -F_{12} & 0 & 0 & -F_{13} & 0\\
    0 & 0 & 0 & v_1 & 0 & 0 & -F_{11} & 0 & 0 & -F_{12} & 0 & 0 & -F_{13}\\
    0 & -F_{11} & 0 & 0 & v_1 & 0 & 0 & 0 & 0 & 0 & 0 & 0 & 0\\
    0 & 0 & -F_{11} & 0 & 0 & v_1 & 0 & 0 & 0 & 0 & 0 & 0 & 0\\
    0 & 0 & 0 & -F_{11} & 0 & 0 & v_1 & 0 & 0 & 0 & 0 & 0 & 0\\
    0 & -F_{12} & 0 & 0 & 0 & 0 & 0 & v_1 & 0 & 0 & 0 & 0 & 0\\
    0 & 0 & -F_{12} & 0 & 0 & 0 & 0 & 0 & v_1 & 0 & 0 & 0 & 0\\
    0 & 0 & 0 &-F_{12} & 0 & 0 & 0 & 0 & 0 & v_1 & 0 & 0 & 0\\
    0 & -F_{13} & 0 & 0 & 0 & 0 & 0 & 0 & 0 & 0 & v_1 & 0 & 0\\
    0 & 0 & -F_{13} & 0 & 0 & 0 & 0 & 0 & 0 & 0 & 0 & v_1 & 0\\
    0 & 0 & 0 &-F_{13} & 0 & 0 & 0 & 0 & 0 & 0 & 0 & 0 & v_1\\
\end{smallmatrix}
\end{bmatrix}
,\\[5pt]
 &A_2(U) :=\begin{bmatrix}\begin{smallmatrix}
    v_2 & 0 & \rho & 0 & 0 & 0 & 0 & 0 & 0 & 0 & 0 & 0 & 0\\
    0 & v_2 & 0 & 0 & -F_{21} & 0 & 0 & -F_{22} & 0 & 0 & -F_{23} & 0 & 0\\
    \frac{p'}{\rho} & 0 & v_2 & 0 & 0 & -F_{21} & 0 & 0 & -F_{22} & 0 & 0 & -F_{23} & 0\\
    0 & 0 & 0 & v_2 & 0 & 0 & -F_{21} & 0 & 0 & -F_{22} & 0 & 0 & -F_{23}\\
    0 & -F_{21} & 0 & 0 & v_2 & 0 & 0 & 0 & 0 & 0 & 0 & 0 & 0\\
    0 & 0 & -F_{21} & 0 & 0 & v_2 & 0 & 0 & 0 & 0 & 0 & 0 & 0\\
    0 & 0 & 0 & -F_{21} & 0 & 0 & v_2 & 0 & 0 & 0 & 0 & 0 & 0\\
    0 & -F_{22} & 0 & 0 & 0 & 0 & 0 & v_2 & 0 & 0 & 0 & 0 & 0\\
    0 & 0 & -F_{22} & 0 & 0 & 0 & 0 & 0 & v_2 & 0 & 0 & 0 & 0\\
    0 & 0 & 0 &-F_{22} & 0 & 0 & 0 & 0 & 0 & v_2 & 0 & 0 & 0\\
    0 & -F_{23} & 0 & 0 & 0 & 0 & 0 & 0 & 0 & 0 & v_2 & 0 & 0\\
    0 & 0 & -F_{23} & 0 & 0 & 0 & 0 & 0 & 0 & 0 & 0 & v_2 & 0\\
    0 & 0 & 0 &-F_{23} & 0 & 0 & 0 & 0 & 0 & 0 & 0 & 0 & v_2\\
\end{smallmatrix}
\end{bmatrix}
,\\[5pt]
\text{ and }
&A_3(U):=\begin{bmatrix}\begin{smallmatrix}
    v_3 & 0 & 0 & \rho & 0 & 0 & 0 & 0 & 0 & 0 & 0 & 0 & 0\\
    0 & v_3 & 0 & 0 & -F_{31} & 0 & 0 & -F_{32} & 0 & 0 & -F_{33} & 0 & 0\\
    0 & 0 & v_3 & 0 & 0 & -F_{31} & 0 & 0 & -F_{32} & 0 & 0 & -F_{33} & 0\\
    \frac{p'}{\rho} & 0 & 0 & v_3 & 0 & 0 & -F_{31} & 0 & 0 & -F_{32} & 0 & 0 & -F_{33}\\
    0 & -F_{31} & 0 & 0 & v_3 & 0 & 0 & 0 & 0 & 0 & 0 & 0 & 0\\
    0 & 0 & -F_{31} & 0 & 0 & v_3 & 0 & 0 & 0 & 0 & 0 & 0 & 0\\
    0 & 0 & 0 & -F_{31} & 0 & 0 & v_3 & 0 & 0 & 0 & 0 & 0 & 0\\
    0 & -F_{32} & 0 & 0 & 0 & 0 & 0 & v_3 & 0 & 0 & 0 & 0 & 0\\
    0 & 0 & -F_{32} & 0 & 0 & 0 & 0 & 0 & v_3 & 0 & 0 & 0 & 0\\
    0 & 0 & 0 &-F_{32} & 0 & 0 & 0 & 0 & 0 & v_3 & 0 & 0 & 0\\
    0 & -F_{33} & 0 & 0 & 0 & 0 & 0 & 0 & 0 & 0 & v_3 & 0 & 0\\
    0 & 0 & -F_{33} & 0 & 0 & 0 & 0 & 0 & 0 & 0 & 0 & v_3 & 0\\
    0 & 0 & 0 &-F_{33} & 0 & 0 & 0 & 0 & 0 & 0 & 0 & 0 & v_3\\
\end{smallmatrix}
\end{bmatrix}
.\\
\end{split}
\end{equation}
This choice simplifies the expression of the nonlinear problem in the fixed domain and guarantees the constant rank property of boundary matrix in the whole domain.

It is obvious that the system of conservation laws \eqref{elastic2} admits trivial vortex sheet solutions consisting of two constant states separated by a planar front as follows:
\begin{align}\label{background.E}
U(t,x_1,x_2,x_3)=\begin{cases}
(\bar{\rho},\bar{v},0,0,\bar{F}^+_{11},\bar{F}^{+}_{21},0,\bar{F}^+_{12},\bar{F}^{+}_{22},0,\bar{F}^+_{13},\bar{F}^{+}_{23},0) &\text{ if } x_3>0,\\
(\bar{\rho},-\bar{v},0,0,\bar{F}^-_{11},\bar{F}^{-}_{21},0,\bar{F}^-_{12},\bar{F}^{-}_{22},0,\bar{F}^-_{13},\bar{F}^{-}_{23},0) &\text{ if } x_3<0.\\
\end{cases}
\end{align}
Every planar elastic vortex sheet {(namely piecewise-constant vortex sheet)} is of this form up to a  Galilean transformation.
For simplicity we assume that
$\bar{F}_{ij}^{+}=-\bar{F}_{ij}^{-}=\bar{F}_{ij},$
for $i\in \{1,2\}, \ j\in \{1,2,3\}.$

Then we need to solve the following
initial-boundary value problem for $U^{\pm}_{\sharp}$ in a fixed domain:
\begin{subequations}   \label{EVS}
\begin{alignat}{2}
\label{EVS.a}
&\mathbb{L}(U^{\pm},\varPhi^{\pm}) :=L(U^{\pm},\varPhi^{\pm})U^{\pm} =\mathbf{0}, \quad x_3>0,\\
\label{EVS.b}
&\mathbb{B}(U^{+},U^{-},\varphi)|_{x_3=0}=\mathbf{0}, \\
\label{EVS.c}
&(U^{+},U^{-},\varphi)|_{t=0}=(U^{+}_0,U^{-}_0,\varphi_0),
\end{alignat}
\end{subequations}
where we have dropped the index $``\sharp"$ for convenience,
  $L(U,\varPhi)$ and $\mathbb{B}$ are given by
\begin{align}
\label{L.def}
&L(U,\varPhi):=I \partial_t+A_1(U)\partial_1+A_2(U)\partial_2+\widetilde{A}_3(U,\varPhi)\partial_3,
\\ \label{B.def}
&\mathbb{B}(U^+,U^-,\varphi):=
\begin{bmatrix}
[v_1] \partial_1\varphi+[v_2] \partial_2\varphi-[v_3]\\[0.5mm]
\partial_t\varphi+v_1^+|_{x_3=0}\partial_1\varphi+v_2^+|_{x_3=0}\partial_2\varphi-v_3^+|_{x_3=0}\\[0.5mm]
[\rho]
\end{bmatrix},
\end{align}
with
\begin{align}  \notag
\widetilde{A}_3(U, \varPhi):=
\frac{1}{\partial_3\varPhi}\big(A_3(U)-\partial_t\varPhi I-\partial_1\varPhi A_1(U)-\partial_2\varPhi A_2(U)\big).
\end{align}

By \eqref{F} and \eqref{Phi.eq},  we obtain that
the boundary matrix
of problem \eqref{EVS}, {\it i.e.},
$$\mathrm{diag}\,\big(\!-\widetilde{A}_3(U^+, \varPhi^+),\,-\widetilde{A}_3(U^-, \varPhi^-)\big),$$
has constant rank on $\{x_3\geq 0\}$ if and only if
\begin{align} \label{inv5}
F_{3j}^{\pm}=F_{1j}^{\pm} \p_1\varPhi^{\pm}+F_{2j}^{\pm} \p_2\varPhi^{\pm}
\quad \textrm{for }\   j=1,2,3
\quad \textrm{if }\   x_3\geq 0.
\end{align}
In the new variables, \eqref{intrinsic} becomes
\begin{alignat}{4}
\label{inv2}&\partial_{\ell}^{\varPhi^{\pm}} (\rho^{\pm}{F}_{\ell  j}^{\pm})=0
\quad &&\textrm{for }\  j=1,2,3 & \quad&\textrm{if }\  x_3> 0,
\end{alignat}
where
we denote the partial derivatives with respect to the lifting function $\varPhi$ by
\begin{align} \label{differential}
\partial_t^{\varPhi}:=\partial_t-\frac{\partial_t\varPhi}{\partial_3\varPhi}\p_3,\quad
\partial_1^{\varPhi}:=\partial_1-\frac{\partial_1\varPhi}{\partial_3\varPhi}\partial_3,\quad
\partial_2^{\varPhi}:=\partial_2-\frac{\partial_2\varPhi}{\partial_3\varPhi}\partial_3,\quad
\partial_3^{\varPhi}:=\frac{1}{\partial_3\varPhi}\partial_3.
\end{align}
The following proposition shows
that identities \eqref{inv5}--\eqref{inv2} are involutions
for vortex sheet problem \eqref{Phi.eq}--\eqref{EVS}.
The proof follows from a straightforward computation and hence is omitted.

\begin{proposition}
\label{pro1.1}
For every sufficiently smooth solution of problem \eqref{Phi.eq}--\eqref{EVS}
on time interval $[0,T]$,
constraints \eqref{inv5}--\eqref{inv2} hold for all $t\in[0,T],$ if
they are satisfied initially.
\end{proposition}

\subsection{Main Result and Discussion}

In the straightened variables, the piecewise constant vortex sheet \eqref{background.E} corresponds to the stationary solution of \eqref{EVS.a}-\eqref{EVS.c} and \eqref{inv5},\eqref{inv2} as follows:
\begin{align}    \label{background}
\bar{U}^{\pm}:=
\big(\bar{\rho},\,\pm\bar{v},\,0,\,0,\,\pm\bar{F}_{11},\pm\bar{F}_{21},\,0,\,\pm\bar{F}_{12},\pm\bar{F}_{22},\,0,\pm\bar{F}_{13},\pm\bar{F}_{23},\,0\big)^{\top},
\quad  \bar{\varphi}:=0,
\,  \bar{\varPhi}^{\pm}:=\pm x_3.
\end{align}

For $j = 1, 2, 3,$ we denote
\begin{equation}\label{row vecs}
\mathrm{F}_{j} := \text{ the $j$th row of the deformation matrix } \mathbf{F}.
\end{equation}
Note that this differs from the column vectors $\mathcal{F}_j$ in \eqref{elastic2}.
From \eqref{background} we know that $\mathrm{F_3} = \mathbf{0}$. We further define the vector projections (see Fig. \ref{fig vectors})
\begin{equation}\label{proj}
\begin{split}
\Pi_a(b) & := \textup{ the parallel projection of $b$ onto $a$}, \\
\pe_a(b) & := b - \Pi_a(b) = \textup{ the perpendicular projection of $b$ onto $a$}.
\end{split}
\end{equation}
\begin{figure}[h!]
  \centering
  \vspace{-1ex}
  \includegraphics[scale=0.4]{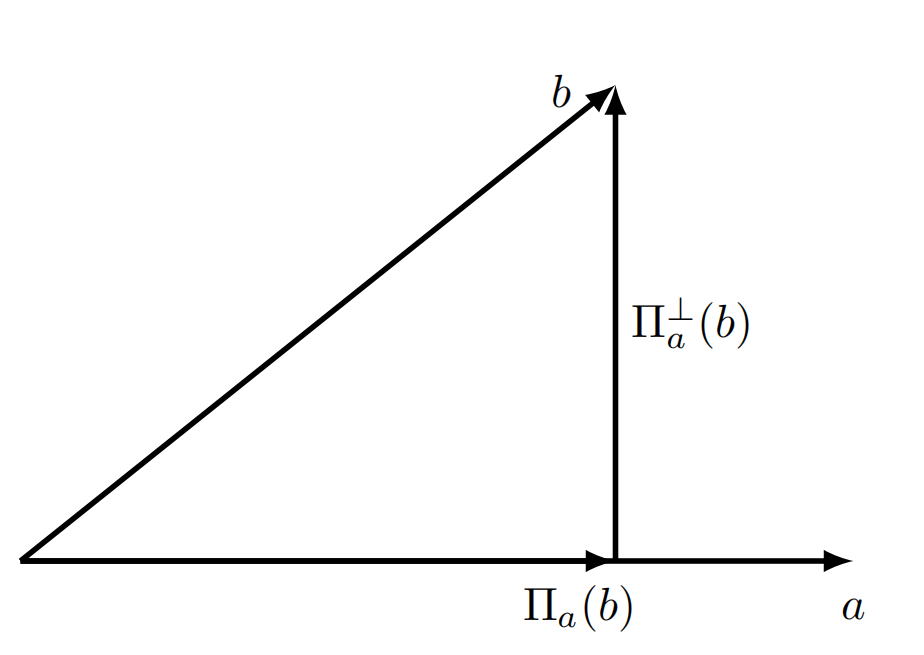}
  \vspace{-1ex}
  \caption{Vector projections}
  \label{fig vectors}
\end{figure}

In order to prove the nonlinear stability of elastic vortex sheets,
we only need to show the existence of solutions to
problem \eqref{Phi.eq}--\eqref{EVS}
on account of transform \eqref{transform}.
The main result of this paper is stated as follows:

\begin{theorem}
 \label{thm}
Let $T>0$ and $s_0\geq 14$ be an integer. Suppose that the background state \eqref{background} satisfies $\mathrm{F}_1\times\mathrm{F}_2\neq\mathbf{0},$ and the following stability conditions{\rm:}
\begin{align}
\label{H1}
\bar{v}^2<\frac{|\pe_{\mathrm{F}_2}(\mathrm{F}_1)|^2}{4},
\end{align}
and
\begin{align}\label{stability4}
\begin{aligned}
&\bar{v}^2<\mathrm{G}(\mathrm{F}_1,\mathrm{F}_2),
\end{aligned}
\end{align}
where $\mathrm{G}(\mathrm{F}_1,\mathrm{F}_2)$ is defined in \eqref{G}.
Suppose further that the initial data $U^{\pm}_0$ and $\varphi_0$ satisfy
constraints \eqref{inv5}--\eqref{inv2} and
the compatibility conditions up to order $s_0$ {\rm(}{\it cf.}\;{\rm Definition \ref{def.compa}}{\rm)},
and that $(U^{\pm}_0-\bar{U}^{\pm},\varphi_0)
\in H^{s_0+1/2}(\mathbb{R}^3_+)\times H^{s_0+1}(\mathbb{R}^2)$
has a compact support.
Then there exists a positive constant $\epsilon$ such that,
if
$$\|U^{\pm}_0-\bar{U}^{\pm}\|_{H^{s_0+1/2}(\mathbb{R}^3_+)}
+\|\varphi_0\|_{H^{s_0+1}(\mathbb{R}^2)}\leq \epsilon,$$
then problem \eqref{Phi.eq}--\eqref{EVS}
admits a solution $(U^{\pm},\varPhi^{\pm},\varphi)$ on the time interval $[0,T]$
satisfying
$$(U^{\pm}-\bar{U}^{\pm},\varPhi^{\pm}-\bar{\varPhi}^{\pm})
\in H^{s_0-8}((0,T)\times\mathbb{R}^3_+), \quad
\varphi\in H^{s_0-7}((0,T)\times\mathbb{R}^2).
$$
\end{theorem}
\begin{remark}
The stability conditions \eqref{H1}--\eqref{stability4} are more restrictive than the full stability domain obtained in \cite{RChen2021} for the constant-coefficient linearized problem, where a necessary and sufficient condition for linear stability was established. In the present work, we derive a \emph{sufficient} condition ensuring \emph{nonlinear stability}. We do not expect these conditions to be necessary; extending the nonlinear stability to the full linear stability domain remains beyond the scope of the current paper.

The additional restrictions arise from the need to perform a delicate microlocal analysis within the para-differential framework. More precisely, the Nash--Moser iteration relies on a refined $L^2$ estimate for the effective variable-coefficient linearized system with only \emph{one} tangential derivative loss. To obtain such an estimate, it is necessary to exclude portions of the parameter domain where the Lopatinski$\breve{\mathrm{i}}$ determinant exhibits stronger degeneracies caused by root-pole collisions (see the discussion in Section \ref{subsec deg}). In particular, condition \eqref{H1} excludes certain collisions between roots and poles of the Lopatinski$\breve{\mathrm{i}}$ determinant, while condition \eqref{stability4} prevents the coincidence of roots of eigenvalues with poles or roots of the Lopatinski$\breve{\mathrm{i}}$ determinant (see more details at the end of Section \ref{microlocalization}). Such separations of critical frequencies are essential in the variable-coefficient setting, whereas they are not required in the constant-coefficient case, where Fourier analysis can be carried out pointwise in frequency space.
\end{remark}

\begin{remark}
The solution constructed in Theorem \ref{thm} is unique among sufficiently smooth solutions in a small neighborhood of the background state \eqref{background}. Indeed, if two such solutions share the same initial data, their difference satisfies a linearized system with source terms that are quadratic in the difference. Using the same $L^2$ energy estimate as in the Nash--Moser iteration, one can close the estimates for the difference with zero initial data, which implies that the two solutions coincide. We refer to \cite{CoulombelSecchi2009} for a similar uniqueness argument in related free boundary problems.
\end{remark}

\subsection{Functional Spaces}\label{notation}

Now we introduce some necessary functional spaces, $\it{i.e.}$, weighted Sobolev spaces in preparation for our main theorem. Let $\mathcal{D}'$ denote the distributions and define

$$H^{s}_{\gamma}(\R^3):=\{u(t,x_1,x_2)\in\mathcal{D}'(\R^3):e^{-\gamma t}u(t,x_1,x_2)\in H^s(\R^3)\},$$
$$H^{s}_{\gamma}(\R^4_+):=\{v(t,x_1,x_2,x_3)\in\mathcal{D}'(\R^4_+):e^{-\gamma t}v(t,x_1,x_2,x_3)\in H^s(\R^4_+)\},$$
for $s\in \R,\gamma\geq1.$

These spaces are endowed with the parameter-dependent norms
$$\| u \|_{H^s_{\gamma}(\R^3)} := \| e^{-\gamma t}u \|_{s,\gamma},$$
and, for integer $s$,
$$\| v \|_{H^s_{\gamma}(\R^4_+)} := \sum_{|\alpha|\le s}\gamma^{s-|\alpha|}
\|e^{-\gamma t} \partial^{\alpha}v\|_{L^2(\R^4_+)},$$
respectively, where
$$\R^4_+ :=\{(t,x_1,x_2,x_3)\in\R^4:x_3>0\}.$$
We define the norm
$$
\| u \|^2_{s,\gamma} :=\frac{1}{(2\pi)^3}\int_{\R^3}(\gamma^2+|\xi|^2)^s|\widehat{u}(\xi)|^2\,\dd \xi, \; \text{ for any } u\in H^s(\R^3),
$$
with $\widehat{u}(\xi)$ being the Fourier transform of $u$ with respect to $(t,x_1,x_2).$

Setting $\tilde{u}=e^{-\gamma t}u,$ we have by definition
$\| u \|_{H^s_{\gamma}(\R^3)}= \| \tilde{u} \|_{s,\gamma}$. Now, we can define the space $L^2(\R_+; H^s_{\gamma}(\R^3)),$ endowed with the norm
$$
\vertiii{v}^2_{L^2(H^s_{\gamma})} :=\int^{+\infty}_0 \| v(\cdot,x_3) \|^2_{H^s_{\gamma}(\R^3)}\,\dd x_3.
$$
We also have
$$
\vertiii{v}^2_{L^2(H^s_{\gamma})}\simeq \vertiii{\tilde{v}}^2_{s,\gamma} :=\int^{+\infty}_0 \| \tilde{v}(\cdot,x_3) \|^2_{s,\gamma}\,\dd x_3.
$$
It is easy to see that when $s=0,$ $\| \cdot \|_0 := \| \cdot \|_{0,\gamma} = \| \cdot \|_{L^2(\R^3)}$ and $\vertiii{\cdot}_{0,\gamma}$ ($\vertiii{\cdot}_{0}$ for simplicity) is the usual norm of $L^2(\R^4_+).$
We denote $\nabla:=(\p_t,\p_1,\p_2)$ when applying it to functions of $(t,x_1,x_2).$
	For multi-index $\alpha=(\alpha_0,\alpha_1,\alpha_2,\alpha_3)\in \mathbb{N}^4$, we define
	$
	\p^{\alpha}:=\p_t^{\alpha_0}\p_1^{\alpha_1}\p_2^{\alpha_2}\p_3^{\alpha_3}
	\quad  \textrm{and }\quad
	|\alpha|:=\alpha_0+\alpha_1+\alpha_2+\alpha_3.
	$
	For $m\in\mathbb{N}$, we denote $\nabla^{m} :=\{\p^{\alpha}: |\alpha|= m\}$.

For $T>0$, we set
\[
\Omega_T:=(0,T)\times\R^3_+,\qquad \omega_T:=(0,T)\times\R^2,
\]
where $\R^3_+:=\{(x_1,x_2,x_3)\in\R^3:\ x_3>0\}$. Moreover,
$$
H^m_{\gamma}(\Omega_T):=
\big\{u\in\mathcal{D}'(\Omega_T)\,:\, {e}^{-\gamma t}u\in H^{m}(\Omega_T)\big\}
$$
 {is introduced with} the norm
\begin{align}  \notag  
\|u\|_{H^m_{\gamma}(\Omega_T)}:=\sum_{|\alpha|\leq m}\gamma^{m-|\alpha|}\|{e}^{-\gamma t} \partial^{\alpha}u\|_{L^2(\Omega_T)}.
\end{align}
{Similarly,} the space $H^m_{\gamma}(\omega_T)$ and its norm are defined.

{Furthermore, we abbreviate} 
$L^2(\mathbb{R}_+;H^m_{\gamma}(\omega_T))$
{to} $L^2(H^m_{\gamma}(\omega_T))$,
which is equipped with the norm
\begin{align*}
\vertiii{u}_{L^2(H^m_{\gamma}(\omega_T))}:=
\sum_{\alpha_0+\alpha_1+\alpha_2\leq m}
\gamma^{m-\alpha_0-\alpha_1-\alpha_2}\|{e}^{-\gamma t} \partial_t^{\alpha_0}\partial_1^{\alpha_1}\partial_2^{\alpha_2}u\|_{L^2(\Omega_T)} .
\end{align*}
{and} $L^2_{\gamma}(\Omega_T):=L^2(H^0_{\gamma}(\omega_T)),$
$\|u\|_{L^2_{\gamma}(\Omega_T)}=\|{e}^{-\gamma t} u\|_{L^2(\Omega_T)}$.

We use the following notation: $A\lesssim B$ ($B\gtrsim A$) if $A\leq CB$ ($B\geq CA$) holds uniformly for some positive constant $C$ that is {\it independent} of $\gamma.$

In the following, we present the Moser-type calculus inequalities in weighted Sobolev spaces, which will be used in proving the higher-order tame estimates
and convergence of the Nash--Moser iterative scheme.

\begin{lemma}
	\label{lem.Moser}
	Let $m\in \mathbb{N}$, $\gamma\geq 1$, $T > 0$,
	and $u,v\in H_{\gamma}^m(\Omega_T)\cap L^{\infty}(\Omega_T)$.
	Let $b$ denote a $C^{\infty}$--function defined
	in a neighborhood of the origin.

	\noindent {\rm (a)}
	If $|\beta_1| + |\beta_2|\leq m$ and $b(0)=0$, then
	\begin{align}
	\big\|\p^{\beta_1} u\p^{\beta_2} v\big\|_{L_{\gamma}^2(\Omega_T)}+
	\big\| uv\big\|_{H_{\gamma}^m(\Omega_T)}
	 & \lesssim
	\|u\|_{L^{\infty}(\Omega_T)}\|v\|_{H_{\gamma}^m(\Omega_T)}
	+\|u\|_{H_{\gamma}^m(\Omega_T)}\|v\|_{L^{\infty}(\Omega_T)}, \label{Moser1}
\\
\label{Moser2}
	\|b(u)\|_{H_{\gamma}^m(\Omega_T)}
&	\leq  C\big(\|u\|_{L^{\infty}(\Omega_T)}\big)\|u\|_{H_{\gamma}^m(\Omega_T)};
	\end{align}
	
	\noindent {\rm (b)}
	If $|\beta_1| + |\beta_2| + |\beta_3|\leq m$, then
	\begin{align}
       \big\|\p^{\beta_1}[\p^{\beta_2},b(u)]\p^{\beta_3}v\big\|_{L^{2}_{\gamma}(\Omega_T)}
	\label{Moser3}  \leq
	C\big(\|u\|_{L^{\infty}(\Omega_T)} \big)
	\left(\|v\|_{H_{\gamma}^m(\Omega_T)}
	+\|u\|_{H_{\gamma}^m(\Omega_T)}\|v\|_{L^{\infty}(\Omega_T)}\right).
	\end{align}
	\mbox{\quad\ }
	Moreover, if $u\in W^{1,\infty}(\Omega_T)$, then
	\begin{align}
	  \big\|\p^{\beta_1}[\p^{\beta_2},b(u)]\p^{\beta_3}v\big\|_{L^{2}_{\gamma}(\Omega_T) }
	 \leq      C\big(\|u\|_{W^{1,\infty}(\Omega_T)}\big)
	\left(\|v\|_{H_{\gamma}^{m-1}(\Omega_T)}
	+\|u\|_{H_{\gamma}^m(\Omega_T)}\|v\|_{L^{\infty}(\Omega_T)}\right). \label{Moser4}
	\end{align}
	Here $\beta_i$ {\rm(}for $i=1,2,3${\rm)} are multi-indices,
	$[a,b]c:=a(bc)-b(ac)$ represents the standard commutator,
	and
the increasing function $C$ is independent of $u$, $v$, $\gamma$, and $T$.
	The same conclusions remain valid when $\Omega_T$ is replaced by $\omega_T$.
\end{lemma}

We refer the proof of the inequalities \eqref{Moser1} and \eqref{Moser2} to \cite[Section 4.5]{M01MR1842775} and \cite[Appendix C]{CS08MR2423311}.
We also omit the proof of the inequalities \eqref{Moser3} and \eqref{Moser4}, which can be derived from  \eqref{Moser1} and \eqref{Moser2}  by a direct calculation.

 \section{Variable Coefficient Linearized Problem}\label{variablecoe}

 In this section we introduce the effective linear problem and its formulation with variable coefficients.
  We first write \eqref{EVS.a} as
 \begin{equation}\label{ee}
 \begin{split}
&\mathbb{L}(U^{\pm},\varPhi^{\pm}):=\partial_tU^{\pm}+A_1(U^{\pm})\partial_1U^{\pm}+A_2(U^{\pm})\partial_2U^{\pm}\\
&\qquad\qquad\qquad+\frac{1}{\partial_3\varPhi^{\pm}}\big(A_3(U^{\pm})-\partial_t\varPhi^{\pm}I-\partial_1\varPhi^{\pm}A_1(U^{\pm})-\partial_2\varPhi^{\pm}A_2(U^{\pm})\big)\partial_3U^{\pm}=\mathbf{0},
\end{split}
\end{equation}
for $x_3>0.$
Using Rankine-Hugoniot conditions, we derive that
\begin{equation}\label{bd}
\mathbb{B}(U|_{x_3=0},\varphi)=
\begin{cases}
(v^+_1-v^-_1)\partial_1\varphi+(v^+_2-v^-_2)\partial_2\varphi-(v^+_3-v^-_3)=0,\\
\partial_t\varphi+v^+_1\partial_1\varphi+v^+_2\partial_2\varphi-v^+_3=0,\\
\rho^+-\rho^-=0,
\end{cases}
\end{equation}
where $\varPhi^{\pm}=\varphi$ at $x_3=0.$
We also have that
\begin{equation}\nonumber
\begin{cases}
(F^+_{11}-F^-_{11})\partial_1\varphi+(F^+_{21}-F^-_{21})\partial_2\varphi-(F^+_{31}-F^-_{31})=0,\\
F^+_{11}\partial_1\varphi+F^+_{21}\partial_2\varphi-F^+_{31}=0,\\
(F^+_{12}-F^-_{12})\partial_1\varphi+(F^+_{22}-F^-_{22})\partial_2\varphi-(F^+_{32}-F^-_{32})=0,\\
F^+_{12}\partial_1\varphi+F^+_{22}\partial_2\varphi-F^+_{32}=0,\\
(F^+_{13}-F^-_{13})\partial_1\varphi+(F^+_{23}-F^-_{23})\partial_2\varphi-(F^+_{33}-F^-_{33})=0,\\
F^+_{13}\partial_1\varphi+F^+_{23}\partial_2\varphi-F^+_{33}=0,\\
\end{cases}
\end{equation}
Now, we consider the following background states:
\begin{equation}\label{background2}
\begin{split}
U^{r,l}&=(\rho^{r,l},v^{r,l}_1,v^{r,l}_2,v^{r,l}_3,F^{r,l}_{11},F^{r,l}_{21},F^{r,l}_{31},F^{r,l}_{12},F^{r,l}_{22},F^{r,l}_{32},F^{r,l}_{13},F^{r,l}_{23},F^{r,l}_{33})^{\top}=\bar{U}^{r,l}+\dot{U}^{r,l}\\
&=(\bar{\rho},\pm\bar{v},0,0,\pm\bar{F}_{11},\pm\bar{F}_{21},0,\pm\bar{F}_{12},\pm\bar{F}_{22},0,\pm\bar{F}_{13},\pm\bar{F}_{23},0)^{\top}+\\
&\quad \ (\dot{\rho}^{r,l},\dot{v}^{r,l}_1,\dot{v}^{r,l}_2,\dot{v}^{r,l}_3,\dot{F}^{r,l}_{11},\dot{F}^{r,l}_{21},\dot{F}^{r,l}_{31},\dot{F}^{r,l}_{12},\dot{F}^{r,l}_{22},\dot{F}^{r,l}_{32},\dot{F}^{r,l}_{13},\dot{F}^{r,l}_{23},\dot{F}^{r,l}_{33})^{\top},\\
\varPhi^{r,l}&(t,x_1,x_2,x_3):=\pm x_3+\dot{\varPhi}^{r,l},
\end{split}
\end{equation}
where $U^{r,l}$ and $\varPhi^{r,l}$ represent the states and changes of variables on each side of vortex sheets separately. $\bar{\rho}>0,\bar{v},\bar{F}_{11}, \bar{F}_{21}, \bar{F}_{12}, \bar{F}_{22}, \bar{F}_{13},\bar{F}_{23}$ are constants. $\dot{U}^{r,l}$ and $\dot{\varPhi}^{r,l}$ are functions which denote the perturbation around the constant states.
We assume the perturbation of the background states satisfying
\begin{equation}\label{pp}
\dot{U}^{r,l}\in W^{2,\infty}(\Omega),\quad\dot{\varPhi}^{r,l}\in W^{3,\infty}(\Omega),\quad \|\dot{U}^{r,l}\|_{W^{2,\infty}(\Omega)}+\|\dot{\varPhi}^{r,l}\|_{W^{3,\infty}(\Omega)} \leq K.
\end{equation}
Here $K$ is a positive constant. $\dot{U}^{r,l}$ and $\dot{\varPhi}^{r,l}$ have compact support in the domain $$\Omega := \{(t,x_1,x_2,x_3)\in \R^4:x_3>0\}.$$
 We also require the perturbed states \eqref{background2} to satisfy the Rankine-Hugoniot conditions:
\begin{equation}\label{boundary2}
\begin{cases}
(v^r_1-v^l_1)\partial_1\varphi+(v^r_2-v^l_2)\partial_2\varphi-(v^r_3-v^l_3)=0,\\
\partial_t\varphi+v^r_1\partial_1\varphi+v^r_2\partial_2\varphi-v^r_3=0,\\
\rho^r-\rho^l=0,
\end{cases}
\end{equation}
on $x_3=0,$ where $\varphi=\varPhi^r|_{x_3=0}=\varPhi^l|_{x_3=0}.$

Now, we assume that the following conditions on the perturbed states \eqref{background2} hold:
\begin{equation}\label{eikonal}
\partial_t\varPhi^{r,l}+v^{r,l}_1\partial_1\varPhi^{r,l}+v^{r,l}_2\partial_2\varPhi^{r,l}-v^{r,l}_3=0,
\end{equation}
\begin{equation}\label{lower}
\partial_3\varPhi^r\geq\kappa_0,\quad \partial_3\varPhi^l\leq -\kappa_0,
\end{equation}
for all $(t,x_1,x_2,x_3)\in\Omega$ and some positive constant $\kappa_0.$
We also assume from initial data that
\begin{equation}
\begin{cases}
F^{r,l}_{11}\partial_1\varPhi^{r,l}+F^{r,l}_{21}\partial_2\varPhi^{r,l}-F^{r,l}_{31}=0,\\
F^{r,l}_{12}\partial_1\varPhi^{r,l}+F^{r,l}_{22}\partial_2\varPhi^{r,l}-F^{r,l}_{32}=0,\\
F^{r,l}_{13}\partial_1\varPhi^{r,l}+F^{r,l}_{23}\partial_2\varPhi^{r,l}-F^{r,l}_{33}=0.
\end{cases}
\end{equation}
Now, we linearize the \eqref{ee} around the basic states \eqref{background2} and denote by $(V^{\pm},\Psi^{\pm})$ the perturbation of the states $(U^{r,l},\varPhi^{r,l}).$ Then, the linearized equations are
\begin{equation}\nonumber
\begin{split}
&\partial_t V^{\pm}+A_1(U^{r,l})\partial_1 V^{\pm}+A_2(U^{r,l})\partial_2 V^{\pm}\\
&\quad+\frac{1}{\partial_3\varPhi^{r,l}}\big(A_3(U^{r,l})-\partial_t\varPhi^{r,l}I-\partial_1\varPhi^{r,l}A_1(U^{r,l})-\partial_2\varPhi^{r,l}A_2(U^{r,l})\big)\partial_3V^{\pm}\\
&\quad+[dA_1(U^{r,l})V^{\pm}]\partial_1U^{r,l}+[dA_2(U^{r,l})V^{\pm}]\partial_2U^{r,l}\\
&\quad-\frac{\partial_3\Psi^{\pm}}{(\partial_3\varPhi^{r,l})^2}\big(A_3(U^{r,l})-\partial_t\varPhi^{r,l}I-\partial_1\varPhi^{r,l}A_1(U^{r,l})-\partial_2\varPhi^{r,l}A_2(U^{r,l})\big)\partial_3U^{r,l}\\
&\quad+\frac{1}{\partial_3\varPhi^{r,l}}\big(dA_3(U^{r,l})V^{\pm}-\partial_t\Psi^{\pm}I-\partial_1\Psi^{\pm}A_1(U^{r,l})-\partial_2\Psi^{\pm}A_2(U^{r,l})\\
&\quad-\partial_1\varPhi^{r,l}dA_1(U^{r,l})V^{\pm}-\partial_2\varPhi^{r,l}dA_2(U^{r,l})V^{\pm}\big)\partial_3U^{r,l}=f,
\end{split}
\end{equation}
for $x_3>0.$ We define the first-order linear operator
\begin{equation}\nonumber
\begin{split}
&L(U^{r,l},\nabla\varPhi^{r,l})V^{\pm}:=\partial_tV^{\pm}+A_1(U^{r,l})\partial_1V^{\pm}+A_2(U^{r,l})\partial_2V^{\pm}\\
&\qquad\qquad\qquad\qquad+\frac{1}{\partial_3\varPhi^{r,l}}\big(A_3(U^{r,l})-\partial_t\varPhi^{r,l}I-\partial_1\varPhi^{r,l}A_1(U^{r,l})-\partial_2\varPhi^{r,l}A_2(U^{r,l})\big)\partial_3V^{\pm},\\
\end{split}
\end{equation}
and introduce the Alinhac's``{\it good unknown}'' \cite{A89MR976971}:
\begin{equation}\nonumber
\begin{split}
\dot{V}^{\pm}=(\dot{\rho}^{\pm},\dot{v}^{\pm}_1,\dot{v}^{\pm}_2,\dot{v}^{\pm}_3,\dot{F}^{\pm}_{11},\dot{F}^{\pm}_{21},\dot{F}^{\pm}_{31},\dot{F}^{\pm}_{12},\dot{F}^{\pm}_{22},\dot{F}^{\pm}_{32},\dot{F}^{\pm}_{13},\dot{F}^{\pm}_{23},\dot{F}^{\pm}_{33})^{\top}:=V^{\pm}-\frac{\Psi^{\pm}}{\partial_3\varPhi^{r,l}}\partial_3U^{r,l}.
\end{split}
\end{equation}
Then, we can rewrite the above equations as
\begin{equation}\nonumber
L(U^{r,l},\nabla\varPhi^{r,l})\dot{V}^{\pm}+C(U^{r,l},\nabla U^{r,l},\nabla \varPhi^{r,l})\dot{V}^{\pm}+\frac{\Psi^{\pm}}{\partial_3\varPhi^{r,l}}\partial_3[L(U^{r,l},\nabla \varPhi^{r,l})U^{r,l}]=f^{r,l},
\end{equation}
where
\begin{equation}\nonumber
\begin{split}
&C(U^{r,l},\nabla U^{r,l},\nabla \varPhi^{r,l})\dot{V}^{\pm}:=[dA_1(U^{r,l})\dot{V}^{\pm}]\partial_1U^{r,l}+[dA_2(U^{r,l})\dot{V}^{\pm}]\partial_2U^{r,l}\\
&\qquad\qquad\qquad\qquad+\frac{1}{\partial_3\varPhi^{r,l}}[dA_3(U^{r,l})\dot{V}^{\pm}-\partial_1\varPhi^{r,l}dA_1(U^{r,l})\dot{V}^{\pm}-\partial_2\varPhi^{r,l}dA_2(U^{r,l})\dot{V}^{\pm}]\partial_3U^{r,l}.
\end{split}
\end{equation}
Neglecting the zeroth-order terms of $\Psi^{\pm}$ and considering the following equations:
\begin{equation}\nonumber
\begin{split}
L'_{r,l}\dot{V}^{\pm}:=L(U^{r,l},\nabla\varPhi^{r,l})\dot{V}^{\pm}+C(U^{r,l},\nabla U^{r,l},\nabla \varPhi^{r,l})\dot{V}^{\pm}=f^{r,l}.
\end{split}
\end{equation}
Since $U^{r,l} \in W^{2,\infty}(\Omega),$ we have $L(U^{r,l},\nabla\varPhi^{r,l}) \in W^{2,\infty}(\Omega)$ and $C(U^{r,l},\nabla U^{r,l},\nabla \varPhi^{r,l}) \in W^{1,\infty}(\Omega).$
Now, we linearize the boundary conditions around the same perturbed states and obtain that
\begin{equation}\nonumber
\begin{cases}
(v^r_1-v^l_1)\partial_1\psi +(v^+_1-v^-_1)\partial_1 \varphi+(v^r_2-v^l_2)\partial_2 \psi +(v^+_2-v^-_2)\partial_2 \varphi-(v^+_3-v^-_3)=g_1,\\
\partial_t \psi + v^r_1\partial_1\psi +v^+_1\partial_1\varphi +v^r_2\partial_2\psi +v^+_2\partial_2\varphi -v^+_3=g_2,\\
\rho^+-\rho^-=g_3,
\end{cases}
\end{equation}
at $x_3=0, \psi =\Psi^+|_{x_3=0}=\Psi^-|_{x_3=0}.$
We also have that
\begin{equation}\nonumber
\begin{cases}
(F^r_{11}-F^l_{11})\partial_1\psi +(F^+_{11}-F^-_{11})\partial_1\varphi+(F^r_{21}-F^{l}_{21})\partial_2 \psi +(F^{+}_{21}-F^{-}_{21})\partial_2\varphi -(F^+_{31}-F^{-}_{31})=g_4,\\
F^r_{11}\partial_1\psi +F^+_{11}\partial_1\varphi +F^r_{21}\partial_2 \psi +F^+_{21}\partial_2\varphi -F^+_{31}=g_5,\\
(F^r_{12}-F^l_{12})\partial_1\psi +(F^+_{12}-F^-_{12})\partial_1\varphi +(F^r_{22}-F^{l}_{22})\partial_2\psi +(F^{+}_{22}-F^{-}_{22})\partial_2\varphi -(F^+_{32}-F^{-}_{32})=g_6,\\
F^r_{12}\partial_1\psi +F^+_{12}\partial_1\varphi +F^r_{22}\partial_2\psi +F^+_{22}\partial_2\varphi -F^+_{32}=g_7,\\
(F^r_{13}-F^l_{13})\partial_1\psi +(F^+_{13}-F^-_{13})\partial_1\varphi +(F^r_{23}-F^{l}_{23})\partial_2\psi +(F^{+}_{23}-F^{-}_{23})\partial_2\varphi -(F^+_{33}-F^{-}_{33})=g_8,\\
F^r_{13}\partial_1\psi +F^+_{13}\partial_1\varphi +F^r_{23}\partial_2\psi +F^+_{23}\partial_2\varphi -F^+_{33}=g_9.\\
\end{cases}
\end{equation}
We write the above system of equations into the following {\it enlarged} form:
$$\underline{\mathbf{b}}\nabla \psi +\underline{M}V|_{x_3=0}=g,$$
where $V=(V^+,V^-)^{\top},\nabla\psi=(\partial_t \psi ,\partial_1 \psi,\partial_2 \psi)^{\top},g=(g_1,g_2,g_3,g_4,g_5,g_6,g_7,g_8,g_9)^{\top},$
\begin{equation}\nonumber
\underline{\mathbf{b}}(t,x_1,x_2)=\left[\begin{array}{ccc}
0 & (v^r_1-v^l_1)|_{x_3=0} & (v^r_2-v^l_2)|_{x_3=0}\\
1 & v^r_1|_{x_3=0} &  v^r_2|_{x_3=0}\\
0 &  0 & 0\\
0 & (F^r_{11}-F^l_{11})|_{x_3=0} & (F^r_{21}-F^l_{21})|_{x_3=0}\\
0 & F^r_{11}|_{x_3=0} & F^r_{21}|_{x_3=0}\\
0 & (F^r_{12}-F^l_{12})|_{x_3=0} & (F^r_{22}-F^l_{22})|_{x_3=0}\\
0 & F^r_{12}|_{x_3=0} & F^r_{22}|_{x_3=0}\\
0 & (F^r_{13}-F^l_{13})|_{x_3=0} & (F^r_{23}-F^l_{23})|_{x_3=0}\\
0 & F^r_{13}|_{x_3=0} & F^r_{23}|_{x_3=0}\\
\end{array}\right]
:=[\mathbf{b}_0,\mathbf{b}_1,\mathbf{b}_2],
\end{equation}
\begin{equation}\nonumber
\begin{split}
 &\underline{M}(t,x_1,x_2):=\\
 &\left[\begin{smallmatrix}
 0 & \partial_1\varphi & \partial_2\varphi & -1 & 0 & 0 & 0 & 0 & 0 & 0 & 0 & 0 & 0 & 0 & -\partial_1\varphi & -\partial_2\varphi & 1 & 0 & 0 & 0 & 0 & 0 & 0 & 0 & 0 & 0\\
 0 & \partial_1\varphi & \partial_2\varphi & -1 & 0 & 0 & 0 & 0 & 0 & 0 & 0 & 0 & 0 & 0 & 0 & 0 & 0 & 0 & 0 & 0 & 0 & 0 & 0 & 0 & 0 & 0\\
 1 & 0 & 0 & 0 & 0 & 0 & 0 & 0 & 0 & 0 & 0 & 0 & 0 & -1 & 0 & 0 & 0 & 0 & 0 & 0 & 0 & 0 & 0 & 0 & 0 & 0\\
 0 & 0 & 0 & 0 & \partial_1\varphi & \partial_2\varphi & -1 & 0 & 0 & 0 & 0 & 0 & 0 & 0 & 0 & 0 & 0 & -\partial_1\varphi & -\partial_2\varphi & 1 & 0 & 0 & 0 & 0 & 0 & 0\\
  0 & 0 & 0 & 0 & \partial_1\varphi & \partial_2\varphi & -1 & 0 & 0 & 0 & 0 & 0 & 0 & 0 & 0 & 0 & 0 & 0 & 0 & 0 & 0 & 0 & 0 & 0 & 0 & 0\\
 0 & 0 & 0 & 0 & 0 & 0 & 0 & \partial_1\varphi & \partial_2\varphi & -1 & 0 & 0 & 0 & 0 & 0 & 0 & 0 & 0 & 0 & 0 & -\partial_1\varphi & -\partial_2\varphi & 1 & 0 & 0 & 0\\
 0 & 0 & 0 & 0 & 0 & 0 & 0 & \partial_1\varphi & \partial_2\varphi & -1 & 0 & 0 & 0 & 0 & 0 & 0 & 0 & 0 & 0 & 0 & 0 & 0 & 0 & 0 & 0 & 0\\
 0 & 0 & 0 & 0 & 0 & 0 & 0 & 0 & 0 & 0 & \partial_1\varphi & \partial_2\varphi & -1 & 0 & 0 & 0 & 0 & 0
 & 0 & 0 & 0 & 0 & 0 & -\partial_1\varphi & -\partial_2\varphi & 1\\
  0 & 0 & 0 & 0 & 0 & 0 & 0 & 0 & 0 & 0 & \partial_1\varphi & \partial_2\varphi & -1 & 0 & 0 & 0 & 0 & 0
 & 0 & 0 & 0 & 0 & 0 & 0 & 0 & 0\\
 \end{smallmatrix}\right].
 \end{split}
\end{equation}
Using the Alinhac's ``{\it good unknown}'' \cite{A89MR976971}, we obtain that
\begin{equation}\label{lb}
B'(\dot{V},\psi):=\underline{\mathbf{b}}\nabla\psi+\underline{M}\left[\begin{array}{c}
\frac{\partial_3U^r}{\partial_3\varPhi^r}\\
\frac{\partial_3U^l}{\partial_3\varPhi^l}\\
\end{array}\right]\psi+\underline{M}\dot{V}|_{x_3=0}=g.
\end{equation}
Therefore, we have the following linearized problem:
\begin{equation}\label{le}
\begin{cases}
L'_{r,l}\dot{V}^{\pm}=f^{r,l},&x_3>0,\\
B'(\dot{V},\psi)=g,&x_3=0.
\end{cases}
\end{equation}
Note that the boundary condition in \eqref{le} does not contain the tangential components of $\dot{V}|_{x_3=0}.$
We write the components of $\dot{V}|_{x_3=0}$ that are contained in \eqref{le} by $\dot{V}^{n}|_{x_3=0}$ as
\begin{equation}\label{nc}
\begin{split}
\dot{V}^{n}|_{x_3=0}=(\dot{V}^{n+},\dot{V}^{n-})|_{x_3=0},
\end{split}
\end{equation}
where
\begin{align}\nonumber
\dot{V}^{n+} = & (\dot{\rho}^+,\dot{v}^+_3-\dot{v}^+_1\partial_1\varPhi^r-\dot{v}^+_2\partial_2\varPhi^r,\dot{F}^+_{31}-\dot{F}^+_{11}\partial_1\varPhi^r-\dot{F}^+_{21}\partial_2\varPhi^r,\dot{F}^+_{32}-\dot{F}^+_{12}\partial_1\varPhi^r-\dot{F}^+_{22}\partial_2\varPhi^r, \nonumber \\ 
& \ \dot{F}^+_{33}-\dot{F}^+_{13}\partial_1\varPhi^r-\dot{F}^+_{23}\partial_2\varPhi^r)^{\top},\nonumber\\
\dot{V}^{n-} = & (\dot{\rho}^-,\dot{v}^-_3-\dot{v}^-_1\partial_1\varPhi^l-\dot{v}^-_2\partial_2\varPhi^l,\dot{F}^-_{31}-\dot{F}^-_{11}\partial_1\varPhi^l-\dot{F}^-_{21}\partial_2\varPhi^l, \dot{F}^-_{32}-\dot{F}^-_{12}\partial_1\varPhi^l-\dot{F}^-_{22}\partial_2\varPhi^l, \nonumber \\
& \ \dot{F}^-_{33} - \dot{F}^-_{13}\partial_1\varPhi^l - \dot{F}^-_{23}\partial_2\varPhi^l)^{\top}.\nonumber
\end{align}
We obtain the main theorem for the variable coefficients.
\begin{theorem}\label{vt}
Suppose that the background solution defined by \eqref{background2} satisfies $\mathrm{F}_1\times\mathrm{F}_2\neq\mathbf{0},$ and \eqref{H1}--\eqref{stability4} hold.
Moreover, the perturbation $\dot{U}^{r,l}$ and $\dot{\varPhi}^{r,l}$ have compact support, and $K$ in \eqref{pp} is small enough.  Then, there are two constants $C_0$ and $\gamma_0$ which are determined by particular solution, such that for all $\dot{V}$ and $\psi$ and all $\gamma\geq\gamma_0,$ the following estimate holds{\rm:}
\begin{equation}\nonumber
\begin{split}
&\gamma \vertiii{\dot{V}}^2_{L^2(H^0_{\gamma})} + \|\dot{V}^{n}|_{x_3=0} \|^2_{L^2_{\gamma}(\R^3)} + \| \psi \|^2_{H^1_{\gamma}(\R^3)}\\
&\quad\leq C_0\left(\frac{1}{\gamma^3} \vertiii{L'\dot{V}}^2_{L^2(H^1_{\gamma})}+\frac{1}{\gamma^2} \| B'(\dot{V},\psi) \|^2_{H^1_{\gamma}(\R^3)}\right),
\end{split}
\end{equation}
where $L'\dot{V} := (L'_r\dot{V}^+,L'_{l}\dot{V}^-).$
\end{theorem}

\subsection{Reduction of the System}\label{reduction}
In this section, we will transform the system \eqref{le} into ODEs. This is achieved through linear transformation on the unknown variables and a paralinearization on the equations of the transformed variables. For the system \eqref{le}, we can find the symmetrizer
$$
S^{r,l}:= \mathrm{diag} \left\{\frac{p'(\rho^{r,l})}{\rho^{r,l}},\rho^{r,l},\rho^{r,l},\rho^{r,l},1,1,1,1,1,1,1,1,1 \right\}.
$$
We multiply $(S^{r,l}\dot{V}^{\pm})^{\top}$ to the interior equations of \eqref{le} and then integrate by parts to obtain the following Lemma \ref{estimate3}.
\begin{lemma}\label{estimate3}
There are two positive constants $C$ and $\gamma_1\geq1$ such that for any $\gamma\geq\gamma_1,$ the following estimate holds{\rm:}
$$
\gamma \vertiii{\dot{V}^{\pm}}^2_{L^2(H^0_{\gamma})}\leq C\left(\frac{1}{\gamma} \vertiii{L'_{r,l}\dot{V}^{\pm}}^2_{L^2_{\gamma}(\R^4_+)} + \left\|\dot{V}^{n}|_{x_3=0} \right\|^2_{L^2_{\gamma}(\R^3)}\right).
$$
\end{lemma}
First, we transform the linearized problem \eqref{le} into a problem with a constant and diagonal boundary matrix. This is essential since the boundary matrix has constant rank on the whole half-plane $x_3\geq0.$
$$
\tilde{A}^{r,l}_3 := \frac{1}{\partial_3\varPhi^{r,l}} \left(A_3(U^{r,l})-\partial_t\varPhi^{r,l}I-\partial_1\varPhi^{r,l}A_1(U^{r,l})-\partial_2\varPhi^{r,l}A_2(U^{r,l}) \right).$$
Recall that $c(\rho)$ is the sound speed defined in \eqref{c.def}. We consider the following transformation
\begin{equation}\nonumber\\
\begin{split}
&T(U^{r,l},\nabla\varPhi^{r,l}):=\\
&\left[\begin{smallmatrix}
0 & 0 & \langle\partial_{\rm tan}\varPhi^{r,l}\rangle & \langle\partial_{\rm tan}\varPhi^{r,l}\rangle & 0 & 0 & 0 & 0 & 0 & 0 & 0 & 0 & 0\\
 1 & 0 & -\frac{c(\rho^{r,l})}{\rho^{r,l}}\partial_1\varPhi^{r,l} & \frac{c(\rho^{r,l})}{\rho^{r,l}}\partial_1\varPhi^{r,l} & 0 & 0 & 0 & 0 & 0 & 0 & 0 & 0 & 0\\
 0 & 1 & -\frac{c(\rho^{r,l})}{\rho^{r,l}}\partial_2\varPhi^{r,l} & \frac{c(\rho^{r,l})}{\rho^{r,l}}\partial_2\varPhi^{r,l} & 0 & 0 & 0 & 0 & 0 & 0 & 0 & 0 & 0\\
\partial_1\varPhi^{r,l} & \partial_2\varPhi^{r,l} & \frac{c(\rho^{r,l})}{\rho^{r,l}} & -\frac{c(\rho^{r,l})}{\rho^{r,l}} & 0 & 0 & 0 & 0 & 0 & 0 & 0 & 0 & 0\\
0 & 0 & 0 & 0 & 1 & 0 & -\partial_1\varPhi^{r,l} & 0 & 0 & 0 & 0 & 0 & 0\\
0 & 0 & 0 & 0 & 0 & 1 & -\partial_2\varPhi^{r,l} & 0 & 0&  0 & 0 & 0 & 0\\
0 & 0 & 0 & 0 & \partial_1\varPhi^{r,l} & \partial_2\varPhi^{r,l} & 1 & 0 & 0 & 0 & 0 & 0 & 0\\
0 & 0 & 0 & 0 & 0 & 0 & 0 & 1 & 0 & -\partial_1\varPhi^{r,l} & 0 & 0 & 0\\
0 & 0 & 0 & 0 & 0 & 0 & 0 & 0 & 1 & -\partial_2\varPhi^{r,l} & 0 & 0 & 0\\
0 & 0 & 0 & 0 & 0 & 0 & 0 & \partial_1\varPhi^{r,l} & \partial_2\varPhi^{r,l} & 1 & 0 & 0 & 0\\
0 & 0 & 0 & 0 & 0 & 0 & 0 & 0 & 0 & 0 & 1 & 0 & -\partial_1\varPhi^{r,l}\\
0 & 0 & 0 & 0 & 0 & 0 & 0 & 0 & 0 & 0 & 0 & 1 & -\partial_2\varPhi^{r,l}\\
0 & 0 & 0 & 0 & 0 & 0 & 0 & 0 & 0 & 0 & \partial_1\varPhi^{r,l} & \partial_2\varPhi^{r,l} & 1\\
\end{smallmatrix}\right],
\end{split}
\end{equation}
where $\langle\partial_{\rm tan}\varPhi^{r,l}\rangle := \sqrt{1+(\partial_1\varPhi^{r,l})^2+(\partial_2\varPhi^{r,l})^2}.$
Then, we can obtain
\begin{equation}\nonumber
\begin{split}
&T^{-1}(U^{r,l},\nabla\varPhi^{r,l})\tilde{A}^{r,l}_3T(U^{r,l},\nabla\varPhi^{r,l})\\
&=\mathrm{diag} \left\{0,0,\frac{c(\rho^{r,l})\langle\partial_{\rm tan}\varPhi^{r,l}\rangle}{\partial_3\varPhi^{r,l}},-\frac{c(\rho^{r,l})\langle\partial_{\rm tan}\varPhi^{r,l}\rangle}{\partial_3\varPhi^{r,l}},0,0,0,0,0,0,0,0,0 \right\}.
\end{split}
\end{equation}
Multiplying the above by the following matrix:
\begin{equation}\label{def A0}
A^{r,l}_0 := \mathrm{diag} \left\{1,1,\frac{\partial_3\varPhi^{r,l}}{c(\rho^{r,l})\langle\partial_{\rm tan}\varPhi^{r,l}\rangle},-\frac{\partial_3\varPhi^{r,l}}{c(\rho^{r,l})\langle\partial_{\rm tan}\varPhi^{r,l}\rangle},1,1,1,1,1,1,1,1,1 \right\},
\end{equation}
we obtain the interior equations of \eqref{le} for the new unknowns $W^{\pm}:=T^{-1}(U^{r,l},\nabla\varPhi^{r,l})\dot{V}^{\pm}$ are
\begin{equation}\label{equation2}
A^{r,l}_0\partial_tW^{\pm}+A^{r,l}_{1}\partial_1W^{\pm}+A^{r,l}_2\partial_2W^{\pm}+I_2\partial_3W^{\pm}+A^{r,l}_0C^{r,l}W^{\pm}=F^{r,l},
\end{equation}
where
\begin{equation}\label{def I2}
I_2 := \mathrm{diag}\{0,0,1,1,0,0,0,0,0,0,0,0,0\}.
\end{equation}
\begin{equation}\label{def A1A2}
\begin{split}
&A^{r,l}_1=A^{r,l}_0T^{-1}(U^{r,l},\nabla\varPhi^{r,l})A_1(U^{r,l})T(U^{r,l},\nabla\varPhi^{r,l}),\\
&A^{r,l}_2=A^{r,l}_0T^{-1}(U^{r,l},\nabla\varPhi^{r,l})A_2(U^{r,l})T(U^{r,l},\nabla\varPhi^{r,l}),\\
&C^{r,l}=T^{-1}(U^{r,l},\nabla\varPhi^{r,l})\partial_tT(U^{r,l},\nabla\varPhi^{r,l})+T^{-1}(U^{r,l},\nabla\varPhi^{r,l})A_1(U^{r,l})\partial_1T(U^{r,l},\nabla\varPhi^{r,l})\\
&\quad\qquad+T^{-1}(U^{r,l},\nabla\varPhi^{r,l})A_2(U^{r,l})\partial_2T(U^{r,l},\nabla\varPhi^{r,l})\\
&\quad\qquad+T^{-1}(U^{r,l},\nabla\varPhi^{r,l})C(U^{r,l},\nabla U^{r,l},\nabla\varPhi^{r,l})T(U^{r,l},\nabla\varPhi^{r,l})\\
&\quad\qquad+T^{-1}(U^{r,l},\nabla\varPhi^{r,l})\tilde{A}^{r,l}_3\partial_3T(U^{r,l},\nabla\varPhi^{r,l}),\\
&F^{r,l}=A^{r,l}_0T^{-1}(U^{r,l},\nabla\varPhi^{r,l})f^{r,l}.
\end{split}
\end{equation}
We consider  the weighted unknown $\tilde{W}^{\pm}=e^{-\gamma t}W^{\pm}$ and rewrite \eqref{equation2} as
\begin{equation}\nonumber
\begin{split}
&\mathcal{L}^{\gamma}_{r,l}\tilde{W}^{\pm}:=
\gamma A^{r,l}_0\tilde{W}^{\pm}+A^{r,l}_0\partial_t\tilde{W}^{\pm}+A_1^{r,l}\partial_1\tilde{W}^{\pm}+A^{r,l}_2\partial_2\tilde{W}^{\pm}+I_2\partial_3\tilde{W}^{\pm}+A^{r,l}_0C^{r,l}\tilde{W}^{\pm}=e^{-\gamma t}F^{r,l},\\
\end{split}
\end{equation}
where $A^{r,l}_{j}\in W^{2,\infty}(\Omega)$ and $C^{r,l}\in W^{1,\infty}(\Omega).$
Then, we obtain the equivalent form of the boundary condition \eqref{le}:
\begin{equation}\label{le2}
\underline{\mathbf{b}}\nabla\psi+\underline{M}\left[\begin{array}{c}
\frac{\partial_3U^r}{\partial_3\varPhi^r}\\
\frac{\partial_3U^l}{\partial_3\varPhi^l}\\
\end{array}\right]\psi
+\underline{M}\left[\begin{array}{cc}
T(U^r,\nabla\varPhi^r)& \mathbf{0}\\
\mathbf{0} & T(U^l,\nabla\varPhi^l)\\
\end{array}\right]W|_{x_3=0}=g.
\end{equation}
Note that $\tilde{W}=e^{-\gamma t}W$ and $\tilde{\psi}=e^{-\gamma t}\psi$. It follows that
\begin{equation}\nonumber
\begin{split}
&\mathcal{B}^{\gamma}(\tilde{W}|_{x_3=0},\tilde{\psi})\\
&:=\gamma \mathbf{b}_0\tilde{\psi}+\underline{\mathbf{b}}\nabla\tilde{\psi}+\underline{M}\left[\begin{array}{c}
\frac{\partial_3U^r}{\partial_3\varPhi^r}\\
\frac{\partial_3U^l}{\partial_3\varPhi^l}\\
\end{array}\right]\tilde{\psi}
+\underline{M}\left[\begin{array}{cc}
T(U^r,\nabla\varPhi^r)& \mathbf{0}\\
\mathbf{0} & T(U^l,\nabla\varPhi^l)\\
\end{array}\right]\tilde{W}|_{x_3=0}=e^{-\gamma t}g,
\end{split}
\end{equation}
where $\mathbf{\mathbf{b}}_0$ is the first column of $\underline{\mathbf{b}}.$ Then, we have $\underline{\mathbf{b}},\underline{M}$ and $T \in W^{2,\infty}(\Omega),$
and $$\check{\mathbf{b}}:=\underline{M}\left[\begin{array}{c}
\frac{\partial_3U^r}{\partial_3\varPhi^r}\\
\frac{\partial_3U^l}{\partial_3\varPhi^l}\\
\end{array}\right]\in W^{1,\infty}(\Omega).$$
We first denote the ``noncharacteristic'' components of $\tilde{W},$ which will be used in Section \ref{Paralinearization}:
\begin{equation}\label{noncharacter2}
\tilde{W}^{nc}|_{x_3=0}:=(\tilde{W}_3,\tilde{W}_4,\tilde{W}_{16},\tilde{W}_{17}).
\end{equation}
Then, we denote the normal components of $\tilde{W}$ by checking the matrix coefficient in front of $\tilde{W}$ in the boundary conditions,
\begin{equation}\label{normalcom}
\tilde{W}^{n}|_{x_3=0}:=(\tilde{W}_3,\tilde{W}_4,\tilde{W}_{7},\tilde{W}_{10},\tilde{W}_{13},\tilde{W}_{16},\tilde{W}_{17},\tilde{W}_{20},\tilde{W}_{23},\tilde{W}_{26}).
\end{equation}
Thus, we write the boundary conditions
$$\mathcal{B}(\tilde{W}^{n}|_{x_3=0},\tilde{\psi})=e^{-\gamma t}g,$$
By defining $\tilde{F}^{r,l}=e^{-\gamma t}F^{r,l}$ and $\tilde{g}=e^{-\gamma t}g,$
the system \eqref{le2} can be rewritten as
\begin{equation}\label{le3}
\begin{cases}
\mathcal{L}^{\gamma}_{r,l}\tilde{W}^{\pm}=\tilde{F}^{r,l},\\
\mathcal{B}^{\gamma}(\tilde{W}^{n}|_{x_3=0},\tilde{\psi})=\tilde{g}.
\end{cases}
\end{equation}

We will perform the paralinearization of the interior equations and the boundary conditions for \eqref{le3}.

\subsection{Some Results on Paradifferential Calculus}\label{paradifferentialcalculus}

For the sake of self-containedness, we present some necessary definitions and results concerning paradifferential calculus with parameters, as utilized in this paper. For rigorous proofs, refer to \cite[Appendix C]{Gavage2007} and the references therein.

\begin{definition}\label{definition para}
Let $m\in \R$ and $k\in \mathbb{N},$ we define the following notions{\rm:}
\begin{enumerate}
   \item[(i)]  A function $a(\mathbf{x},\xi,\gamma):\R^3\times\R^3\times[1,\infty)\rightarrow \mathbb{C}^{N\times N}$ is called a paradifferential symbol of degree $m$ and regularity $k$ if $a$ is $C^{\infty}$ in $\xi$ and satisfies
       $$\left\| \partial^{\alpha}_{\xi}a(\cdot,\xi,\gamma)\right\|_{W^{k,\infty}(\R^3)}\leq C_{\alpha}\lambda^{m-|\alpha|,\gamma}(\xi),$$
       for all $(\xi,\gamma)\in \R^3\times[1,\infty),$ where $\lambda^{s,\gamma}(\xi):=(\gamma^2+|\xi|^2)^{\frac{s}{2}},$ $s\in \R,$ and $C_{\alpha}$ is a constant.
      
   \item[(ii)] The set of paradifferential symbols of degree $m$ and regularity $k$ is denoted by $\Gamma^m_k.$
   \item[(iii)] A family of operators $\{\mathcal{P}^{\gamma}\}_{\gamma\geq1}$ is said to be of order$\leq m$ if, for all $s\in\R$ and $\gamma\geq1,$ there
       exists a constant $C(s,m)$ such that
       $$\left\| \mathcal{P}^{\gamma}u \right\|_{s,\gamma}\leq C(s,m) \| u \|_{s+m,\gamma},$$
       for all $u\in H^{s+m}(\R^3).$ A generic family of such operators is denoted by $\mathcal{R}_m.$
     \item[(iv)] For $s\in\R,$ define the operator $\Lambda^{s,\gamma}$ by
         $$\Lambda^{s,\gamma}u(\mathbf{x}):=\frac{1}{(2\pi)^3}\int_{\R^3}e^{i\mathbf{x}\cdot\xi}\lambda^{s,\gamma}(\xi)\hat{u}(\xi)\,\dd \xi$$
         for all $u\in \mathcal{S}$ (the Schwartz class).
     \item[(v)] To any symbol $a\in \Gamma^m_0,$ we associate a family of paradifferential operators $\{T^{\gamma}_{a}\}_{\gamma\geq1},$ defined by
         $$T^{\gamma}_au(\mathbf{x}):=\frac{1}{(2\pi)^3}\int_{\R^3}\int_{\R^3}e^{i\mathbf{x}\cdot\xi}K^{\psi}(\mathbf{x}-\mathbf{y},\xi,\gamma)a(\mathbf{y},\xi,\gamma)\hat{u}(\xi)\,\dd \mathbf{y}\dd \xi,$$
         where $K^{\psi}(\cdot,\xi,\gamma)$ is the inverse Fourier transform of $\psi(\cdot,\xi,\gamma).$ The function $\psi$ is defined as
         $$\psi(\mathbf{x},\xi,\gamma):=\sum_{q\in \mathbb{N}}\chi(2^{2-q}\mathbf{x},0)\varPhi(2^{-q}\xi,2^{-q}\gamma),$$
         where $\varPhi(\xi,\gamma):=\chi(2^{-1}\xi,2^{-1}\gamma)-\chi(\xi,\gamma),$ and $\chi$ being a $C^{\infty}$-function on $\R^4$ satisfying
         \begin{align}\nonumber
         \chi(\mathbf{z})=\begin{cases}
         1, & \text{ if } |\mathbf{z}|\leq\frac{1}{2},\\
         0, & \text{ if } |\mathbf{z}|\geq1,
         \end{cases} \quad \text{ and } \chi(\mathbf{z})\geq\chi(\mathbf{z}'), \text{ if } |\mathbf{z}|\leq|\mathbf{z}'|.
         \end{align}
\end{enumerate}
\end{definition}

We have the following properties for the paradifferential calculus:
\begin{lemma}\label{lemmapara}
The following statements hold{\rm:}
\begin{enumerate}
  \item[\rm{(i)}] If $a\in W^{1,\infty}(\R^3),$ $u\in L^2(\R^3),$ and $\gamma\geq1,$ then
      $$\gamma \|au-T^{\gamma}_au \|_0 + \|a\partial_ju-T^{\gamma}_{i\xi_ja}u \|_0 + \|au-T^{\gamma}_au \|_{1,\gamma} \lesssim \|a\|_{W^{1,\infty}(\R^3)} \|u\|_0;$$
  \item[\rm{(ii)}] If $a\in W^{2,\infty}(\R^3),$ $u\in L^2(\R^3),$ and $\gamma\geq1,$ then
      $$\gamma \|au-T^{\gamma}_a u \|_{1,\gamma} + \|a\partial_j u-T^{\gamma}_{i\xi_j a}u \|_{1,\gamma} \lesssim \|a\|_{W^{2,\infty}(\R^3)} \|u \|_0;$$
  \item[\rm{(iii)}] If $a\in \Gamma^m_k,$ then $T^{\gamma}_a$ is of order $\leq m.$ In particular, if $a\in L^{\infty}(\R^3)$ and is independent of $\xi,$ then
      $$\|T^{\gamma}_au \|_{s,\gamma} \lesssim \|a\|_{L^{\infty}(\R^3)} \|u\|_{s,\gamma}, \text{ for all } s\in \R, u\in H^s(\R^3);$$
  \item[\rm{(iv)}] If $a\in\Gamma^m_1$ and $b\in \Gamma^{m'}_1,$ then $ab\in \Gamma^{m+m'}_1,$ the family $\{T^{\gamma}_aT^{\gamma}_b-T^{\gamma}_{ab}\}_{\gamma\geq1}$ is of order $\leq m+m'-1,$ and the family $\{(T^{\gamma}_a)^{\ast}-T^{\gamma}_{a^{\ast}}\}_{\gamma\geq1}$ is of order $\leq m-1;$
  \item[\rm{(v)}] If $a\in \Gamma^m_2$ and $b\in \Gamma^{m'}_2,$ then $\{T^{\gamma}_aT^{\gamma}_b-T^{\gamma}_{ab}-T^{\gamma}_{-i\sum_j\partial_{\xi_j}a\partial_{x_j}b}\}_{\gamma\geq1}$ is of order $\leq m+m'-2;$
  \item[\rm{(vi)}] G${\rm\mathring a}$rding's inequality{\rm:} If $a\in \Gamma^{2m}_1$ is a square matrix symbol satisfying
      $$\Re a(\mathbf{x},\xi,\gamma)\geq c(\gamma^2+|\xi|^2)^mI \quad\text{ for all } (\mathbf{x},\xi,\gamma)\in\R^6\times[1,\infty),$$
      for some constant $c,$ then there exists $\gamma_0\geq1$ such that
      $$\Re\langle T^{\gamma}_au,u\rangle\geq\frac{c}{4} \|u\|^2_{m,\gamma} \quad\text{ for all }u\in H^m(\R^3) \text{ and } \gamma\geq\gamma_0;$$
  \item[\rm{(vii)}] Microlocalized G${\rm\mathring a}$rding's inequality{\rm:}
      Let $a\in \Gamma^{2m}_1$ be a square matrix symbol and $\chi\in\Gamma^0_1.$ If there exist a scalar real symbol $\tilde{\chi}\in \Gamma^0_1$ and a constant $c>0$ such that $\tilde{\chi}\geq0,$ $\chi\tilde{\chi}\equiv\chi,$ and
      $$\tilde{\chi}^2(\mathbf{x},\xi,\gamma)\Re a(\mathbf{x},\xi,\gamma)\geq c\tilde{\chi}^2(\mathbf{x},\xi,\gamma)(\gamma^2+|\xi|^2)^mI$$
      for all $(\mathbf{x},\xi,\gamma)\in \R^6\times[1,\infty),$ then there exist $\gamma_0\geq 1$ and $C>0$ such that
      $$\Re\langle T^{\gamma}_aT^{\gamma}_{\chi}u,T^{\gamma}_{\chi}u \rangle\geq\frac{c}{2} \|T^{\gamma}_{\chi}u \|^2_{m,\gamma}-C \|u\|^2_{m-1,\gamma}$$
      for all $u\in H^m(\R^3),$ $\gamma\geq\gamma_0.$
      Here, $\Re B :=\frac{B+B^{\ast}}{2}$ for any complex square matrix $B,$ $B^{\ast}$ denotes its conjugate transpose.
\end{enumerate}
\end{lemma}

We remark that the proof of Lemma \ref{lemmapara} {\rm (vii)} can be found in \cite[Theorem B.18]{Meltivier-Zumbrun}.

\smallskip
Similar to the constant coefficient case, the key idea in proving Theorem \ref{vt} is to transform the variable-coefficient linear problem \eqref{le3} into an ODE. However, instead of applying a Fourier transform as in the constant-coefficient case, we employ paralinearization for \eqref{le3}. In the following section, we derive the para-linearized form of \eqref{le3} and estimate the errors introduced by replacing the original system with its para-linearized counterpart.

\subsection{Paralinearization}\label{Paralinearization}
For the frequency space defined in the variable coefficient case:
$$\Pi:= \left\{(\tau,\eta,\tilde{\eta}):\tau=\gamma+i\delta\in \mathbb{C}, \eta, \tilde{\eta}\in \R, |\tau|^2+\eta^2+\tilde{\eta}^2\neq0, \Re\tau\geq0 \right\},$$
where $\delta,\eta,\tilde{\eta}$ represents the Fourier variables with respect to $t,x_1,x_2$ separately.
Using the homogeneity structure of the system, we will focus on the following unit hemisphere in the frequency space
$$\Sigma=\{(\tau,\eta,\tilde{\eta}):|\tau|^2+\eta^2+\tilde{\eta}^2=1, \Re\tau\geq0\}.$$
This argument will later be extended to the entire frequency space $\Pi.$
For simplicity, we omit the tilde notation in the system. We denote
\begin{equation}\nonumber
\begin{split}
&\mathbf{b}_0:=\left[\begin{array}{c}
0\\
1\\
0\\
0\\
0\\
0\\
0\\
0\\
0\\
\end{array}\right]
,\quad \mathbf{b}_1(t,x_1,x_2):=\left[\begin{array}{c}
v^r_1-v^l_1\\
v^r_1\\
0\\
F^r_{11}-F^l_{11}\\
F^r_{11}\\
F^r_{12}-F^l_{12}\\
F^r_{12}\\
F^r_{13}-F^l_{13}\\
F^r_{13}\\
\end{array}
\right](t,x_1,x_2,0),\\
&\mathbf{b}_2(t,x_1,x_2):=\left[\begin{array}{c}
v^r_2-v^l_2\\
v^r_2\\
0\\
F^r_{21}-F^l_{21}\\
F^r_{21}\\
F^r_{22}-F^l_{22}\\
F^r_{22}\\
F^r_{23}-F^l_{23}\\
F^r_{23}\\
\end{array}
\right](t,x_1,x_2,0),\quad \mathbf{b}=\tau \mathbf{b}_0+i\eta \mathbf{b}_1+i\tilde{\eta}\mathbf{b}_2.
\end{split}
\end{equation}
Using Lemma \ref{lemmapara} {\rm (i)-(iii)}, we perform paralinearization and obtain that
\begin{equation}\nonumber
\begin{split}
&\gamma \mathbf{b}_0\psi+\mathbf{b}_0\partial_t\psi=T^{\gamma}_{\tau \mathbf{b}_0}\psi,\\
&\|\mathbf{b}_1\partial_1\psi-T^{\gamma}_{i\eta \mathbf{b}_1}\psi \|_{1,\gamma}\leq C \|\mathbf{b}_1 \|_{W^{2,\infty}(\R^3)} \|\psi \|_0 \leq \frac{C}{\gamma} \|\psi \|_{1,\gamma},\\
&\|\mathbf{b}_2\partial_2\psi-T^{\gamma}_{i\tilde{\eta} \mathbf{b}_2}\psi \|_{1,\gamma}\leq C \|\mathbf{b}_2 \|_{W^{2,\infty}(\R^3)} \|\psi \|_0 \leq \frac{C}{\gamma} \|\psi \|_{1,\gamma},\\
&\|\check{\mathbf{b}}\psi-T^{\gamma}_{\check{\mathbf{b}}}\psi \|_{1,\gamma}\leq C \|\check{\mathbf{b}} \|_{W^{1,\infty}(\R^3)} \|\psi \|_0\leq\frac{C}{\gamma} \|\psi \|_{1,\gamma},\\
&\|T^{\gamma}_{\check{\mathbf{b}}}\psi \|_{1,\gamma}\leq C \|\check{\mathbf{b}} \|_{L^{\infty}(\R^3)} \|\psi \|_{1,\gamma}\leq C \|\psi \|_{1,\gamma},\\
\end{split}
\end{equation}
where $C$ are positive constants. Then, we consider the coefficients of $W^{n},$
\begin{equation}\nonumber\\
\begin{split}
&\underline{M} \text{diag}\{T^r,T^l\}W=: \mathbf{M}W^{n}\\
&=\left[\begin{smallmatrix}
0 & 0 &-m_r & m_r & 0 & 0 & 0 & 0& 0 & 0 & 0 & 0 & 0 & 0 & 0 & m_l & -m_l & 0 & 0 & 0 & 0 & 0 & 0 & 0 & 0 & 0\\
0 & 0 & -m_r &m_r & 0 & 0 & 0 & 0& 0 & 0 & 0 & 0 & 0 & 0 & 0 & 0 & 0 & 0 & 0 & 0 & 0 & 0 & 0 & 0 & 0 & 0\\
0 & 0 & k & k & 0 & 0 & 0 & 0& 0 & 0 & 0 & 0 & 0 & 0 & 0 & -k & -k & 0 & 0 & 0 & 0 & 0 & 0 & 0 & 0 & 0\\
0 & 0 & 0 & 0 & 0 & 0 & -k^2 & 0 & 0 & 0 & 0 & 0 & 0 & 0 & 0 & 0 & 0 & 0 & 0 & k^2 & 0 & 0 & 0 & 0 & 0 & 0\\
0 & 0 & 0 & 0 & 0 & 0 & -k^2 & 0 & 0 & 0 & 0 & 0 & 0 & 0 & 0 & 0 & 0 & 0 & 0 & 0 & 0 & 0 & 0 & 0 & 0 & 0\\
0 & 0 & 0 & 0 & 0 & 0 & 0 & 0 & 0 & -k^2 & 0 & 0 & 0 & 0 & 0 & 0 & 0 & 0 & 0 & 0 & 0 & 0 & k^2 & 0 & 0 & 0\\
0 & 0 & 0 & 0 & 0 & 0 & 0 & 0 & 0 & -k^2 & 0 & 0 & 0 & 0 & 0 & 0 & 0 & 0 & 0 & 0 & 0 & 0 & 0 & 0 & 0 & 0\\
0 & 0 & 0 & 0 & 0 & 0 & 0 & 0 & 0 & 0 & 0 & 0 & -k^2 & 0 & 0 & 0 & 0 & 0 & 0 & 0 & 0 & 0 & 0 & 0 & 0 & k^2\\
0 & 0 & 0 & 0 & 0 & 0 & 0 & 0 & 0 & 0 & 0 & 0 & -k^2 & 0 & 0 & 0 & 0 & 0 & 0 & 0 & 0 & 0 & 0 & 0 & 0 & 0\\
\end{smallmatrix}
\right]W^{n}.
\end{split}
\end{equation}
We denote $m_r=\frac{c^r}{\rho^r}\langle\partial_{\rm tan}\varphi\rangle^2,$ $m_l=\frac{c^l}{\rho^l}\langle\partial_{\rm tan}\varphi\rangle^2,$
$k=\langle\partial_{\rm tan}\varphi\rangle$ and
\begin{equation*}
\| \mathbf{M}W^{n}|_{x_3=0}-T^{\gamma}_{\mathbf{M}}W^{n}|_{x_3=0} \|_{1,\gamma}\leq\frac{C}{\gamma} \| \mathbf{M} \|_{W^{2,\infty}(\R^3)} \| W^{n}|_{x_3=0} \|_0\leq\frac{C}{\gamma} \| W^{n}|_{x_3=0} \|_0,
\end{equation*}
where $C$ are some positive constants. Combing the above estimates, we have
\begin{equation}\label{estimate4}
\| \mathcal{B}^{\gamma}(W^{n}|_{x_3=0},\psi)-T^{\gamma}_\mathbf{b}\psi-T^{\gamma}_{\mathbf{M}}W^{n}|_{x_3=0} \|_{1,\gamma} \leq C(\|\psi \|_{1,\gamma} + \frac{1}{\gamma} \|W^{n}|_{x_3=0}\|_0).
\end{equation}
Next, for the interior differential equations, using Lemma \ref{lemmapara} {\rm (i)-(ii)}, we have
\begin{equation*}
\begin{split}
\vertiii{\gamma A^r_0W^+-T^{\gamma}_{\gamma A^r_0}W^+}^2_{1,\gamma} & =\int^{\infty}_0\gamma^2 \| A^r_0W^+(\cdot,x_3)-T^{\gamma}_{A^r_0}W^+(\cdot,x_3) \|^2_{1,\gamma}\,\dd x_3\\
& \leq C \|A^r_0 \|^2_{W^{2,\infty}(\Omega)} \vertiii{W^+}^2_0\leq C \vertiii{W^+}^2_0, \\
\vertiii{A^r_0\partial_tW^+-T^{\gamma}_{i\delta A^r_0}W^+}_{1,\gamma} & \leq C\vertiii{W^+}_0, \\
\vertiii{A^r_1\partial_1W^+-T^{\gamma}_{i\eta A^r_1}W^+}_{1,\gamma} & \leq C\vertiii{W^+}_0, \\
\vertiii{A^r_2\partial_2W^+-T^{\gamma}_{i\tilde{\eta} A^r_2}W^+}_{1,\gamma} & \leq C\vertiii{W^+}_0, \\
\vertiii{A^r_0C^rW^+-T^{\gamma}_{A^r_0C^r}W^+}_{1,\gamma} & \leq C\vertiii{W^+}_0.
\end{split}
\end{equation*}
Similar estimates also holds for $W^-$. Hence, we obtain that
\begin{equation}\label{estimate5}
\vertiii{\mathcal{L}^{\gamma}_{r,l}W^{\pm}-T^{\gamma}_{\tau A^{r,l}_0+i\eta A^{r,l}_1+i\tilde{\eta}A^{r,l}_2+A^{r,l}_0C^{r,l}}W^{\pm}-I_2\partial_3W^{\pm}}_{1,\gamma}\leq C \vertiii{W^{\pm}}_{0}.
\end{equation}
Now, we start to derive the specific expression for the paralinearized system.
We write
\begin{equation}\label{def tildeF}
\mathbf{b}=\begin{bmatrix}
i\eta(v^r_1-v^l_1)+i\tilde{\eta}(v^r_2-v^l_2)\\
\tau+i\eta v^r_1+i\tilde{\eta}v^r_2\\
0\\
i\eta(F^r_{11}-F^l_{11})+i\tilde{\eta}(F^r_{21}-F^l_{21})\\
i\eta F^r_{11}+i\tilde{\eta}F^r_{21}\\
i\eta(F^r_{12}-F^l_{12})+i\tilde{\eta}(F^r_{22}-F^l_{22})\\
i\eta F^r_{12}+i\tilde{\eta}F^r_{22}\\
i\eta(F^r_{13}-F^l_{13})+i\tilde{\eta}(F^r_{23}-F^l_{23})\\
i\eta F^r_{13}+i\tilde{\eta}F^r_{23}\\
\end{bmatrix}:=\begin{bmatrix}
ib\\
a\\
0\\
iF_a\\
iF_b\\
iF_c\\
iF_d\\
iF_e\\
iF_f
\end{bmatrix}.
\end{equation}
Using the fact that $\mathrm{F}_1\times\mathrm{F}_2\neq\mathbf{0},$ we have
\begin{align}\label{bfront}
&|\mathbf{b}(t,x_1,x_2,\delta,\eta,\tilde{\eta},\gamma)|^2\nonumber\\
&= |\eta(v^r_1-v^l_1)+\tilde{\eta}(v^r_2-v^l_2)|^2 +\gamma^2+|\delta+\eta v^r_1+\tilde{\eta}v^r_2|^2\nonumber\\
&\quad+|\eta(F^r_{11}-F^l_{11})+\tilde{\eta}(F^r_{21}-F^l_{21})|^2+|\eta F^r_{11}+\tilde{\eta}F^r_{21}|^2\nonumber\\
&\quad+|\eta(F^r_{12}-F^l_{12})+\tilde{\eta}(F^r_{22}-F^l_{22})|^2+|\eta F^r_{12}+\tilde{\eta}F^r_{22}|^2\nonumber\\
&\quad+|\eta(F^r_{13}-F^l_{13})+\tilde{\eta}(F^r_{23}-F^l_{23})|^2+|\eta F^r_{13}+\tilde{\eta}F^r_{23}|^2\nonumber\\
&\geq C(\gamma^2+\delta^2+\eta^2+\tilde{\eta}^2),
\end{align}
for some positive constant $C.$
Then, using G${\rm \mathring a}$rding's inequality (Lemma \ref{lemmapara} {\rm (vi)}), we obtain that
$$\Re\langle T^{\gamma}_{\mathbf{b}^{\ast}\mathbf{b}}\psi,\psi\rangle_{L^2(\R^3)}\geq\frac{c}{2} \| \psi \|^2_{1,\gamma},$$
for all $\gamma\geq\gamma_0,$ where $\gamma_0$ depends only on $K.$ Using the properties of paradifferential operators (Lemma \ref{lemmapara} {\rm (iv)}), we obtain that
$T^{\gamma}_{\mathbf{b}^{\ast}\mathbf{b}}=(T^{\gamma}_\mathbf{b})^{\ast}T^{\gamma}_\mathbf{b}+\mathcal{R}_1,$ where $\mathcal{R}_1$ is an operator of order 1, then we have
$$
\| \psi \|_{1,\gamma}\leq C \| T^{\gamma}_\mathbf{b}\psi \|_0,
$$
for all $\gamma\geq\gamma_0.$ Using \eqref{estimate4} and Lemma \ref{lemmapara}, we have
\begin{equation}\label{estimate6}
\begin{split}
\| \psi \|_{1,\gamma}&\leq C \left( \left\| T^{\gamma}_\mathbf{b}\psi+T^{\gamma}_{\mathbf{M}}W^{n}|_{x_3=0} \right\|_0 + \left\| T^{\gamma}_{\mathbf{M}}W^{n}|_{x_3=0} \right\|_0 \right)\\
&\leq C \left( \frac{1}{\gamma}\left\| T^{\gamma}_\mathbf{b}\psi+T^{\gamma}_{\mathbf{M}}W^{n}|_{x_3=0} \right\|_{1,\gamma} + \| W^{n}|_{x_3=0} \|_0 \right)\\
&\leq C\left( \frac{1}{\gamma} \left\| \mathcal{B}^{\gamma}(W^{n}|_{x_3=0},\psi)-T^{\gamma}_\mathbf{b}\psi-T^{\gamma}_{\mathbf{M}}W^{n}|_{x_3=0} \right\|_{1,\gamma} \right. \\
&\qquad \  \left. +\frac{1}{\gamma} \left\| \mathcal{B}^{\gamma}(W^{n}|_{x_3=0},\psi) \right\|_{1,\gamma}+\| W^{n}|_{x_3=0} \|_0 \right)\\
&\leq C\left( \frac{1}{\gamma}\left\| \mathcal{B}^{\gamma}(W^{n}|_{x_3=0},\psi) \right\|_{1,\gamma}+\| W^{n}|_{x_3=0} \|_0 \right).
\end{split}
\end{equation}

Now, we want to analyze the part of the boundary condition where the front function $\psi$ is not involved.
We recall from \eqref{def tildeF} that 
$$
a =\tau+iv^r_1\eta+iv^r_2\tilde{\eta},\qquad b =\eta(v^r_1-v^l_1)+\tilde{\eta}(v^r_2-v^l_2).
$$
It is worth pointing out that $a$ and $b$ could equal to 0 for 3D elastic flow, due to the possible frequency interactions. However, in 2D elastic flow, it is natural to obtain the ellipticity for the front symbol and no such kind of degeneracy of $a$ and $b$ happen.
Now, we consider the following matrix:
\begin{equation}\label{PI}
\mathbb{P}(t,x_1,x_2,\tau,\eta,\tilde{\eta}):=\left[\begin{array}{ccccccccc}
0 & 0 & 1 & 0 & 0 & 0 & 0 & 0 & 0\\
a & -ib & 0 & 0 & 0 & 0 & 0 & 0 & 0 \\
-\tilde{F}_a & 0 & 0 & b_1 & 0 & 0 & 0 & 0 & 0\\
-\tilde{F}_b & 0 & 0 & 0 & b_2 & 0 & 0 & 0 & 0\\
-\tilde{F}_c & 0 & 0 & 0 & 0 & b_3 & 0 & 0 & 0\\
-\tilde{F}_d & 0 & 0 & 0 & 0 & 0 & b_4 & 0 & 0\\
-\tilde{F}_e & 0 & 0 & 0 & 0 & 0 & 0 & b_5 & 0\\
-\tilde{F}_f & 0 & 0 & 0 & 0 & 0 & 0 & 0 & b_6\\
\end{array}\right],
\end{equation}
for $(\tau,\eta,\tilde{\eta})\in \Sigma$, and extend it as a homogeneous mapping of degree of 0 with respect to $(\tau,\eta,\tilde{\eta}).$ Here $\tilde F_a, \ldots, \tilde F_f$, and $b_1, \ldots, b_6$ will be determined microlocally as below. It is easy to check that
\begin{align}\nonumber
\mathbb{P} \mathbf{b}=\begin{bmatrix}
0\\
0\\
i(-\tilde{F}_ab+F_ab_1)\\
i(-\tilde{F}_b b+F_bb_2)\\
i(-\tilde{F}_c b+F_cb_3)\\
i(-\tilde{F}_d b+F_db_4)\\
i(-\tilde{F}_e b+F_eb_5)\\
i(-\tilde{F}_f b+F_fb_6)\\
\end{bmatrix}.
\end{align}
The construction of $\mathbb{P}$ is to be understood \emph{microlocally}. The
cancellations producing $\mathbb{P} \mathbf{b} = \mathbf{0}$ amount to identities of the form
$-\tilde F_\bullet\, b+F_\bullet\, b_j=0$, and hence involve a division by
one of the elastic components $F_a, F_c, F_e$. Since $\mathrm{F}_1\times \mathrm{F}_2\neq\mathbf{0}$, these three components never vanish simultaneously for $(\eta,\tilde\eta)\neq
(0,0)$, so at every point of $\Sigma$ at least one of them is non-zero; on the
other hand, it is not clear that any of them remains bounded away from zero globally on all of $\Sigma$. Therefore $\mathbb{P}$, and likewise the symbols $\mathbb{P} \mathbf{M}$, $\mathbb{P}_6$, the auxiliary matrix $B$ and the cut-off $\chi_B$ derived from it below, are defined microlocally: near each direction one divides by whichever of $F_a,F_c,F_e$ is non-vanishing there, and the resulting local symbols, all of class $\Gamma^0_2$, are patched together. On the pure temporal region $(\eta,\tilde{\eta})=(0,0)$, it is immediately seen from \eqref{def tildeF} that $\mathbb{P} \mathbf{b}=\mathbf{0}$.

Simple calculation yields that
\begin{equation*}
\begin{split}
&\mathbb{P} \mathbf{M}=\left[\begin{smallmatrix}
k& k & 0 & 0 & 0 & -k & -k & 0 & 0 & 0 \\
-(a-ib)m_r & (a-ib)m_r & 0 & 0 & 0 & am_l & -am_l & 0 & 0 & 0 \\
\tilde{F}_a m_r & -\tilde{F}_am_r & -b_1 k^2 & 0 & 0 & -\tilde{F}_am_l & \tilde{F}_a m_l & b_1 k^2 & 0 & 0\\
F_bm_r & -F_bm_r & bk^2 & 0 & 0 & -F_bm_l & F_bm_l & 0 & 0 & 0\\
F_cm_r & -F_cm_r & 0 & -bk^2 & 0 & -F_cm_l & F_cm_l & 0 & bk^2 & 0\\
F_dm_r & -F_dm_r & 0 & -bk^2 & 0 & -F_dm_l & F_dm_l & 0 & 0 & 0\\
F_em_r & -F_em_r & 0 & 0 & -bk^2 & -F_em_l& F_em_l & 0 & 0 & bk^2\\
F_fm_r & -F_fm_r & 0 & 0 & -bk^2 & -F_fm_l& F_fm_l & 0 & 0 & 0\\
\end{smallmatrix}\right]
\end{split}
\end{equation*}
for $(\tau,\eta,\tilde{\eta})\in \Sigma$. It is easily seen that $\mathbb{P} \mathbf{M}$ is homogeneous of degree $0$ with respect to $(\tau,\eta,\tilde{\eta}).$

We denote the last six rows in $\mathbb{P}$ by $\mathbb{P}_6.$ Then we can separate the $\mathbb{P}\mathbf{M}$. We have
 $$
 T^{\gamma}_{\mathbb{P}_6\mathbf{M}}W^n|_{x_3=0}=T^{\gamma}_{A}W^{nc}|_{x_3=0}+T^{\gamma}_{B}(W_7,W_{10},W_{13},W_{20},W_{23},W_{26})^{\top}|_{x_3=0},
 $$
{where $B$ is a $6\times 6$ matrix symbol homogeneous of degree $0$. Note that in general $B$ is \emph{not} invertible on the whole $\Sigma$: a direct computation gives $\det B=b^5 b_1$, hence $B$ degenerates on $\{b=0\}\cup\{b_1=0\}$. Therefore we will only use $B$ on a microlocal region where it is elliptic, and we will apply the microlocal G${\rm\mathring a}$rding inequality in Lemma \ref{lemmapara}(vii).
For simplicity, we write
\begin{align*}
B=\begin{bmatrix}
-b_1 & 0 & 0 & b_1 & 0 & 0\\
b & 0 & 0 & 0 & 0 & 0\\
0 & -b & 0 & 0 & b& 0\\
0 & -b& 0 & 0 & 0 & 0\\
0 & 0 & -b & 0 & 0 & b\\
0 & 0 & -b & 0 & 0 & 0
\end{bmatrix}.
\end{align*}
Let $\mathcal{W}:=(W_7,W_{10},W_{13},W_{20},W_{23},W_{26})^{\top}$.  Choose a cut-off symbol $\chi_B\in\Gamma^0_k$ supported in an open set where $|b|\ge\delta_B$ and $|b_1|\ge\delta_B$ for some $\delta_B>0$.
Then, since the coefficients are bounded by a constant depending only on $K$ ({\it cf.} \eqref{pp}) and $\chi_B$ is supported away from $\{b=0\}\cup\{b_1=0\}$, there exists $c_B=c_B(\delta_B,K)>0$ such that
\[
\chi_B^2\, B^{\ast}B\ \ge\ c_B\, \chi_B^2\,I\qquad\text{on }\mathbb{R}^3_{(t,x_1, x_2)}\times\Sigma.
\]
Using the microlocal G${\rm\mathring a}$rding inequality (Lemma \ref{lemmapara} {\rm (vii)}) with $m=0$, $a=B^{\ast}B$ and $\chi=\chi_B$, we obtain
\begin{align}\label{Blower}
& \Re \left\langle T^{\gamma}_{B^{\ast}B}T^{\gamma}_{\chi_B}\mathcal{W},\, T^{\gamma}_{\chi_B}\mathcal{W} \right\rangle \Big|_{x_3=0}
\geq \frac{c_B}{2}\left\| T^{\gamma}_{\chi_B}\mathcal{W}|_{x_3=0} \right\|^2_0
- C\left\| \mathcal{W}|_{x_3=0} \right\|^2_{-1,\gamma}
\end{align}
for all $\gamma\geq \gamma_0$. In particular, since $\|\mathcal{W}|_{x_3=0}\|_{-1,\gamma}\lesssim \gamma^{-1}\|\mathcal{W}|_{x_3=0}\|_0$, for $\gamma$ large enough we have
\[
\left\| T^{\gamma}_{\chi_B} \mathcal{W}|_{x_3=0} \right\|_0
\ \lesssim\ \left\| T^{\gamma}_{B}T^{\gamma}_{\chi_B} \mathcal{W}|_{x_3=0} \right\|_0+\frac{1}{\gamma}\left\| \mathcal{W}|_{x_3=0} \right\|_0.
\]
Then, using \eqref{Blower} and applying $T^{\gamma}_{\chi_B}$ to the decomposition of $T^{\gamma}_{\mathbb{P}_6\mathbf{M}}W^n|_{x_3=0}$, we have
\begin{align*}
\left\| T^{\gamma}_{\chi_B}\mathcal{W}|_{x_3=0} \right\|_0
\lesssim \left\| T^{\gamma}_{\chi_B}T^{\gamma}_{\mathbb{P}_6\mathbf{M}}W^n|_{x_3=0} \right\|_0 + \left\| T^{\gamma}_{\chi_B}T^{\gamma}_{A}W^{nc}|_{x_3=0} \right\|_0 + \frac{1}{\gamma}\left\| \mathcal{W}|_{x_3=0} \right\|_0.
\end{align*}}
Using Lemma \ref{lemmapara} and \eqref{estimate4}, we can also obtain the estimate for the front function $\psi,$ one has
\begin{align*}
\left\| T^{\gamma}_{\mathbb{P}_6\mathbf{M}}W^n|_{x_3=0} \right\|_0 & \leq \left\| T^{\gamma}_{\mathbb{P}_6}\mathcal{B}^{\gamma}(W^n|_{x_3=0},\psi)-T^{\gamma}_{\mathbb{P}_6}T^{\gamma}_{\mathbf{M}}W^n|_{x_3=0}-T^{\gamma}_{\mathbb{P}_6}T^{\gamma}_\mathbf{b}\psi \right\|_0 \\
&\quad + \left\| T^{\gamma}_{\mathbb{P}_6}\mathcal{B}^{\gamma}(W^n|_{x_3=0},\psi) \right\|_0 + \left\| W^n|_{x_3=0} \right\|_{-1,\gamma}+\| \psi \|_0 \\
& \leq \frac{1}{\gamma} \left\| \mathcal{B}^{\gamma}(W^n|_{x_3=0},\psi)-T^{\gamma}_{\mathbf{M}}W^n|_{x_3=0}-T^{\gamma}_{\mathbf{b}} \psi\right\|_{1,\gamma}+\frac{1}{\gamma} \left\| \mathcal{B}^{\gamma}(W^n|_{x_3=0},\psi) \right\|_{1,\gamma} \\
&\quad + \left\| W^n|_{x_3=0} \right\|_{-1,\gamma} + \| \psi \|_0 \\
&\leq\frac{1}{\gamma}\| \psi \|_{1,\gamma} + \frac{1}{\gamma}\left\| \mathcal{B}^{\gamma}(W^n|_{x_3=0},\psi) \right\|_{1,\gamma}+\frac{1}{\gamma}\left\| W^n|_{x_3=0} \right\|_0.
\end{align*}
Taking $\gamma$ sufficiently large, we obtain 
\begin{equation}\label{We}
\begin{split}
\left\| T^{\gamma}_{\chi_B}\mathcal{W}|_{x_3=0} \right\|_0 & \lesssim\left( \frac{1}{\gamma}\| \psi \|_{1,\gamma}+\frac{1}{\gamma}\left\| \mathcal{B}^{\gamma}(W^n|_{x_3=0},\psi) \right\|_{1,\gamma} + \left\| W^{nc}|_{x_3=0} \right\|_{0} \right) \\
& \quad + \frac{1}{\gamma}\left\| T^{\gamma}_{1-\chi_B}\mathcal{W}|_{x_3=0} \right\|_0.
\end{split}
\end{equation}
Indeed, compared with the ``global'' estimate in the non-degenerate case, the microlocal G${\rm\mathring a}$rding inequality (Lemma \ref{lemmapara} (vii)) yields an $H^{-1}$ remainder. Using symbolic calculus, one has
\[
\|\mathcal{W}|_{x_3=0}\|_0\ \lesssim\ \|T^{\gamma}_{\chi_B} \mathcal{W}\|_0+\|T^{\gamma}_{1-\chi_B} \mathcal{W}\|_0,
\]
and hence the lower-order term can be split into a part supported on $\supp\chi_B$ (absorbed into the left-hand side for $\gamma$ large) and a complementary part supported on $\supp(1-\chi_B)$, which is exactly the last term of \eqref{We}.

Combining \eqref{estimate6} with the microlocal estimate \eqref{We} gives, on $\supp\chi_B$,
\begin{equation}\label{Wephi}
\begin{split}
\left\| T^{\gamma}_{\chi_B} \mathcal{W}|_{x_3=0} \right\|_0+\| \psi \|_{1,\gamma} & 
\lesssim \frac{1}{\gamma}\left\| \mathcal{B}^{\gamma}(W^n|_{x_3=0},\psi) \right\|_{1,\gamma}+\left\| W^{nc}|_{x_3=0} \right\|_0 \\
& \quad +\ \frac{1}{\gamma}\left\| T^{\gamma}_{1-\chi_B} \mathcal{W}|_{x_3=0} \right\|_0.
\end{split}
\end{equation}

\begin{remark}[Control of $T^\gamma_{1-\chi_B} \mathcal{W}$ by the case estimates]\label{rem:control-1mchiB}
The last term in \eqref{Wephi} reflects the possible degeneracy of the auxiliary matrix $B$ on $\supp(1-\chi_B)$
({\it e.g.} when $b=0$ or $b_1=0$, so that $\det B=b^5b_1 = 0$).
To control $T^\gamma_{1-\chi_B} \mathcal{W}$, we use the microlocal partition of unity on $\Sigma$ introduced in Sections \ref{subsec est on each}--\ref{subsec proof}:
there exist smooth homogeneous cut-offs $\bar{\chi}_{p_1},\bar{\chi}_{p_2},\bar{\chi}_{rt},\chi_{re}$ such that
\begin{equation}\label{chi.partition}
\bar{\chi}_{p_1}+\bar{\chi}_{p_2}+\bar{\chi}_{rt}+\chi_{re}\equiv 1\qquad\text{on }\Sigma.
\end{equation}
Choose them so that $\supp\chi_B\subset\supp\chi_{re}$ and set $\chi_{re,B}:=(1-\chi_B)\chi_{re}$.
Then
\begin{equation}\label{chiB.decomp}
1-\chi_B=\chi_{re,B}+\bar{\chi}_{p_1}+\bar{\chi}_{p_2}+\bar{\chi}_{rt}\qquad\text{on }\Sigma,
\end{equation}
and consequently, up to lower-order commutators,
\[
T^\gamma_{1-\chi_B}\mathcal{W}
=
T^\gamma_{\chi_{re,B}}\mathcal{W}
+T^\gamma_{\bar{\chi}_{p_1}}\mathcal{W}
+T^\gamma_{\bar{\chi}_{p_2}}\mathcal{W}
+T^\gamma_{\bar{\chi}_{rt}}\mathcal{W}
+\mathcal{R}_{-1}\mathcal{W}.
\]

We now state the boundary-trace estimates which will be proved in each microlocal region, together with the corresponding change of unknowns:
\begin{enumerate}[leftmargin=*, label=(\arabic*)]
\item On $\supp\bar{\chi}_{p_1}$ (Case~1, poles possibly \emph{coinciding} for $W^+$ and $W^-$), we introduce the transformed unknowns
$Z^+$ and $Z^-$ by \eqref{def:Zplus.p1} and \eqref{def:Zminus.p1}, where the leading-order transformation matrix $Q^r_0$ and the
symmetrizer $R^r_0$ are defined by \eqref{def:Q0r.p1} and \eqref{def:R0r.p1}.
The boundary reduction uses the matrices $P_1,P_2$ in \eqref{P1P2} (hence $\beta_{\mathrm{in}}$ in \eqref{beta}) to prove the key boundary estimate \eqref{para28}, and combining it with the interior energy estimates yields \eqref{para30}.

\item On $\supp\bar{\chi}_{rt}$ (Case~2, roots), we define the microlocal variables by \eqref{def:Zplus.rt} and \eqref{def:Zminus.rt}.
The associated matrices $Q^{r,l}_0,Q^{r,l}_{-1}$ and symmetrizers $R^{r,l}_0,R^{r,l}_{-1}$ are constructed exactly as in Case~1, but the boundary reduction is {\it not} elliptic on this region: the boundary symbol is only weakly stable on $\supp\bar{\chi}_{rt}$. However, the stable eigenvectors are still smooth because $\supp\bar{\chi}_{rt}$ is disjoint from the pole set $\Upsilon_p$, which allows us to use the localized G$\rm\mathring{a}$rding inequality to obtain \eqref{para35}.

\item On $\supp\bar{\chi}_{p_2}$ (Case 3, poles away from roots), we use the change of unknowns \eqref{def:Zplus.p2} and \eqref{def:Zminus.p2}
(again built from the same $Q$/$R$ construction), and the pole analysis gives \eqref{para37}.
\item On $\supp\chi_{re,B}$ (Other case, which corresponds to a regular region away from poles and roots), we work directly with $W^{nc}_{re}$ solving \eqref{para38}; the Kreiss-type symmetrizer yields \eqref{para39}, which controls $\|W^{nc}_{re}|_{x_3=0}\|_{1,\gamma}$ and therefore $\|T^\gamma_{\chi_{re,B}}\mathcal{W}\|_0$.
\end{enumerate}

In the singular microlocal regions ($\supp\bar{\chi}_{p_1}$, $\supp\bar{\chi}_{rt}$ and $\supp\bar{\chi}_{p_2}$), the change of unknowns is elliptic on the support of the corresponding cut-off:
for instance, on $\supp\bar{\chi}_{p_1}$ the symbols $\chi_1Q^r_0$ and $\chi_1R^r_0$ are invertible (see \eqref{def:Q0r.p1}--\eqref{def:R0r.p1}),
hence Lemma \ref{lemmapara} (iv) yields a parametrix and we may pass between traces of $W$ and traces of $Z$ up to $\mathcal{R}_{-1}$ remainders.
Since $\mathcal{W}$ is a fixed subvector of $W^{n}|_{x_3=0}$, we have $\|T^\gamma_{\chi}\mathcal{W}\|_0\le C\|T^\gamma_{\chi}W^{n}|_{x_3=0}\|_0$ for each cut-off $\chi$.
Therefore, \eqref{para30}, \eqref{para35}, \eqref{para37} and \eqref{para39} provide a bound for $\|T^\gamma_{1-\chi_B}\mathcal{W}\|_0$,
and after summing the microlocal inequalities associated with \eqref{chi.partition} and taking $\gamma$ sufficiently large,
all $H^{-1}$ remainders (including the last term in \eqref{Wephi}) can be absorbed.
\end{remark}

Denote that
\begin{equation}\label{def k1}
k^{r,l}_1:=\tau+i\eta v^{r,l}_1+i\tilde{\eta}v^{r,l}_2.
\end{equation}
Now we estimate $W^{nc}|_{x_3=0}$ using the remaining boundary conditions,
$$T^{\gamma}_{\beta}W^{nc}|_{x_3=0}=\tilde{G},$$
where
\begin{equation}\label{def beta}
\begin{split}
&\beta := \left[\begin{matrix}
k & k & -k & -k\\\\
\displaystyle -\frac{c^r}{\rho^r}k^2k^l_1 & \displaystyle \frac{c^r}{\rho^r}k^2k^l_1  & \displaystyle \frac{c^l}{\rho^l}k^2k^r_1 & \displaystyle -\frac{c^l}{\rho^l}k^2k^r_1 \\
\end{matrix}\right].
\end{split}
\end{equation}
Note that $\beta \in \Gamma^0_2.$
We define the para-linearized system
\begin{equation}\label{para}
\begin{cases}
T^{\gamma}_{\tau A^r_0+i\eta A^r_1+i\tilde{\eta}A^r_2+A^r_0C^r}W^++I_2\partial_3W^+=\tilde{F}^+,\\
T^{\gamma}_{\tau A^l_0+i\eta A^l_1+i\tilde{\eta}A^l_2+A^l_0C^l}W^-+I_2\partial_3W^-=\tilde{F}^-,\\
T^{\gamma}_{\beta}W^{nc}|_{x_3=0}=\tilde{G}.
\end{cases}
\end{equation}
We remain to prove the following estimate:
\begin{equation}\label{estimate7}
\left\| W^{nc}|_{x_3=0} \right\|^2_0\leq C_0 \left(\frac{1}{\gamma^3} \left\| \tilde{F}^{\pm} \right\|^2_{1,\gamma}+\frac{1}{\gamma^2} \|\tilde{G}\|^2_{1,\gamma} \right).
\end{equation}

Once \eqref{estimate7} is obtained, we insert it into the microlocal trace bounds for $(W^n|_{x_3=0},\psi)$: on $\supp\chi_B$ we use \eqref{Wephi}, and on the complementary cut-offs we use the boundary estimates proved in Sections \ref{subsec est on each}--\ref{subsec proof}. Summing by a partition of unity on $\Sigma$ and absorbing lower-order remainders for $\gamma$ large yields the full-frequency estimate for $(W^n|_{x_3=0},\psi)$ required in Theorem~\ref{vt}.

First, we have
\begin{equation*}\nonumber
\begin{split}
&\vertiii{\mathcal{L}^{\gamma}_{r,l}W^{\pm}-T^{\gamma}_{\tau A^{r,l}_0+i\eta A^{r,l}_1+i\tilde{\eta} A^{r,l}_2+A^{r,l}_0 C^{r,l}}W^{\pm}-I_2\partial_3W^{\pm}}_{1,\gamma}\leq C\vertiii{W^{\pm}}_{0},\\
&\left\| \mathcal{B}^{\gamma}(W^{n}|_{x_3=0},\psi)-T^{\gamma}_{\mathbf b}\psi-T^{\gamma}_{\mathbf{M}}W^{n}|_{x_3=0} \right\|_{1,\gamma }\leq C \left(\| \psi \|_{1,\gamma}+\frac{1}{\gamma}\left\| W^{n}|_{x_3=0} \right\|_0 \right),\\
&\|\tilde{G}\|_{1,\gamma} = \left\|T^{\gamma}_{\beta}W^{nc}|_{x_3=0} \right\|_{1,\gamma} = \left\|T^{\gamma}_{\mathbb{P}_2 \mathbf{M}}W^{nc}|_{x_3=0}\right\|_{1,\gamma}
=\left\|T^{\gamma}_{\mathbb{P}_2 {\mathbf b}}\psi+T^{\gamma}_{\mathbb{P}_2 \mathbf{M}}W^{nc}|_{x_3=0} \right\|_{1,\gamma},
\end{split}
\end{equation*}
where $\mathbb{P}_2$ denotes the first two rows of $\mathbb{P}.$
Then
\begin{equation}\nonumber
\begin{split}
\|\tilde{G}\|_{1,\gamma}&\leq\left\|T^{\gamma}_{\mathbb{P}_2}(T^{\gamma}_{\mathbf b}\psi+T^{\gamma}_{\mathbf{M}}W^{nc}|_{x_3=0})\right\|_{1,\gamma}+\| \psi \|_{1,\gamma}+\left\| W^{nc}|_{x_3=0} \right\|_0\\
&\leq\left\|\mathcal{B}^{\gamma}(W^{n}|_{x_3=0},\psi)-T^{\gamma}_{\mathbf b}\psi-T^{\gamma}_{\mathbf{M}}W^{n}|_{x_3=0}\right\|_{1,\gamma}
+\left\| \mathcal{B}^{\gamma}(W^{n}|_{x_3=0},\psi) \right\|_{1,\gamma}+\| \psi \|_{1,\gamma}+\left\| W^{nc}|_{x_3=0} \right\|_0\\
&\leq \| \psi \|_{1,\gamma}+\frac{1}{\gamma}\| W^{n}|_{x_3=0} \|_0+\left\| W^{nc}|_{x_3=0} \right\|_0+\left\| \mathcal{B}^{\gamma}(W^{n}|_{x_3=0},\psi) \right\|_{1,\gamma}.
\end{split}
\end{equation}
Moreover,
\begin{equation*}
\begin{split}
\vertiii{ \tilde{F}^{\pm} }_{1,\gamma} & \lesssim \vertiii{\mathcal{L}^{\gamma}_{r,l}W^{\pm}}_{1,\gamma} + \vertiii{\tilde{F}^{\pm}-\mathcal{L}^{\gamma}_{r,l}W^{\pm}}_{1,\gamma}\\
&\leq \vertiii{\mathcal{L}^{\gamma}_{r,l}W^{\pm}}_{1,\gamma} + C \vertiii{W^{\pm}}_0.
\end{split}
\end{equation*}
Therefore, from \eqref{estimate7}, we obtain
\begin{align*}
\left\| W^{nc}|_{x_3=0} \right\|^2_0 & \leq C_0 \left( \frac{1}{\gamma^3}\vertiii{\mathcal{L}^{\gamma}_{r,l}W^{\pm}}^2_{1,\gamma} + \frac{1}{\gamma^3}\vertiii{W^{\pm}}^2_0+\frac{1}{\gamma^2}\left\| \mathcal{B}^{\gamma}(W^n|_{x_3=0},\psi) \right\|^2_{1,\gamma} \right. \\
& \qquad \ + \left. \frac{1}{\gamma^4}\left\| W^n|_{x_3=0} \right\|^2_0+\frac{1}{\gamma^2}\| \psi \|^2_{1,\gamma} \right).
\end{align*}

Finally, combining the above bounds with the microlocal trace inequalities for $(W^n|_{x_3=0},\psi)$ ({\it i.e.,} \eqref{Wephi} on $\supp\chi_B$ and the estimates in Sections \ref{subsec est on each}--\ref{subsec proof} on the remaining cut-offs) and summing by a partition of unity on $\Sigma$, we obtain that for $\gamma$ large
\begin{equation*}
\begin{split}
\| W^{n}|_{x_3=0} \|^2_0+\| \psi \|^2_{1,\gamma}
&\leq C \left( \frac{1}{\gamma^3}\vertiii{\mathcal{L}^{\gamma}_{r,l}W^{\pm}}^2_{1,\gamma}+\frac{1}{\gamma^3}\vertiii{W^{\pm}}^2_0 + \frac{1}{\gamma^2}\left\| \mathcal{B}^{\gamma}(W^n|_{x_3=0},\psi) \right\|^2_{1,\gamma} \right. \\
&\qquad \ \left. + \frac{1}{\gamma^4}\left\| W^n|_{x_3=0} \right\|^2_0+\frac{1}{\gamma^2}\| \psi \|^2_{1,\gamma} \right) \\
&\leq C \left(\frac{1}{\gamma^3} \vertiii{\mathcal{L}^{\gamma}_{r,l}W^{\pm} }^2_{1,\gamma}+\frac{1}{\gamma^3}\vertiii{W^\pm}^2_0 + \frac{1}{\gamma^2}\left\| \mathcal{B}^{\gamma}(W^{n}|_{x_3=0},\psi) \right\|^2_{1,\gamma}\right),
\end{split}
\end{equation*}
where in the last step we have absorbed the terms $\frac{1}{\gamma^4}\|W^n|_{x_3=0}\|_0^2$ and $\frac{1}{\gamma^2}\|\psi\|_{1,\gamma}^2$ into the left-hand side by choosing $\gamma$ large enough. Therefore, the key step is to obtain \eqref{estimate7} from \eqref{para}. Using Lemma~\ref{estimate3} together with the microlocal estimates in the next Sections, we obtain Theorem \ref{vt}.

\subsection{Microlocalization}\label{microlocalization}

For simplicity, we concentrate our analysis on the unit hemisphere $\Sigma=\{(\tau,\eta,\tilde{\eta}):|\tau|^2+\eta^2+\tilde{\eta}^2=1 \text{ and } \Re\tau\geq0 \}.$
\subsubsection{Poles}
Considering the following differential equation:
\begin{equation}\label{para2}
\begin{cases}
(\tau A^r_0+i\eta A^r_1+i\tilde{\eta}A^r_2)W^++I_2\partial_3W^+={\mathbf 0},\\
(\tau A^l_0+i\eta A^l_1+i\tilde{\eta}A^l_2)W^-+I_2\partial_3W^-={\mathbf 0},\\
\beta W^{nc}|_{x_3=0}={\mathbf 0},
\end{cases}
\end{equation}
where $A^{r,l}_0, A^{r,l}_1, A^{r,l}_2, I_2,\beta$ are defined in \eqref{def A0}, \eqref{def I2}, \eqref{def A1A2}, and \eqref{def beta}.  Denote
\begin{equation}\label{def k2}
k^{r,l}_2 := (\eta F^{r,l}_{11}+\tilde{\eta} F^{r,l}_{21})^2+(\eta F^{r,l}_{12}+\tilde{\eta} F^{r,l}_{22})^2+(\eta F^{r,l}_{13}+\tilde{\eta} F^{r,l}_{23})^2.
\end{equation}
Now, we consider the algebraic equations $\mathcal{T}W^+={\mathbf 0}$ for $W^+$, where
\begin{equation}
\mathcal{T} := \left[\begin{smallmatrix}
k^r_1 & 0 & \mathcal{T}_{1,3}  & \mathcal{T}_{1,4} & \mathcal{T}_{1,5} & 0 & 0 & \mathcal{T}_{1,8} & 0 & 0 & \mathcal{T}_{1,11} & 0 & 0\\
0 & k^r_1 & \mathcal{T}_{2,3} & \mathcal{T}_{2,4} & 0 & \mathcal{T}_{2,6} & 0 & 0 & \mathcal{T}_{2,9} & 0 & 0 & \mathcal{T}_{2,12} & 0\\
\mathcal{T}_{3,1} & \mathcal{T}_{3,2} & \mathcal{T}_{3,3} & 0 & 0 & 0 & \mathcal{T}_{3,7} & 0 & 0 & \mathcal{T}_{3,10} & 0 & 0 & \mathcal{T}_{3,13}\\
\mathcal{T}_{4,1} & \mathcal{T}_{4,2}  & 0 & \mathcal{T}_{4,4} & 0 & 0 & \mathcal{T}_{4,7} & 0 & 0 & \mathcal{T}_{4,10} & 0 & 0 & \mathcal{T}_{4,13}\\
\mathcal{T}_{5,1} & 0 & 0 & 0 & k^r_1 & 0 & 0 & 0 & 0 & 0 & 0 & 0 & 0\\
0 & \mathcal{T}_{6,2} & 0 & 0 & 0 & k^r_1 & 0 & 0 & 0 & 0 & 0 & 0 & 0\\
0 & 0 & \mathcal{T}_{7,3} & \mathcal{T}_{7,4} &0 & 0 & k^r_1 & 0 & 0 & 0 & 0 & 0 & 0\\
\mathcal{T}_{8,1} & 0 & 0 & 0 & 0 & 0 & 0 & k^r_1 & 0 & 0 & 0 & 0 & 0\\
0 & \mathcal{T}_{9,2} & 0 & 0 & 0 & 0 & 0 & 0 & k^r_1 & 0 & 0 & 0 & 0\\
0 & 0 & \mathcal{T}_{10,3} & \mathcal{T}_{10,4} &0 & 0 & 0 & 0 & 0 & k^r_1 & 0 & 0 & 0\\
\mathcal{T}_{11,1} & 0 & 0 & 0 & 0 & 0 & 0 & 0 & 0 & 0 & k^r_1 & 0 & 0\\
0 & \mathcal{T}_{12,2} & 0 & 0 & 0 & 0 & 0 & 0 & 0 & 0 & 0 & k^r_1 & 0\\
0 & 0 & \mathcal{T}_{13,3} & \mathcal{T}_{13,4} &0 & 0 & 0 & 0 & 0 & 0 & 0 & 0 & k^r_1\\
\end{smallmatrix}
\right],
\end{equation}
where
$$
\mathcal{T}_{1,3}=\mathcal{T}_{1,4}=\frac{c^2}{\rho^r\langle\partial_{\rm tan}\varPhi^r\rangle}\{i\eta[(\partial_2\varPhi^{r})^2+1]-i\tilde{\eta}\partial_1\varPhi^r\partial_2\varPhi^r\},
$$
$$\mathcal{T}_{2,3}=\mathcal{T}_{2,4}=\frac{c^2}{\rho^r\langle\partial_{\rm tan}\varPhi^r\rangle}\{-i\eta\partial_1\varPhi^r\partial_2\varPhi^r+i\tilde{\eta}[(\partial_1\varPhi^{r})^2+1]\},
$$
$$\mathcal{T}_{1,5}=\mathcal{T}_{2,6}=\mathcal{T}_{5,1}=\mathcal{T}_{6,2}=-i\eta F^r_{11}-i\tilde{\eta}F^r_{21},$$
$$\mathcal{T}_{1,8}=\mathcal{T}_{2,9}=\mathcal{T}_{8,1}=\mathcal{T}_{9,2}=-i\eta F^r_{12}-i\tilde{\eta}F^r_{22},$$
$$\mathcal{T}_{1,11}=\mathcal{T}_{2,12}=\mathcal{T}_{11,1}=\mathcal{T}_{12,2}=-i\eta F^{r}_{13}-i\tilde{\eta} F^r_{23},$$
$$\mathcal{T}_{3,7}=\mathcal{T}_{4,7}=-\frac{\rho^r}{2(c^r)^2\langle\partial_{\rm tan}\varPhi^r\rangle}\mathcal{T}_{1,5},$$
$$\mathcal{T}_{3,10}=\mathcal{T}_{4,10}=-\frac{\rho^r}{2(c^r)^2\langle\partial_{\rm tan}\varPhi^r\rangle}\mathcal{T}_{1,8},$$
$$\mathcal{T}_{3,13}=\mathcal{T}_{4,13}=-\frac{\rho^r}{2(c^r)^2\langle\partial_{\rm tan}\varPhi^r\rangle}\mathcal{T}_{1,11},$$
$$\mathcal{T}_{7,3}=-\mathcal{T}_{7,4}=\frac{c^r}{\rho^r}\mathcal{T}_{1,5},\quad\mathcal{T}_{10,3}=-\mathcal{T}_{10,4}=\frac{c^r}{\rho^r}\mathcal{T}_{1,8},\quad\mathcal{T}_{13,3}=-\mathcal{T}_{13,4}=\frac{c^r}{\rho^r}\mathcal{T}_{1,11},$$
$$\mathcal{T}_{3,1}=-\mathcal{T}_{4,1}=\frac{\rho^r\partial_3\varPhi^ri\eta}{2c^r\langle\partial_{\rm tan}\varPhi^r\rangle^2},\quad\mathcal{T}_{3,2}=-\mathcal{T}_{4,2}=\frac{\rho^r \partial_3\varPhi^ri\tilde{\eta}}{2c^r\langle\partial_{\rm tan}\varPhi^r\rangle^2},$$
$$\mathcal{T}_{3,3}=\frac{\partial_3\varPhi^r}{c^r\langle\partial_{\rm tan}\varPhi^r\rangle}\tau+i\eta \partial_3\varPhi^r \left( \frac{-\partial_1\varPhi^r}{\langle\partial_{\rm tan}\varPhi^r\rangle^2}+\frac{v^r_1}{c^r\langle\partial_{\rm tan}\varPhi^r\rangle} \right)+i\tilde{\eta} \partial_3\varPhi^r \left( \frac{-\partial_2\varPhi^r}{\langle\partial_{\rm tan}\varPhi^r\rangle^2}+\frac{v^r_2}{c^r\langle\partial_{\rm tan}\varPhi^r\rangle} \right),$$
$$\mathcal{T}_{4,4}=\frac{\partial_3\varPhi^r}{c^r\langle\partial_{\rm tan}\varPhi^r\rangle}\tau+i\eta \partial_3\varPhi^r \left( \frac{-\partial_1\varPhi^r}{\langle\partial_{\rm tan}\varPhi^r\rangle^2}-\frac{v^r_1}{c^r\langle\partial_{\rm tan}\varPhi^r\rangle} \right) + i\tilde{\eta} \partial_3\varPhi^r \left( \frac{-\partial_2\varPhi^r}{\langle\partial_{\rm tan}\varPhi^r\rangle^2}-\frac{v^r_2}{c^r\langle\partial_{\rm tan}\varPhi^r\rangle} \right).$$
Similar equations also hold for $W^-.$ Similar to the constant coefficient case, if $(k^r_1)^7((k^r_1)^2+k^r_2)\neq0,$ we can solve
$W_1, W_2, W_5,\cdots,W_{13}$ by $W_3, W_{4},$
Then, using differential equations \eqref{para2}, we can obtain the differential equations only involve $W_3$ and $W_4,$
\begin{equation}\nonumber
\partial_3\left[\begin{array}{c}
W_3\\
W_4
\end{array}
\right]=\mathbb{A}^r
\left[\begin{array}{c}
W_3\\
W_4
\end{array}
\right],
\end{equation}
where
\begin{equation*}
\mathbb{A}^r:=\left[\begin{array}{cc}
\mu^r & -m^r\\
m^r & -\mu^r\\
\end{array}
\right]+i\eta\frac{\partial_1\varPhi^r\partial_3\varPhi^r}{\langle\partial_{\rm tan}\varPhi^r\rangle^2}\left[\begin{array}{cc}
1 & 0 \\
0 & 1\\
\end{array}
\right]
+i\tilde{\eta}\frac{\partial_2\varPhi^r\partial_3\varPhi^r}{\langle\partial_{\rm tan}\varPhi^r\rangle^2}\left[\begin{array}{cc}
1 & 0 \\
0 & 1\\
\end{array}
\right],
\end{equation*}
\begin{equation*}
\begin{split}
\mu^r=-\frac{\partial_3\varPhi^rk^r_1}{c^r\langle\partial_{\rm tan}\varPhi^r\rangle}-\frac{\partial_3\varPhi^rk^r_2}{2\langle\partial_{\rm tan}\varPhi^r\rangle c^rk^r_1}-\frac{\partial_3\varPhi^rc^rk^r_1[(\eta\partial_2\varPhi^r-\tilde{\eta}\partial_1\varPhi^r)^2+\eta^2+\tilde{\eta}^2]}{2\langle\partial_{\rm tan}\varPhi^r\rangle^3[(k^r_1)^2+k^r_2]},\\
m^r=-\frac{\partial_3\varPhi^rk^r_2}{2\langle\partial_{\rm tan}\varPhi^r\rangle c^rk^r_1}+\frac{\partial_3\varPhi^rc^rk^r_1[(\eta\partial_2\varPhi^r-\tilde{\eta}\partial_1\varPhi^r)^2+\eta^2+\tilde{\eta}^2]}{2\langle\partial_{\rm tan}\varPhi^r\rangle^3[(k^r_1)^2+k^r_2]}.\\
\end{split}
\end{equation*}
Similar arguments hold for $W^-.$ Then the points in frequency space at which the system cannot be reduced to the non-characteristic form are the points:
\begin{equation}\label{pole}
(k^{r,l}_1)^7[(k^{r,l}_1)^2+k^{r,l}_2]=0,
\end{equation}
where $k^{r,l}_1$ and $k^{r,l}_2$ are defined in \eqref{def k1} and \eqref{def k2} respectively. These points are exactly the poles of the system \eqref{para2} and we denote
\begin{equation*}
\Upsilon_p := \left\{(t,x_1,x_2,x_3,\tau,\eta,\tilde{\eta})\in \R^4_+\times \Sigma: (k^{r,l}_1)^7[(k^{r,l}_1)^2+k^{r,l}_2]=0 \right\}.
\end{equation*}

\subsubsection{Roots of the Lopatinski$\breve{\mathrm{i}}$ determinant}

Now, we derive the Lopatinski$\breve{\mathrm{i}}$ determinant. We write the eigenvalue of $\mathbb{A}^r$ with negative real part by $\omega^r+i\frac{\partial_1\varPhi^r\partial_3\varPhi^r}{\langle\partial_{\rm tan}\varPhi^r\rangle^2}\eta+i\frac{\partial_2\varPhi^r\partial_3\varPhi^r}{\langle\partial_{\rm tan}\varPhi^r\rangle^2}\tilde{\eta},$ which satisfies
\begin{equation*}
\begin{split}
(\omega^r)^2 & = (\mu^r)^2-(m^r)^2 \\
& =\frac{(\partial_3\varPhi^r)^2}{(c^r)^2\langle\partial_{\rm tan}\varPhi^r\rangle^4} \left\{ \langle\partial_{\rm tan}\varPhi^r\rangle^2 \left( (k^r_1)^2+k^r_2 \right) + (c^r)^2 \left[ (\eta\partial_2\varPhi^r-\tilde{\eta}\partial_1\varPhi^r)^2+\eta^2+\tilde{\eta}^2 \right] \right\}.
\end{split}
\end{equation*}
The corresponding eigenvector is
\begin{equation*}
E^r=\left[\begin{array}{c}
-\alpha^r(\mu^r+\omega^r)\\
-\alpha^rm^r
\end{array}
\right],
\end{equation*}
where $$\alpha^r=k^r_1[(k^r_1)^2+k^r_2].$$

The case is similar for $W^-.$ Denote the eigenvalue of $\mathbb{A}^l$ with negative real part by $\omega^l+i\frac{\partial_1\varPhi^l\partial_3\varPhi^l}{\langle\partial_{\rm tan}\varPhi^l\rangle^2}\eta+i\frac{\partial_2\varPhi^l\partial_3\varPhi^l}{\langle\partial_{\rm tan}\varPhi^l\rangle^2}\tilde{\eta},$
which satisfies
\begin{equation*}
\begin{split}
(\omega^l)^2 & = (\mu^l)^2-(m^l)^2 \\
& = \frac{(\partial_3\varPhi^l)^2}{(c^l)^2\langle\partial_{\rm tan}\varPhi^l\rangle^4} \left\{ \langle\partial_{\rm tan}\varPhi^l\rangle^2 \left( (k^l_1)^2+k^l_2 \right) +(c^l)^2 \left[ (\eta\partial_2\varPhi^l-\tilde{\eta}\partial_1\varPhi^l)^2+\eta^2+\tilde{\eta}^2 \right] \right\}.
\end{split}
\end{equation*}
The corresponding eigenvector is
\begin{equation}\nonumber
E^l=\left[\begin{array}{c}
-\alpha^l(\mu^l+\omega^l)\\
-\alpha^lm^l
\end{array}
\right],
\end{equation}
where $$\alpha^l=k^l_1[(k^l_1)^2+k^l_2].$$
The above eigenvalues are well-defined and smooth on $(\R^4_+\times \Sigma)\setminus \Upsilon_p$, while the eigenvectors are well-defined and smooth on $\R^4_+\times \Sigma$, similar to the constant coefficient case. Hence, the Lopatinski$\breve{\mathrm{i}}$ determinant given below is well-defined for all points in $\R^4_+\times\Sigma$,
\begin{equation}\label{lopatinskii2}
\begin{split}
\text{det} \left(\beta\left[\begin{array}{cc}
E^r & \mathbf{0}\\
\mathbf{0} & E^l
\end{array}
\right]\right)\Big|_{x_3=0}
= & \frac{c^4k^2}{\rho}k^r_1k^l_1\cdot \left[ \frac{k^4}{\iota^r_3\iota^l_3}\omega^r\omega^l+(\eta \iota_2-\tilde{\eta}\iota_1)^2+\eta^2+\tilde{\eta}^2 \right] \left( \frac{\omega^r}{\iota^r_3}-\frac{\omega^l}{\iota^l_3} \right) \\
&\cdot \left[ \frac{\iota^r_3}{kc}((k^r_1)^2+k^r_2)-k^r_1\omega^r \right] \cdot \left[ \frac{\iota^l_3}{kc}\left( (k^l_1)^2+k^l_2 \right)+k^l_1\omega^l \right].
\end{split}
\end{equation}
It is homogeneous of degree 0 with respect to $(\tau,\eta,\tilde{\eta}),$ where
$$
\iota_1=\partial_1\psi=\partial_1\varPhi^{r,l}|_{x_3=0}, \quad \iota_2=\partial_2\psi=\partial_2\varPhi^{r,l}|_{x_3=0}, \quad \iota^{r,l}_3=\partial_3\varPhi^{r,l}|_{x_3=0}, \quad c=c^r|_{x_3=0}=c^l|_{x_3=0}.
$$
On $\R^4_+\times\Sigma$ and for $K$ sufficiently small (where $K$ is introduced in \eqref{pp}),  $\frac{c^4k^2}{\rho}$ and the last two factors in \eqref{lopatinskii2} do not vanish. Hence the possible zeros of the Lopatinski$\breve{\mathrm{i}}$ determinant only come from the remaining factors; in particular, the factors $k^{r}_1$ and $k^{l}_1$ yield the roots $\tau=-iv^{r,l}_1\eta-iv^{r,l}_2\tilde{\eta}$. Therefore, it remains to discuss the zeros of the third and fourth factors.
All the coefficients in the factors of the Lopatinski$\breve{\mathrm{i}}$ determinant are continuous with respect to the background state $U|_{x_3=0}:=(U^r|_{x_3=0},U^l|_{x_3=0})$ and $\varPhi|_{x_3=0}:=(\varPhi^r|_{x_3=0},\varPhi^l|_{x_3=0}).$
These factors reduce to the corresponding factors in the constant coefficient case, if the perturbation in \eqref{background2} is zero. Assuming that $K$ in \eqref{pp} is small enough and using a continuity argument, we obtain that the number of the roots in the third and fourth factors in \eqref{lopatinskii2} is the same as the number of the roots in the corresponding factors of the constant coefficient case. Hence, there are two roots in the third factor and we denote $\tau=iV_1\sqrt{\eta^2+\tilde{\eta}^2}$ and $\tau=iV_2\sqrt{\eta^2+\tilde{\eta}^2}.$
 \begin{remark}\nonumber
 The term on the right-hand side of $\omega^r$, $(\eta\partial_2\varPhi^r-\tilde{\eta}\partial_1\varPhi^r)^2$, is small under the small perturbation assumption. The same argument holds for $\omega^l.$
 \end{remark}

 It is easy to see that $V_1(U|_{x_3=0},\nabla\varPhi|_{x_3=0})$ and $V_2(U|_{x_3=0},\nabla\varPhi|_{x_3=0})$ are real and depend continuously on the background states $U|_{x_3=0}$ and $\nabla\varPhi|_{x_3=0}.$
The roots of the Lopatinski$\breve{\mathrm{i}}$ determinants can be written by the following set according to the information provided by the boundary:
\begin{equation}\label{root}
\Upsilon^0_r:= \left\{ (t,x_1,x_2,\tau,\eta,\tilde{\eta})\in \R^3\times\Sigma: \Re\tau=0 \text{ and } \sigma=0 \right\},
\end{equation}
where
\begin{equation*}
\sigma := \frac{(\delta+v^r_1|_{x_3=0}\eta+v^r_2|_{x_3=0}\tilde{\eta})(\delta+v^l_1|_{x_3=0}\eta+v^l_2|_{x_3=0}\tilde{\eta})(\delta-V_1\sqrt{\eta^2+\tilde{\eta}^2})(\delta-V_2\sqrt{\eta^2+\tilde{\eta}^2})}{(\delta^2+\eta^2+\tilde{\eta}^2)^{\frac{3}{2}}}
\end{equation*}
on $\Sigma.$ We can extend the set $\Upsilon^0_r$ and $\sigma$ defined by the data of the boundary into the interior of the domain $x_3>0.$ The coefficients in $\sigma$ for $x_3>0$ can be defined by continuity of $V_1$ and $V_2$ on the background state $U$ and $\nabla\varPhi.$
Denote the extended set
$$
\Upsilon_r:= \left\{(t,x_1,x_2,x_3,\tau,\eta,\tilde{\eta})\in \R^4_+\times\Sigma: \Re\tau=0 \text{ and } \sigma=0 \right\}.
$$
Similar to the poles, the roots of $\sigma$ can be viewed as the strip in the frequency space $\Sigma$ parameterized by $x_3.$ It originates from the boundary $x_3=0$ and propagates into the interior domain $x_3>0.$
Moreover, we need that the roots of the eigenvalues $\omega^r=0$ and $\omega^l=0$ do not coincide with the poles and the roots of the Lopatinski$\breve{\mathrm{i}}$ determinant.

For simplicity, we write
$$
\cos \theta=\frac{\eta}{\sqrt{\eta^2+\tilde{\eta}^2}},\quad \sin \theta=\frac{\tilde{\eta}}{\sqrt{\eta^2+\tilde{\eta}^2}}, \quad \text{for } (\eta,\tilde{\eta})\neq (0,0),
$$
and define
\begin{equation}\label{eq g}
g_{r,l}(\theta) = \left( \cos \theta F^{r,l}_{11}+\sin\theta  F^{r,l}_{21} \right)^2 + \left( \cos\theta F^{r,l}_{12}+\sin \theta F^{r,l}_{22} \right)^2 + \left( \cos \theta F^{r,l}_{13}+\sin\theta F^{r,l}_{23} \right)^2.
\end{equation}

Now, we consider the following frequency sets:\\
\begin{enumerate}
\item[\rm (1)] $\Upsilon^{(1)}_p = \Upsilon^{(1)}_r := \left\{(t,x_1,x_2,x_3,\tau,\eta,\tilde{\eta}):\tau=-iv^{r,l}_1\eta-iv^{r,l}_2\tilde{\eta} \right\},$\\
\item[\rm (2)] $\Upsilon^{(2)}_r := \left\{(t,x_1,x_2,x_3,\tau,\eta,\tilde{\eta}): \tau=iV_1\sqrt{\eta^2+\tilde{\eta}^2} \text{ or } \tau=iV_2\sqrt{\eta^2+\tilde{\eta}^2} \right\},$\\
\item[\rm (3)] $\Upsilon^{(2)}_p := \left\{(t,x_1,x_2,x_3,\tau,\eta,\tilde{\eta}):\tau=-i\Big(v^{r,l}_1\eta+v^{r,l}_2\tilde{\eta}\pm\sqrt{(\eta^2+\tilde{\eta}^2) g_{r,l}(\theta)} \Big) \right\}$.
\end{enumerate}
Here, $(1)$ denotes the roots of the first factor in $\Upsilon_p$ and the roots of the first two factors of $\sigma$ in $\Upsilon_r$; $(2)$ represents the last two factors of $\sigma$ in $\Upsilon_r$; and $(3)$ represents the roots of second factors in $\Upsilon_p.$

Note that $\Upsilon_p=\Upsilon^{(1)}_p\cup \Upsilon^{(2)}_p$ and $\Upsilon_r=\Upsilon^{(1)}_r\cup\Upsilon^{(2)}_r.$
As mentioned before, we write
$$
\Upsilon_{\omega} := \left\{(t,x_1,x_2,x_3,\tau,\eta,\tilde{\eta}): \omega^{r,l}=0 \right\},
$$
which clearly does not intersect with $\Upsilon_p$ and $\Upsilon_r$, {\it i.e.,} $\Upsilon_{\omega}\cap(\Upsilon_p\cup\Upsilon_r)=\emptyset.$

Under the assumption \eqref{stability4}, $\bar{v}^2<\mathrm{G}(\mathrm{F}_1,\mathrm{F}_2),$ where
\begin{align}\label{G}
\mathrm{G}(\mathrm{F}_1,\mathrm{F}_2) = \frac{1}{4}\inf_{\cos \theta\neq0}\frac{1}{\cos^2\theta}\left( \sqrt{g(\theta)+c^2}-\sqrt{g(\theta)} \right)^2,
\end{align}
and $g(\theta)$ is defined analogously to \eqref{eq g}, we obtain that $$\Upsilon^{(2)}_p\cap \Upsilon_{\omega}=\emptyset.$$
Notice that
\begin{equation*}
\frac{g(\theta)}{\cos^2\theta}=|\mathrm{F}_1|^2+|\mathrm{F}_2|^2\tan^2\theta+2(\mathrm{F}_1\cdot \mathrm{F}_2)\tan \theta.
\end{equation*}
We can introduce
$$
t:=\frac{|\mathrm{F}_2|}{|\mathrm{F}_1|},\quad \alpha:=\text{the angle between $\mathrm{F}_1$ and $\mathrm{F}_2$}.
$$
Then, one has
\begin{align*}
g(\theta)& = \cos^2\theta \left[ |\mathrm{F}_1|^2+|\mathrm{F}_2|^2\tan^2\theta+2 (\mathrm{F}_1\cdot\mathrm{F}_2)\tan\theta \right] \\
&=\cos^2\theta |\mathrm{F}_1|^2 \left( 1+t^2\tan^2\theta+2t\cos \alpha\tan \theta \right) \\
&=\cos^2\theta |\mathrm{F}_1|^2 \left[ \sin^2\alpha+(\cos\alpha+t\tan \theta)^2 \right] \\
&\leq|\mathrm{F}_1|^2 \left[ \sin^2\alpha+(\cos\theta\cos\alpha+t\sin \theta)^2 \right] \\
&\leq|\mathrm{F}_1|^2(1+t^2)\leq |\mathrm{F}_1|^2+|\mathrm{F}_2|^2.
\end{align*}
From this it follows that
\begin{align*}
\mathrm{G}(\mathrm{F}_1,\mathrm{F}_2)&=\frac{1}{4}\inf_{\cos \theta\neq 0}
\frac{g(\theta)}{\cos^2\theta}\frac{c^4}{g(\theta) \left(\sqrt{g(\theta)+c^2}+\sqrt{g(\theta)} \right)^2} \\
&\geq\frac{1}{4}\inf_{\cos\theta\neq0}\frac{g(\theta)}{\cos^2\theta}\frac{c^4}{(|\mathrm{F}_1|^2+|\mathrm{F}_2|^2) \left( \sqrt{|\mathrm{F}_1|^2+|\mathrm{F}_2|^2+c^2}+\sqrt{|\mathrm{F}_1|^2+|\mathrm{F}_2|^2} \right)^2} \\
&\geq\frac{|\pe_{\mathrm{F}_2}(\mathrm{F}_1)|^2}{4}\frac{c^4}{(|\mathrm{F}_1|^2+|\mathrm{F}_2|^2) \left( \sqrt{|\mathrm{F}_1|^2+|\mathrm{F}_2|^2+c^2}+\sqrt{|\mathrm{F}_1|^2+|\mathrm{F}_2|^2} \right)^2}.
\end{align*}
Moreover, the condition
$$
\bar{v}^2<\frac{|\pe_{\mathrm{F}_2}(\mathrm{F}_1)|^2}{4}
$$
guarantees that
$$
\Upsilon^{(1)}_p\cap\Upsilon^{(2)}_p=\Upsilon^{(1)}_r\cap \Upsilon^{(2)}_p=\Upsilon^{(1)}_p\cap \Upsilon_{\omega}=\Upsilon^{(1)}_r\cap\Upsilon_{\omega}=\Upsilon^{(1)}_p\cap\Upsilon^{(2)}_r=\Upsilon^{(1)}_r\cap\Upsilon^{(2)}_r=\emptyset.
$$
Note that $$\Upsilon^{(2)}_p\cap\Upsilon^{(2)}_r=\Upsilon^{(2)}_r\cap\Upsilon_{\omega}=\emptyset$$ holds with no restriction on the background solutions. Hence, except for the special case $\Upsilon^{(1)}_p=\Upsilon^{(1)}_r$ in which there is always interaction between the poles (roots), any two of the remaining strips in $\Upsilon_p,\Upsilon_r,\Upsilon_{\omega}$ do not intersect with each other in the whole domain $\R^4\times\Sigma,$ unless they are identical.

\subsection{Estimates in Each Case}\label{subsec est on each}

Recall from Section \ref{Paralinearization} that on $\supp\chi_a$ the auxiliary $6\times 6$ matrix $B$ satisfies $\det B=b^5 b_1$, and on $\supp\chi_c$ and $\supp\chi_e$ the corresponding matrix satisfies the analogous formulas obtained from the same construction on those regions. Hence whenever we use estimate \eqref{Blower}, it should be understood after applying a cut-off $\chi_B\in\Gamma^0_k$ supported on a microlocal region where the corresponding matrix is invertible; on such a support one has $B^\ast B\ge c_B I$ for some $c_B>0$. On the complementary region given by $1-\chi_B$ we do not need to invert $B$; the estimates in Sections \ref{subsec est on each}--\ref{subsec proof} are obtained directly from Kreiss symmetrizers and the full boundary operator. In this section, we derive the estimates for each case and obtain the desired estimate for the paralinearized system. The relation among $\tau,\eta,\tilde{\eta}$ corresponds to a strip on $\Sigma$ with fixed $(\tau,\eta,\tilde{\eta}).$
Now, we focus on the situation where these strips do not intersect with each other and construct neighbourhoods around them, except for Case 1, in which the poles and roots always intersect.

After possibly shrinking these neighborhoods, they are pairwise disjoint and contain no point in $\Upsilon_{\omega}$.

Denote, on $\R^4\times\Sigma,$
\begin{equation*}
\begin{split}
 & \mathcal{V}^r_{p_1} := \text{ the open neighborhood around } \tau=-iv^r_1\eta-iv^r_2\tilde{\eta}, \\
 & \mathcal{V}^l_{p_1} := \text{ the open neighborhood around } \tau=-iv^l_1\eta-iv^l_2\tilde{\eta}, \\
 & \mathcal{V}^1_{p_2} := \text{ the open neighborhood around } \tau=-i\left(v^r_1\eta+v^r_2\tilde{\eta}+\sqrt{(\eta^2+\tilde{\eta}^2)g_{r}(\theta)}\right), \\
 & \mathcal{V}^2_{p_2} := \text{ the open neighborhood around } \tau=-i\left(v^r_1\eta+v^r_2\tilde{\eta}-\sqrt{(\eta^2+\tilde{\eta}^2)g_{r}(\theta)}\right), \\
 & \mathcal{V}^3_{p_2} := \text{ the open neighborhood around } \tau=-i\left(v^l_1\eta+v^l_2\tilde{\eta}+\sqrt{(\eta^2+\tilde{\eta}^2)g_{l}(\theta)}\right), \\
 & \mathcal{V}^4_{p_2} := \text{ the open neighborhood around } \tau=-i\left(v^l_1\eta+v^l_2\tilde{\eta}-\sqrt{(\eta^2+\tilde{\eta}^2)g_{l}(\theta)}\right), \\
 & \mathcal{V}^1_r := \text{ the open neighborhood around } \tau=iV_1\sqrt{\eta^2+\tilde{\eta}^2}, \\
 & \mathcal{V}^2_r := \text{ the open neighborhood around } \tau=iV_2\sqrt{\eta^2+\tilde{\eta}^2}.
\end{split}
\end{equation*}
\begin{remark}\label{nonintersection}
Due to the stability condition imposed on the background solutions for non-parallel elastic deformation gradients, certain neighborhoods remain disjoint. For example, the neighborhood of $\tau=-iv^r_1\eta-iv^r_2\tilde{\eta}$, denoted as $\mathcal{V}^r_{p_1}$, cannot intersect with the neighborhood of $\tau=-i(v^r_1\eta+v^r_2\tilde{\eta}+\sqrt{(\eta^2+\tilde{\eta}^2)g_{r}(\theta)})$, denoted as $\mathcal{V}^1_{p_2}$. This separation indicates the stabilizing effect of elasticity, which is consistent with the linear analysis of constant coefficients, see \cite{RChen2021}.
\end{remark}

\subsection{Case 1: Points in $\Upsilon^{(1)}_{p}=\Upsilon^{(1)}_r$}\label{subsec case1}
We consider the kind of frequencies that are both poles and roots of the Lopatinski$\breve{\mathrm{i}}$ determinant. Consider $\mathcal{V}^r_{p_1}$ as an example, since the other cases can be discussed similarly.

Different from 2D case, $\mathcal{V}^r_{p_1}$ not only contains the poles of the equations $W^+$, but also contains the poles for $W^-$ in \eqref{para2}. Hence we derive the estimates for $W^+$. The estimates for $W^-$ will follow the same way.
Introducing the smooth cut-off function $\chi_{p_1}$ whose range is $[0,1].$ On $\R^4\times\Sigma,$ the support of $\chi_{p_1}$ is contained in $\mathcal{V}^r_{p_1}$ and equals $1$ on a smaller neighborhood of the strip satisfying $\tau=-iv^r_1\eta-iv^r_2\tilde{\eta}.$ We can extend $\chi_{p_1}$ by homogeneity of degree 0 with respect to $(\tau, \eta,\tilde{\eta})$ into the whole domain $\R^4_+\times\Pi$. We know that $\chi_{p_1}\in \Gamma^0_k$ for any integer $k.$ Define
$$
W^+_{p_1}:=T^{\gamma}_{\chi_{p_1}}W^+.
$$
From \eqref{para2}, we have
$$
I_2\partial_3W^+_{p_1}=I_2T^{\gamma}_{\partial_3\chi_{p_1}}W^++T^{\gamma}_{\chi_{p_1}}F^+-T^{\gamma}_{\chi_{p_1}}T^{\gamma}_{\tau A^r_0+i\eta A^r_1+i\tilde{\eta}A^r_2+A^r_0C^r} W^+.
$$
 Then
 \begin{equation}\label{para3}
 T^{\gamma}_{\mathcal{A}^r}W^+_{p_1}+T^{\gamma}_{A^r_0C^r}W^+_{p_1}+T^{\gamma}_rW^++I_2\partial_3W^+_{p_1}=T^{\gamma}_{\chi_{p_1}}F^++\mathcal{R}_{-1}W^+,
 \end{equation}
 where $\mathcal{A}^r:=\tau A^r_0+i\eta A^r_1+i\tilde{\eta}A^r_2$ and $r\in\Gamma^0_1$ whose support is where $\chi_{p_1}\in(0,1)$. 
 Consider two cut-off functions $\chi_1$ and $\chi_2$ in $\Gamma^0_1,$ such that both of the supports are in
 \begin{equation*}
 \begin{split}
\mathcal{V}^r_{p_1}\cdot \R_+:= & \Big\{(t,x_1,x_2,x_3,\tau,\eta,\tilde{\eta})\in\Omega\times \Pi:\\
& \quad \left. \left( t,x_1,x_2,x_3,\frac{\tau}{\sqrt{|\tau|^2+\eta^2+\tilde{\eta}^2}},\frac{\eta}{\sqrt{|\tau|^2+\eta^2+\tilde{\eta}^2}},\frac{\tilde{\eta}}{\sqrt{|\tau|^2+\eta^2+\tilde{\eta}^2}} \right) \in \mathcal{V}^r_{p_1} \right\},
\end{split}
\end{equation*}
where $\chi_1=1$ on $\text{supp}\chi_{p_1}$ and $\chi_2=1$ on $\text{supp}\chi_1.$ Now we multiply \eqref{para3} by $\chi_2$ and obtain that
\begin{equation}\label{para4}
T^{\gamma}_{\chi_2\mathcal{A}^r}W^+_{p_1}+ T^{\gamma}_{\chi_2A^r_0C^r}W^+_{p_1}+T^{\gamma}_rW^++I_2\partial_3W^+_{p_1}=\mathcal{R}_0F^++\mathcal{R}_{-1}W^+.
\end{equation}
Here the support of $\chi_2\mathcal{A}^r$ is in $\text{supp}\chi_2,$ which is the subset of  $\mathcal{V}^r_{p_1}\cdot \R_+.$ Now we can uppertriangularize the first symbol $\chi_2\mathcal{A}^r$ on $\text{supp}\chi_2.$ Define the transformation matrix $Q^r_0$ on $\mathcal{V}^r_{p_1}:$
\begin{equation}\label{def:Q0r.p1}
Q^r_0 = \left[\begin{smallmatrix}
1 & 0 & \hat{W}^r_1 & 0 & 0 & 0 & 0 & 0 & 0 & 0 & 0 & 0 & 0\\
0 & 1 & \hat{W}^r_2 & 0 & 0 & 0 & 0 & 0 & 0 & 0 & 0 & 0 & 0\\
0 & 0 & -\alpha^r(\mu^r+\omega^r) & U^r_3 &  0 & 0 & 0 & 0 & 0 & 0 & 0 & 0& 0\\
0 & 0 & -\alpha^rm^r & U^r_4 &  0 & 0 & 0 & 0 & 0 & 0 & 0 & 0& 0\\
0 & 0 & \hat{W}^r_5 & 0 &  1 & 0 & 0 & 0 & 0 & 0 & 0 & 0& 0\\
0 & 0 & \hat{W}^r_6 & 0 &  0 & 1 & 0 & 0 & 0 & 0 & 0 & 0& 0\\
0 & 0 & \hat{W}^r_7 & 0 &  0 & 0 & 1 & 0 & 0 & 0 & 0 & 0& 0\\
0 & 0 & \hat{W}^r_8 & 0 &  0 & 0 & 0 & 1 & 0 & 0 & 0 & 0& 0\\
0 & 0 & \hat{W}^r_9 & 0 &  0 & 0 & 0 & 0 & 1 & 0 & 0 & 0 & 0\\
0 & 0 & \hat{W}^r_{10} & 0 &  0 & 0 & 0 & 0& 0 & 1 & 0 & 0 & 0\\
0 & 0 & \hat{W}^r_{11} & 0 &  0 & 0 & 0 & 0 & 0 & 0 & 1 & 0 & 0\\
0 & 0 & \hat{W}^r_{12} & 0 &  0 & 0 & 0 & 0 & 0 & 0 & 0 & 1 & 0\\
0 & 0 & \hat{W}^r_{13} & 0 &  0 & 0 & 0 & 0 & 0 & 0 & 0 & 0 & 1\\
\end{smallmatrix}
\right].
\end{equation}

Thus $Q^r_0$ is homogeneous of degree $0$ with respect to $(\tau,\eta,\tilde{\eta}).$
The entries at the third column in the third and fourth rows are from the eigenvector $E^r$, while $\hat{W}^r_i,i=1,2,5,6,\cdots,13$, are chosen to make the third column of $\mathcal{A}^rQ^r_0$ zero except for the third and fourth rows.
 $$
 \mathcal{T} \big( \hat{W}^r_1,\hat{W}^r_2,-\alpha^r(\mu^r+\omega^r),-\alpha^rm^r,\hat{W}^{r}_5,\hat{W}^{r}_6,\hat{W}^{r}_7,\hat{W}^{r}_8,\hat{W}^{r}_9,\hat{W}^{r}_{10},\hat{W}^{r}_{11},\hat{W}^{r}_{12},\hat{W}^{r}_{13} \big)^{\top}={\mathbf 0}.
 $$
 It is noted that $\hat{W}^r_i,i=1,2,5,\cdots,13$, can be solved at all points in $\R^4_+\times\Pi.$ In the following, we introduce $\chi_1Q^r_0$ to exclude the frequencies at which $\omega^r$ degenerates. To ensure the invertibility of $Q^r_0$, we can take $U^r_3=1,U^r_4=0$ for simplicity (a more precise argument would require microlocalization). So $\chi_1Q^r_0\in \Gamma^0_2.$
 In order to uppertriangularize the first order operator $\mathcal{A}^r$ in $\mathcal{V}^r_{p_1}\cdot \R_+,$ we need to construct $R^r_0$ in $\mathcal{V}^r_{p_1}:$
\begin{equation}\label{def:R0r.p1}
 R^r_0=\left[
 \begin{smallmatrix}
 1 & 0 & 0 & 0 & 0 & 0 & 0 & 0 & 0 & 0 & 0 & 0 & 0\\
0 & 1 & 0 & 0 & 0 & 0 & 0 & 0 & 0 & 0 & 0 & 0 & 0\\
0 & 0 & 0 & -\frac{1}{\xi}& 0 & 0 & 0 & 0 & 0 & 0 & 0 & 0 & 0\\
\bar{W}^r_1 & \bar{W}^r_2 & \frac{\alpha^rm^r}{\xi} & -\frac{\alpha^r(\mu^r+\omega^r)}{\xi} & \bar{W}^r_5 & \bar{W}^r_6 & \bar{W}^r_7 & \bar{W}^r_8 & \bar{W}^r_9 & \bar{W}^r_{10} & \bar{W}^r_{11} & \bar{W}^r_{12} & \bar{W}^r_{13} \\
0 & 0 & 0 & 0 & 1 & 0 & 0 & 0 & 0 & 0 & 0 & 0 & 0\\
0 & 0 & 0 & 0 & 0 & 1 & 0 & 0 & 0 & 0 & 0 & 0 & 0\\
0 & 0 & 0 & 0 & 0 & 0 & 1 & 0 & 0 & 0 & 0 & 0 & 0\\
0 & 0 & 0 & 0 & 0 & 0 & 0 & 1 & 0 & 0 & 0 & 0 & 0\\
0 & 0 & 0 & 0 & 0 & 0 & 0 & 0 & 1 & 0 & 0 & 0 & 0\\
0 & 0 & 0 & 0 & 0 & 0 & 0 & 0 & 0 & 1 & 0 & 0 & 0\\
0 & 0 & 0 & 0 & 0 & 0 & 0 & 0 & 0 & 0 & 1 & 0 & 0\\
0 & 0 & 0 & 0 & 0 & 0 & 0 & 0 & 0 & 0 & 0 & 1 & 0\\
0 & 0 & 0 & 0 & 0 & 0 & 0 & 0 & 0 & 0 & 0 & 0 & 1\\
\end{smallmatrix}
\right],
 \end{equation}
where $\xi=\alpha^rm^r$ which equals the determinant of $Q^r_0.$

Since $\Upsilon_{\omega}\cap\mathcal{V}^r_{p_1}=\emptyset$, the eigenvalue $\omega^r$ does not vanish on $\mathcal{V}^r_{p_1}$. Moreover, although $m^r$ and $\mu^r$ contain terms of the form $k^r_2/k^r_1$, the prefactor $\alpha^r=k^r_1((k^r_1)^2+k^r_2)$ cancels the potential $1/k^r_1$ singularity; hence $\xi=\alpha^rm^r$ extends smoothly across $k^r_1=0$. Shrinking $\mathcal{V}^r_{p_1}$ if necessary, we may assume $|\xi|\ge c_0>0$ on $\supp\chi_1$, so that $\chi_1Q^r_0$ is elliptic and $\chi_1{Q^r_0}^{-1}\in\Gamma^0_2$.
 Here $\bar{W}^r_1,\bar{W}^r_2,\bar{W}^r_5,\cdots,\bar{W}^r_{13}$ are defined to be homogeneous of degree 0.
 Then we have
 \begin{equation}\nonumber
 \left[\begin{smallmatrix}
 \bar{W}^r_1\\
  \bar{W}^r_2\\
   \frac{\alpha^rm^r}{\xi}\\
    -\frac{\alpha^r(\mu^r+\omega^r)}{\xi} \\
     \bar{W}^r_5 \\
      \bar{W}^r_6 \\
       \bar{W}^r_7\\
        \bar{W}^r_8 \\
         \bar{W}^r_9 \\
          \bar{W}^r_{10}\\
         \bar{W}^r_{11} \\
          \bar{W}^r_{12} \\
          \bar{W}^r_{13} \end{smallmatrix}\right]^{\top}
          \left[\begin{smallmatrix}
   k^r_1 & 0 & \mathcal{T}_{1,5} & 0 & 0 & \mathcal{T}_{1,8} & 0 & 0 & \mathcal{T}_{1,11} & 0 & 0\\
0 & k^r_1  & 0 & \mathcal{T}_{2,6} & 0 & 0 & \mathcal{T}_{2,9} & 0 & 0 & \mathcal{T}_{2,12} & 0\\
\mathcal{T}_{3,1} & \mathcal{T}_{3,2}  & 0 & 0 & \mathcal{T}_{3,7} & 0 & 0 & \mathcal{T}_{3,10} & 0 & 0 & \mathcal{T}_{3,13}\\
\mathcal{T}_{4,1} & \mathcal{T}_{4,2}  & 0 & 0 & \mathcal{T}_{4,7} & 0 & 0 & \mathcal{T}_{4,10} & 0 & 0 & \mathcal{T}_{4,13}\\
\mathcal{T}_{5,1} & 0 & k^r_1 & 0 & 0 & 0 & 0 & 0 & 0 & 0 & 0\\
0 & \mathcal{T}_{6,2} & 0 & k^r_1 & 0 & 0 & 0 & 0 & 0 & 0 & 0\\
0 & 0 &0 & 0 & k^r_1 & 0 & 0 & 0 & 0 & 0 & 0\\
\mathcal{T}_{8,1} & 0  & 0 & 0 & 0 & k^r_1 & 0 & 0 & 0 & 0 & 0\\
0 & \mathcal{T}_{9,2} & 0 & 0 & 0 & 0 & k^r_1 & 0 & 0 & 0 & 0\\
0 & 0  &0 & 0 & 0 & 0 & 0 & k^r_1 & 0 & 0 & 0\\
\mathcal{T}_{11,1} & 0  & 0 & 0 & 0 & 0 & 0 & 0 & k^r_1 & 0 & 0\\
0 & \mathcal{T}_{12,2} & 0 & 0 & 0 & 0 & 0 & 0 & 0 & k^r_1 & 0\\
0 & 0 &0 & 0 & 0 & 0 & 0 & 0 & 0 & 0 & k^r_1\\
 \end{smallmatrix}
\right]=\mathbf{0}.
 \end{equation}
It follows that $\chi_1R^r_0\in \Gamma^0_2.$ We can finally obtain the first-order symbol for the upper triangularization.
 \begin{equation*}
 \begin{split}
 \tilde{A}^r&:=R^r_0\mathcal{A}^rQ^r_0\\
 &=\left[\begin{smallmatrix}
k^r_1 & 0 & 0  & \Theta_1 & \mathcal{T}_{1,5} & 0 & 0 & \mathcal{T}_{1,8} & 0 & 0 & \mathcal{T}_{1,11} & 0 & 0\\
0 & k^r_1 & 0 & \Theta_1 & 0 & \mathcal{T}_{2,6} & 0 & 0 & \mathcal{T}_{2,9} & 0 & 0 & \mathcal{T}_{2,12} & 0\\
\Theta_1 & \Theta_1 & \tilde{A}^r_{3,3} & 0 & 0 & 0 & \Theta_1 & 0 & 0 & \Theta_1 & 0 & 0 & \Theta_1\\
0 & 0  & 0 & \tilde{A}^r_{4,4} & 0 & 0 & 0 & 0 & 0 & 0 & 0 & 0 & 0\\
\mathcal{T}_{5,1} & 0 & 0 & 0 & k^r_1 & 0 & 0 & 0 & 0 & 0 & 0 & 0 & 0\\
0 & \mathcal{T}_{6,2} & 0 & 0 & 0 & k^r_1 & 0 & 0 & 0 & 0 & 0 & 0 & 0\\
0 & 0 & 0 & \Theta_1 &0 & 0 & k^r_1 & 0 & 0 & 0 & 0 & 0 & 0\\
\mathcal{T}_{8,1} & 0 & 0 & 0 & 0 & 0 & 0 & k^r_1 & 0 & 0 & 0 & 0 & 0\\
0 & \mathcal{T}_{9,2} & 0 & 0 & 0 & 0 & 0 & 0 & k^r_1 & 0 & 0 & 0 & 0\\
0 & 0 & 0 & \Theta_1 &0 & 0 & 0 & 0 & 0 & k^r_1 & 0 & 0 & 0\\
\mathcal{T}_{11,1} & 0 & 0 & 0 & 0 & 0 & 0 & 0 & 0 & 0 & k^r_1 & 0 & 0\\
0 & \mathcal{T}_{12,2} & 0 & 0 & 0 & 0 & 0 & 0 & 0 & 0 & 0 & k^r_1 & 0\\
0 & 0 & 0 & \Theta_1 &0 & 0 & 0 & 0 & 0 & 0 & 0 & 0 & k^r_1\\
\end{smallmatrix}
\right],
 \end{split}
 \end{equation*}
 where $$\tilde{A}^r_{3,3}=-\omega^r-\frac{i\partial_3\varPhi^r\partial_1\varPhi^r\eta}{\langle\partial_{\rm tan}\varPhi^r\rangle^2}-\frac{i\partial_3\varPhi^r\partial_2\varPhi^r\tilde{\eta}}{\langle\partial_{\rm tan}\varPhi^r\rangle^2},$$
  $$\tilde{A}^r_{4,4}=\omega^r-\frac{i\partial_3\varPhi^r\partial_1\varPhi^r\eta}{\langle\partial_{\rm tan}\varPhi^r\rangle^2}-\frac{i\partial_3\varPhi^r\partial_2\varPhi^r\tilde{\eta}}{\langle\partial_{\rm tan}\varPhi^r\rangle^2}.$$
   Here, $\Theta_1\in \Gamma^1_2$ whose exact expression is not important for our analysis.

\medskip

The following Lemma \ref{paralemma}, which plays a key role in paradifferential calculus, can be proved similarly to the approach in \cite{Coulombel2002}.
 \begin{lemma}\label{paralemma}
 With the appropriate choice of $Q^r_{-1}$ and $R^r_{-1}$ in $\Gamma^{-1}_1,$ there is a symbol $D_0=(d_{i,j})_{13\times13}$
 in $\Gamma^0_1$ satisfying $d_{3,4}=d_{4,3}=0,$ such that
 $$R^r_{-1}{R^{r}_0}^{-1}\tilde{A}^r-\tilde{A}^rQ^r_{-1}Q^r_0-(\partial_3 (Q^{r}_0)^{-1}-R^rA^r_0C^r-[R^r,\chi_2A^r]+[\chi_2A^r,Q^r])Q^r_0-D_0$$
is a symbol in $\Gamma^{-1}_1$ on $\chi_2=1;$
moreover, $$R^r_{-1}I_2=I_2Q^r_{-1},$$
where $[\cdot,\cdot]$ is defined as
$$[A,B]:=\frac{1}{i}\left(\frac{\partial A}{\partial\delta}\frac{\partial B}{\partial t}+\frac{\partial A}{\partial\eta}\frac{\partial B}{\partial x_1}+\frac{\partial A}{\partial\tilde{\eta}}\frac{\partial B}{\partial x_2}\right)$$
 for any symbols $A$ and $B.$
 \end{lemma}
Now we set
\begin{equation}\label{def:Zplus.p1}
Z^+:=T^{\gamma}_{\chi_1((Q^{r}_0)^{-1}+Q^r_{-1})}W^{+}_{p_1}.
\end{equation}
Define $Q^r={Q^{r}_0}^{-1}+Q^r_{-1}$ and $R^r=R^r_0+R^r_{-1},$ we obtain
\begin{equation}\nonumber
\begin{split}
I_2\partial_3Z^+&=I_2T^{\gamma}_{(\partial_3\chi_1)Q^r}W^+_{p_1}+I_2T^{\gamma}_{\chi_1\partial_3Q^r}W^+_{p_1}+I_2T^{\gamma}_{\chi_1Q^r}\partial_3W^+_{p_1}\\
&=I_2T^{\gamma}_{(\partial_3\chi_1)Q^r}W^+_{p_1}+I_2T^{\gamma}_{\chi_1\partial_3Q^r}W^+_{p_1}+T^{\gamma}_{\chi_1R^r}I_2\partial_3W^+_{p_1}.
\end{split}
\end{equation}
$\partial_3\chi_1$ is supported where $\chi_1\in(0,1)$ and it is disjoint with the support of $\chi_{p_1}.$
Then,  from asymptotic expansion of the symbols,  we have
$$T^{\gamma}_{(\partial_3\chi_1)Q^r}W^+_{p_1}=\mathcal{R}_{-1}W^+.$$
Using \eqref{para4}, we have
\begin{equation}\nonumber
\begin{split}
I_2\partial_3Z^+&=I_2T^{\gamma}_{\chi_1\partial_3 (Q_0^r)^{-1}} W^+_{p_1}-T^{\gamma}_{\chi_1R^rA^r}W^+_{p_1}-T^{\gamma}_{[\chi_1R^r,\chi_2A^r]}W^+_{p_1}-T^{\gamma}_{\chi_1R^rA^r_0C^r}W^+_{p_1}\\
&\quad+T^{\gamma}_rW^++\mathcal{R}_0F+\mathcal{R}_{-1}W^+\\
&=I_2T^{\gamma}_{\chi_1\partial_3 (Q_0^r)^{-1}} W^+_{p_1}-T^{\gamma}_{\chi_1R^rA^r}W^+_{p_1}+T^{\gamma}_{\chi_2\tilde{A}^r}T^{\gamma}_{\chi_1Q^r}W^+_{p_1}-T^{\gamma}_{\chi_2\tilde{A}^r}T^{\gamma}_{\chi_1Q^r}W^+_{p_1}\\
&\quad-T^{\gamma}_{[\chi_1R^r,\chi_2A^r]}W^+_{p_1}-T^{\gamma}_{\chi_1R^rA^r_0C^r}W^+_{p_1}+T^{\gamma}_rW^++\mathcal{R}_0F+\mathcal{R}_{-1}W^+\\
&=-T^{\gamma}_{\chi_1(R^r_{-1}{R^{r}_0}^{-1}\tilde{A}^r-\tilde{A}^rQ^r_{-1}Q^r_0)}Z^+-T^{\gamma}_{\chi_2\tilde{A}^r}Z^++I_2T^{\gamma}_{\chi_1\partial_3{Q^{r}_0}^{-1}}W^+_{p_1}-T^{\gamma}_{\chi_1R^rA^r_0C^r}W^+_{p_1}\\
&\quad-T^{\gamma}_{\chi_1([R^r,\chi_2\mathcal{A}^r]-[\chi_2\tilde{A}^r,Q^r])}W^+_{p_1}+T^{\gamma}_rW^++\mathcal{R}_0F+\mathcal{R}_{-1}W^+.\\
\end{split}
\end{equation}
Then, we obtain that
\begin{equation}\nonumber
\begin{split}
I_2\partial_3Z^+&=-T^{\gamma}_{\chi_1(R^r_{-1}{R^{r}_0}^{-1}\tilde{A}^r-\tilde{A}^rQ^r_{-1}Q^r_0)}Z^+-T^{\gamma}_{\chi_2\tilde{A}^r}Z^++I_2T^{\gamma}_{\chi_1\partial_3{Q^{r}_0}^{-1}Q^r_0}Z^+-T^{\gamma}_{\chi_1R^rA^r_0C^rQ^r_0}Z^+\\
&\quad-T^{\gamma}_{\chi_1([R^r,\chi_2A^r]-[\chi_2\tilde{A}^r,Q^r])Q^r_0}Z^++T^{\gamma}_rW^++\mathcal{R}_0F^++\mathcal{R}_{-1}W^+.\\
\end{split}
\end{equation}
From Lemma \ref{lemmapara} and Lemma \ref{paralemma}, we know
\begin{equation}\label{para5}
\begin{split}
I_2\partial_3Z^+=-T^{\gamma}_{\chi_2\tilde{A}^r}Z^++T^{\gamma}_{D_0}Z^++T^{\gamma}_rW^++\mathcal{R}_0F^++\mathcal{R}_{-1}W^+.
\end{split}
\end{equation}
Since the support of the Fourier transform of $Z^+$ is in the support of $\chi_{p_1},$ we have
\begin{equation}\label{para6}
I_2\partial_3Z^+=-T^{\gamma}_{\tilde{D}_1}Z^++T^{\gamma}_{\tilde{D}_0}Z^++T^{\gamma}_rW^++\mathcal{R}_0F^++\mathcal{R}_{-1}W^+,
\end{equation}
where $\tilde{D}_1$ is the same as $\tilde{A}^r$ except replacing $\omega^r$ in each element by $\tilde{\omega}^r,$ $\tilde{\omega}^r\in\Gamma^1_2.$ It is equal to $\omega^r$ on the support of $\chi_2,$ $\tilde{D}_0$ is an extension of $D_0$ with $d_{3,4}=d_{4,3}=0$ to the whole space. Moreover, we see that $\Re\omega^r\leq-c\Lambda$ in $\mathcal{V}^r_{p_1}$, where $\Lambda \in \Gamma^1_2$ is defined as
\begin{equation}\label{def Lambda}
\Lambda := \sqrt{\gamma^2+\delta^2+\eta^2+\tilde{\eta}^2}.
\end{equation}
This suggests that we can extend as $\Re\tilde{\omega}^r\leq -c\Lambda$ to the whole space.
For simplicity, we will write $\omega^r$ instead of $\tilde{\omega}^r$ in later arguments.

Denote $Z^+ := (Z_1,Z_2,\cdots,Z_{13})^{\top}$. From the fourth equation in \eqref{para6} we find
$$\partial_3Z_4=T^{\gamma}_{-\omega^r+i\bar{\omega}^r}Z_4+T^{\gamma}_{\Theta_0}Z_4+\sum_{i\neq3,4}T^{\gamma}_{\Theta_0} Z_i+T^{\gamma}_rW^++\mathcal{R}_0F^++\mathcal{R}_{-1}W^+,$$
where $$\bar{\omega}^r=\frac{\partial_3\varPhi^r\partial_1\varPhi^r}{\langle\partial_{\rm tan}\varPhi^r\rangle^2}\eta+\frac{\partial_3\varPhi^r\partial_2\varPhi^r}{\langle\partial_{\rm tan}\varPhi^r\rangle^2}\tilde{\eta}.$$

Consider two symmetrizers $(T^{\gamma}_{\sigma})^{\ast}T^{\gamma}_{\Lambda}T^{\gamma}_{\sigma}$ and $(T^{\gamma}_{\Lambda})^{\ast}T^{\gamma}_{\Lambda},$ where
$\sigma$ is defined in $\R^4_+\times\Sigma$ by \eqref{root} and extended to $\R^4_{+}\times\Pi$ by homogeneity of degree 1. Thus $\sigma\in \Gamma^1_2$, and
\begin{equation}\label{para7}
\begin{split}
\Re\langle T^{\gamma}_{\sigma}\partial_3Z_4,T^{\gamma}_{\Lambda}T^{\gamma}_{\sigma}Z_4\rangle&
=\Re\langle T^{\gamma}_{\Lambda}T^{\gamma}_{\sigma}Z_4,T^{\gamma}_{\sigma}\partial_3Z_4\rangle\\
&=\Re\langle T^{\gamma}_{\Lambda}T^{\gamma}_{\sigma}Z_4,T^{\gamma}_{\sigma}T^{\gamma}_{-\omega^r+i\bar{\omega}^r}Z_4\rangle+\Re\langle T^{\gamma}_{\Lambda}T^{\gamma}_{\sigma}Z_4,T^{\gamma}_{\sigma}T^{\gamma}_{\Theta_0}Z_4\rangle\\
&\quad+\sum_{i\neq3,4}\Re\langle T^{\gamma}_{\Lambda}T^{\gamma}_{\sigma}Z_4,T^{\gamma}_{\sigma}T^{\gamma}_{\Theta_0}Z_i\rangle+\Re\langle T^{\gamma}_{\Lambda}T^{\gamma}_{\sigma}Z_4,T^{\gamma}_{\sigma}T^{\gamma}_rW^+\rangle\\
&\quad+\Re\langle T^{\gamma}_{\Lambda}T^{\gamma}_{\sigma}Z_4,T^{\gamma}_{\sigma}\mathcal{R}_{-1}W^+\rangle+\Re\langle T^{\gamma}_{\Lambda}T^{\gamma}_{\sigma}Z_4,T^{\gamma}_{\sigma}F^+\rangle.
\end{split}
\end{equation}
For the first term on the right-hand side of \eqref{para7}, using Lemma \ref{lemmapara}, we have
$$\Re\langle T^{\gamma}_{\Lambda}T^{\gamma}_{\sigma}Z_4,T^{\gamma}_{\sigma}T^{\gamma}_{-\omega^r+i\bar{\omega}^r}Z_4\rangle=
\Re\langle T^{\gamma}_{\Lambda}T^{\gamma}_{\sigma}Z_4,T^{\gamma}_{\frac{-\omega^r+i\bar{\omega}^r}{\Lambda}}T^{\gamma}_{\Lambda}T^{\gamma}_{\sigma}Z_4\rangle\\
+\Re\langle T^{\gamma}_{\Lambda}T^{\gamma}_{\sigma}Z_4,\,\mathcal{R}_1 Z_4\rangle.\\
$$
From the extension of $\omega^r,$ we obtain that
$$\Re\frac{-\omega^r+i\bar{\omega}^r}{\Lambda}\geq c,$$ for some positive $c$ depending on the background states.
Using G${\rm \mathring a}$rding's inequality (Lemma \ref{lemmapara}(vi)), we have
$$
\Re\langle T^{\gamma}_{\Lambda}T^{\gamma}_{\sigma}Z_4,T^{\gamma}_{\frac{-\omega^r+i\bar{\omega}^r}{\Lambda}}T^{\gamma}_{\Lambda}T^{\gamma}_{\sigma}Z_4\rangle \geq c \|T^{\gamma}_{\Lambda}T^{\gamma}_{\sigma}Z_4\|^2_0 = c \|T^{\gamma}_{\sigma}Z_4\|^2_{1,\gamma}.
$$
Using Lemma \ref{lemmapara}
 (iii)-(iv), for the rest of the terms on the right-hand side of \eqref{para7},
\begin{equation*}
\begin{split}
\Re\langle T^{\gamma}_{\Lambda}T^{\gamma}_{\sigma}Z_4,\mathcal{R}_1Z_4\rangle & \leq \varepsilon\|T^{\gamma}_{\Lambda}T^{\gamma}_{\sigma}Z_4\|^2_0+\frac{1}{\varepsilon} \|Z_4\|^2_{1,\gamma}, \\
\Re\langle T^{\gamma}_{\Lambda}T^{\gamma}_{\sigma}Z_4,T^{\gamma}_{\sigma}T^{\gamma}_{\Theta_0}Z_4\rangle & \leq \varepsilon\|T^{\gamma}_{\Lambda}T^{\gamma}_{\sigma}Z_4\|^2_0+\frac{1}{\varepsilon}\|Z_4\|^2_{1,\gamma}, \\
\Re\langle T^{\gamma}_{\Lambda}T^{\gamma}_{\sigma}Z_4,T^{\gamma}_{\sigma}T^{\gamma}_{\Theta_0}Z_i\rangle & = \Re\langle T^{\gamma}_{\Lambda}T^{\gamma}_{\sigma}Z_4,T^{\gamma}_{\Theta_0}T^{\gamma}_{\sigma}Z_i\rangle+\Re\langle T^{\gamma}_{\Lambda}T^{\gamma}_{\sigma}Z_4,\mathcal{R}_0Z_i\rangle \\
& \leq \varepsilon \|T^{\gamma}_{\Lambda}T^{\gamma}_{\sigma}Z_4\|^2_0 + \frac{1}{\varepsilon}\Big( \|T^{\gamma}_{\sigma}Z_i\|^2_{0} + \|Z_i \|^2_0\Big), \\
\Re\langle T^{\gamma}_{\Lambda}T^{\gamma}_{\sigma}Z_4, T^{\gamma}_{\sigma}T^{\gamma}_rW^+\rangle & \leq \varepsilon \|T^{\gamma}_{\Lambda}T^{\gamma}_{\sigma}Z_4\|^2_0+\frac{1}{\varepsilon} \|T^{\gamma}_rW^+ \|^2_{1,\gamma}, \\
\Re\langle T^{\gamma}_{\Lambda}T^{\gamma}_{\sigma}Z_4, T^{\gamma}_{\sigma}\mathcal{R}_{-1}W^+\rangle & \leq \varepsilon \|T^{\gamma}_{\Lambda}T^{\gamma}_{\sigma}Z_4\|^2_0+\frac{1}{\varepsilon} \| W^+ \|^2_{0}, \\
\Re\langle T^{\gamma}_{\Lambda}T^{\gamma}_{\sigma}Z_4, T^{\gamma}_{\sigma}F^+\rangle & \leq \varepsilon \|T^{\gamma}_{\Lambda}T^{\gamma}_{\sigma}Z_4\|^2_0+\frac{1}{\varepsilon} \| F^+ \|^2_{1,\gamma},
\end{split}
\end{equation*}
where $\varepsilon>0$ is taken to be small enough.

Note that
\begin{equation}\label{para8}
\begin{split}
\partial_3\Re\langle Z_4,(T^{\gamma}_{\sigma})^{\ast}T^{\gamma}_{\Lambda}T^{\gamma}_{\sigma}Z_4\rangle&=\Re\partial_3\langle T^{\gamma}_{\sigma}Z_4, T^{\gamma}_{\Lambda}T^{\gamma}_{\sigma}Z_4\rangle\\
&=\Re\langle T^{\gamma}_{\partial_3\sigma}Z_4,T^{\gamma}_{\Lambda}T^{\gamma}_{\sigma}Z_4\rangle+\Re\langle T^{\gamma}_{\sigma}Z_4,T^{\gamma}_{\Lambda}T^{\gamma}_{\partial_3\sigma}Z_4\rangle\\
&\quad+\Re\langle T^{\gamma}_{\sigma}Z_4,T^{\gamma}_{\Lambda}T^{\gamma}_{\sigma}\partial_3Z_4\rangle+\Re\langle T^{\gamma}_{\sigma}\partial_3Z_4,T^{\gamma}_{\Lambda}T^{\gamma}_{\sigma}Z_4\rangle.
\end{split}
\end{equation}
For the first three terms on the right-hand side of \eqref{para8},
\begin{equation*}
\begin{split}
\Re\langle T^{\gamma}_{\partial_3\sigma}Z_4,T^{\gamma}_{\Lambda}T^{\gamma}_{\sigma}Z_4\rangle & \leq \varepsilon\|T^{\gamma}_{\Lambda}T^{\gamma}_{\sigma}Z_4\|^2_0+\frac{1}{\varepsilon}\|Z_4\|^2_{1,\gamma}, \\
\Re\langle T^{\gamma}_{\sigma}Z_4,T^{\gamma}_{\Lambda}T^{\gamma}_{\partial_3\sigma}Z_4\rangle & \leq \varepsilon\|T^{\gamma}_{\Lambda}T^{\gamma}_{\sigma}Z_4\|^2_0+\frac{1}{\varepsilon}\|Z_4\|^2_{1,\gamma}, \\
\Re\langle T^{\gamma}_{\sigma}Z_4,T^{\gamma}_{\Lambda}T^{\gamma}_{\sigma}\partial_3Z_4\rangle & = \Re\langle T^{\gamma}_{\Lambda}T^{\gamma}_{\sigma}Z_4,T^{\gamma}_{\sigma}\partial_3Z_4\rangle+\Re\langle T^{\gamma}_{\sigma}Z_4,\mathcal{R}_0T^{\gamma}_{\sigma}\partial_3Z_4\rangle.
\end{split}
\end{equation*}
The terms $\Re\langle T^{\gamma}_{\Lambda}T^{\gamma}_{\sigma}Z_4,T^{\gamma}_{\sigma}\partial_3Z_4\rangle,$ $\Re\langle T^{\gamma}_{\sigma}Z_4,\mathcal{R}_0T^{\gamma}_{\sigma}\partial_3Z_4\rangle$ can be treated similarly.
Summing up \eqref{para7} and \eqref{para8} and integrating with respect to $x_3,$ we have
\begin{equation*}
\begin{split}
\vertiii{T^{\gamma}_{\Lambda}T^{\gamma}_{\sigma}Z_4}^2_0+\Re\left\langle T^{\gamma}_{\sigma}Z_4,T^{\gamma}_{\Lambda}T^{\gamma}_{\sigma}Z_4 \right\rangle|_{x_3=0}
& \lesssim \left(C+\frac{1}{\varepsilon} \right) \vertiii{Z_4}^2_{1,\gamma} + \sum_{i\neq 3,4}\frac{1}{\varepsilon} \left( \vertiii{T^{\gamma}_{\sigma}Z_i}^2_0 + \vertiii{Z_i }^2_0 \right) \\
&\quad\ +\frac{1}{\varepsilon} \left( \vertiii{T^{\gamma}_{r}W^+}^2_{1,\gamma} + \vertiii{W^+}^2_0 + \vertiii{ F^+ }^2_{1,\gamma} \right).
\end{split}
\end{equation*}
Using the fact that
$$
\Re \left\langle T^{\gamma}_{\sigma}Z_4,T^{\gamma}_{\Lambda}T^{\gamma}_{\sigma}Z_4 \right\rangle|_{x_3=0} = \Re \left. \left\langle T^{\gamma}_{\Lambda^{\frac{1}{2}}}T^{\gamma}_{\sigma}Z_4,T^{\gamma}_{\Lambda^{\frac{1}{2}}}T^{\gamma}_{\sigma}Z_4 \right\rangle \right|_{x_3=0} + \Re \left. \left\langle T^{\gamma}_{\Lambda^{\frac{1}{2}}}T^{\gamma}_{\sigma}Z_4,\mathcal{R}_0T^{\gamma}_{\sigma}Z_4 \right\rangle \right|_{x_3=0},
$$
we can obtain that
\begin{equation}\label{para9}
\begin{split}
&\vertiii{T^{\gamma}_{\Lambda}T^{\gamma}_{\sigma}Z_4}^2_0 + \| T^{\gamma}_{\Lambda^{\frac{1}{2}}}T^{\gamma}_{\sigma}Z_4|_{x_3=0} \|^2_0\\
&\quad\lesssim \| T^{\gamma}_{\sigma}Z_4|_{x_3=0} \|^2_0 + \left( C+\frac{1}{\varepsilon} \right)\vertiii{Z_4}^2_{1,\gamma}+\sum_{i\neq 3,4}\frac{1}{\varepsilon} \left(\vertiii{T^{\gamma}_{\sigma}Z_i}^2_0+\vertiii{Z_i }^2_0 \right)\\
&\quad\quad+\frac{1}{\varepsilon} \left(\vertiii{T^{\gamma}_{r}W^+}^2_{1,\gamma}+\vertiii{W^+}^2_0+\vertiii{ F^+ }^2_{1,\gamma} \right).
\end{split}
\end{equation}
In the following we apply the second symmetrizer $(T^{\gamma}_{\Lambda})^{\ast}T^{\gamma}_{\Lambda}$ to have
\begin{equation*}
\begin{split}
\partial_3\Re\langle T^{\gamma}_{\Lambda}Z_4,T^{\gamma}_{\Lambda}Z_4\rangle & = 2\Re\langle T^{\gamma}_{\Lambda}Z_4,T^{\gamma}_{\Lambda}\partial_3Z_4\rangle \\
& = 2\Re\langle T^{\gamma}_{\Lambda}Z_4, T^{\gamma}_{\Lambda}T^{\gamma}_{-\omega^r+i\bar{\omega}^r}Z_4\rangle + 2\Re\langle T^{\gamma}_{\Lambda}Z_4,T^{\gamma}_{\Lambda}T^{\gamma}_{\Theta_0}Z_4\rangle+2\sum_{i\neq3,4}\Re\langle T^{\gamma}_{\Lambda}Z_4,T^{\gamma}_{\Lambda}T^{\gamma}_{\Theta_0}Z_i\rangle\\
&\quad+2\Re\langle T^{\gamma}_{\Lambda}Z_4,T^{\gamma}_{\Lambda}T^{\gamma}_{r}W^+\rangle+2\Re\langle T^{\gamma}_{\Lambda}Z_4,T^{\gamma}_{\Lambda}\mathcal{R}_{-1}W^+\rangle+2\Re\langle T^{\gamma}_{\Lambda}Z_4,T^{\gamma}_{\Lambda}F^+\rangle.\\
\end{split}
\end{equation*}
Hence, we get
\begin{equation}\label{para10}
\begin{split}
\vertiii{Z_4}^2_{\frac{3}{2},\gamma} + \| Z_4|_{x_3=0} \|^2_{1,\gamma} & \lesssim C\vertiii{Z_4}^2_{1,\gamma}+\frac{1}{\varepsilon}\vertiii{Z_4}^2_{\frac{1}{2},\gamma}+\sum_{i\neq 3,4} \left(\frac{1}{\varepsilon}\vertiii{Z_i }^2_{\frac{1}{2},\gamma}+\frac{1}{\varepsilon\gamma}\vertiii{Z_i }^2_0 \right)\\
&\quad + \frac{1}{\varepsilon\gamma} \left( \vertiii{T^{\gamma}_{r}W^+}^2_{1,\gamma}+\vertiii{W^+}^2_0+\vertiii{ F^+ }^2_{1,\gamma} \right).
\end{split}
\end{equation}

Then, we consider 1st, 2nd, 5th, 6th, 8th, 9th, 11th, 12th of the system \eqref{para6} which can be written as
\begin{equation}\label{para11}
T^{\gamma}_{\mathbf{a}}(Z_1,Z_2,Z_5,Z_6,Z_8,Z_9,Z_{11},Z_{12})^{\top}+T^{\gamma}_{\Theta_1}Z_4+T^{\gamma}_{\Theta_0}Z^++T^{\gamma}_rW^++\mathcal{R}_{-1}W^+=\mathcal{R}_0F^+,
\end{equation}
where $\Theta_1$ is an $8\times1$ matrix symbol and belongs to $\Gamma^1_2,$ $\Theta_0$ is an $8\times13$ matrix symbol and belongs to $\Gamma^0_1,$ and $\mathbf{a} \in \Gamma^1_2$ is an $8\times8$ matrix symbol given as follows
\begin{equation}\nonumber
\mathbf{a}=\left[\begin{array}{cccccccc}
k^r_1 & 0 & \mathbf{a}_{1,3} & 0 & \mathbf{a}_{1,5} & 0 & \mathbf{a}_{1,7}  & 0\\
0 & k^r_1 & 0 & \mathbf{a}_{2,4} & 0 &\mathbf{a}_{2,6} & 0 & \mathbf{a}_{2,8}\\
\mathbf{a}_{3,1} & 0 & k^r_1 & 0  & 0 & 0 & 0 & 0\\
0 & \mathbf{a}_{4,2} & 0 & k^r_1 &0 & 0 & 0 & 0\\
\mathbf{a}_{5,1} & 0& 0 & 0 & k^r_1  & 0 & 0 &0\\
0 & \mathbf{a}_{6,2} & 0 & 0& 0 &k^r_1 & 0 & 0\\
\mathbf{a}_{7,1} & 0 & 0 & 0& 0 & 0 & k^r_1 & 0\\
0 & \mathbf{a}_{8,2} & 0 & 0& 0 & 0 & 0 & k^r_1\\
\end{array}
\right].
\end{equation}
Here, we denote that
\begin{equation*}
\begin{split}
\mathbf{a}_{1,3} & =\mathbf{a}_{2,4}=\mathbf{a}_{3,1}=\mathbf{a}_{4,2}=-i\eta F^r_{11}-i\tilde{\eta} F^r_{21}, \\
\mathbf{a}_{1,5} & =\mathbf{a}_{2,6}=\mathbf{a}_{5,1}=\mathbf{a}_{6,2}=-i\eta F^r_{12}-i\tilde{\eta} F^r_{22}, \\
\mathbf{a}_{1,7} & =\mathbf{a}_{2,8}=\mathbf{a}_{7,1}=\mathbf{a}_{8,2}=-i\eta F^r_{13}-i\tilde{\eta} F^r_{23}.
\end{split}
\end{equation*}
Now, we apply the symbol $\frac{\mathbf{a}^{\ast}}{\Lambda^5}\in \Gamma^0_2$ to \eqref{para11} with $\mathbf{a}^{\ast}$ the adjoint of $\mathbf{a}.$
Hence, we obtain that
$$
T^{\gamma}_aZ_j+T^{\gamma}_{\Theta_1}Z_4+ \sum_i T^{\gamma}_{\Theta_0}Z_i+T^{\gamma}_rW^++\mathcal{R}_{-1}W^+=\mathcal{R}_0F^+,
$$
where $j=1,2,5,6,8,9,11,12$ and $a=(k^r_1)^4\frac{(k^r_1)^2+k^r_2}{\Lambda^5}.$
From the definition of cut-off function, $\frac{(k^r_1)^2+k^r_2}{\Lambda^5}$ is non-zero in the support of $\chi_2.$
Hence, we have
$$a=(1-\chi_2)a+\chi_2(k^r_1)^4\frac{(k^r_1)^2+k^r_2}{\Lambda^5}.$$
Since $Z^+=T^{\gamma}_{\chi_1(Q^r_0+Q^r_{-1})}W^+_{p_1},$
we have
$$T^{\gamma}_{(1-\chi_2)a}Z_j=T^{\gamma}_{(1-\chi_2)a}T^{\gamma}_{\chi_1 Q^r_j}T^{\gamma}_{\chi_{p_1}}W^+=\mathcal{R}_{-1}T^{\gamma}_{\chi_{p_1}}W^+.$$
Here, the support of $(1-\chi_2)a$ is disjoint with $\chi_{1} Q^r_j$. Also $\sigma=0$ holds at the frequency points where $k^r_1=0$ (also possibly $k^l_1=0$) in the support of $\chi_{2}.$ We can write $\chi_2a=\chi_2\Theta_0  (\gamma+i\sigma),$ and hence
\begin{equation}\label{para12}
T^{\gamma}_{\chi_2\Theta_0 (\gamma+i\sigma)}Z_j+T^{\gamma}_{\Theta_1}Z_4 + \sum_i T^{\gamma}_{\Theta_0}Z_i+T^{\gamma}_rW^++\mathcal{R}_{-1}W^+=\mathcal{R}_0F^+.
\end{equation}
Applying the symmetrizer $(T^{\gamma}_{\sigma})^{\ast}T^{\gamma}_{\sigma},$ we obtain that
\begin{equation}\label{para122}
\begin{split}
&\Re\langle T^{\gamma}_{\sigma}Z_j,T^{\gamma}_{\sigma}T^{\gamma}_{\chi_2\Theta_0 (\gamma+i\sigma)}Z_j\rangle+\Re\langle T^{\gamma}_{\sigma}Z_j,T^{\gamma}_{\sigma}T^{\gamma}_{\Theta_1}Z_4\rangle + \sum_i \Re\langle T^{\gamma}_{\sigma}Z_j, T^{\gamma}_{\sigma}T^{\gamma}_{\Theta_0}Z_i\rangle\\
&\quad+\Re\langle T^{\gamma}_{\sigma}Z_j, T^{\gamma}_{\sigma}T^{\gamma}_rW^+\rangle+\Re\langle T^{\gamma}_{\sigma}Z_j, T^{\gamma}_{\sigma}\mathcal{R}_{-1}W^+\rangle=\Re\langle T^{\gamma}_{\sigma}Z_j,T^{\gamma}_{\sigma}F^+\rangle.
\end{split}
\end{equation}
Now, using Lemma \ref{lemmapara} (iii)-(iv), we estimate the above terms one by one,
\begin{equation*}
\begin{split}
\Re\langle T^{\gamma}_{\sigma}Z_j,T^{\gamma}_{\sigma}T^{\gamma}_{\Theta_1}Z_4\rangle & = \Re\langle T^{\gamma}_{\sigma}Z_j,T^{\gamma}_{\frac{\Theta_1}{\Lambda}}T^{\gamma}_{\Lambda}T^{\gamma}_{\sigma}Z_4\rangle+\Re\langle T^{\gamma}_{\sigma}Z_j,\mathcal{R}_1Z_4\rangle \\
& \leq \varepsilon\gamma \|T^{\gamma}_{\sigma}Z_j \|^2_0+\frac{1}{\varepsilon\gamma}\|T^{\gamma}_{\Lambda}T^{\gamma}_{\sigma}Z_4\|^2_0+\frac{1}{\varepsilon\gamma}\|Z_4\|^2_{1,\gamma}, \\
\Re\langle T^{\gamma}_{\sigma}Z_j,T^{\gamma}_{\sigma}T^{\gamma}_{\Theta_0}Z_i\rangle & = \Re\langle T^{\gamma}_{\sigma}Z_j,T^{\gamma}_{\Theta_0}T^{\gamma}_{\sigma}Z_i\rangle+\Re\langle T^{\gamma}_{\sigma}Z_j,\mathcal{R}_0Z_i\rangle\\
& \leq\varepsilon\gamma \|T^{\gamma}_{\sigma}Z_j \|^2_0+\frac{1}{\varepsilon\gamma}\|T^{\gamma}_{\sigma}Z_i\|^2_0+\frac{1}{\varepsilon\gamma}\|Z_i \|^2_{0}, \\
\Re\langle T^{\gamma}_{\sigma}Z_j, T^{\gamma}_{\sigma}T^{\gamma}_rW^+\rangle & \leq \varepsilon\gamma\|T^{\gamma}_{\sigma}Z_j \|^2_0+\frac{1}{\varepsilon\gamma}\|T^{\gamma}_rW^+ \|^2_{1,\gamma}, \\
\Re\langle T^{\gamma}_{\sigma}Z_j, T^{\gamma}_{\sigma}\mathcal{R}_{-1}W^+\rangle & \leq \varepsilon\gamma\|T^{\gamma}_{\sigma}Z_j \|^2_0+\frac{1}{\varepsilon\gamma}\| W^+ \|^2_{0}, \\
\Re\langle T^{\gamma}_{\sigma}Z_j, T^{\gamma}_{\sigma}F^+\rangle & \leq \varepsilon\gamma\|T^{\gamma}_{\sigma}Z_j \|^2_0+\frac{1}{\varepsilon\gamma}\| F^+ \|^2_{1,\gamma}.
\end{split}
\end{equation*}
For the first term on the left-hand side of \eqref{para122}, we obtain that
\begin{equation*}
\Re\langle T^{\gamma}_{\sigma}Z_j, T^{\gamma}_{\sigma}T^{\gamma}_{\chi_2\Theta_0 (\gamma+i\sigma)}Z_j\rangle =\Re\langle T^{\gamma}_{\sigma}Z_j, T^{\gamma}_{\sigma}T^{\gamma}_{\chi_2\Theta_0}T^{\gamma}_{\gamma+i\sigma}Z_j\rangle+\Re\langle T^{\gamma}_{\sigma}Z_j, T^{\gamma}_{\Theta_0}T^{\gamma}_{\sigma}Z_j\rangle+\Re\langle T^{\gamma}_{\sigma}Z_j, \mathcal{R}_0Z_j\rangle.
\end{equation*}
Then, we have
\begin{equation*}
\begin{split}
\Re\langle T^{\gamma}_{\sigma}Z_j, \mathcal{R}_0Z_j\rangle & \leq \varepsilon \gamma\|T^{\gamma}_{\sigma}Z_j \|^2_0+\frac{1}{\varepsilon\gamma}\| Z_j \|^2_0,\\
\Re\langle T^{\gamma}_{\sigma}Z_j, T^{\gamma}_{\Theta_0}T^{\gamma}_{\sigma}Z_j\rangle & \leq C\|T^{\gamma}_{\sigma}Z_j \|^2_0.
\end{split}
\end{equation*}
Note that
\begin{equation}\label{para13}
\Re\langle T^{\gamma}_{\sigma}Z_j, T^{\gamma}_{\sigma}T^{\gamma}_{\chi_2\Theta_0}T^{\gamma}_{\gamma+i\sigma}Z_j\rangle=\Re\left\langle T^{\gamma}_{\sigma}Z_j, T^{\gamma}_{\chi_2\Theta_0}T^{\gamma}_{\sigma}T^{\gamma}_{\gamma+i\sigma}Z_j \right\rangle+\Re\langle T^{\gamma}_{\sigma}Z_j, \mathcal{R}_0T^{\gamma}_{\gamma+i\sigma}Z_j\rangle.\\
\end{equation}
For the second term on the right-hand side of \eqref{para13}, we have
\begin{equation*}
\begin{split}
\Re\langle T^{\gamma}_{\sigma}Z_j, \mathcal{R}_0T^{\gamma}_{\gamma+i\sigma}Z_j\rangle&=
\Re\langle T^{\gamma}_{\sigma}Z_j, \mathcal{R}_0T^{\gamma}_{\gamma}Z_j\rangle+
\Re\langle T^{\gamma}_{\sigma}Z_j, \mathcal{R}_0T^{\gamma}_{i\sigma}Z_j\rangle\\
&\leq\varepsilon \gamma\|T^{\gamma}_{\sigma}Z_j \|^2_0+\frac{1}{\varepsilon\gamma} \|T^{\gamma}_{\gamma}Z_j\|^2_0 + C\|T^{\gamma}_{\sigma}Z_j \|^2_0.
\end{split}
\end{equation*}
The first term on the right-hand side of \eqref{para13} can be written into
\begin{equation}\label{para14}
\begin{split}
&\Re\left\langle T^{\gamma}_{\sigma}Z_j, T^{\gamma}_{\chi_2\Theta_0}T^{\gamma}_{\sigma}T^{\gamma}_{\gamma+i\sigma}Z_j \right\rangle =
\Re \left\langle T^{\gamma}_{\sigma}Z_j, T^{\gamma}_{\chi_2\Theta_0}T^{\gamma}_{\sigma}T^{\gamma}_{\gamma}Z_j \right\rangle +
\Re \left\langle T^{\gamma}_{\sigma}Z_j, T^{\gamma}_{\chi_2\Theta_0}T^{\gamma}_{\sigma}T^{\gamma}_{i\sigma}Z_j \right\rangle \\
&=\gamma\Re \left\langle T^{\gamma}_{\sigma}Z_j, T^{\gamma}_{\chi_2\Theta_0}\mathcal{R}_0Z_j \right\rangle +
\Re\langle T^{\gamma}_{\sigma}Z_j, T^{\gamma}_{\chi_2\Theta_0}T^{\gamma}_{\gamma+i\sigma}T^{\gamma}_{\sigma}Z_j\rangle\\
&=\gamma\Re\left\langle T^{\gamma}_{\sigma}Z_j, T^{\gamma}_{\chi_2\Theta_0}\mathcal{R}_0Z_j \right\rangle +
\Re\left\langle T^{\gamma}_{\sigma}Z_j, T^{\gamma}_{\chi_2\Theta_0(\gamma+i\sigma)}T^{\gamma}_{\sigma}Z_j \right\rangle + \Re \left\langle T^{\gamma}_{\sigma}Z_j,\mathcal{R}_0T^{\gamma}_{\sigma}Z_j \right\rangle.\\
\end{split}
\end{equation}
The first and third terms can be estimated by Cauchy-Schwarz inequality, for the second term
\begin{equation}\label{para15}
\begin{split}
\Re \left\langle T^{\gamma}_{\sigma}Z_j, T^{\gamma}_{\chi_2\Theta_0(\gamma+i\sigma)}T^{\gamma}_{\sigma}Z_j \right\rangle & = \Re \left\langle T^{\gamma}_{\sigma}Z_j, T^{\gamma}_{\chi_2 a}T^{\gamma}_{\sigma}Z_j \right\rangle \\
& = \Re \langle T^{\gamma}_{\sigma}Z_j, T^{\gamma}_{\tilde{a}}T^{\gamma}_{\sigma}Z_j\rangle + \Re\left\langle T^{\gamma}_{\sigma}Z_j, T^{\gamma}_{(\chi_2-1)\tilde{a}}T^{\gamma}_{\sigma}Z_j \right\rangle,
\end{split}
\end{equation}
where $\tilde{a}$ is the extension of $\chi_2 a$ to the whole space with $|\Re\tilde{a}|\geq c\gamma,$ for some fixed positive constant $c.$
For the second term on the right-hand side of \eqref{para15},
\begin{equation}\nonumber
\begin{split}
T^{\gamma}_{(\chi_2-1)\tilde{a}}T^{\gamma}_{\sigma}Z_j&=
T^{\gamma}_{(\chi_2-1)\tilde{a}}T^{\gamma}_{\sigma}T^{\gamma}_{\chi_1 Q^r_j}T^{\gamma}_{\chi_{p_1}}W^+\\
&=T^{\gamma}_{(\chi_2-1)\tilde{a}}T^{\gamma}_{\chi_1 Q^r_j}T^{\gamma}_{\sigma}T^{\gamma}_{\chi_{p_1}}W^++T^{\gamma}_{(\chi_2-1)\tilde{a}}T^{\gamma}_{O_0}T^{\gamma}_{\chi_{p_1}}W^+\\
&\quad+T^{\gamma}_{(\chi_2-1)\tilde{a}}T^{\gamma}_{O_{-1}}T^{\gamma}_{\chi_{p_1}}W^++T^{\gamma}_{(\chi_2-1)\tilde{a}}\mathcal{R}_{-2}T^{\gamma}_{\chi_{p_1}}W^+,\\
\end{split}
\end{equation}
where $O_0$ and $O_{-1}$ are only supported on the support of $\chi_1$ which is disjoint with the support of $(\chi_2-1)\tilde{a}.$
Hence, we obtain that
\begin{equation}\nonumber
T^{\gamma}_{(\chi_2-1)\tilde{a}}T^{\gamma}_{O_0}T^{\gamma}_{\chi_{p_1}}W^+=\mathcal{R}_{-1}W^+,
\end{equation}
\begin{equation}\nonumber
T^{\gamma}_{(\chi_2-1)\tilde{a}}T^{\gamma}_{O_{-1}}T^{\gamma}_{\chi_{p_1}}W^+=\mathcal{R}_{-1}W^+,
\end{equation}
\begin{equation}\nonumber
T^{\gamma}_{(\chi_2-1)\tilde{a}}T^{\gamma}_{\chi_1 Q^r_j}T^{\gamma}_{\sigma}T^{\gamma}_{\chi_{p_1}}W^+=\mathcal{R}_{-1}W^+.
\end{equation}
Then, we obtain that
\begin{equation}\label{para16}
\Re \left\langle T^{\gamma}_{\sigma}Z_j, T^{\gamma}_{(\chi_2-1)\tilde{a}}T^{\gamma}_{\sigma}Z_j \right\rangle \leq\varepsilon\gamma\|T^{\gamma}_{\sigma}Z_j \|^2_0+\frac{1}{\varepsilon\gamma}\| W^+ \|^2_0.
\end{equation}
Note that $|\Re \tilde{a}|\geq c\gamma,$ we have
\begin{equation}\label{para17}
\qquad |\Re\langle T^{\gamma}_{\sigma}Z_j,T^{\gamma}_{\tilde{a}}T^{\gamma}_{\sigma}Z_j\rangle|\geq c\gamma\|T^{\gamma}_{\sigma}Z_j \|^2_0.
\end{equation}
Estimates \eqref{para15}-\eqref{para17} yield
\begin{equation}\label{para18}
\begin{split}
\gamma\|T^{\gamma}_{\sigma}Z_j \|^2_0 & \leq\frac{1}{\varepsilon\gamma} \| Z_j \|^2_0+C\|T^{\gamma}_{\sigma}Z_j \|^2_0+\frac{1}{\varepsilon}\| Z_j \|^2_0+\frac{1}{\varepsilon\gamma} \left(\|T^{\gamma}_{\Lambda}T^{\gamma}_{\sigma}Z_4\|^2_0+\|Z_4\|^2_{1,\gamma} \right)\\
&\quad+\sum_i\frac{1}{\varepsilon\gamma} \left(\|T^{\gamma}_{\sigma}Z_i\|^2_0+\|Z_i \|^2_0 \right) + \frac{1}{\varepsilon\gamma} \left( \|T^{\gamma}_rW^+ \|^2_{1,\gamma}+\| W^+ \|^2_0+\| F^+ \|^2_{1,\gamma} \right),
\end{split}
\end{equation}
for $j=1,2,5,6,8,9,11,12.$

Similar to the outgoing mode $Z_4,$ we apply symmetrizer $T^{\gamma}_{\Lambda}$ to \eqref{para12},
\begin{equation*}
\begin{split}
&\Re \left\langle Z_j,T^{\gamma}_{\Lambda}T^{\gamma}_{\chi_2\Theta_0\times(\gamma+i\sigma)}Z_j \right\rangle + \Re \left\langle Z_j,T^{\gamma}_{\Lambda}T^{\gamma}_{\Theta_1}Z_4 \right\rangle + \sum_i \Re\left\langle Z_j, T^{\gamma}_{\Lambda}T^{\gamma}_{\Theta_0}Z_i \right\rangle\\
&\quad + \Re \left\langle Z_j, T^{\gamma}_{\Lambda}T^{\gamma}_rW^+ \right\rangle + \Re\left\langle Z_j, T^{\gamma}_{\Lambda}\mathcal{R}_{-1}W^+\right\rangle = \Re\left\langle Z_j,T^{\gamma}_{\Lambda}F^+ \right\rangle.
\end{split}
\end{equation*}
In the above, the first term can be written as
\begin{equation*}
\begin{split}
\Re \left\langle Z_j, T^{\gamma}_{\Lambda}T^{\gamma}_{\chi_2\Theta_0\times(\gamma+i\sigma)}Z_j \right\rangle & = \Re \left\langle T^{\gamma}_{\Lambda^{\frac{1}{2}}}Z_j, T^{\gamma}_{\Lambda^{\frac{1}{2}}}T^{\gamma}_{\chi_2\Theta_0\times(\gamma+i\sigma)}Z_j \right\rangle\\
&\quad + \Re \left\langle Z_j, \mathcal{R}_0T^{\gamma}_{\chi_2\Theta_0\times(\gamma+i\sigma)}Z_j \right\rangle\\
& = \Re \left\langle T^{\gamma}_{\Lambda^{\frac{1}{2}}}Z_j, T^{\gamma}_{\chi_2\Theta_0\times(\gamma+i\sigma)}T^{\gamma}_{\Lambda^{\frac{1}{2}}}Z_j \right\rangle + \Re \left\langle T^{\gamma}_{\Lambda^{\frac{1}{2}}} Z_j, \mathcal{R}_{\frac{1}{2}}Z_j \right\rangle\\
&\quad + \Re \left\langle Z_j, \mathcal{R}_0T^{\gamma}_{\chi_2\Theta_0\times(\gamma+i\sigma)}Z_j \right\rangle.
\end{split}
\end{equation*}
For $ \Re\left\langle T^{\gamma}_{\Lambda^{\frac{1}{2}}}Z_j, T^{\gamma}_{\chi_2\Theta_0\times(\gamma+i\sigma)}T^{\gamma}_{\Lambda^{\frac{1}{2}}}Z_j \right\rangle,$ we can split into two terms:
\begin{equation*}
\begin{split}
\Re\left\langle T^{\gamma}_{\Lambda^{\frac{1}{2}}}Z_j, T^{\gamma}_{\chi_2\Theta_0\times(\gamma+i\sigma)}T^{\gamma}_{\Lambda^{\frac{1}{2}}}Z_j \right\rangle & = \Re \left\langle T^{\gamma}_{\Lambda^{\frac{1}{2}}}Z_j, T^{\gamma}_{\tilde{a}}T^{\gamma}_{\Lambda^{\frac{1}{2}}}Z_j \right\rangle + \Re \left\langle T^{\gamma}_{\Lambda^{\frac{1}{2}}}Z_j, T^{\gamma}_{(\chi_2-1)\tilde{a}}T^{\gamma}_{\Lambda^{\frac{1}{2}}}Z_j \right\rangle.
\end{split}
\end{equation*}
We can estimate that
\begin{equation*}
\begin{split}
\Re\left\langle T^{\gamma}_{\Lambda^{\frac{1}{2}}}Z_j, T^{\gamma}_{(\chi_2-1)\tilde{a}}T^{\gamma}_{\Lambda^{\frac{1}{2}}}Z_j \right\rangle & \leq C\| Z_j \|^2_{\frac{1}{2},\gamma}+\varepsilon\gamma\| Z_j \|^2_{\frac{1}{2},\gamma}+\frac{1}{\varepsilon\gamma}\| W^+ \|^2_{-\frac{1}{2},\gamma}, \\
\left| \Re\left\langle T^{\gamma}_{\Lambda^{\frac{1}{2}}}Z_j, T^{\gamma}_{\tilde{a}}T^{\gamma}_{\Lambda^{\frac{1}{2}}}Z_j \right\rangle \right| & \geq c\gamma\| Z_j \|^2_{\frac{1}{2},\gamma}, \\
\Re\left\langle Z_j, \mathcal{R}_0T^{\gamma}_{\chi_2\Theta_0\times(\gamma+i\sigma)}Z_j \right\rangle & \leq C\| Z_j \|^2_{\frac{1}{2},\gamma}, \\
\Re\langle Z_j,T^{\gamma}_{\Lambda}T^{\gamma}_{\Theta_1}Z_4\rangle & \leq \varepsilon\gamma\| Z_j \|^2_{\frac{1}{2},\gamma}+\frac{1}{\varepsilon\gamma}\|Z_4\|^2_{\frac{3}{2},\gamma}, \\
\Re \left\langle Z_j,T^{\gamma}_{\Lambda}T^{\gamma}_{\Theta_0}Z_i \right\rangle & \leq \varepsilon\gamma\| Z_j \|^2_{\frac{1}{2},\gamma}+\frac{1}{\varepsilon\gamma}\|Z_i \|^2_{\frac{1}{2},\gamma}, \\
\Re \left\langle Z_j,T^{\gamma}_{\Lambda}T^{\gamma}_{r}W^+ \right\rangle & \leq \varepsilon\gamma^2\| Z_j \|^2_{0}+\frac{1}{\varepsilon\gamma^2}\|T^{\gamma}_rW^+ \|^2_{1,\gamma}, \\
\Re \left\langle Z_j,T^{\gamma}_{\Lambda}\mathcal{R}_{-1}W^+ \right\rangle & \leq \varepsilon\gamma^2\| Z_j \|^2_{0}+\frac{1}{\varepsilon\gamma^2}\| W^+ \|^2_{0}, \\
\Re \left\langle Z_j,T^{\gamma}_{\Lambda}F^+ \right\rangle & \leq \varepsilon\gamma^2\| Z_j \|^2_{0}+\frac{1}{\varepsilon\gamma^2}\| F^+ \|^2_{1,\gamma}.
\end{split}
\end{equation*}
Thus we obtain
\begin{equation}\label{para20}
\begin{split}
\gamma\| Z_j \|^2_{\frac{1}{2},\gamma}&\leq C\| Z_j \|^2_{\frac{1}{2},\gamma}+\frac{1}{\varepsilon\gamma}\|Z_4\|^2_{\frac{3}{2},\gamma}\\
&\quad + \frac{1}{\varepsilon\gamma} \sum_i \|Z_i \|^2_{\frac{1}{2},\gamma}+\frac{1}{\varepsilon\gamma^2} \left(\|T^{\gamma}_rW^+ \|^2_{1,\gamma} + \| W^+ \|^2_0+\| F^+ \|^2_{1,\gamma} \right),
\end{split}
\end{equation}
for $j=1,2,5,6,8,9,11,12,$ by taking $\varepsilon$ small enough.

For $j = 7,10,13$ in \eqref{para6}, we have
$$
T^{\gamma}_{\tau+iv^r_1\eta+iv^r_2\tilde{\eta}}Z_j+T^{\gamma}_{\Theta_1}Z_4 + \sum_i  T^{\gamma}_{\Theta_0}Z_i+T^{\gamma}_rW^++\mathcal{R}_{-1}W^+=\mathcal{R}_0F^+.$$
Following the same estimates for $Z_j$ with $j=1,2,5,6,8,9,11,12,$ we have
\begin{equation}\label{para21}
\begin{split}
\gamma\|T^{\gamma}_{\sigma}Z_j \|^2_0 & \leq\frac{1}{\varepsilon\gamma}\| Z_j \|^2_0+C\|T^{\gamma}_{\sigma}Z_j \|^2_0+\frac{1}{\varepsilon}\| Z_j \|^2_0+\frac{1}{\varepsilon\gamma} \left(\|T^{\gamma}_{\Lambda}T^{\gamma}_{\sigma}Z_4\|^2_0+\|Z_4\|^2_{1,\gamma} \right)\\
&\quad + \sum_i\frac{1}{\varepsilon\gamma} \left(\|T^{\gamma}_{\sigma}Z_i\|^2_0+\|Z_i \|^2_0 \right)\\
&\quad+\frac{1}{\varepsilon\gamma} \left(\|T^{\gamma}_rW^+ \|^2_{1,\gamma}+\| W^+ \|^2_0+\| F^+ \|^2_{1,\gamma} \right),
\end{split}
\end{equation}
and
\begin{equation}\label{para22}
\begin{split}
\gamma\| Z_j \|^2_{\frac{1}{2},\gamma}&\leq C\| Z_j \|^2_{\frac{1}{2},\gamma}+\frac{1}{\varepsilon\gamma}\|Z_4\|^2_{\frac{3}{2},\gamma}\\
&\quad+\frac{1}{\varepsilon\gamma} \sum_i \|Z_i \|^2_{\frac{1}{2},\gamma}+\frac{1}{\varepsilon\gamma^2}(\|T^{\gamma}_rW^+ \|^2_{1,\gamma}+\| W^+ \|^2_0+\| F^+ \|^2_{1,\gamma})
\end{split}
\end{equation}
for $j=7,10,13$. Note from \eqref{para18}--\eqref{para20} and \eqref{para21}--\eqref{para22} that the estimates for the terms with $j = 1,2,5,6,7,8,9,10,11,12,13$ are exactly the same.

For the incoming mode $Z_3$ of \eqref{para6},
\begin{equation}\label{para23}
\begin{split}
\partial_3 Z_3&=T^{\gamma}_{\omega^r+i\bar{\omega}^r}Z_3+T^{\gamma}_{\Theta_1}Z_1+T^{\gamma}_{\Theta_1}Z_2+T^{\gamma}_{\Theta_1}Z_7+T^{\gamma}_{\Theta_1}Z_{10}+T^{\gamma}_{\Theta_1}Z_{13}\\
&\quad+T^{\gamma}_{\Theta_0}Z_3+\sum_{i\neq3,4}T^{\gamma}_{\Theta_0}Z_i+T^{\gamma}_rW^++\mathcal{R}_{-1}W^++F^+.
\end{split}
\end{equation}
First, we apply symmetrizer $(T^{\gamma}_{\sigma})^{\ast}T^{\gamma}_{\frac{1}{\Lambda}}T^{\gamma}_{\sigma}$ to \eqref{para23},
\begin{equation*}
\begin{split}
\Re \left\langle T^{\gamma}_{\sigma}Z_3, T^{\gamma}_{\frac{1}{\Lambda}}T^{\gamma}_{\sigma}\partial_3Z_3 \right\rangle & = \Re \left\langle T^{\gamma}_{\sigma}Z_3, T^{\gamma}_{\frac{1}{\Lambda}}T^{\gamma}_{\sigma}T^{\gamma}_{\omega^r+i\bar{\omega}^r}Z_3 \right\rangle + \sum_{j=1,2,7,10,13} \Re \left\langle T^{\gamma}_{\sigma}Z_3, T^{\gamma}_{\frac{1}{\Lambda}}T^{\gamma}_{\sigma}T^{\gamma}_{\Theta_1}Z_j \right\rangle \\
&\quad + \Re \left\langle T^{\gamma}_{\sigma}Z_3, T^{\gamma}_{\frac{1}{\Lambda}}T^{\gamma}_{\sigma}T^{\gamma}_{\Theta_0}Z_3 \right\rangle + \sum_{i\neq3,4} \Re \left\langle T^{\gamma}_{\sigma}Z_3, T^{\gamma}_{\frac{1}{\Lambda}}T^{\gamma}_{\sigma}T^{\gamma}_{\Theta_0}Z_i \right\rangle \\
&\quad + \Re \left\langle T^{\gamma}_{\sigma}Z_3, T^{\gamma}_{\frac{1}{\Lambda}}T^{\gamma}_{\sigma}T^{\gamma}_rW^+ \right\rangle + \Re \left\langle T^{\gamma}_{\sigma}Z_3, T^{\gamma}_{\frac{1}{\Lambda}}T^{\gamma}_{\sigma}\mathcal{R}_{-1}W^+ \right\rangle \\
&\quad + \Re \left\langle T^{\gamma}_{\sigma}Z_3, T^{\gamma}_{\frac{1}{\Lambda}}T^{\gamma}_{\sigma}F^+ \right\rangle.
\end{split}
\end{equation*}
Similar to the case for the outgoing modes, we obtain
\begin{equation*}
\begin{split}
\partial_3 \Re \left\langle T^{\gamma}_{\sigma}Z_3,T^{\gamma}_{\frac{1}{\Lambda}}T^{\gamma}_{\sigma}Z_3 \right\rangle & = \Re \left\langle T^{\gamma}_{\partial_3\sigma}Z_3,T^{\gamma}_{\frac{1}{\Lambda}}T^{\gamma}_{\sigma}Z_3 \right\rangle + \Re \left\langle T^{\gamma}_{\sigma}Z_3,T^{\gamma}_{\frac{1}{\Lambda}}T^{\gamma}_{\partial_3\sigma}Z_3 \right\rangle\\
&\quad + \Re \left\langle T^{\gamma}_{\sigma}\partial_3Z_3,T^{\gamma}_{\frac{1}{\Lambda}}T^{\gamma}_{\sigma}Z_3 \right\rangle + \Re \left\langle T^{\gamma}_{\sigma}Z_3,T^{\gamma}_{\frac{1}{\Lambda}}T^{\gamma}_{\sigma}\partial_3Z_3 \right\rangle,\\
\Re\left\langle T^{\gamma}_{\sigma}Z_3, T^{\gamma}_{\frac{1}{\Lambda}}T^{\gamma}_{\sigma}T^{\gamma}_{\omega^r+i\bar{\omega}^r}Z_3 \right\rangle & = \Re \left\langle T^{\gamma}_{\sigma}Z_3, T^{\gamma}_{\frac{\omega^r+i\bar{\omega}^r}{\Lambda}}T^{\gamma}_{\sigma}Z_3 \right\rangle + \Re \left\langle T^{\gamma}_{\sigma}Z_3,\mathcal{R}_0Z_3 \right\rangle, \\
\Re\left\langle T^{\gamma}_{\partial_3\sigma}Z_3,T^{\gamma}_{\frac{1}{\Lambda}}T^{\gamma}_{\sigma}Z_3 \right\rangle & \leq \varepsilon \|T^{\gamma}_{\sigma}Z_3\|^2_0+\frac{1}{\varepsilon} \|Z_3\|^2_0, \\
\Re\left\langle T^{\gamma}_{\sigma}Z_3,T^{\gamma}_{\frac{1}{\Lambda}}T^{\gamma}_{\partial_3\sigma}Z_3 \right\rangle & \leq \varepsilon\|T^{\gamma}_{\sigma}Z_3\|^2_0+\frac{1}{\varepsilon}\|Z_3\|^2_0, \\
\Re\left\langle T^{\gamma}_{\sigma}\partial_3Z_3,T^{\gamma}_{\frac{1}{\Lambda}}T^{\gamma}_{\sigma}Z_3 \right\rangle & =
\Re \left\langle T^{\gamma}_{\frac{1}{\Lambda}}T^{\gamma}_{\sigma}\partial_3Z_3,T^{\gamma}_{\sigma}Z_3 \right\rangle + \Re \left\langle T^{\gamma}_{\sigma}\partial_3Z_3,\mathcal{R}_{-2}T^{\gamma}_{\sigma}Z_3 \right\rangle, \\
\Re\left\langle T^{\gamma}_{\sigma}Z_3, T^{\gamma}_{\frac{\omega^r+i\bar{\omega}^r}{\Lambda}}T^{\gamma}_{\sigma}Z_3 \right\rangle & \leq -c \|T^{\gamma}_{\sigma}Z_3\|^2_0, \\
\Re\left\langle T^{\gamma}_{\sigma}Z_3,\mathcal{R}_0Z_3 \right\rangle & \leq \varepsilon \|T^{\gamma}_{\sigma}Z_3\|^2_0 + \frac{1}{\varepsilon}\|Z_3\|^2_0, \\
\Re\left\langle T^{\gamma}_{\sigma}Z_3, T^{\gamma}_{\frac{1}{\Lambda}}T^{\gamma}_{\sigma}T^{\gamma}_{\Theta_1}Z_j \right\rangle & =
\Re \left\langle T^{\gamma}_{\sigma}Z_3, T^{\gamma}_{\frac{1}{\Lambda}}T^{\gamma}_{\Theta_1}T^{\gamma}_{\sigma}Z_j \right\rangle + \Re \left\langle T^{\gamma}_{\sigma}Z_3, \mathcal{R}_0Z_j \right\rangle, \\
& \leq \varepsilon \|T^{\gamma}_{\sigma}Z_3\|^2_0+\frac{1}{\varepsilon}\|T^{\gamma}_{\sigma}Z_j \|^2_0+\frac{1}{\varepsilon}\| Z_j \|^2_0, \quad j = 1, 2, 7, 10, 13, \\
\Re\left\langle T^{\gamma}_{\sigma}Z_3, T^{\gamma}_{\frac{1}{\Lambda}}T^{\gamma}_{\sigma}T^{\gamma}_{\Theta_0}Z_3 \right\rangle & =
\Re \left\langle T^{\gamma}_{\sigma}Z_3, T^{\gamma}_{\frac{1}{\Lambda}}T^{\gamma}_{\Theta_0}T^{\gamma}_{\sigma}Z_3 \right\rangle + \Re \left\langle T^{\gamma}_{\sigma}Z_3, \mathcal{R}_{-1}Z_3 \right\rangle \\
& \leq \varepsilon \|T^{\gamma}_{\sigma}Z_3\|^2_0 + \frac{1}{\varepsilon}\|T^{\gamma}_{\sigma}Z_3\|^2_{-1,\gamma} + \frac{1}{\varepsilon}\|Z_3\|^2_{-1,\gamma}, 
\end{split}
\end{equation*}
\begin{equation*}
\begin{split}
\Re\left\langle T^{\gamma}_{\sigma}Z_3, T^{\gamma}_{\frac{1}{\Lambda}}T^{\gamma}_{\sigma}T^{\gamma}_{\Theta_0}Z_i \right\rangle & \leq \varepsilon \|T^{\gamma}_{\sigma}Z_3\|^2_0 + \frac{1}{\varepsilon}\|T^{\gamma}_{\sigma}Z_i\|^2_{-1,\gamma} + \frac{1}{\varepsilon}\|Z_i \|^2_{-1,\gamma}, \quad i \ne 3, 4, \\
\Re \left\langle T^{\gamma}_{\sigma}Z_3, T^{\gamma}_{\frac{1}{\Lambda}}T^{\gamma}_{\sigma}T^{\gamma}_{r}W^+ \right\rangle & \leq \varepsilon\|T^{\gamma}_{\sigma}Z_3\|^2_0+\frac{1}{\varepsilon} \|T^{\gamma}_{r}W^+ \|^2_{0}, \\
\Re\left\langle T^{\gamma}_{\sigma}Z_3, T^{\gamma}_{\frac{1}{\Lambda}}T^{\gamma}_{\sigma}\mathcal{R}_{-1}W^+ \right\rangle & \leq \varepsilon\|T^{\gamma}_{\sigma}Z_3\|^2_0+\frac{1}{\varepsilon}\| W^+ \|^2_{-1,\gamma}, \\
\Re\left\langle T^{\gamma}_{\sigma}Z_3, T^{\gamma}_{\frac{1}{\Lambda}}T^{\gamma}_{\sigma}F^+ \right\rangle & \leq \varepsilon\|T^{\gamma}_{\sigma}Z_3\|^2_0+\frac{1}{\varepsilon}\| F^+ \|^2_{0}.
\end{split}
\end{equation*}
Then, we have
\begin{equation}\label{para24}
\begin{split}
\vertiii{T^{\gamma}_{\sigma}Z_3 }^2_0 & \leq \Re \left\langle T^{\gamma}_{\sigma}Z_3, T^{\gamma}_{\frac{1}{\Lambda}}T^{\gamma}_{\sigma}Z_3 \right\rangle\Big|_{x_3=0} + \frac{1}{\varepsilon}\vertiii{ Z_3 }^2_0 + \frac{1}{\varepsilon} \sum_{j = 1,2,7,10,13} \left(\vertiii{ T^{\gamma}_{\sigma}Z_j }^2_0+\vertiii{ Z_j }^2_0 \right) \\
&\quad + \frac{1}{\varepsilon} \left(\vertiii{T^{\gamma}_{\sigma}Z_3 }^2_{-1,\gamma}+\vertiii{ Z_3 }^2_{-1,\gamma} \right) + \sum_{i\neq3,4}\frac{1}{\varepsilon} \left(\vertiii{T^{\gamma}_{\sigma}Z_i}^2_{-1,\gamma}+\vertiii{Z_i }^2_{-1,\gamma} \right)\\
&\quad + \frac{1}{\varepsilon} \left(\vertiii{ T^{\gamma}_rW^+ }^2_0+\vertiii{W^+}^2_{-1,\gamma}+\vertiii{ F^+ }^2_0 \right).
\end{split}
\end{equation}
Now, we apply symmetrizer $1$ to obtain that
\begin{equation*}
\begin{split}
\partial_3\Re\langle Z_3,Z_3\rangle & = 2\Re \langle Z_3,\partial_3Z_3\rangle\\
& = 2 \Re \left\langle Z_3, T^{\gamma}_{\omega^r+i\bar{\omega}^r}Z_3 \right\rangle + \sum_{j=1,2,7,10,13} 2 \Re \left\langle Z_3,T^{\gamma}_{\Theta_1}Z_j \right\rangle + 2 \Re \left\langle Z_3, T^{\gamma}_{\Theta_0}Z_3 \right\rangle \\
&\quad + \sum_{i\neq3,4} 2\Re \left\langle Z_3,T^{\gamma}_{\Theta_0}Z_i \right\rangle + 2\Re\left\langle Z_3,T^{\gamma}_rW^+ \right\rangle \\
&\quad + 2\Re \left\langle Z_3,\mathcal{R}_{-1}W^+ \right\rangle + 2\Re \left\langle Z_3,F^+ \right\rangle.
\end{split}
\end{equation*}
Note that
\begin{equation*}
\begin{split}
\partial_3\Re\langle Z_3,Z_3\rangle & =2\Re\langle Z_3,\partial_3Z_3\rangle \\
& = 2\Re \left\langle Z_3,T^{\gamma}_{\omega^r+i\bar{\omega}^r}Z_3 \right\rangle + \sum_{j=1,2,7,10,13}  2\Re\left\langle Z_3,T^{\gamma}_{\Theta_1}Z_j \right\rangle + 2\Re \left\langle Z_3,T^{\gamma}_{\Theta_0}Z_3 \right\rangle\\
&\quad + \sum_{i\neq3,4} 2\Re\left\langle Z_3,T^{\gamma}_{\Theta_0}Z_i \right\rangle+2\Re\left\langle Z_3,T^{\gamma}_rW^+ \right\rangle\\
&\quad + 2\Re\left\langle Z_3,\mathcal{R}_{-1}W^+ \right\rangle+2\Re\left\langle Z_3,F^+ \right\rangle,
\end{split}
\end{equation*}
and
\vspace{-1ex}
\begin{equation*}
\begin{split}
2\Re\left\langle Z_3,T^{\gamma}_{\omega^r+i\bar{\omega}^r}Z_3 \right\rangle & = 2\Re \left\langle Z_3,(T^{\gamma}_{\Lambda^{\frac{1}{2}}})^{\ast}T^{\gamma}_{\frac{\omega^r+i\bar{\omega}^r}{\Lambda^{\frac{1}{2}}}}Z_3 \right\rangle + 2\Re \langle Z_3,\mathcal{R}_0Z_3\rangle\\
& = 2\Re \left\langle T^{\gamma}_{\Lambda^{\frac{1}{2}}}Z_3,T^{\gamma}_{\frac{\omega^r+i\bar{\omega}^r}{\Lambda^{\frac{1}{2}}}}T^{\gamma}_{\Lambda^{\frac{1}{2}}}Z_3 \right\rangle + 2 \Re \left\langle T^{\gamma}_{\Lambda^{\frac{1}{2}}}Z_3,\mathcal{R}_{-\frac{1}{2}}Z_3 \right\rangle + 2\Re\langle Z_3,\mathcal{R}_0Z_3\rangle.
\end{split}
\end{equation*}
We can obtain that
\begin{equation*}
\begin{split}
\Re\left\langle T^{\gamma}_{\Lambda^{\frac{1}{2}}}Z_3,T^{\gamma}_{\frac{\omega^r+i\bar{\omega}^r}{\Lambda^{\frac{1}{2}}}}T^{\gamma}_{\Lambda^{\frac{1}{2}}}Z_3 \right\rangle & \leq -c\|Z_3\|^2_{\frac{1}{2},\gamma},\\
\Re\left\langle T^{\gamma}_{\Lambda^{\frac{1}{2}}}Z_3,\mathcal{R}_{-\frac{1}{2}}Z_3 \right\rangle & \leq \varepsilon\|Z_3\|^2_{\frac{1}{2},\gamma}+\frac{1}{\varepsilon}\|Z_3\|^2_{-\frac{1}{2},\gamma},\\
\Re\langle Z_3,\mathcal{R}_0Z_3\rangle & \leq C\|Z_3\|^2_0.
\end{split}
\end{equation*}
Hence it follows that
\begin{equation}\label{para25}
\begin{split}
\vertiii{ Z_3 }^2_{\frac{1}{2},\gamma} & \leq \|Z_3|_{x_3=0} \|^2_0 + \left(C+\frac{1}{\varepsilon} \right)\vertiii{ Z_3 }^2_0+\frac{1}{\varepsilon} \sum_{j=1,2,7,10,13} \vertiii{ Z_j }^2_{\frac{1}{2},\gamma}+\sum_{i\neq3,4}\frac{1}{\varepsilon}\vertiii{Z_i }^2_{-\frac{1}{2},\gamma}\\
&\quad+\frac{1}{\varepsilon} \left( \vertiii{ T^{\gamma}_rW^+ }^2_{-\frac{1}{2},\gamma}+\vertiii{W^+}^2_{-\frac{3}{2},\gamma}+\vertiii{ F^+ }^2_{-\frac{1}{2},\gamma} \right).
\end{split}
\end{equation}

Considering \eqref{para9},\eqref{para10},\eqref{para18},\eqref{para20},\eqref{para24},\eqref{para25}, dividing them by the appropriate power of $\gamma,$ we obtain that
\begin{equation*}
\begin{split}
\frac{1}{\gamma}\vertiii{T^{\gamma}_{\Lambda}T^{\gamma}_{\sigma}Z_4}^2_0+\frac{1}{\gamma}\vertiii{ T^{\gamma}_{\Lambda^{\frac{1}{2}}}T^{\gamma}_{\sigma}Z_4|_{x_3=0} }^2_0 & \lesssim \frac{1}{\gamma}\| T^{\gamma}_{\sigma}Z_4|_{x_3=0} \|^2_0 + \left(C+\frac{1}{\varepsilon} \right) \frac{1}{\gamma}\vertiii{Z_4}^2_{1,\gamma} \\
&\quad + \sum_{i\neq 3,4}\frac{1}{\varepsilon\gamma} \left( \vertiii{T^{\gamma}_{\sigma}Z_i}^2_0+\vertiii{Z_i }^2_0 \right)\\
&\quad +\frac{1}{\varepsilon\gamma} \left( \vertiii{T^{\gamma}_{r}W^+}^2_{1,\gamma}+\vertiii{W^+}^2_0+\vertiii{ F^+ }^2_{1,\gamma} \right), \\
\vertiii{Z_4}^2_{\frac{3}{2},\gamma} + \vertiii{ Z_4|_{x_3=0} }^2_{1,\gamma} & \lesssim C\vertiii{Z_4}^2_{1,\gamma}+\frac{1}{\varepsilon}\vertiii{Z_4}^2_{\frac{1}{2},\gamma} \\
& \quad +\sum_{i\neq 3,4} \left( \frac{1}{\varepsilon}\vertiii{Z_i }^2_{\frac{1}{2},\gamma}+\frac{1}{\varepsilon\gamma}\vertiii{Z_i }^2_0 \right)\\
&\quad + \frac{1}{\varepsilon\gamma} \left( \vertiii{T^{\gamma}_{r}W^+}^2_{1,\gamma}+\vertiii{W^+}^2_0+\vertiii{ F^+ }^2_{1,\gamma} \right), \\
\end{split}
\end{equation*}
\vspace{-1.5ex}
\begin{equation*}
\begin{split}
\gamma\vertiii{T^{\gamma}_{\sigma}Z_3 }^2_0 & \leq \gamma\Re\left\langle T^{\gamma}_{\sigma}Z_3, T^{\gamma}_{\frac{1}{\Lambda}}T^{\gamma}_{\sigma}Z_3 \right\rangle\Big|_{x_3=0}+\frac{\gamma}{\varepsilon}\vertiii{ Z_3 }^2_0 + \frac{\gamma}{\varepsilon} \left( \vertiii{ T^{\gamma}_{\sigma}Z_j }^2_0+\vertiii{ Z_j }^2_0 \right) \\
&\quad + \frac{\gamma}{\varepsilon}\left(\vertiii{T^{\gamma}_{\sigma}Z_3 }^2_{-1,\gamma}+\vertiii{ Z_3 }^2_{-1,\gamma}\right) + \sum_{i\neq3,4}\frac{\gamma}{\varepsilon}\left(\vertiii{T^{\gamma}_{\sigma}Z_i}^2_{-1,\gamma}+\vertiii{Z_i }^2_{-1,\gamma}\right)\\
&\quad + \frac{\gamma}{\varepsilon}(\vertiii{ T^{\gamma}_rW^+ }^2_0+\vertiii{W^+}^2_{-1,\gamma}+\vertiii{ F^+ }^2_0), \\
\gamma^2 \vertiii{ Z_3 }^2_{\frac{1}{2},\gamma}&\leq \gamma^2\|Z_3|_{x_3=0} \|^2_0 + \left( C+\frac{1}{\varepsilon} \right) \gamma^2\vertiii{ Z_3 }^2_0+\frac{\gamma^2}{\varepsilon}\vertiii{ Z_j }^2_{\frac{1}{2},\gamma}+\sum_{i\neq3,4}\frac{\gamma^2}{\varepsilon}\vertiii{Z_i }^2_{-\frac{1}{2},\gamma}\\
& \quad + \frac{\gamma^2}{\varepsilon} \left( \vertiii{ T^{\gamma}_rW^+ }^2_{-\frac{1}{2},\gamma}+\vertiii{W^+}^2_{-\frac{3}{2},\gamma}+\vertiii{ F^+ }^2_{-\frac{1}{2},\gamma} \right).
\end{split}
\end{equation*}
For $j \ne 3,4$,
\begin{equation*}
\begin{split}
\gamma\|T^{\gamma}_{\sigma}Z_j \|^2_0 & \leq\frac{1}{\varepsilon\gamma}\| Z_j \|^2_0+C\|T^{\gamma}_{\sigma}Z_j \|^2_0+\frac{1}{\varepsilon}\| Z_j \|^2_0 + \frac{1}{\varepsilon\gamma}\left(\|T^{\gamma}_{\Lambda}T^{\gamma}_{\sigma}Z_4\|^2_0+\|Z_4\|^2_{1,\gamma}\right)\\
&\quad + \sum_i\frac{1}{\varepsilon\gamma} \left(\|T^{\gamma}_{\sigma}Z_i\|^2_0+\|Z_i \|^2_0 \right) + \frac{1}{\varepsilon\gamma} \left(\vertiii{ T^{\gamma}_rW^+ }^2_{1,\gamma}+\vertiii{W^+}^2_0+\vertiii{ F^+ }^2_{1,\gamma} \right), \\
\gamma^2\| Z_j \|^2_{\frac{1}{2},\gamma} & \leq C\gamma\| Z_j \|^2_{\frac{1}{2},\gamma}+\frac{1}{\varepsilon} \left( \|Z_4\|^2_{\frac{3}{2},\gamma} +\frac{1}{\varepsilon}\|Z_i \|^2_{\frac{1}{2},\gamma} \right) + \frac{1}{\varepsilon\gamma} \left( \vertiii{ T^{\gamma}_rW^+ }^2_{1,\gamma}+\vertiii{W}^+_0+\vertiii{ F^+ }^2_{1,\gamma} \right).
\end{split}
\end{equation*}
Summing up the above estimates and taking $\gamma$ sufficiently large, we have for $1 \le j \le 13$ and $j \ne 4$,
\begin{equation}\label{para26}
\begin{split}
&\frac{1}{\gamma} \left( \vertiii{T^{\gamma}_{\Lambda}T^{\gamma}_{\sigma}Z_4}^2_0 + \left\| T^{\gamma}_{\Lambda^{\frac{1}{2}}}T^{\gamma}_{\sigma}Z_4|_{x_3=0} \right\|^2_0 \right) +\vertiii{Z_4}^2_{\frac{3}{2},\gamma} +\| Z_4|_{x_3=0} \|^2_{1,\gamma} +\gamma\vertiii{T^{\gamma}_{\sigma}Z_j }^2_0 + \gamma^2\vertiii{ Z_j }^2_{\frac{1}{2},\gamma} \\
&\  \leq \gamma\Re\left\langle T^{\gamma}_{\sigma}Z_3,T^{\gamma}_{\frac{1}{\Lambda}}T^{\gamma}_{\sigma}Z_3 \right\rangle \Big|_{x_3=0}+\gamma^2\|Z_3|_{x_3=0} \|^2_0+\frac{1}{\gamma}\vertiii{Z_4}^2_{1,\gamma}\\
&\ \quad + \sum_i \left( \vertiii{Z_i }^2_{\frac{1}{2},\gamma}+\frac{1}{\gamma}\vertiii{T^{\gamma}_{\sigma}Z_i}^2_0 \right) + \frac{1}{\gamma} \left( \vertiii{T^{\gamma}_{r}W^+}^2_{1,\gamma}+\vertiii{W^+}^2_0+\vertiii{ F^+ }^2_{1,\gamma} \right)\\
&\ \leq \gamma\Re\left\langle T^{\gamma}_{\sigma}Z_3,T^{\gamma}_{\frac{1}{\Lambda}}T^{\gamma}_{\sigma}Z_3 \right\rangle \Big|_{x_3=0} + \gamma^2\|Z_3|_{x_3=0} \|^2_0 +\frac{1}{\gamma} \left(\vertiii{T^{\gamma}_{r}W^+}^2_{1,\gamma}+\vertiii{W^+}^2_0+\vertiii{ F^+ }^2_{1,\gamma} \right).\\
\end{split}
\end{equation}

We remark that the extra degree of freedom can cause complicated interaction between the poles of $W^+$ and $W^-.$
Hence, $\tau=-iv^r_1\eta-iv^r_2\tilde{\eta}$ is also the pole of the differential equation for $W^-$ in \eqref{para2}, as long as $(v^r_1-v^l_1,v^r_2-v^l_2)\cdot (\eta,\tilde{\eta})=0.$ This is a key point in 3D analysis, since it is possible that the poles for the two equations coincide. In a similar way as before, we obtain for $14 \le j \le 26$ and $j \ne 17$ that
\begin{equation}\label{para27}
\begin{split}
&\frac{1}{\gamma} \left( \vertiii{ T^{\gamma}_{\Lambda}T^{\gamma}_{\sigma}Z_{17} }^2_0 + \left\|T^{\gamma}_{\Lambda^{\frac{1}{2}}}T^{\gamma}_{\sigma}Z_{17}|_{x_3=0} \right\|^2_0 \right) + \vertiii{ Z_{17} }^2_{\frac{3}{2},\gamma} + \|Z_{17}|_{x_3=0}\|^2_{1,\gamma}  \\
& \quad\quad  + \gamma\vertiii{ T^{\gamma}_{\sigma}Z_j }^2_0+\gamma^2\vertiii{ Z_j }^2_{\frac{1}{2},\gamma} \\
&\quad\leq \gamma \Re \left\langle T^{\gamma}_{\sigma}Z_{16},T^{\gamma}_{\frac{1}{\Lambda}}T^{\gamma}_{\sigma}Z_{16} \right\rangle \Big|_{x_3=0} + \gamma^2 \|Z_{16}|_{x_3=0}\|^2_0\\
&\quad\quad + \frac{1}{\gamma} \left( \vertiii{ T^{\gamma}_{r}W^- }^2_{1,\gamma} + \vertiii{ W^- }^2_0 + \vertiii{ F^- }^2_{1,\gamma} \right),\\
\end{split}
\end{equation}
where
\begin{equation}\label{def:Zminus.p1}
(Z_{14},\cdots, Z_{26})^{\top}:=T^{\gamma}_{\chi_1 Q^l}T^{\gamma}_{\chi_{p_1}}W^-.
\end{equation}
and $Q^l$ is the transformation matrix for $W^-,$ which is defined in a similarly way as for $Q^r$.
We write $Z_4$ and $Z_{17}$ for the outgoing modes and $Z_3$ and $Z_{16}$ for the incoming modes. Then we have
$$
Z_{\mathrm{in}}=(Z_3,Z_{16})^{\top} \quad \text{ and } \quad Z_{\mathrm{out}}=(Z_4,Z_{17})^{\top}.
$$
So the last step is to use the boundary conditions in \eqref{para} to estimate the terms $\|Z_3|_{x_3=0} \|^2_0$,  $\|Z_{16}|_{x_3=0}\|^2_0$, $\gamma\Re\left\langle T^{\gamma}_{\sigma}Z_3, T^{\gamma}_{\frac{1}{\Lambda}}T^{\gamma}_{\sigma}Z_3 \right\rangle \Big|_{x_3=0}$, and $\gamma\Re \left\langle T^{\gamma}_{\sigma}Z_{16}, T^{\gamma}_{\frac{1}{\Lambda}}T^{\gamma}_{\sigma}Z_{16} \right\rangle \Big|_{x_3=0}$. Notice that
\begin{align*}
\gamma\Re\left\langle T^{\gamma}_{\sigma}Z_3, T^{\gamma}_{\frac{1}{\Lambda}}T^{\gamma}_{\sigma}Z_3 \right\rangle \Big|_{x_3=0 } & \lesssim \|T^{\gamma}_{\sigma}Z_3|_{x_3=0}\|^2_0, \\
\gamma\Re \left\langle T^{\gamma}_{\sigma}Z_{16}, T^{\gamma}_{\frac{1}{\Lambda}}T^{\gamma}_{\sigma}Z_{16} \right\rangle \Big|_{x_3=0} & \lesssim \|T^{\gamma}_{\sigma}Z_{16}|_{x_3=0}\|^2_0.
\end{align*}
Therefore we only need to estimate the boundary terms $T^{\gamma}_{\sigma}Z_{\mathrm{in}}|_{x_3=0}$ and $Z_{\mathrm{in}}|_{x_3=0}$. The goal is to use the boundary conditions \eqref{para}
to prove the following estimate:
\begin{equation}\label{para28}
\gamma^2 \|Z_{\mathrm{in}}|_{x_3=0}\|^2_0 + \left\|T^{\gamma}_{\sigma}Z_{\mathrm{in}}|_{x_3=0} \right\|^2_0 \lesssim \|G\|^2_{1,\gamma} + \|Z_{\mathrm{out}}|_{x_3=0}\|^2_{1,\gamma} + \left\| W^{nc}|_{x_3=0} \right\|^2_0.
\end{equation}

\medskip

Let us rewrite the Lopatinski$\breve{\mathrm{i}}$ matrix as
\begin{equation}\label{lop}
\begin{split}
\beta\left[\begin{matrix}
E^r & {\mathbf 0}\\
\mathbf {0} & E^l\\
\end{matrix}\right] =: \left[\begin{matrix}
\varsigma_1 & \varsigma_2\\
\varsigma_3 & \varsigma_4\\
\end{matrix}\right].
\end{split}
\end{equation}
We calculate its determinant at $x_3=0$ to satisfy
$$
\varsigma_1\varsigma_4-\varsigma_2\varsigma_3=k^r_1k^l_1h(t,x_1,x_2,\tau,\eta,\tilde{\eta}),\quad \text{where} \quad h(t,x_1,x_2,\tau,\eta,\tilde{\eta})\neq0
$$
in a neighborhood of $(-iv^r_1\eta-iv^r_2\tilde{\eta},\eta,\tilde{\eta})\in \Sigma$ and in a neighborhood of $(-iv^l_1\eta-iv^l_2\tilde{\eta},\eta,\tilde{\eta})\in \Sigma.$ Similar to the constant-coefficient case \cite[Lemma 3.6]{RChen2021}, let us assume without loss of generality that $\varsigma_1\neq0$.

Define the following matrices in a suitably small neighborhood of $\mathcal{V}^r_{p_1}\cup \mathcal{V}^l_{p_1}$:
\begin{equation}\label{P1P2}
P_1:=\left[\begin{matrix}
\frac{1}{\varsigma_1} & 0\\
-\frac{\varsigma_3}{\varsigma_1 h \Lambda^2} & \frac{1}{h \Lambda^2}
\end{matrix}\right],\qquad
P_2:= \left[\begin{matrix}
1 & -\varsigma_2\\
 0 & \varsigma_1\\
\end{matrix}\right].
\end{equation}
It is easily seen that $P_1, P_2 \in \Gamma^0_2.$

\begin{remark}
The definition \eqref{P1P2} is given under the temporary assumption $\varsigma_1\neq 0$.
If $\varsigma_1=0$ at some point of $\supp\chi_1$, we instead permute the roles of $(\varsigma_1,\varsigma_2)$ and
define $P_1,P_2$ using any non-vanishing component among $\varsigma_1,\varsigma_2$.
By shrinking the microlocal neighborhood if necessary, one can always ensure that the chosen component stays bounded away from $0$
on $\supp\chi_1$, so that $P_1,P_2\in\Gamma^0_2$ remain elliptic there.
\end{remark}
Shrinking further  $\mathcal{V}^r_{p_1}\cup \mathcal{V}^l_{p_1}$ if necessary, we have
\begin{equation}\label{beta}
\beta_{\rm{in}}:=P_1\beta\left[\begin{matrix}
E^r & {\mathbf 0}\\
{\mathbf 0} & E^l
\end{matrix}\right]P_2=\left[\begin{matrix}
1 & 0\\
0 & \Lambda^{-2}k_1^rk_1^l \end{matrix}\right],
\end{equation}
where, by \eqref{def k1},
\[
k_1^r=\gamma+i\sigma_r,\qquad k_1^l=\gamma+i\sigma_l,
\]
with
\[
\sigma_r:=\delta+v^r_1\eta+v^r_2\tilde{\eta},
\qquad
\sigma_l:=\delta+v^l_1\eta+v^l_2\tilde{\eta}.
\]
Thus the second diagonal entry of $\Lambda\beta_{\rm in}$ is a first-order symbol $\Lambda^{-1}k_1^rk_1^l=\Lambda^{-1}(\gamma+i\sigma_r)(\gamma+i\sigma_l)$. Write
\[
\bar{\Lambda}:=(\delta^2+\eta^2+\tilde{\eta}^2)^{1/2}.
\]
Shrinking the microlocal neighborhood if necessary, we may assume that $\Lambda\sim \bar{\Lambda}$ on $\supp\tilde{\chi}_2$. Since $\supp\tilde{\chi}_2$ is disjoint from $\Upsilon_r^{(2)}$, the last two factors in $\sigma$ (see \eqref{root}) are elliptic on $\supp\tilde{\chi}_2$. Hence
\[
e_{p_1}:=\Lambda\bar{\Lambda}^{-3}(\delta-V_1\sqrt{\eta^2+\tilde{\eta}^2})(\delta-V_2\sqrt{\eta^2+\tilde{\eta}^2})\in \Gamma^0_2
\]
is a real-valued elliptic symbol there, and therefore
\begin{equation}\label{sigma.factor.case1}
\sigma=e_{p_1}\Lambda^{-1}\sigma_r\sigma_l
\qquad\text{on }\supp\tilde{\chi}_2.
\end{equation}
In particular,
 \begin{equation}\label{sigma.bound.case1}
 |\sigma|\lesssim \Lambda,\qquad\text{on }\supp\tilde{\chi}_2.
 \end{equation}

We now fix the four cut-off functions $\tilde{\chi}_1,$ $\tilde{\chi}_2,$ $\tilde{\chi}_3$ and $\tilde{\chi}_4$ such that
\begin{align*}
&\tilde{\chi}_1\equiv1 \text{ in a neighborhood of } \supp \chi_1\cap\{x_3=0\}.\nonumber\\
&\tilde{\chi}_j\equiv1 \text{ in a neighborhood of } \supp \tilde{\chi}_{j-1}, \text{ for } j=2,3,4.\\
&\supp \tilde{\chi}_4\subseteq  \mathcal{V}^r_{p_1}\cup \mathcal{V}^l_{p_1}\cap \{x_3=0\}\times \Sigma.
\end{align*}
Following the argument of \cite[Section 3.4.3]{Coulombel2002}, we can obtain the following estimate by using the localized G${\rm\mathring a}$rding's inequality:
\begin{equation}\label{b1}
\left\|T^{\gamma}_{\tilde{\chi}_2\Lambda\beta_{\mathrm{in}}}T^{\gamma}_{\tilde{\chi}_1}T^{\gamma}_{\tilde{\chi}_4P^{-1}_2}Z_{\mathrm{in}}|_{x_3=0} \right\|_0 \lesssim \|G\|_{1,\gamma}+\|Z_{\mathrm{out}}|_{x_3=0}\|_{1,\gamma} + \left\| W^{nc}|_{x_3=0} \right\|_0.
\end{equation}
Now, we use the special structure of $\beta_{\mathrm{in}}$ to obtain a lower bound for the term on the left-hand side of \eqref{b1}.
Define
\begin{align}\label{b2}
(\vartheta_1,\vartheta_2)^{\top}:=T^{\gamma}_{\tilde{\chi}_4P^{-1}_2}Z_{\mathrm{in}}|_{x_3=0}.
\end{align}
From \eqref{beta}, we have
\begin{equation}\label{b3}
\left\|T^{\gamma}_{\tilde{\chi}_2\Lambda\beta_{\mathrm{in}}}T^{\gamma}_{\tilde{\chi}_1}T^{\gamma}_{\tilde{\chi}_4P^{-1}_2}Z_{\mathrm{in}}|_{x_3=0} \right\|^2_0 = \left\|T^{\gamma}_{\tilde{\chi}_2\Lambda}T^{\gamma}_{\tilde{\chi}_1}\vartheta_1 \right\|^2_0 + \left\|T^{\gamma}_{\tilde{\chi}_2\Lambda^{-1}k_1^rk_1^l}T^{\gamma}_{\tilde{\chi}_1}\vartheta_2 \right\|^2_0.
\end{equation}
Applying the localized G${\rm\mathring a}$rding's inequality (see Lemma \ref{lemmapara} {\rm (vii)}), we have
\begin{equation}\label{b4}
\begin{split}
\left\|T^{\gamma}_{\tilde{\chi}_2\Lambda}T^{\gamma}_{\tilde{\chi}_1}\vartheta_1 \right\|^2_0 & = \left\langle(T^{\gamma}_{\tilde{\chi}_2\Lambda})^{\ast}T^{\gamma}_{\tilde{\chi}_2\Lambda}T^{\gamma}_{\tilde{\chi}_1}\vartheta_1,T^{\gamma}_{\tilde{\chi}_1}\vartheta_1 \right\rangle \\
& \geq \Re \left\langle T^{\gamma}_{\tilde{\chi}^2_2\Lambda^2}T^{\gamma}_{\tilde{\chi}_1}\vartheta_1,T^{\gamma}_{\tilde{\chi}_1}\vartheta_1 \right\rangle - C \left\|T^{\gamma}_{\tilde{\chi}_1}\vartheta_1 \right\|_0 \left\|T^{\gamma}_{\tilde{\chi}_1}\vartheta_1 \right\|_{1,\gamma} \\
& \geq c \left\|T^{\gamma}_{\tilde{\chi}_1}\vartheta_1 \right\|^2_{1,\gamma} - C \|\vartheta_1\|^2_0 - C \left\|T^{\gamma}_{\tilde{\chi}_1}\vartheta_1 \right\|^2_0 \\
& \gtrsim \|\vartheta_1\|^2_{1,\gamma} - C\|Z_{\mathrm{in}}|_{x_3=0}\|^2_0 \\
& \gtrsim \gamma^2\|\vartheta_1\|^2_0 + \left\| T^{\gamma}_{\sigma}\vartheta_1 \right\|^2_0 - C\|Z_{\mathrm{in}}|_{x_3=0}\|^2_0,
\end{split}
\end{equation}
for sufficiently large $\gamma$. Similarly
\begin{equation}\label{b4bis}
\left\|T^{\gamma}_{\tilde{\chi}_2\Lambda^{-1}k_1^rk_1^l}T^{\gamma}_{\tilde{\chi}_1}\vartheta_2 \right\|^2_0
\gtrsim
\left\|T^{\gamma}_{\sigma}\vartheta_2 \right\|^2_0
-
C \|Z_{\mathrm{in}}|_{x_3=0}\|^2_0.
\end{equation}
Inserting the above two estimates into \eqref{b3}, we have
\begin{equation}\label{b5}
\left\|T^{\gamma}_{\tilde{\chi}_2\Lambda\beta_{\mathrm{in}}}T^{\gamma}_{\tilde{\chi}_1}T^{\gamma}_{\tilde{\chi}_4P^{-1}_2}Z_{\mathrm{in}}|_{x_3=0} \right\|^2_0 \gtrsim \gamma^2 \|\vartheta_1\|^2_0 + \left\|T^{\gamma}_{\sigma}(\vartheta_1,\vartheta_2)\right\|^2_0-C\|Z_{\mathrm{in}}|_{x_3=0}\|^2_0.
\end{equation}
Using the fact that $\tilde{\chi}_3\chi_1\equiv\chi_1,$ we obtain that
$$T^{\gamma}_{\tilde{\chi}_3}T^{\gamma}_{\sigma}Z_{\mathrm{in}}=T^{\gamma}_{\sigma}T^{\gamma}_{\tilde{\chi}_3}Z_{\mathrm{in}}+\mathcal{R}_0Z_{\mathrm{in}}=T^{\gamma}_{\sigma}Z_{\mathrm{in}}+\mathcal{R}_0Z_{\mathrm{in}}.$$
Then, we have
\begin{align}\nonumber
T^{\gamma}_{\sigma}(\vartheta_1,\vartheta_2)&=T^{\gamma}_{\tilde{\chi}_4P^{-1}_2}T^{\gamma}_{\sigma}Z_{\mathrm{in}}|_{x_3=0}+\mathcal{R}_0Z_{\mathrm{in}}|_{x_3=0} = T^{\gamma}_{\tilde{\chi}_4P^{-1}_2}T^{\gamma}_{\tilde{\chi}_3}T^{\gamma}_{\sigma}Z_{\mathrm{in}}|_{x_3=0}+\mathcal{R}_0Z_{\mathrm{in}}|_{x_3=0}.\nonumber
\end{align}
Using the ellipticity of $(P^{-1}_2)^{\ast}P^{-1}_2$ on the support of $\tilde{\chi}_4$ and that $\sigma \in \mathbb{R}$, we apply the localized G${\rm\mathring a}$rding's inequality (Lemma \ref{lemmapara} {\rm (vii)}) to obtain that, for sufficiently large $\gamma,$
\begin{equation*}
\begin{split}
&\left\|T^{\gamma}_{\sigma}(\vartheta_1,\vartheta_2)\right\|^2_0 \\
& \ \gtrsim \left\langle(T^{\gamma}_{\tilde{\chi}_4P^{-1}_2})^{\ast}T^{\gamma}_{\tilde{\chi}_4P^{-1}_2}T^{\gamma}_{\tilde{\chi}_3}T^{\gamma}_{\sigma}Z_{\mathrm{in}}|_{x_3=0}, \ T^{\gamma}_{\tilde{\chi}_3}T^{\gamma}_{\sigma}Z_{\mathrm{in}}|_{x_3=0} \right\rangle - C\|Z_{\mathrm{in}}|_{x_3=0}\|^2_0\nonumber\\
& \ \gtrsim \left\|T^{\gamma}_{\tilde{\chi}_3}T^{\gamma}_{\sigma}Z_{\mathrm{in}}|_{x_3=0} \right\|^2_0 - C \left\| T^{\gamma}_{\sigma}Z_{\mathrm{in}}|_{x_3=0} \right\|^2_{-1,\gamma} - C\left\|T^{\gamma}_{\tilde{\chi}_3}T^{\gamma}_{\sigma}Z_{\mathrm{in}}|_{x_3=0} \right\|^2_{-1,\gamma}-C\|Z_{\mathrm{in}}|_{x_3=0}\|_0^2.
\end{split}
\end{equation*}
Then, for sufficiently large $\gamma,$ we have
\begin{align}\label{b6}
\left\|T^{\gamma}_{\sigma}(\vartheta_1,\vartheta_2)\right\|^2_0\gtrsim\left\| T^{\gamma}_{\sigma}Z_{\mathrm{in}}|_{x_3=0} \right\|^2_0-C\|Z_{\mathrm{in}}|_{x_3=0}\|^2_0.
\end{align}
Similarly, we can prove
\begin{equation}\label{b7}
\begin{split}
\|(\vartheta_1,\vartheta_2)\|^2_0&\gtrsim\|Z_{\mathrm{in}}|_{x_3=0}\|^2_0-C\|Z_{\mathrm{in}}|_{x_3=0}\|^2_{-1,\gamma} \gtrsim \|Z_{\mathrm{in}}|_{x_3=0}\|^2_0-\frac{C}{\gamma^2}\|Z_{\mathrm{in}}|_{x_3=0}\|^2_0.
\end{split}
\end{equation}
Combining \eqref{b1},\eqref{b5}--\eqref{b7}, and taking $\gamma$ sufficiently large, we derive \eqref{para28}.

From \eqref{para10},
 we have
 \begin{equation}\label{para29}
 \|Z_{\mathrm{out}}|_{x_3=0}\|^2_{1,\gamma}\lesssim \sum_{i\neq3,4}\vertiii{Z_i }^2_{\frac{1}{2},\gamma}+\frac{1}{\gamma}\left(\vertiii{ T^{\gamma}_{r}W }^2_{1,\gamma}+\vertiii{ W }^2_0 + \vertiii{ F }^2_{1,\gamma}\right).
\end{equation}
Combining \eqref{para26}--\eqref{para28},\eqref{para29}, we obtain that
\begin{equation}\label{para30}
\begin{split}
&\frac{1}{\gamma} \vertiii{ T^{\gamma}_{\Lambda}T^{\gamma}_{\sigma}Z_{\mathrm{out}} }^2_0 + \vertiii{ Z_{\mathrm{out}} }^2_{\frac{3}{2},\gamma} + \gamma \vertiii{ T^{\gamma}_{\sigma}Z_{\rm c} }^2_0 + \gamma^2 \vertiii{ Z_{\rm c} }^2_{\frac{1}{2},\gamma} + \gamma \vertiii{ T^{\gamma}_{\sigma}Z_{\mathrm{in}} }^2_0\\
&\quad\quad + \gamma^2 \vertiii{ Z_{\mathrm{in}} }^2_{\frac{1}{2},\gamma} + \frac{1}{\gamma} \left\| T^{\gamma}_{\Lambda^{\frac{1}{2}}}T^{\gamma}_{\sigma}Z_{\mathrm{out}}|_{x_3=0} \right\|^2_0\\
&\quad\quad + \|Z_{\mathrm{out}}|_{x_3=0}\|^2_{1,\gamma} + \gamma^2 \|Z_{\mathrm{in}}|_{x_3=0}\|^2_0+\left\|T^{\gamma}_{\sigma}Z_{\mathrm{in}}|_{x_3=0} \right\|^2_0\\
&\quad \leq \|G\|^2_{1,\gamma}+\left\| W^{nc}|_{x_3=0} \right\|^2_0+\frac{1}{\gamma}\left( \vertiii{ T^{\gamma}_rW }^2_{1,\gamma} + \vertiii{W}^2_0+\vertiii{ F }^2_{1,\gamma}\right),
\end{split}
\end{equation}
where $Z_{\rm c}=(Z_1,Z_2,Z_5,\cdots,Z_{13},Z_{14}, Z_{15},Z_{18},\cdots,Z_{26})^{\top}.$
\subsection{Case 2: Points in $\Upsilon^{(2)}_r$}
We need to estimate the part of $W^{\pm}$ corresponding to $\mathcal{V}^1_r,\mathcal{V}^2_r.$ For simplicity, we discuss the differential equations for $W^+$ in $\mathcal{V}^1_r.$ The remaining neighborhood of $\mathcal{V}^2_r$ and the discussion for $W^-,$ we obtain the same estimates.
Now consider the cut-off function $\chi_{rt}$ in $\Gamma^0_k$ for any integer $k,$ whose support on $\R^4_+\times \Sigma$ is contained in $\mathcal{V}^1_r,$ and is equal to $1$ in a smaller neighborhood of the strip where $\tau=iV_1\sqrt{\eta^2+\tilde{\eta}^2}.$
Denote
$$
W^+_{rt} := T^{\gamma}_{\chi_{rt}}W^+.
$$
Hence, we obtain that
$$T^{\gamma}_{\tau A^r_0+i\eta A^r_1+i\tilde{\eta}A^r_2}W^+_{rt}+T^{\gamma}_{A^r_0C^r}W^{+}_{rt}+T^{\gamma}_rW^++I_2\partial_3W^{+}_{rt}=T^{\gamma}_{\chi_{rt}}F^++\mathcal{R}_{-1}W^+.$$
Here, $r$ is in $\Gamma^0_1$, bounded and supported only in the set where $\chi_{rt}\in(0,1).$ Then, we take two cut-off functions $\chi_1$ and $\chi_2$ in the class $\Gamma^0_k$ for any integer $k.$ Both of the functions are supported in $\mathcal{V}_{rt},$ $\chi_1=1$ on the support of $\chi_{rt}$ and $\chi_2=1$ on the support of $\chi_1.$
Similar to the previous discussion, after applying the cut-off symbol, we can find transformation matrices $Q^r_0$ and $Q^r_{-1}$ and symmetrizers $R^r_0$ and $R^r_{-1}$ to obtain
$$
I_2\partial_3Z^+=-T^{\gamma}_{\chi_2\tilde{A}^r}Z^++T^{\gamma}_{D_0}Z^++T^{\gamma}_rW^++\mathcal{R}_0F^++\mathcal{R}_{-1}W^+,
$$
where $\tilde{A}^r$ is the same as in \eqref{para5} in previous case, $\chi_1Q^r_0$ and $\chi_1R^r_0$ are invertible symbols in $\Gamma^0_2,$ and $Q^r_{-1}, R^r_{-1}\in\Gamma^{-1}_1.$
Define
\begin{equation}\label{def:Zplus.rt}
Z^+:=T^{\gamma}_{\chi_1({Q^{r}_0}^{-1}+Q^r_{-1})}W^+_{rt}.
\end{equation}

After same argument for $\chi_2\tilde{A}^r$ and $D_0,$
 we obtain that
 \begin{equation}\label{para31}
 I_2\partial_3Z^+=-T^{\gamma}_{\tilde{D_1}}Z^++T^{\gamma}_{\tilde{D}_0}Z^++T^{\gamma}_rW^++\mathcal{R}_0F^++\mathcal{R}_{-1}W^+.
 \end{equation}
 In $\tilde{D}_1$ we have $\omega^r\in\Gamma^1_2,$ $\Re \omega^r\leq -c\Lambda,$ and in $\tilde{D}_0$ we have $d_{3,4}=d_{4,3}=0$.
 Denote
 $$
 Z^+=(Z_1,Z_2,\cdots,Z_{13})^{\top}.
 $$
 It follows that
 $$\partial_3Z_4=T^{\gamma}_{-\omega^r+i\bar{\omega}^r}Z_4+T^{\gamma}_{\Theta_0}Z_4+\sum_{i\neq3,4}T^{\gamma}_{\Theta_0} Z_i+T^{\gamma}_rW^++\mathcal{R}_0F^++\mathcal{R}_{-1}W^+.$$
Applying the two symmetrizers $(T^{\gamma}_{\sigma})^{\ast}T^{\gamma}_{\Lambda}T^{\gamma}_{\sigma}$ and $(T^{\gamma}_{\Lambda})^{\ast}T^{\gamma}_{\Lambda}$, we have
\begin{equation*}
\begin{split}
&\vertiii{T^{\gamma}_{\Lambda}T^{\gamma}_{\sigma}Z_4}^2_0 + \vertiii{ T^{\gamma}_{\Lambda^{\frac{1}{2}}}T^{\gamma}_{\sigma}Z_4|_{x_3=0} }^2_0 \\
&\quad \lesssim \| T^{\gamma}_{\sigma}Z_4|_{x_3=0} \|^2_0 + \left( C+\frac{1}{\varepsilon} \right)\vertiii{Z_4}^2_{1,\gamma}+\sum_{i\neq 3,4}\frac{1}{\varepsilon} \left( \vertiii{T^{\gamma}_{\sigma}Z_i}^2_0+\vertiii{Z_i }^2_0 \right) \\
&\quad\quad + \frac{1}{\varepsilon} \left( \vertiii{T^{\gamma}_{r}W^+}^2_{1,\gamma}+\vertiii{W^+}^2_0+\vertiii{ F^+ }^2_{1,\gamma} \right), \\
&\vertiii{Z_4}^2_{\frac{3}{2},\gamma}+\vertiii{ Z_4|_{x_3=0} }^2_{1,\gamma}\\
&\quad\lesssim C\vertiii{Z_4}^2_{1,\gamma}+\frac{1}{\varepsilon}\vertiii{Z_4}^2_{\frac{1}{2},\gamma}+\sum_{i\neq 3,4} \left( \frac{1}{\varepsilon}\vertiii{Z_i }^2_{\frac{1}{2},\gamma}+\frac{1}{\varepsilon\gamma}\vertiii{Z_i }^2_0 \right)\\
&\quad\quad+\frac{1}{\varepsilon\gamma} \left( \vertiii{T^{\gamma}_{r}W^+}^2_{1,\gamma}+\vertiii{W^+}^2_0+\vertiii{ F^+ }^2_{1,\gamma} \right).
\end{split}
\end{equation*}
The roots of the Lopatinski$\breve{\mathrm{i}}$ determinant do not coincide with the poles of the differential equations. Thus we can estimate $Z_j$ for $j=1,2,5,7,\cdots,13$ in the same way. Multiplying \eqref{para31} with some appropriately chosen matrix symbol in $\Gamma^0_1,$ we obtain that
\begin{equation}\label{para32}
T^{\gamma}_aZ_j+T^{\gamma}_{\Theta_1}Z_4+\sum_i T^{\gamma}_{\Theta_0}Z_i+T^{\gamma}_rW^++\mathcal{R}_{-1}W^+=\mathcal{R}_0F^+.\\
\end{equation}
Here, $|\Re a|\geq c\Lambda$ on the support of $\chi_2.$
We can extend $a$ into a new symbol $\tilde{a}$ satisfying $|\Re \tilde{a}|\geq c\Lambda.$
Applying the two symmetrizers $(T^{\gamma}_{\sigma})^{\ast}T^{\gamma}_{\sigma}$ and $T^{\gamma}_{\Lambda},$
we have
\begin{equation*}
\begin{split}
\|T^{\gamma}_{\sigma}Z_j \|^2_{\frac{1}{2},\gamma} & \leq C\| Z_j \|^2_{1,\gamma}+\frac{1}{\varepsilon}\|Z_4\|^2_{\frac{3}{2},\gamma}+\sum_i \frac{1}{\varepsilon}\|Z_i \|^2_{\frac{1}{2},\gamma} \\
&\quad + \frac{1}{\varepsilon} \left( \|T^{\gamma}_rW^+ \|^2_{\frac{1}{2},\gamma}+\| W^+ \|^2_{-\frac{1}{2},\gamma}+\| F^+ \|^2_{\frac{1}{2},\gamma} \right), \\
\| Z_j \|^2_{1,\gamma} & \leq \frac{1}{\varepsilon}\| Z_j \|^2_{0}+\frac{1}{\varepsilon}\|Z_4\|^2_{1,\gamma}\\
&\quad + \sum_i \frac{1}{\varepsilon}\|Z_i \|^2_{0} + \frac{1}{\varepsilon} \Big( \|T^{\gamma}_rW^+ \|^2_{0}+\| W^+ \|^2_{-1,\gamma}+\| F^+ \|^2_{0} \Big),
\end{split}
\end{equation*}
\begin{equation*}
\begin{split}
\vertiii{T^{\gamma}_{\sigma}Z_3 }^2_0&\leq \Re\left\langle T^{\gamma}_{\sigma}Z_3, T^{\gamma}_{\frac{1}{\Lambda}}T^{\gamma}_{\sigma}Z_3 \right\rangle\Big|_{x_3=0} + \frac{1}{\varepsilon}\vertiii{ Z_3 }^2_0 + \frac{1}{\varepsilon} \left( \vertiii{ T^{\gamma}_{\sigma}Z_j }^2_0 + \vertiii{ Z_j }^2_0 \right) \\
&\quad+\frac{1}{\varepsilon}\left(\vertiii{T^{\gamma}_{\sigma}Z_3 }^2_{-1,\gamma}+\vertiii{ Z_3 }^2_{-1,\gamma}\right) +\sum_{i\neq3,4}\frac{1}{\varepsilon}\left(\vertiii{T^{\gamma}_{\sigma}Z_i}^2_{-1,\gamma}+\vertiii{Z_i }^2_{-1,\gamma}\right)\\
&\quad + \frac{1}{\varepsilon} \left( \vertiii{ T^{\gamma}_rW^+ }^2_0+\vertiii{W^+}^2_{-1,\gamma}+\vertiii{ F^+ }^2_0 \right), \\
\vertiii{ Z_3 }^2_{\frac{1}{2},\gamma}&\leq \|Z_3|_{x_3=0} \|^2_0 + \left( C+\frac{1}{\varepsilon} \right)\vertiii{ Z_3 }^2_0+\frac{1}{\varepsilon}\vertiii{ Z_j }^2_{\frac{1}{2},\gamma}+\sum_{i\neq3,4}\frac{1}{\varepsilon}\vertiii{Z_i }^2_{-\frac{1}{2},\gamma}\\
&\quad+\frac{1}{\varepsilon} \left( \vertiii{ T^{\gamma}_rW^+ }^2_{-\frac{1}{2},\gamma}+\vertiii{W^+}^2_{-\frac{3}{2},\gamma}+\vertiii{ F^+ }^2_{-\frac{1}{2},\gamma} \right).
\end{split}
\end{equation*}
Combining the above estimates and dividing by $\gamma$ to an appropriate power and then taking $\gamma$ large enough, we obtain that
\begin{equation}\label{para33}
\begin{split}
&\frac{1}{\gamma}\vertiii{T^{\gamma}_{\Lambda}T^{\gamma}_{\sigma}Z_4}^2_0+\vertiii{Z_4}^2_{\frac{3}{2},\gamma}+\vertiii{ T^{\gamma}_{\sigma}Z_j }^2_{\frac{1}{2},\gamma}+\gamma\vertiii{ Z_j }^2_{1,\gamma}+\gamma\vertiii{T^{\gamma}_{\sigma}Z_3 }^2_0\\
&\quad\quad+\gamma^2\vertiii{ Z_3 }^2_{\frac{1}{2},\gamma}+\frac{1}{\gamma} \left\| T^{\gamma}_{\Lambda^{\frac{1}{2}}}T^{\gamma}_{\sigma}Z_4|_{x_3=0} \right\|^2_0+\| Z_4|_{x_3=0} \|^2_{1,\gamma}\\
&\quad\leq \gamma\Re\left\langle T^{\gamma}_{\sigma}Z_3,T^{\gamma}_{\frac{1}{\Lambda}}T^{\gamma}_{\sigma}Z_3 \right\rangle|_{x_3=0}\\
&\quad\quad+\gamma^2\|Z_3|_{x_3=0} \|^2_0+\frac{1}{\gamma}\left(\vertiii{T^{\gamma}_{r}W^+}^2_{1,\gamma}+\vertiii{W^+}^2_0+\vertiii{ F^+ }^2_{1,\gamma}\right).\\
\end{split}
\end{equation}
\begin{equation}\label{def:Zminus.rt}
Z^-:=(Z_{14},\cdots, Z_{26})^{\top}:=T^{\gamma}_{\chi_1((Q^{l}_0)^{-1}+Q^l_{-1})}T^{\gamma}_{\chi_{rt}}W^-.
\end{equation}
For $Z^-$ we have
\begin{equation}\label{para34}
\begin{split}
&\frac{1}{\gamma}\vertiii{ T^{\gamma}_{\Lambda}T^{\gamma}_{\sigma}Z_{17} }^2_0+\vertiii{ Z_{17} }^2_{\frac{3}{2},\gamma}+\vertiii{ T^{\gamma}_{\sigma}Z_j }^2_{\frac{1}{2},\gamma}+\gamma\vertiii{ Z_j }^2_{1,\gamma}+\gamma\vertiii{ T^{\gamma}_{\sigma}Z_{16} }^2_0\\
&\quad\quad+\gamma^2\vertiii{ Z_{16} }^2_{\frac{1}{2},\gamma}+\frac{1}{\gamma}\|T^{\gamma}_{\Lambda^{\frac{1}{2}}}T^{\gamma}_{\sigma}Z_{17}|_{x_3=0}\|^2_0+\|Z_{17}|_{x_3=0}\|^2_{1,\gamma}\\
&\quad\leq \gamma\Re\left\langle T^{\gamma}_{\sigma}Z_{16},T^{\gamma}_{\frac{1}{\Lambda}}T^{\gamma}_{\sigma}Z_{16} \right\rangle|_{x_3=0}\\
&\quad\quad+\gamma^2\|Z_{16}|_{x_3=0}\|^2_0+\frac{1}{\gamma}\left(\vertiii{ T^{\gamma}_{r}W^- }^2_{1,\gamma}+\vertiii{ W^- }^2_0+\vertiii{ F^- }^2_{1,\gamma}\right).\\
\end{split}
\end{equation}
Similar to Case 1, the boundary terms in \eqref{para2} can be used to estimate $\|Z_3|_{x_3=0} \|^2_0$ and $\|Z_{16}|_{x_3=0}\|^2_0,$
and $\gamma\Re\left\langle T^{\gamma}_{\sigma}Z_3, T^{\gamma}_{\frac{1}{\Lambda}}T^{\gamma}_{\sigma}Z_3 \right\rangle \Big|_{x_3=0}$ and $\gamma\Re\left\langle T^{\gamma}_{\sigma}Z_{16}, T^{\gamma}_{\frac{1}{\Lambda}}T^{\gamma}_{\sigma}Z_{16} \right\rangle \Big|_{x_3=0}.$ By a similar argument as in \cite[Section 5.2, Case 2]{RChen2018}, one can obtain  
$$
\gamma^2\|Z_{\rm in}|_{x_3=0}\|_0^2+\|T^{\gamma}_{\sigma}Z_{\rm in}|_{x_3=0}\|_0^2
\lesssim \|G\|_{1,\gamma}^2+\|Z_{\rm out}|_{x_3=0}\|_{1,\gamma}^2+\|W^{nc}|_{x_3=0}\|_0^2.
$$

Using \eqref{para33} and \eqref{para34}, we have
\begin{equation}\label{para35}
\begin{split}
&\frac{1}{\gamma} \vertiii{ T^{\gamma}_{\Lambda}T^{\gamma}_{\sigma}Z_{\rm out} }^2_0 + \vertiii{ Z_{\rm out} }^2_{\frac{3}{2},\gamma} + \vertiii{ T^{\gamma}_{\sigma}Z_{\rm c} }^2_{\frac{1}{2},\gamma}+\gamma\vertiii{ Z_{\rm c} }^2_{1,\gamma}\\
&\quad\quad + \gamma \vertiii{ T^{\gamma}_{\sigma}Z_{\rm in} }^2_{0} + \gamma^2 \vertiii{ Z_{\rm in} }^2_{\frac{1}{2},\gamma} + \frac{1}{\gamma} \left\| T^{\gamma}_{\Lambda^{\frac{1}{2}}}T^{\gamma}_{\sigma}Z_{\rm out}|_{x_3=0} \right\|^2_0 + \| Z_{\rm out}|_{x_3=0} \|^2_{1,\gamma}\\
&\quad\quad + \gamma^2 \| Z_{\rm in}|_{x_3=0} \|^2_0 + \left\|T^{\gamma}_{\sigma}Z_{\rm in}|_{x_3=0} \right\|^2_0\\
&\quad\lesssim \|G\|^2_{1,\gamma}+\left\| W^{nc}|_{x_3=0} \right\|^2_0+\frac{1}{\gamma} \left( \vertiii{ T^{\gamma}_{r}W }^2_{1,\gamma} + \vertiii{W}^2_0+\vertiii{ F }^2_{1,\gamma} \right),\\
\end{split}
\end{equation}
where $Z_{\rm c}=(Z_1,Z_2,Z_5,\cdots,Z_{13},Z_{14},Z_{15},Z_{18},\cdots,Z_{26})^{\top}.$

\subsection{Case 3: Points in $\Upsilon^{(2)}_p$}
In this section, we discuss the poles that are  not the roots of Lopatinski$\breve{\mathrm{i}}$ determinant. Our discussion focuses on the neighborhoods $\mathcal{V}^1_{p_2},\mathcal{V}^2_{p_2},\mathcal{V}^3_{p_2}$ and $\mathcal{V}^4_{p_2}.$ As an example, consider $\mathcal{V}^1_{p_2}$.
This neighborhood contains a strip in the frequency space where $\tau=-i(v^r_1\eta+v^r_2\tilde{\eta}+\sqrt{(\eta^2+\tilde{\eta}^2)g_r(\theta)}).$ Here $\tau$ represents a pole of the differential equations for $W^+,$ but not for $W^-.$
In this case, the equation for $W^-$ can be reduced to a non-characteristic one.
The arguments apply to both $W^+$ and $W^-$, provided that the points where $\omega^{r,l}=0$ are excluded from the neighborhood. For simplicity, we restrict our discussion to the case of $W^{+}.$
First, we introduce the cut-off functions $\chi_{p_2},$ $\chi_1$ and $\chi_2,$ which are defined in the pole case. Using the matrices $Q^r_0$ and $Q^r_{-1}$ and symmetrizers $R^r_0$ and $R^r_{-1},$ along with appropriate adjustments to $\chi_2\tilde{A}^r$ and $D_0$, we derive the following equation:
\begin{equation}\label{para36}
I_2\partial_3Z^+=-T^{\gamma}_{\tilde{D}_1}Z^++T^{\gamma}_{\tilde{D}_0}Z^++T^{\gamma}_rW^++\mathcal{R}_0F^++\mathcal{R}_{-1}W^+,
\end{equation}
where
\begin{equation}\label{def:Zplus.p2}
Z^+:=T^{\gamma}_{\chi_1((Q^{r}_0)^{-1}+Q^r_{-1})}T^{\gamma}_{\chi_{p_1}}W^+.
\end{equation}

The symbols in this equation are consistent with those introduced in the previous cases.
The equation for $Z_4$ yields
$$
\partial_3Z_4=T^{\gamma}_{-\omega^r+i\bar{\omega}^r}Z_4+T^{\gamma}_{\Theta_0}Z_4+\sum_{i\neq3,4}T^{\gamma}_{\Theta_0} Z_i+T^{\gamma}_rW^++\mathcal{R}_0F^++\mathcal{R}_{-1}W^+.
$$
Consider the symmetrizer $(T^{\gamma}_{\Lambda})^{\ast}T^{\gamma}_{\Lambda}T^{\gamma}_{\Lambda},$ we obtain that
\begin{equation*}
\begin{split}
\Re \left\langle T^{\gamma}_{\Lambda}T^{\gamma}_{\Lambda}Z_4,T^{\gamma}_{\Lambda}\partial_3Z_4 \right\rangle & = \Re \left\langle T^{\gamma}_{\Lambda}T^{\gamma}_{\Lambda}Z_4,T^{\gamma}_{\Lambda}T^{\gamma}_{-\omega^r+i\bar{\omega}^r}Z_4 \right\rangle + \Re \left\langle T^{\gamma}_{\Lambda}T^{\gamma}_{\Lambda}Z_4,T^{\gamma}_{\Lambda}T^{\gamma}_{\Theta_0}Z_4 \right\rangle \\
&\quad + \sum_{i\neq3,4} \Re \left\langle T^{\gamma}_{\Lambda}T^{\gamma}_{\Lambda}Z_4,T^{\gamma}_{\Lambda}T^{\gamma}_{\Theta_0}Z_i \right\rangle + \Re \left\langle T^{\gamma}_{\Lambda}T^{\gamma}_{\Lambda}Z_4,T^{\gamma}_{\Lambda}T^{\gamma}_rW^+ \right\rangle \\
&\quad + \Re \left\langle T^{\gamma}_{\Lambda}T^{\gamma}_{\Lambda}Z_4,T^{\gamma}_{\Lambda}\mathcal{R}_{-1}W^+ \right\rangle + \Re \left\langle T^{\gamma}_{\Lambda}T^{\gamma}_{\Lambda}Z_4,T^{\gamma}_{\Lambda}F^+ \right\rangle.
\end{split}
\end{equation*}
Taking $\varepsilon$ small enough, we have
\begin{equation*}
\begin{split}
&\vertiii{Z_4}^2_{2,\gamma}+\vertiii{ Z_4|_{x_3=0} }^2_{\frac{3}{2},\gamma}\\
&\quad\lesssim \| Z_4|_{x_3=0} \|^2_{1,\gamma} + \left( \frac{1}{\varepsilon}+C \right) \vertiii{Z_4}^2_{1,\gamma}+\sum_{i\neq 3,4}\frac{1}{\varepsilon}\vertiii{Z_i }^2_{1,\gamma}\\
&\quad\quad + \frac{1}{\varepsilon} \left( \vertiii{T^{\gamma}_{r}W^+}^2_{1,\gamma}+\vertiii{W^+}^2_0+\vertiii{ F^+ }^2_{1,\gamma} \right).
\end{split}
\end{equation*}
For $j=1,2,5\cdots,13$ in \eqref{para36}, applying the symmetrizer $(T^{\gamma}_{\Lambda})^{\ast}T^{\gamma}_{\Lambda},$ we have
\begin{equation*}
\begin{split}
\gamma\| Z_j \|^2_{1,\gamma}&\leq C\| Z_j \|^2_{1,\gamma}+\frac{1}{\gamma}\|Z_4\|^2_{2,\gamma}\\
&\quad+ \sum_i \frac{1}{\gamma}\|Z_i \|^2_{1,\gamma}+\frac{1}{\gamma} \Big( \|T^{\gamma}_rW^+ \|^2_{1,\gamma}+\| W^+ \|^2_0+\| F^+ \|^2_{1,\gamma} \Big).
\end{split}
\end{equation*}
For the incoming mode $Z_3,$ we take the symmetrizer $T^{\gamma}_{\Lambda},$
\begin{equation*}
\begin{split}
\vertiii{ Z_3 }^2_{1,\gamma}&\leq \Re \left\langle Z_3, T^{\gamma}_{\Lambda}Z_3 \right\rangle\big|_{x_3=0}+\frac{1}{\varepsilon}\vertiii{ Z_3 }^2_0+\frac{1}{\varepsilon}\vertiii{ Z_j }^2_{1,\gamma}+\sum_{i\neq3,4}\frac{1}{\varepsilon}\vertiii{Z_i }^2_0\\
&\quad+\frac{1}{\varepsilon} \left( \vertiii{ T^{\gamma}_rW^+ }^2_0+\vertiii{W^+}^2_{-1,\gamma}+\vertiii{ F }^2_0 \right).
\end{split}
\end{equation*}
Combining all the estimates above, we obtain that
\begin{equation*}
\begin{split}
&\frac{1}{\gamma} \vertiii{ Z_{4} }^2_{1,\gamma} + \gamma\| Z_j \|^2_{1,\gamma} + \gamma\vertiii{ Z_3 }^2_{1,\gamma} + \frac{1}{\gamma}\| Z_4|_{x_3=0} \|^2_{\frac{3}{2},\gamma} \\
&\quad\lesssim \gamma \Re \left\langle Z_3,T^{\gamma}_{\Lambda}Z_3 \right\rangle|_{x_3=0} + \frac{1}{\gamma} \Big( \|T^{\gamma}_rW^+ \|^2_{1,\gamma}+\| W^+ \|^2_0+\| F^+ \|^2_{1,\gamma} \Big).
\end{split}
\end{equation*}
For
\begin{equation}\label{def:Zminus.p2}
(Z_{14},\cdots,Z_{26})^{\top}:=T^{\gamma}_{\chi_1Q^l}T^{\gamma}_{\chi_{p_2}}W^-,
\end{equation}
we obtain that
\begin{equation*}
\begin{split}
&\frac{1}{\gamma}\vertiii{ Z_{17} }^2_{1,\gamma}+\gamma\vertiii{ Z_j }^2_{1,\gamma}+\gamma\vertiii{ Z_{16} }^2_{1,\gamma}+\frac{1}{\gamma}\|Z_{17}|_{x_3=0}\|^2_{\frac{3}{2},\gamma} \\
&\quad\lesssim\gamma \Re \left\langle Z_{16},T^{\gamma}_{\Lambda}Z_{16} \right\rangle|_{x_3=0} + \frac{1}{\gamma} \Big( \|T^{\gamma}_rW^- \|^2_{1,\gamma} + \| W^- \|^2_0 + \| F^- \|^2_{1,\gamma} \Big).
\end{split}
\end{equation*}
Now, we estimate $\gamma \Re\langle Z_3,T^{\gamma}_{\Lambda}Z_3\rangle|_{x_3=0}$ and $\gamma \Re\langle Z_{16},T^{\gamma}_{\Lambda}Z_{16}\rangle|_{x_3=0}$.
These terms can be controlled by $\| Z_{\rm in}|_{x_3=0} \|^2_{1,\gamma}.$ 

Using the boundary conditions \eqref{para2} and using the fact that the Lopatinski$\breve{\mathrm{i}}$ determinant has a positive lower bound in the open neighourhood $\mathcal{V}^1_{p_2}$,
we have
$$\| Z_{\rm in}|_{x_3=0} \|^2_{1,\gamma}\lesssim \|G\|^2_{1,\gamma}+\| Z_{\rm out}|_{x_3=0} \|^2_{1,\gamma}+\left\| W^{nc}|_{x_3=0} \right\|^2_0.$$
Putting together, we have
\begin{equation}\label{para37}
\begin{split}
&\frac{1}{\gamma}\vertiii{ Z_{\rm out} }^2_{2,\gamma}+\gamma\vertiii{ Z_{\rm c} }^2_{1,\gamma}+\gamma\vertiii{ Z_{\rm in} }^2_{1,\gamma}+\frac{1}{\gamma}\| Z_{\rm out}|_{x_3=0} \|^2_{\frac{3}{2},\gamma}+\| Z_{\rm in}|_{x_3=0} \|^2_{1,\gamma}\\
&\quad\lesssim \|G\|^2_{1,\gamma}+\left\| W^{nc}|_{x_3=0} \right\|^2_0 + \frac{1}{\gamma} \Big( \|T^{\gamma}_rW \|^2_{1,\gamma} + \| W \|^2_0 + \| F \|^2_{1,\gamma} \Big),
\end{split}
\end{equation}
where $Z_{\rm c}=(Z_1,Z_2,Z_5,\cdots,Z_{13},Z_{14},Z_{15},Z_{18},\cdots,Z_{26})^{\top}.$

\subsection{Other Case}
The remaining points are those where the Lopatinski$\breve{\mathrm{i}}$ determinant is non-zero, allowing the system to be reduced into a non-characteristic form. In this case, a Kreiss's symmetrizer can be constructed. This corresponds to the {\it good frequency} case in \cite{Coulombel2004}. Consider the cut-off symbol $\chi_{re}=1-\bar{\chi}_{p_1}-\bar{\chi}_{p_2}-\bar{\chi}_{rt}$ in $\Gamma^0_k$ for any integer $k,$ where, as in Section \ref{subsec est on each}, $\bar{\chi}_{p_1}$ is the sum of two cut-off functions for two neighborhoods $\mathcal{V}^l_{p_1}$ and $\mathcal{V}^r_{p_1}$,  $\bar{\chi}_{p_2}$ is the sum of four cut-off functions for four neighborhoods $\mathcal{V}^i_{p_2}$ with $i = 1, 2, 3, 4$, and $\bar{\chi}_{rt}$ is the sum of two cut-off functions for the two neighborhoods $\mathcal{V}^i_{r},i=1,2.$ The cut-off function
$\chi_{re}$ is 0 near the roots of the Lopatinski$\breve{\mathrm{i}}$ determinant $\Upsilon_r$ and $\Upsilon_p.$
We can construct an open neighborhood $\mathcal{V}_{re}$ that contains the support of $\chi_{re}$ but does not contain a small neighborhood of $\Upsilon_r$ and $\Upsilon_p.$
Denote that
$$
W^{\pm}_{re}:=T^{\gamma}_{\chi_{re}}W^{\pm}, \text{ and } W_{re}:=(W^+_{re},W^-_{re})^{\top}.
$$
Following the approach in \cite{Coulombel2004}, we can eliminate all components of $W^{\pm}_{re}$ in the kernel of $I_2.$ This leads to a differential equation for $W^{nc}_{re}:=T^{\gamma}_{\chi_{re}}W^{nc}$ of the form
\begin{equation}\label{para38}
\partial_3 W^{nc}_{re}=T^{\gamma}_{\chi_2\mathbb{A}} W^{nc}_{re} +T^{\gamma}_{\mathbb{E}} W^{nc}_{re}+T^{\gamma}_rW+\mathcal{R}_0F+\mathcal{R}_{-1}W,
\end{equation}
where $\mathbb{A}=\mathrm{diag}\{\mathbb{A}^r,\mathbb{A}^l\}$ and $\mathbb{E}, r \in\Gamma^0_1$ which are supported in the place where $\chi_{re}\in(0,1).$ Using \eqref{para}, similar to \cite{Coulombel2002,Coulombel2004},
we have the following estimate
\begin{equation}\label{para39}
\begin{split}
&\gamma \vertiii{ W_{re} }^2_{1,\gamma} + \| W^{nc}_{re}|_{x_3=0} \|^2_{1,\gamma}\\
&\quad\lesssim\|G\|^2_{1,\gamma}+\left\| W^{nc}|_{x_3=0} \right\|^2_0+\frac{1}{\gamma}\left(\vertiii{ F }^2_{1,\gamma}+\vertiii{W}^2_0+\vertiii{ T^{\gamma}_rW }^2_{1,\gamma}\right).
\end{split}
\end{equation}

\subsection{Proof of Theorem \ref{vt}}\label{subsec proof}

\begin{proof}
We now summarize all the estimates from the four cases discussed above. Taking $\gamma$ sufficiently large and summing up \eqref{para30}, \eqref{para35}, \eqref{para37} and \eqref{para39}, we have that the left-hand side of the sum is bounded by
$$
\gamma^3\vertiii{W}^2_0+\gamma^2\left\| W^{nc}|_{x_3=0} \right\|^2_0.
$$
The support of $r$ is contained in the following set:
\begin{equation*}
\left\{ (t,x_1,x_2,x_3,\delta,\eta,\tilde{\eta})\in \R^4_+\times \Pi: \bar{\chi}_{p_1}\in(0,1) \text{ or } \bar{\chi}_{p_2}\in(0,1) \text{ or } \bar{\chi}_{rt}\in(0,1) \text{ or }\bar{\chi}_{re}\in(0,1) \right\}.
\end{equation*}
Note that $\bar{\chi}_{p_1}+\bar{\chi}_{p_2}+\bar{\chi}_{rt}+\chi_{re}=1$. Then $r=0$ when $\bar{\chi}_{p_1},\bar{\chi}_{p_2},\bar{\chi}_{rt}$, or $\chi_{re}$ equals $1.$ We also have that
$\sigma$ vanishes only at some points where $\bar{\chi}_{p_1}=1$ or $\bar{\chi}_{rt}=1.$ Thus
$\sigma$ has a lower bound on the support of $r$ and we write
\begin{equation*}
r=a_{p_2}\bar{\chi}_{p_2}+a_{re}\chi_{re}+a_{p_1}\sigma \chi^{p_1}_1\left[\begin{array}{cc}
Q^r & {\mathbf 0}\\
{\mathbf 0} & Q^l
\end{array}
\right]\bar{\chi}_{p_1}+
a_{rt}\sigma \chi^{rt}_1\left[\begin{array}{cc}
Q^r & {\mathbf 0}\\
{\mathbf 0} & Q^l
\end{array}
\right]\bar{\chi}_{rt},
\end{equation*}
where $a_{p_2},a_{re},a_{p_1},a_{rt}$ all have block diagonal structures in $\Gamma^0_1.$ Let $\chi^{p_1}_1$ and $\chi^{rt}_1$ denote the corresponding cut-off functions in Case 1 and 2. So the term $\frac{1}{\gamma}\vertiii{ T^{\gamma}_rW }^2_{1,\gamma}$ can be absorbed by
$$
\gamma \vertiii{ T^{\gamma}_{\sigma}T^{\gamma}_{\bar{\chi}_{p_1}}W }^2_0 + \gamma \vertiii{ T^{\gamma}_{\sigma}T^{\gamma}_{\bar{\chi}_{rt}}W }^2_0 + \gamma \vertiii{ T^{\gamma}_{\bar{\chi}_{p_2}}W }^2_{1,\gamma} + \gamma \vertiii{ T^{\gamma}_{\chi_{re}}W }^2_{1,\gamma}.
$$
This term can also be controlled by the left-hand side of the sum of  \eqref{para30},\eqref{para35},\eqref{para37},\eqref{para39}.

We choose the cut-offs so that $\supp\chi_B\subset\supp\chi_{re}$. Then the trace vector 
\[
\mathcal{W} := (W_7,W_{10},W_{13},W_{20},W_{23},W_{26})^{\top}|_{x_3=0} 
\]
decomposes as
\[
\mathcal{W}=T^{\gamma}_{\chi_B} \mathcal{W}+T^{\gamma}_{(1-\chi_B)\chi_{re}} \mathcal{W} +T^{\gamma}_{\bar{\chi}_{p_1}}\mathcal{W} +T^{\gamma}_{\bar{\chi}_{p_2}}\mathcal{W} +T^{\gamma}_{\bar{\chi}_{rt}} \mathcal{W}+\mathcal R_{-1} \mathcal{W}.
\]
The first term is controlled by \eqref{We}. For the remaining terms, Cases~1--3 yield coercive bounds for the corresponding microlocal traces after the changes of unknowns, while in the regular region (corresponding to the Other case) the term $T^{\gamma}_{(1-\chi_B)\chi_{re}}\mathcal{W}$ can be controlled. Hence the complementary cut-offs provide a bound for $T^{\gamma}_{1-\chi_B}\mathcal{W}$, and together with \eqref{estimate6} this controls $(W^n|_{x_3=0},\psi)$. Putting together the resulting estimates and absorbing the lower-order $H^{-1}$ remainders for $\gamma$ sufficiently large gives the full-frequency estimate for $(W^n|_{x_3=0},\psi)$ stated above.

Hence it follows that
$$
\left\| W^{nc}|_{x_3=0} \right\|^2_0\leq C_0 \left(\frac{1}{\gamma^3} \vertiii{ \tilde{F} }^2_{1,\gamma}+\frac{1}{\gamma^2}\|\tilde{G}\|^2_{1,\gamma} \right),
$$
which completes the proof of Theorem \ref{vt}.
\end{proof}

\section{Well-posedness of the Linearized Problem}\label{sec.well-posed}

In this section we analyze the  linearized problem for \eqref{EVS} and establish the well-posedness of solutions in the standard Sobolev spaces $H^m$.

\subsection{Variable Coefficient Linearized Problem}\label{sec.vari1}

We begin by linearizing problem \eqref{EVS} around a given basic
state $ (\check{U}^{\pm},\check{\varPhi}^{\pm} )$.
We suppose that
\begin{align}
&\mathrm{supp}\, (\check{V}^{\pm},\check{\varPsi}^{\pm} )\subset \{-T\le t\le 2T,\ x_3\geq 0,\ |x|\leq 2\},\label{bas.c1}\\ \label{bas.c2}
& \big\|  \check{V}^{\pm}  \big\|_{W^{2,\infty}(\Omega)}
+\big\|   \check{\varPsi}^{\pm} \big\|_{W^{3,\infty}(\Omega)}\leq K,
\end{align}
for $\check{V}^{\pm}:=\check{U}^{\pm}-\bar{U}^{\pm}$
and $\check{\varPsi}^{\pm}:=\check{\varPhi}^{\pm}-\bar{\varPhi}^{\pm}$,
where
$T$ and $K$ are positive constants,
and $(\bar{U}^{\pm},\bar{\varPhi}^{\pm})$ is the background {state} given by \eqref{background}.
Moreover, we assume the basic state
$(\check{U}^{\pm},\check{\varPhi}^{\pm} )$
satisfies \eqref{Phi.eq}, \eqref{EVS.b}, and \eqref{inv5}, {\it i.e.},
\begin{subequations}\label{bas}
	\begin{alignat}{2}
	\label{bas.1}&\pm\partial_3\check{\varPhi}^{\pm}\geq {\kappa_0}>0 &\qquad  \text{ if } x_3\geq 0,\\
	\label{bas.2}&\partial_t\check{\varPhi}^{\pm}+\check{v}_1^{\pm}\partial_1\check{\varPhi}^{\pm}+\check{v}_2^{\pm}\partial_2\check{\varPhi}^{\pm}-\check{v}_3^{\pm}=0& \text{ if } x_3\geq 0,\\
	\label{bas.5}&\check{F}_{3j}^{\pm}=\check{F}_{1j}^{\pm} \p_1\check\varPhi^{\pm}+\check{F}_{2j}^{\pm} \p_2\check\varPhi^{\pm}\quad \textrm{for }\   j=1,2,3& \text{ if } x_3\geq  0,\\
	\label{bas.3}&\check{\varPhi}^{+}=\check{\varPhi}^{-}=\check{\varphi} &\qquad \text{ if }  x_3= 0,\\
	\label{bas.4}&\mathbb{B}\big(\check{U}^+,\check{U}^-,\check{\varphi}\big)={\mathbf 0} &\qquad \text{ if }  x_3= 0,
	\end{alignat}
\end{subequations}
 for some constant $\kappa_0>0$.
Constraints \eqref{bas.2} and \eqref{bas.5} ensure that the rank of the boundary matrix for the linearized problem remains constant
on the domain $\bar{\Omega}$.
Denote
$\check{U}:=(\check{U}^+,\check{U}^-)^{\top}$,
$\check{V}:=(\check{V}^+,\check{V}^-)^{\top}$,
$\check{\varPhi}:=(\check{\varPhi}^+,\check{\varPhi}^-)^{\top}$,
and
$\check{\varPsi}:=(\check{\varPsi}^+,\check{\varPsi}^-)^{\top}$
for simplicity.

We need to absorb \eqref{bas.5} into \eqref{bas.4} and write the boundary conditions in an {\it enlarged} form, which will later be analyzed separately for the Nash-Moser iteration.
The linearized operators can be defined as follows:
\begin{align} \label{L'.bb}
&\mathbb{L}'(U,\varPhi)(V,\varPsi)
:=\left(L(U,\varPhi) +\mathcal{C}(U,\varPhi) \right)V -{{\frac{1}{\partial_3 \varPhi}(L(U,\varPhi)\varPsi)\partial_3U}},\\
\label{B'.bb}
&\mathbb{B}'\big(\check{U} ,\check{\varphi} \big)(V,\psi)
:=\check{b}\nabla\psi+\check{\textup{B}} V|_{x_3=0},
\end{align}
where $V:=(V^+,V^-)^{\top}$, and
$\mathcal{C}(U,\varPhi)$,
$\check{b}$ and $\check{\textup{B}}$ are defined separately by
\setlength{\arraycolsep}{2pt}
\begin{align}
\mathcal{C}(U,\varPhi)V&:=
\left(\p_{U_i}A_1(U) \partial_1 U+\p_{U_i}A_2(U) \partial_2 U+\p_{U_i}\widetilde{A}_3(U,\varPhi) \partial_3 U\right)V_i,
\label{C.cal}
\\
\label{b.ring}
\check{b}(t,x_1,x_2)&:=
\begin{bmatrix}
0  &(\check{v}_1^+-\check{v}_1^-)|_{x_3=0}& (\check{v}_2^+-\check{v}_2^-)|_{x_3=0}\\
1 &\check{v}_{1}^+|_{x_3=0} & \check{v}_{2}^+|_{x_3=0}\\
0 &0 & 0\\
0 & (\check{F}^+_{11}-\check{F}^-_{11})|_{x_3=0} & (\check{F}^+_{21}-\check{F}^-_{21})|_{x_3=0}\\
0 & \check{F}^+_{11}|_{x_3=0} & \check{F}^+_{21}|_{x_3=0}\\
0 & (\check{F}^+_{12}-\check{F}^-_{12})|_{x_3=0} & (\check{F}^+_{22}-\check{F}^-_{22})|_{x_3=0}\\
0 & \check{F}^+_{12}|_{x_3=0} & \check{F}^+_{22}|_{x_3=0}\\
0 & (\check{F}^+_{13}-\check{F}^-_{13})|_{x_3=0} & (\check{F}^+_{23}-\check{F}^-_{23})|_{x_3=0}\\
0 & \check{F}^+_{13}|_{x_3=0} & \check{F}^+_{23}|_{x_3=0}
\end{bmatrix},
\end{align}
and
\begin{equation}\label{B.ring}
\check{\textup{B}}(t,x_1,x_2) := [\check{\textup{B}}_1 \   \check{\textup{B}}_2],
\end{equation}
where
\begin{align}
\nonumber
&\check{\textup{B}}_1 := 
 \ \left[\begin{smallmatrix}
 0 & \p_1\check{\varphi} & \p_2\check{\varphi} &-1 &0 & 0&0 & 0& 0 & 0 & 0 & 0 & 0 \\
 0 & \p_1\check{\varphi} & \p_2\check{\varphi} &-1 &0 & 0&0 & 0& 0 & 0 & 0 & 0 & 0 \\
 1 & 0 & 0 & 0 &0 & 0&0 & 0& 0 & 0 & 0 & 0 & 0 \\
 0 & 0 & 0 & 0 & \partial_1\check{\varphi} & \partial_2\check{\varphi} & -1 & 0 & 0 & 0 & 0 & 0 & 0 \\
 0 & 0 & 0 & 0 & \partial_1\check{\varphi} & \partial_2\check{\varphi} & -1 & 0 & 0 & 0 & 0 & 0 & 0 \\ 
 0 & 0 & 0 & 0 & 0 & 0 & 0 & \partial_1\check{\varphi} & \partial_2\check{\varphi} & -1 & 0 & 0 & 0 \\
 0 & 0 & 0 & 0 & 0 & 0 & 0 & \partial_1\check{\varphi} & \partial_2\check{\varphi} & -1 & 0 & 0 & 0 \\
 0 & 0 & 0 & 0 & 0 & 0 & 0 & 0 & 0 & 0 & \partial_1\check{\varphi} & \partial_2\check{\varphi} & -1 \\
 0 & 0 & 0 & 0 & 0 & 0 & 0 & 0 & 0 & 0 & \partial_1\check{\varphi} & \partial_2\check{\varphi} & -1  
\end{smallmatrix}\right],
\end{align}

\begin{align*}
\check{\textup{B}}_2 := 
& \ \left[\begin{smallmatrix}
 0 & -\p_1\check{\varphi} & -\p_2\check{\varphi} &1 &0 & 0&0 & 0 & 0 & 0 & 0 & 0 & 0 \\
 0 & 0 & 0 & 0 &0 & 0&0 & 0 & 0 & 0 & 0 & 0 & 0 \\
-1 & 0 & 0 & 0 &0 & 0&0 & 0 & 0 & 0 & 0 & 0 & 0 \\
 0 & 0 & 0 & 0 & -\partial_1\check{\varphi} & -\partial_2\check{\varphi} & 1 & 0 & 0 & 0 & 0 & 0 & 0\\
 0 & 0 & 0 & 0 & 0 & 0 & 0 & 0 & 0 & 0 & 0 & 0 & 0\\
 0 & 0 & 0 & 0 & 0 & 0 & 0 & -\partial_1\check{\varphi} & -\partial_2\check{\varphi} & 1 & 0 & 0 & 0\\
 0 & 0 & 0 & 0 & 0 & 0 & 0 & 0 & 0 & 0 & 0 & 0 & 0\\
 0 & 0 & 0 & 0 & 0 & 0 & 0 & 0 & 0 & 0 & -\partial_1\check{\varphi} & -\partial_2\check{\varphi} & 1\\
 0 & 0 & 0 & 0 & 0 & 0 & 0 & 0 & 0 & 0 & 0 & 0 & 0
\end{smallmatrix}\right].
\end{align*}

Motivated by {\rm Alinhac} \cite{A89MR976971},
we get
\begin{align*}
\mathbb{L}'(\check{U}^{\pm}, \check{\varPhi}^{\pm})(V^{\pm},\varPsi^{\pm})
=L(\check{U}^{\pm}, \check{\varPhi}^{\pm})\dot{V}^{\pm}
+\mathcal{C}( \check{U}^{\pm},\check{\varPhi}^{\pm})\dot{V}^{\pm}
+\frac{\varPsi^{\pm}}{\partial_3\check{\varPhi}^{\pm}}\partial_3
\mathbb{L}(\check{U}^{\pm} ,\check{\varPhi}^{\pm} ),
\end{align*}
where
$\dot{V}^{\pm}$ are the ``{\it good unknowns}'':
\begin{align} \label{good}
\dot{V}^{\pm}:=
V^{\pm}-\frac{\partial_3\check{U}^{\pm}}{\partial_3 \check{\varPhi}^{\pm}}\varPsi^{\pm}.
\end{align}
We now consider the effective linear system: 
\begin{subequations} \label{effective}
\begin{alignat}{3}   \label{effective.1}
&\mathbb{L}'_e\big(\check{U}^{\pm},\check{\varPhi}^{\pm}\big)\dot{V}^{\pm}
:=
L\big(\check{U}^{\pm},\check{\varPhi}^{\pm}\big)\dot{V}^{\pm}
+\mathcal{C}( \check{U}^{\pm},\check{\varPhi}^{\pm})\dot{V}^{\pm}
=f^{\pm} \quad  & \text{ if } x_3>0,\\
&\mathbb{B}'_e\big(\check{U} ,\check{\varPhi} \big)(\dot{V},\psi)
:=\check{b}\nabla\psi
+\check{b}_{\natural}\psi
+\check{\textup{B}} \dot{V} |_{x_3=0}
=g \quad  & \text{ if } x_3=0, \label{effective.2}\\
&\varPsi^+=\varPsi^-=\psi \quad  & \text{ if } x_3=0, \label{effective.3}
\end{alignat}
\end{subequations}
where $L(\check{U}^{\pm},\check{\varPhi}^{\pm})$,
$\mathcal{C}(\check{U}^{\pm},\check{\varPhi}^{\pm})$,
$\check{b}$
and $\check{\textup{B}}$ are
given in \eqref{L.def}, \eqref{C.cal}, \eqref{b.ring} and \eqref{B.ring}, separately,
$\dot{V}:=(\dot{V}^+,\dot{V}^-)^{\top}$,
and
\begin{align}\label{b.natural}
\check{b}_{\natural}(t,x_1,x_2) := \check{\textup{B}}(t,x_1,x_2)
\left.\begin{bmatrix}
{\partial_3 \check{U}^+}/{\partial_3 \check{\varPhi}^+}\\[1mm]
{\partial_3 \check{U}^-}/{\partial_3 \check{\varPhi}^-}
\end{bmatrix} \right|_{x_3=0}.
\end{align}
{Here,}
$\mathcal{C}(\check{U}^{\pm},\check{\varPhi}^{\pm})$
{are two}  smooth functions
of $(\check{V}^{\pm},\nabla \check{V}^{\pm},\nabla \check{\varPsi}^{\pm})$ that vanish
at the origin, $\check{b}$ is a smooth function of trace $\check{V}|_{x_3=0},$ while
  $\check{b}_{\natural}$ is a smooth vector-function of $(\nabla \check{V}|_{x_3=0},
\nabla \check{\varPsi}|_{x_3=0})$, which also vanishes at the origin. Additionally, the matrix $\check{\textup{B}}$ is a smooth matrix function of
$\nabla \check{\varphi}$.
 It is important to note that the boundary condition \eqref{effective.2} depends on the traces of $\dot{V}$ solely through $\mathbb{P}(\check{\varphi})\dot{V}^{\pm}|_{x_3=0}$, where
\begin{align} \label{P.bb}
\mathbb{P}(\check{\varphi})V^{\pm}
:=\left(\dot{V}^{nc+},\quad \dot{V}^{nc-}\right)^{\top}.
\end{align}

To transform the linearized problem \eqref{effective}
into a form with a constant diagonal boundary matrix, we introduce the following matrices:
\begin{equation} \label{R.def}
R(U,\varPhi ):=\,
\left[\begin{smallmatrix}
0 & 0 & \langle \p_{\rm tan}\varPhi\rangle & \langle \p_{\rm tan}\varPhi\rangle & 0 & 0 & 0 & 0 & 0 & 0 & 0 & 0 & 0\\[1mm]
1 & 0 &  -\frac{c(\rho)}{\rho}\p_1\varPhi & \frac{c(\rho)}{\rho}\p_1\varPhi   & 0 & 0 & 0 & 0 & 0 & 0 & 0 & 0 & 0\\[1mm]
 0 & 1 & -\frac{c(\rho)}{\rho}\p_2\varPhi & \frac{c(\rho)}{\rho}\p_2\varPhi & 0 & 0  & 0 & 0 & 0 & 0 & 0 & 0 & 0 \\[1mm]
 \partial_1\varPhi & \partial_2\varPhi & \frac{c(\rho)}{\rho} &-\frac{c(\rho)}{\rho} &0& 0 & 0  & 0 & 0   & 0 & 0 & 0 & 0\\[1mm]
 0 & 0 & 0 &0 &1 &0  & 0 &  0 & 0  & 0 & 0 & 0 & 0\\[1mm]
0 & 0 & 0 &0 &0 &1  & 0 &  0 & 0  & 0 & 0 & 0 & 0\\[1mm]
0 & 0 & 0 &0 &0 &0  & 1&  0 & 0  & 0 & 0 & 0 & 0\\[1mm]
0 & 0 & 0&0 &0 & 0& 0  & 1& 0  & 0 & 0 & 0 & 0\\[1mm]
0 & 0 & 0 & 0 & 0 & 0 & 0 & 0  & 1  & 0 & 0 & 0 & 0\\[1mm]
0 & 0 & 0 & 0 & 0 & 0 & 0 & 0  & 0  & 1 & 0 & 0 & 0\\[1mm]
0 & 0 & 0 & 0 & 0 & 0 & 0 & 0  & 0  & 0 & 1 & 0 & 0\\[1mm]
0 & 0 & 0 & 0 & 0 & 0 & 0 & 0  & 0  & 0 & 0 & 1 & 0\\[1mm]
0 & 0 & 0 & 0 & 0 & 0 & 0 & 0  & 0  & 0 & 0 & 0 & 1\\
\end{smallmatrix}\right]
\end{equation}
and
\begin{align}
\label{A0t.def}
\widetilde{A}_0(U,\varPhi ):=\, &
\mathrm{diag} \left(1, 1,\, \frac{\p_3\varPhi}{c(\rho) \langle \p_{\rm tan}\varPhi\rangle},\,
-\frac{\p_3\varPhi}{c(\rho) \langle \p_{\rm tan}\varPhi\rangle},\,  1,\,  1,\,  1,\,  1,\,  1,\,  1,\,1,\,1,\, 1 \right),
\end{align}
where
$\langle \p_{\rm tan}\varPhi\rangle :=(1+(\p_1 \varPhi)^2+(\p_2 \varPhi)^2)^{1/2} $
and $c(\rho)$ is the sound speed given in \eqref{c.def}.
Then it follows from constraints \eqref{bas.2} and \eqref{bas.5} that
\begin{align*}
\widetilde{A}_0R^{-1}\widetilde{A}_3R\big(\check{U}^{\pm},\check{\varPhi}^{\pm}\big)
=I_2.
\end{align*}
Using the new variables
\begin{align}\label{W.def}
W^{\pm}:=R^{-1}\big(\check{U}^{\pm},\check{\varPhi}^{\pm}\big)\dot{V}^{\pm},
\end{align}
the problem \eqref{effective} can be equivalently reformulated as
\begin{subequations}\label{reformulation}
\begin{alignat}{3}
\label{reformulation.1}&\mathcal{A}_0^{\pm}\partial_t W^{\pm}+\mathcal{A}_1^{\pm}\partial_1 W^{\pm}+\mathcal{A}_2^{\pm}\partial_2 W^{\pm}+I_2\partial_3 W^{\pm}+\mathcal{A}_4^{\pm} W^{\pm}=F^{\pm} \quad  & \text{ if } x_3>0, \\
\label{reformulation.2}&\check{b}\nabla\psi+\check{b}_{\natural}\psi+\bm{B} W^{\rm nc}=g \quad  & \text{ if } x_3=0, \\
\label{reformulation.3}&\varPsi^+=\varPsi^-=\psi \quad  & \text{ if } x_3=0,
\end{alignat}
\end{subequations}
where
\begin{equation*}
\begin{split}  
\mathcal{A}_0^{\pm} & :=\widetilde{A}_0 \left( \check{U}^{\pm},\check{\varPhi}^{\pm} \right),\qquad\mathcal{A}_1^{\pm}:=\widetilde{A}_0R^{-1}A_1R \left( \check{U}^{\pm},\check{\varPhi}^{\pm} \right), \\
\mathcal{A}_2^{\pm} & := \widetilde{A}_0 R^{-1}A_2R \left( \check{U}^{\pm},\check{\varPhi}^{\pm} \right), \qquad F^{\pm} := \widetilde{A}_0 R^{-1} \left( \check{U}^{\pm},\check{\varPhi}^{\pm} \right)f^{\pm}, \\
\mathcal{A}_4^{\pm} & := \widetilde{A}_0 \left( R^{-1}\partial_tR
+R^{-1}A_1\partial_1R+R^{-1}A_2\partial_2R+R^{-1}{\widetilde{A}}_3\partial_3R+R^{-1}\mathcal{C}R \right) \left( \check{U}^{\pm},\check{\varPhi}^{\pm} \right).
\end{split}
\end{equation*}
In \eqref{reformulation.2}, the coefficients $\check{b}$ and $\check{b}_{\natural}$ are
defined by \eqref{b.ring} and \eqref{b.natural} respectively. The matrix is given by
\begin{align}
\label{B.bm}
\bm{B}(t,x_1,x_2) &:= \left. \left[
\begin{matrix}
 -\dfrac{c(\check{\rho})}{\check{\rho}}\langle\p_{\rm tan}\check{\varphi}\rangle^2 \phantom{\,}
&\phantom{\,}\dfrac{c(\check{\rho})}{\check{\rho}}\langle\p_{\rm tan}\check{\varphi}\rangle^2
 \phantom{\,}
 & \phantom{\,}\dfrac{c(\check{\rho})}{\check{\rho}}\langle\p_{\rm tan}\check{\varphi}\rangle^2 \phantom{\,}
&\phantom{\,}
-\dfrac{c(\check{\rho})}{\check{\rho}}\langle\p_{\rm tan}\check{\varphi}\rangle^2 \\[4mm]
 -\dfrac{c(\check{\rho})}{\check{\rho}}\langle\p_{\rm tan}\check{\varphi}\rangle^2
&\dfrac{c(\check{\rho})}{\check{\rho}}\langle\p_{\rm tan}\check{\varphi}\rangle^2
 & 0 &0 \\[4mm]
  \langle\p_{\rm tan}\check{\varphi}\rangle
& \langle\p_{\rm tan}\check{\varphi}\rangle
&   - \langle\p_{\rm tan}\check{\varphi}\rangle
&- \langle\p_{\rm tan}\check{\varphi}\rangle
\end{matrix}\right] \right|_{x_3=0},
\end{align}
and $W^{\mathrm{nc}}:=(W^{\mathrm{nc}}_+,W^{\mathrm{nc}}_-)^{\top}$
 represents the non-characteristic part of 
$W:=(W^+,W^-)^{\top}$
with $W^{\mathrm{nc}}_{\pm}:=(W^{\pm}_3, W^{\pm}_4)^{\top}$.
It is evident that
$\mathcal{A}_0^{\pm},$ $\mathcal{A}_1^{\pm}$ and $\mathcal{A}_2^{\pm}$ are smooth functions of
$(\check{V}^{\pm},\nabla \check{\varPsi}^{\pm})$,
$\mathcal{A}_4^{\pm}$ are  smooth matrix-functions of
$(\check{V}^{\pm},\nabla\check{V}^{\pm},\nabla \check{\varPsi}^{\pm},\nabla^2 \check{\varPsi}^{\pm})$,
and $\bm{B}$ is a smooth {matrix-}function of $(\check{V}|_{x_3=0},\nabla \check{\varphi})$.

We are now prepared to state the following theorem. The proof of the theorem will comprise the rest of the section.
\begin{theorem}
\label{thm2}
Let $T>0$ and $m\in\mathbb{N}$ with $ m\geq 2$ being fixed.
Suppose that the background state \eqref{background} satisfies \eqref{H1} and \eqref{stability4},
and that
$(\check{V}^{\pm},\check{\varPsi}^{\pm})$
belong to $H^{m+3}_{\gamma}(\Omega_T)$
for all $\gamma\geq 1$,
and satisfy \eqref{bas.c1}--\eqref{bas}
and
\begin{equation} \label{thm2.H}
\|(\check{V}^{\pm}, \check{\varPsi}^{\pm})\|_{H^6_{\gamma}(\Omega_T)}
+\|(\check{V}^{\pm}, \check{\varPsi}^{\pm}) \|_{H^5_{\gamma}(\omega_T)}\leq K.
\end{equation}
Suppose\ further that the source terms $(f,g)\in H^{m+1}(\Omega_T)\times H^{m+1}(\omega_T)$ vanish in the past. Then there exist constants $K_0>0$ and $\gamma_0\geq 1$ such that,
if $K\leq K_0$ and $\gamma\geq \gamma_0$, the problem \eqref{effective} has a unique solution
$(\dot{V}^{\pm},\psi)\in H^{m}(\Omega_T)\times H^{m+1}(\omega_T)$
vanishing in the past and satisfying the tame estimates
\begin{equation} \label{thm2.est}
\begin{split}
&\|\dot{V}\|_{H^{m}_{\gamma}(\Omega_T)} +\|\mathbb{P}(\check{\varphi})\dot{V}^{\pm}\|_{H^{m}_{\gamma}(\omega_T)}
+\|\psi\|_{H^{m+1}_{\gamma}(\omega_T)} \\
&\quad   \lesssim
\|f\|_{H^{m+1}_{\gamma}(\Omega_T)}
+\|g\|_{H^{m+1}_{\gamma}(\omega_T)}
+
\left( \|f\|_{H^3_{\gamma}(\Omega_T)} +\|g\|_{H^3_{\gamma}(\omega_T)}\right)
\|(\check{V}^{\pm},\check{\varPsi}^{\pm})\|_{H^{m+3}_{\gamma}(\Omega_T)}.
\end{split}
\end{equation}
\end{theorem}

When  $f$ and $g$ vanish in the past (which is equivalent to zero initial data), Theorem \ref{thm2} applies. The case of general initial data will be addressed in Section \ref{sec.compa}, where approximate solutions are constructed prior to applying the Nash-Moser iteration scheme.

\subsection{Well-Posedness in $L^2$}\label{sec.well-posed1}

We recall the following $L^2$ {\it a priori} energy estimate
for the linearized problem \eqref{effective}, as derived in Theorem \ref{vt}.

\begin{theorem} \label{thm.CHW}
Suppose that the background state $(\bar{U}^{\pm},\bar{\varPhi}^{\pm})$
defined by \eqref{background}
satisfies \eqref{H1} and \eqref{stability4},
and the basic state $\big(\check{U}^{\pm},\check{\varPhi}^{\pm}\big)$
satisfies \eqref{bas.c1}--\eqref{bas}.
Then there exist constants $K_0>0$ and $\gamma_0\geq 1$ such that,
if $K\leq K_0$ and $\gamma\geq \gamma_0$,
then
 \begin{equation} \label{CHW.est}
 \begin{split}
 &\gamma \|\dot{V}\|_{L^2_{\gamma}(\Omega)}^2 + \|\mathbb{P}(\check{\varphi})\dot{V} \|_{L^2_{\gamma}(\p\Omega)}^2 + \|\psi\|_{H^1_{\gamma}(\mathbb{R}^3)}^2  \\
 & \qquad \lesssim  \gamma^{-3} \vertiii{ \mathbb{L}'_e \left( \check{U}^{\pm},\check{\varPhi}^{\pm} \right)\dot{V}^{\pm} }_{L^2(H_{\gamma}^1)}^2 + \gamma^{-2} \left\| \mathbb{B}_e' \left( \check{U} ,\check{\varPhi} \right)(\dot{V} ,\psi) \right\|_{H^1_{\gamma}(\mathbb{R}^3)}^2
\end{split}
\end{equation}
for all $(\dot{V},\psi)\in H^2_{\gamma}(\Omega)\times H^2_{\gamma}(\mathbb{R}^3)$, where the operators $\mathbb{P}(\check{\varphi})$, $\mathbb{L}'_e$, and $\mathbb{B}'_e$ are defined by \eqref{P.bb}, \eqref{effective.1}, and \eqref{effective.2}, respectively.
\end{theorem}

System \eqref{effective.1} is symmetrizable hyperbolic, with coefficients
satisfying the regularity assumptions by {\rm  Coulombel} \cite{C05MR2138641}. Consequently, it is necessary
 to construct a dual problem that satisfies an appropriate energy estimate.
To this end,  we define

\begin{equation*}
\check{\varsigma}_1^{\pm}
:=-\frac{\check{\rho}^{\pm}}{\p_3 \check{\varPhi}^{\pm}},\qquad
\check{\varsigma}_2^{\pm}
:=-\frac{c(\check{\rho}^{\pm})^2\p_1\check{\varphi}}
{2\check{\rho}^{\pm}\p_3 \check{\varPhi}^{\pm}},\qquad
\check{\varsigma}_3^{\pm}
:=-\frac{c(\check{\rho}^{\pm})^2\p_2\check{\varphi} }
{2\check{\rho}^{\pm}\p_3 \check{\varPhi}^{\pm}},\qquad \check{\varsigma}_4^{\pm}
:=\frac{c(\check{\rho}^{\pm})^2 }
{2\check{\rho}^{\pm}\p_3 \check{\varPhi}^{\pm}}.
\end{equation*}
We use \eqref{bas.2} and \eqref{bas.5} to calculate
\begin{equation*}  
\check{B}_1^{\top}\check{B}
+\check{D}_1^{\top}\check{D}
=\mathrm{diag}\,
\big(\widetilde{A}_3(\check{U}^{+},\check{\varPhi}^{+}),
\,\widetilde{A}_3(\check{U}^{-},\check{\varPhi}^{-})\big)
\big|_{x_3=0},
\end{equation*}
where $\check{B}$ is given in \eqref{B.ring}, $\check{D}$ is given as follows:
\begin{align}
\label{checkD}\nonumber
\check{D}:= \nonumber
&  \left[\begin{smallmatrix}
 0 & \p_1\check{\varphi} & \p_2\check{\varphi} &-1 &0 & 0&0 & 0& 0 & 0 & 0 & 0 & 0 &
0 & \p_1\check{\varphi} & \p_2\check{\varphi} &-1 &0 & 0&0 & 0 & 0 & 0 & 0 & 0 & 0 \\
 0 & \p_1\check{\varphi} & \p_2\check{\varphi} &-1 &0 & 0&0 & 0& 0 & 0 & 0 & 0 & 0 &
0 & 0 & 0 & 0 &0 & 0&0 & 0 & 0 & 0 & 0 & 0 & 0 \\
 1 & 0 & 0 & 0 &0 & 0&0 & 0& 0 & 0 & 0 & 0 & 0 &
1 & 0 & 0 & 0 &0 & 0&0 & 0 & 0 & 0 & 0 & 0 & 0 \\
\end{smallmatrix}\right],
\end{align}
Following \cite[Section 3.2]{M01MR1842775},
we define the dual problem for \eqref{effective} as:
\begin{align}\nonumber
\begin{cases}
\mathbb{L}'_e\big(\check{U}^{\pm},\check{\varPhi}^{\pm}\big)^{\ast}U^{\pm}=f^{\ast}_{\pm}, & \text{ if } x_3>0,\\
\check{D}^{\top}_1U={\mathbf 0},\quad \mathrm{div}(\check{b}^{\top}\check{B}_1U)-\check{b}_{\natural}\check{B}_1U={\mathbf 0}, & \text{ if } x_3=0,
\end{cases}
\end{align}
where $\check{b}$, $\check{b}_{\natural}$ are given in \eqref{b.ring}, \eqref{b.natural} respectively.
$\check{B}_1$ and  $\check{D}_1$ are defined as follows:
\begin{align*}
\check{B}_1& := \left[\begin{smallmatrix}
0 & 0 & 0 & 0 & 0 & 0& 0& 0& 0& 0& 0& 0& 0& -\check{\varsigma}_1^{-} & 0 & 0 & 0  & 0 & 0 & 0 & 0 & 0 & 0 & 0 & 0 & 0\\
\check{\varsigma}_1^{+} & 0& 0& 0& 0& 0& 0& 0& 0& 0& 0& 0& 0 & \check{\varsigma}_1^{-} & 0& 0& 0& 0& 0& 0& 0& 0& 0& 0& 0& 0\\
0& \check{\varsigma}_2^{+}& \check{\varsigma}_3^{+}& \check{\varsigma}_4^{+}& 0& 0& 0& 0& 0& 0& 0& 0& 0 & 0 &  -\check{\varsigma}_2^{-}& -\check{\varsigma}_3^{-}& -\check{\varsigma}_4^{-}& 0& 0& 0& 0& 0& 0& 0& 0& 0\\
0 & 0& 0& 0& 0& 0& 0& 0& 0& 0& 0& 0& 0 & 0& 0& 0& 0& 0& 0& 0& 0& 0& 0& 0& 0& 0\\
0 & 0& 0& 0& 0& 0& 0& 0& 0& 0& 0& 0& 0 & 0& 0& 0& 0& 0& 0& 0& 0& 0& 0& 0& 0& 0\\
0 & 0& 0& 0& 0& 0& 0& 0& 0& 0& 0& 0& 0 & 0& 0& 0& 0& 0& 0& 0& 0& 0& 0& 0& 0& 0\\
0 & 0& 0& 0& 0& 0& 0& 0& 0& 0& 0& 0& 0 & 0& 0& 0& 0& 0& 0& 0& 0& 0& 0& 0& 0& 0\\
0 & 0& 0& 0& 0& 0& 0& 0& 0& 0& 0& 0& 0 & 0& 0& 0& 0& 0& 0& 0& 0& 0& 0& 0& 0& 0\\
0 & 0& 0& 0& 0& 0& 0& 0& 0& 0& 0& 0& 0 & 0& 0& 0& 0& 0& 0& 0& 0& 0& 0& 0& 0& 0\\
0 & 0& 0& 0& 0& 0& 0& 0& 0& 0& 0& 0& 0 & 0& 0& 0& 0& 0& 0& 0& 0& 0& 0& 0& 0& 0\\
0 & 0& 0& 0& 0& 0& 0& 0& 0& 0& 0& 0& 0 & 0& 0& 0& 0& 0& 0& 0& 0& 0& 0& 0& 0& 0\\
0 & 0& 0& 0& 0& 0& 0& 0& 0& 0& 0& 0& 0 & 0& 0& 0& 0& 0& 0& 0& 0& 0& 0& 0& 0& 0\\
\end{smallmatrix}\right], \\
\check{D}_1 & := \left[\begin{smallmatrix}
0 & 0 & 0 & 0 & 0 & 0& 0& 0& 0& 0& 0& 0& 0& 0 & 0 & 0 & 0  & 0 & 0 & 0 & 0 & 0 & 0 & 0 & 0 & 0\\
0 & 0& 0& 0& 0& 0& 0& 0& 0& 0& 0& 0& 0 & 0& 0& 0& 0& 0& 0& 0& 0& 0& 0& 0& 0& 0\\
0& \check{\varsigma}_2^{+}& \check{\varsigma}_3^{+}& \check{\varsigma}_4^{+}& 0& 0& 0& 0& 0& 0& 0& 0& 0 & 0& \check{\varsigma}_2^{-}&\check{\varsigma}_3^{-}& \check{\varsigma}_4^{-}& 0& 0& 0& 0& 0& 0& 0& 0& 0\\
\end{smallmatrix}\right],
\end{align*}
and the symbol $\mathrm{div}$ represents the divergence operator in $\mathbb{R}^3$.
$\mathbb{L}_e'\big(\check{U}^{\pm},\check{\varPhi}^{\pm}\big)^*$
are the adjoints of $\mathbb{L}_e'\big(\check{U}^{\pm},\check{\varPhi}^{\pm}\big)$, respectively.

Using the same analysis as in \cite[Section 3.4]{CS08MR2423311},
we can obtain the well-posedness of the linearized problem
\eqref{effective} in $L^2$.

 \begin{theorem}
 \label{thm3}
Let $T>0$ be fixed. Suppose that
$f \in L^2(\mathbb{R}_+;H^1(\omega_T))$
and $g\in H^1(\omega_T)$ vanish in the past
and all the hypotheses in {\rm Theorem \ref{thm.CHW}} are satisfied.
Then there exist constants $K_0>0$ and $\gamma_0\geq 1$ such that, if $K\leq K_0$ and $\gamma\geq \gamma_0$, then there exists a unique solution $(\dot{V}^+,\dot{V}^-,\psi)\in L^2(\Omega_T)\times L^2(\Omega_T)\times H^1(\omega_T)$ to the problem \eqref{effective.1}--\eqref{effective.2} that vanishes in the past and satisfies
\begin{equation}\label{thm3.e}
\gamma\|\dot{V}\|^2_{L^2_{\gamma}(\Omega_t)}
+\|\mathbb{P}(\check{\varphi})\dot{V}  \|^2_{L^2_{\gamma}(\omega_t)}
+\|\psi\|^2_{H^1_{\gamma}(\omega_t)}
 \lesssim \gamma^{-3} \vertiii{ f }^2_{L^2(H_{\gamma}^1(\omega_t))}
+\gamma^{-2}\|g\|^2_{H^1_{\gamma}(\omega_t)}
\end{equation}
for all $\gamma\geq \gamma_0$ and $t\in[0,T]$.
\end{theorem}

For the reformulated problem \eqref{reformulation}, Theorem \ref{thm3}
implies  that
\begin{equation}\label{thm3.e2}
 \gamma\|W\|^2_{L^2_{\gamma}(\Omega_T)} + \|W^{\mathrm{nc}} \|^2_{L^2_{\gamma}(\omega_T)}
+ \|\psi\|^2_{H^1_{\gamma}(\omega_T)}
\lesssim \gamma^{-3} \vertiii{ F^{\pm} }^2_{L^2(H^1_{\gamma}(\omega_T))}
+\gamma^{-2}\|g\|^2_{H^1_{\gamma}(\omega_T)}.
\end{equation}
For any nonnegative integer $m$, a generic smooth matrix-valued function of $\{(\p^{\alpha} \check{V},\p^{\alpha}\check{\varPsi}):|\alpha|\leq m\}$ is denoted by $\check{\rm c}_m$,
and we write $\underline{\check{\rm c}}_m$ to denote such a function that vanishes at the origin.
For example, the equations for $\dot{\rho}^{\pm}$ in \eqref{effective.1} can be rewritten as:
\begin{align}
\label{rho.eq}
(\p_t^{\check{\varPhi}^{\pm}}
+\check{v}^{\pm}_{\ell}\p_{\ell}^{\check{\varPhi}^{\pm}} )\dot{\rho}^{\pm}
+\check{\rho}^{\pm} \p_{\ell}^{\check{\varPhi}^{\pm}} \dot{v}_{\ell}^{\pm}
=\check{\rm c}_0 {f}+\underline{\check{\rm c}}_1 \dot{V}.
\end{align}
The precise forms of $\check{\rm c}_m$ and
$\underline{\check{\rm c}}_m$ may vary from line to line.

\subsection{Tangential Derivatives}\label{sec.tan}

The following lemma provides
an estimate for the tangential derivatives:
\begin{lemma}\label{lem.tan}
If the hypotheses of {\rm Theorem {\rm \ref{thm2}}} hold, then there exists a constant $\gamma_{m}\geq 1$,
 independent of $T$, such that
 \begin{align}
 &\notag
 \gamma^{1/2}\vertiii{W}_{L^2(H^{m}_{\gamma}(\omega_T))}
 +\|W^{\mathrm{nc}} \|_{H^{m}_{\gamma}(\omega_T)}
 +\|\psi\|_{H^{m+1}_{\gamma}(\omega_T)}\\
 &\quad  \lesssim
 \gamma^{-{3}/{2}}\vertiii{ F^{\pm} }_{L^2(H^{m+1}_{\gamma}(\omega_T))}
 +\gamma^{-{3}/{2}}\|W\|_{L^{\infty}(\Omega_T)}
  \| (\check{V}, \check{\varPsi} ) \|_{H^{m+3}_{\gamma}(\Omega_T)}
 \notag \\
 &\quad  \quad \,
+ \gamma^{-1}\|g\|_{H^{m+1}_{\gamma}(\omega_T)}
 +\gamma^{-1}
 \|(W^{\mathrm{nc}} ,\psi)\|_{L^{\infty}(\omega_T)}
 \|(\check{V},\check{\varPsi})\|_{{H^{m+2}_{\gamma}(\omega_T)}},
 \label{tan.est}
 \end{align}
for all $\gamma\geq \gamma_{m}$ and solutions  $(W,\psi)\in H^{m+2}_{\gamma}(\Omega_T)\times H^{m+2}_{\gamma}(\omega_T)$ to the problem \eqref{reformulation}.
\end{lemma}
\begin{proof}
We will follow the approach of \cite[Proposition 1]{CS08MR2423311} to consider the enlarged system. However, for estimating the source terms, we will use the Moser-type calculus inequalities \eqref{Moser1}--\eqref{Moser4} instead of the Gagliardo--Nirenberg and H\"{o}lder inequalities used in \cite[Proposition 1]{CS08MR2423311}.

Let $\ell\in\mathbb{N}$ with $1\leq \ell\leq m$.
Let $\alpha=(\alpha_0,\alpha_1,\alpha_2,0)\in\mathbb{N}^4$ with $|\alpha|=\ell$
so that $\p^{\alpha}=\p_t^{\alpha_0}\p_1^{\alpha_1}\p_2^{\alpha_2}$ is a tangential derivative
satisfying $\alpha_0+\alpha_1+\alpha_2=\ell$.
We then apply the operator $\p^{\alpha}$ to \eqref{reformulation.1} to obtain
\begin{equation}\label{tan.id1}
\begin{split}
& \mathcal{A}_0^{\pm}\partial_t \p^{\alpha}W^{\pm}
+\mathcal{A}_1^{\pm}\partial_1 \p^{\alpha}W^{\pm}+\mathcal{A}_2^{\pm}\partial_2 \p^{\alpha}W^{\pm} +I_2\partial_3 \p^{\alpha}W^{\pm} +\mathcal{A}_4^{\pm} \p^{\alpha} W^{\pm}
\\
&\quad
+ \sum_{\substack{|\beta|=1,\,\beta\leq \alpha}}
C_{\alpha,\beta}\big(
\p^{\beta}\mathcal{A}_0^{\pm}\partial_t \p^{\alpha-\beta}W^{\pm}
+\p^{\beta}\mathcal{A}_1^{\pm}\partial_1 \p^{\alpha-\beta}W^{\pm}+\p^{\beta}\mathcal{A}_2^{\pm}\partial_2 \p^{\alpha-\beta}W^{\pm}
\big)
=\mathscr{F}^{\alpha}_{\pm},
\end{split}
\end{equation}
where
\begin{equation*}
\begin{split}
\mathscr{F}^{\alpha}_{\pm}
:= & \, \p^{\alpha}F^{\pm}
+\sum_{\substack{0<\beta\leq \alpha}}
C_{\alpha,\beta}\p^{\beta}\mathcal{A}_4^{\pm} \p^{\alpha-\beta}W^{\pm}\nonumber\\
&+\sum_{\substack{|\beta|\geq 2,\,\beta\leq \alpha}}
C_{\alpha,\beta}\big(
\p^{\beta}\mathcal{A}_0^{\pm}\partial_t \p^{\alpha-\beta}W^{\pm}
+\p^{\beta}\mathcal{A}_1^{\pm}\partial_1 \p^{\alpha-\beta}W^{\pm}+\p^{\beta}\mathcal{A}_2^{\pm}\partial_2 \p^{\alpha-\beta}W^{\pm}
\big).
\end{split}
\end{equation*}
Similarly, from \eqref{reformulation.2} we have
\begin{equation}\label{tan.id2}
\check{b}\nabla \p^{\alpha}\psi
+\check{b}_{\natural}\p^{\alpha}\psi
+\bm{B} \p^{\alpha}W^{\rm nc}=\mathscr{G}^{\alpha}
\quad \textrm{on }\ \omega_T,
\end{equation}
where
\begin{equation*}
\mathscr{G}^{\alpha}:=
\p^{\alpha} g-[\p^{\alpha}, \check{b}]\nabla  \psi
-[\p^{\alpha},\check{b}_{\natural}]\psi -[\p^{\alpha},\bm{B}]  W^{\rm nc}.
\end{equation*}

Since the terms involving tangential derivatives of order $\ell$ in \eqref{tan.id1}
do not solely contain $\p^{\alpha} W^{\pm}$,
as in \cite[Proposition 1]{CS08MR2423311},
we write an enlarged system that accounts for all the tangential derivatives of order $\ell$. This allows us to apply the $L^2$ {\it a priori} estimate of Theorem \ref{thm.CHW}.
It is important to note that the last term on the left-hand side of \eqref{tan.id1} cannot be treated
simply as source terms due to the loss of derivatives in \eqref{CHW.est}.
We define
\begin{equation*}
{W}^{(\ell)}_{\pm} := \left\{\p_t^{\alpha_0}\p_1^{\alpha_1}\p_2^{\alpha_2} W^{\pm}:
\alpha_0+\alpha_1+\alpha_2=\ell \right\},\quad
{\psi}^{(\ell)} := \left\{\p_t^{\alpha_0}\p_1^{\alpha_1} \p_2^{\alpha_2} \psi:
\alpha_0+\alpha_1+\alpha_2=\ell \right\},
\end{equation*}
and from \eqref{tan.id1}--\eqref{tan.id2}, we obtain the following system:
\begin{subequations}
\label{tan.eq}
\begin{alignat}{2}
&
\mathscr{A}_0^{\pm}\partial_t  W^{(\ell)}_{\pm}
+\mathscr{A}_1^{\pm}\partial_1  W^{(\ell)}_{\pm}+\mathscr{A}_2^{\pm}\partial_2  W^{(\ell)}_{\pm}
+\mathscr{I}\partial_3  W^{(\ell)}_{\pm}
+\mathscr{C}^{\pm}   W^{(\ell)}_{\pm}
=\mathscr{F}^{(\ell)}_{\pm},\\
&
\mathscr{B}\nabla\psi^{(\ell)}
+\mathscr{B}_{\natural}\psi^{(\ell)}
+\mathscr{M} W^{(\ell)}_{\rm nc}=\mathscr{G}^{(\ell)},
\end{alignat}
\end{subequations}
where $\mathscr{A}_0^{\pm}$, $\mathscr{A}_1^{\pm}$, $\mathscr{A}_2^{\pm}$ and $\mathscr{I}$
are block diagonal matrices with blocks $\mathcal{A}_0^{\pm}$, $\mathcal{A}_1^{\pm}$, $\mathcal{A}_2^{\pm}$
and $I_2$, respectively.
Matrices $\mathscr{C}^{\pm}$ belong to $W^{1,\infty}(\Omega)$.
The source terms $\mathscr{F}^{(\ell)}_{\pm}$ and $\mathscr{G}^{(\ell)}$ consist of
$\mathscr{F}^{\alpha}_{\pm}$  and $\mathscr{G}^{\alpha}$ for all $\alpha=(\alpha_0,\alpha_1,\alpha_2,0)$
with $|\alpha|=\ell$, respectively.
The enlarged problem \eqref{tan.eq} satisfies an energy estimate similar to
\eqref{thm3.e2}, {\it i.e.},
\begin{equation}\label{tan.e}
\begin{split}
 &{\gamma}^{1/2}\|W^{(\ell)}\|_{L^2_{\gamma}(\Omega_T)}
+\|W^{(\ell)}_{\mathrm{nc}} \|_{L^2_{\gamma}(\omega_T)}
+\|\psi^{(\ell)}\|_{H^1_{\gamma}(\omega_T)}  \\
&\qquad \qquad  \qquad
\lesssim \gamma^{-{3}/{2}} \vertiii{ \mathscr{F}^{(\ell)} }_{L^2(H^1_{\gamma}(\omega_T))}
+\gamma^{-1}\|\mathscr{G}^{(\ell)}\|_{H^1_{\gamma}(\omega_T)}.
\end{split}
\end{equation}

Now, by using Moser-type calculus inequalities \eqref{Moser1}--\eqref{Moser4}, we estimate the source terms $\mathscr{F}^{(\ell)}_{\pm}$ and $\mathscr{G}^{(\ell)}.$ First,  we have from definition,
\begin{equation*}
\begin{split}
\vertiii{ \p^{\alpha}F }_{L^2(H^1_{\gamma}(\omega_T) )}
& \lesssim
\|(\gamma\p^{\alpha}F,
\p_t\p^{\alpha}F,\p_1\p^{\alpha}F,\p_2\p^{\alpha}F)\|_{L^2_{\gamma}(\Omega_T )}
\lesssim \vertiii{ F }_{L^2(H^{\ell +1}_{\gamma}(\omega_T) )},
\label{tan.e1}\\
\|\p^{\alpha}g\|_{H^1_{\gamma}(\omega_T) }
&\lesssim \| g\|_{H^{\ell +1}_{\gamma}(\omega_T )}.
\end{split}
\end{equation*}

For $0<\beta\leq \alpha$, we obtain
\begin{equation}\label{tan.e3a}
\|\p^{\beta}\mathcal{A}_4   \p^{\alpha-\beta}W \|_{H^1_{\gamma}(\omega_T) }
 \lesssim
\| (\gamma\p^{\beta}\mathcal{A}_4   \p^{\alpha-\beta}W,
\nabla_{t,x_1,x_2}(\p^{\beta}\mathcal{A}_4   \p^{\alpha-\beta}W)
)
\|_{L^2_{\gamma}(\omega_T) }.
\end{equation}
Applying Moser-type calculus inequality \eqref{Moser1} yields that
\begin{equation}\label{tan.e3b}
\begin{split}
&\|  \p^{\beta}\mathcal{A}_4 \p^{\alpha-\beta}W
\|_{L^2_{\gamma}(\omega_T) }
=  \| \p^{\beta-\beta'}(\p^{\beta'}\mathcal{A}_4 )  \p^{\alpha-\beta}W
\|_{L^2_{\gamma}(\omega_T) } \\
&\quad
\lesssim
\|\p^{\beta'}\mathcal{A}_4\|_{L^{\infty}(\omega_T)}\|W\|_{H^{\ell-1}_{\gamma}(\omega_T)}
+\|\p^{\beta'}\mathcal{A}_4\|_{H^{\ell-1}_{\gamma}(\omega_T)}\|W\|_{L^{\infty}(\omega_T)} \\
&\quad
\lesssim
\|W\|_{H^{\ell-1}_{\gamma}(\omega_T)}
+\|(\check{V},\check{\varPsi} )\|_{H^{\ell+2}_{\gamma}(\omega_T)}
\|W\|_{L^{\infty}(\omega_T)},
\end{split}
\end{equation}
where $\beta'\leq \beta$ with $|\beta'|=1$.
Moreover, we have
\begin{equation*}
\| \nabla_{t,x_1,x_2}(\p^{\beta}\mathcal{A}_4   \p^{\alpha-\beta}W)
\|_{L^2_{\gamma}(\omega_T) }
\lesssim
\|W\|_{H^{\ell}_{\gamma}(\omega_T)}
+\|(\check{V},\check{\varPsi} )\|_{H^{\ell+3}_{\gamma}(\omega_T)}
\|W\|_{L^{\infty}(\omega_T)}.
\end{equation*}
Combining \eqref{tan.e3a} and \eqref{tan.e3b}, we have
\begin{equation}\label{tan.e3}
\vertiii{ \p^{\beta}\mathcal{A}_4   \p^{\alpha-\beta}W }_{L^2(H^1_{\gamma}(\omega_T) )}
  \lesssim
\vertiii{W}_{L^2(H^{\ell}_{\gamma}(\omega_T))}
+\|(\check{V},\check{\varPsi} )\|_{H^{\ell+3}_{\gamma}(\Omega_T)}
\|W\|_{L^{\infty}(\Omega_T)}.
\end{equation}

For $\beta\leq \alpha$ with $|\alpha|\geq 2$, similar to \eqref{tan.e3},
we use \eqref{Moser1} to derive
\begin{align}
\notag &
\vertiii{ \p^{\beta}\mathcal{A}_0 \partial_t \p^{\alpha-\beta}W }_{L^2(H^1_{\gamma}(\omega_T) )}
+ \vertiii{ \p^{\beta}\mathcal{A}_1 \partial_1 \p^{\alpha-\beta}W }_{L^2(H^1_{\gamma}(\omega_T) )}+ \vertiii{ \p^{\beta}\mathcal{A}_2 \partial_2 \p^{\alpha-\beta}W }_{L^2(H^1_{\gamma}(\omega_T) )}
\\
&\qquad \lesssim
\vertiii{W}_{L^2(H^{\ell}_{\gamma}(\omega_T))}
+\|(\check{V},\check{\varPsi} )\|_{H^{\ell+3}_{\gamma}(\Omega_T)}
\|W\|_{L^{\infty}(\Omega_T)}.
\label{tan.e4}
\end{align}
Combining \eqref{tan.e1}, \eqref{tan.e3}, and \eqref{tan.e4} leads to
\begin{equation}\label{F.scr.e}
\vertiii{ \mathscr{F}^{(\ell)} }_{L^2(H^1_{\gamma}(\omega_T) )}
\lesssim  \vertiii{ F }_{L^2(H^{\ell +1}_{\gamma}(\omega_T) )}
+\vertiii{W}_{L^2(H^{\ell}_{\gamma}(\omega_T))}+\|(\check{V},\check{\varPsi} )\|_{H^{\ell+3}_{\gamma}(\Omega_T)}
\|W\|_{L^{\infty}(\Omega_T)}.
\end{equation}

Using \eqref{Moser3}--\eqref{Moser4}, we obtain
\begin{align*}
\left\| [\p^{\alpha}, \check{b}]\nabla  \psi \right\|_{H^1_{\gamma}(\omega_T) }
&\lesssim
\gamma \left\| [\p^{\alpha}, \check{b}]\nabla  \psi \right\|_{L^2_{\gamma}(\omega_T) }
+
\sum_{|\beta|= 1}
\left\| \p^{\beta}[\p^{\alpha}, \check{b}]\nabla  \psi \right\|_{L^2_{\gamma}(\omega_T) }\\
&\lesssim
\|\psi\|_{H^{\ell+1}_{\gamma}(\omega_T) }
+\|\underline{\check{\rm c}}_0\|_{H^{\ell+2}_{\gamma}(\omega_T) }
\|\psi\|_{L^{\infty}(\omega_T) }
\notag
\\[1.5mm]
&\lesssim
\|\psi\|_{H^{\ell+1}_{\gamma}(\omega_T) }
+\|(\check{V},\check{\varPsi} )\|_{H^{\ell+2}_{\gamma}(\omega_T) }
\|\psi\|_{L^{\infty}(\omega_T) }.
\end{align*}
Applying Moser-type calculus inequalities \eqref{Moser3}--\eqref{Moser4} to
the other terms in $\mathscr{G}^{\alpha}$, we get
\begin{equation}\label{G.scr.e}
\begin{split}
\|\mathscr{G}^{(\ell)}\|_{H^1_{\gamma}(\omega_T) }
\lesssim \; &\| g\|_{H^{\ell +1}_{\gamma}(\omega_T) }
+\|W^{\rm nc}\|_{H^{\ell}_{\gamma}(\omega_T)}
+\|\psi\|_{H^{\ell+1}_{\gamma}(\omega_T)}
\\
&+\|(\check{V},\check{\varPsi} )\|_{H^{\ell+2}_{\gamma}(\omega_T)}
\|(W^{\rm nc},\psi)\|_{L^{\infty}(\omega_T)}.
\end{split}
\end{equation}

Substitute equations \eqref{F.scr.e} and \eqref{G.scr.e} into \eqref{tan.e}, multiply
the resulting estimate by $\gamma^{m-\ell}$. By choosing $\gamma$ large enough, we conclude the desired tame estimate \eqref{tan.est}.
This completes the proof of the lemma.
\end{proof}

\subsection{Normal Derivatives of the Noncharacteristic Variables}\label{sec.non}

Following \cite{RChen2020}, we compensate for
the loss of normal derivatives by utilizing the estimates
of the linearized divergences and vorticities.
From equation \eqref{reformulation.1}, we obtain:
\begin{equation}\label{W.nc.id}
\left[\begin{matrix}
0\\
\partial_3 W^{\rm nc}_{\pm}\\
{\mathbf 0}
\end{matrix}\right]
=F^{\pm}
-\mathcal{A}_0^{\pm}\partial_t W^{\pm}-\mathcal{A}_1^{\pm}\partial_1 W^{\pm}-\mathcal{A}_2^{\pm}\partial_2 W^{\pm}
-\mathcal{A}_4^{\pm} W^{\pm}.
\end{equation}
This leads to
\begin{equation*}
\vertiii{ \partial_3 W^{\rm nc} }_{L^2(H^{m-1}_{\gamma}(\omega_T)) }
\lesssim
\vertiii{ (F,  \check{\rm c}_1 \partial_t W,  \check{\rm c}_1\partial_1 W, \check{\rm c}_1\partial_2 W,
 \underline{\check{\rm c}}_2 W) }_{L^2(H^{m-1}_{\gamma}(\omega_T)) }.
\end{equation*}
It follows from \eqref{Moser1}--\eqref{Moser2} that
\begin{align*}
\|  \underline{\check{\rm c}}_2  W  \|_{ H^{m-1}_{\gamma}(\omega_T) } &\lesssim
\|\underline{\check{\rm c}}_2\|_{L^{\infty}(\omega_T) }\|W\|_{ H^{m-1}_{\gamma}(\omega_T) }
+\|\underline{\check{\rm c}}_2\|_{ H^{m-1}_{\gamma}(\omega_T) } \|W\|_{L^{\infty}(\omega_T) }\\
&\lesssim
 \|W\|_{ H^{m-1}_{\gamma}(\omega_T) }
+\|(\check{V},\check{\varPsi})\|_{ H^{m+1}_{\gamma}(\omega_T) } \|W\|_{L^{\infty}(\omega_T) },
\end{align*}
and
\begin{equation*}
\|  \underline{\check{\rm c}}_1  W  \|_{ H^{m}_{\gamma}(\omega_T) }
\lesssim
\|W\|_{ H^{m}_{\gamma}(\omega_T) }
+\|(\check{V},\check{\varPsi})\|_{ H^{m+1}_{\gamma}(\omega_T) } \|W\|_{L^{\infty}(\omega_T) }.
\end{equation*}
It is easy to check that
\begin{align*}
\| \check{\rm c}_1 \nabla_{t,x_1,x_2} W  \|_{ H^{m-1}_{\gamma}(\omega_T) }
&\lesssim
\| \check{\rm c}_1  W  \|_{ H^{m}_{\gamma}(\omega_T) }
+\| \nabla_{t,x_1,x_2}\check{\rm c}_1  W  \|_{ H^{m-1}_{\gamma}(\omega_T) } \\
&\lesssim
\|   W  \|_{ H^{m}_{\gamma}(\omega_T) }
+\| \underline{\check{\rm c}}_1  W  \|_{ H^{m}_{\gamma}(\omega_T) }
+\|  \underline{\check{\rm c}}_2  W  \|_{ H^{m-1}_{\gamma}(\omega_T) }.
\end{align*}
Using the estimate above, we get
\begin{equation}\label{W.nc.est}
\begin{split}
\vertiii{ \partial_3 W^{\rm nc} }_{L^2(H^{m-1}_{\gamma}(\omega_T)) }
\lesssim
\;& \|F\|_{H^{m-1}_{\gamma}(\Omega_T) }
+\|W\|_{ L^2(H^{m}_{\gamma}(\omega_T)) } \\
& + \vertiii{ (\check{V},\check{\varPsi}) }_{ L^2(H^{m+1}_{\gamma}(\omega_T)) }
\|W\|_{L^{\infty}(\Omega_T) }.
\end{split}
\end{equation}

Next, we introduce the linearized divergences and vorticities,
whose estimates allow us to
recover the normal derivatives of the characteristic variables
\begin{equation}\label{W.c}
\begin{split}
&W_1=\frac{1}{{\langle\partial_{\rm tan}\varPhi\rangle}^2}\Big[\big(1+(\partial_2\varPhi)^2\big)\dot{v}_1-(\partial_1\varPhi\partial_2\varPhi)\dot{v}_2+\partial_1\varPhi\dot{v}_3\Big], \\
&W_2=\frac{1}{{\langle\partial_{\rm tan}\varPhi\rangle}^2}\Big[\big(-\partial_1\varPhi\partial_2\varPhi\big)\dot{v}_1+\big(1+(\partial_1\varPhi)^2\big)\dot{v}_2+\partial_2\varPhi\dot{v}_3\Big], \\
&(W^{\pm}_5,W^{\pm}_6,W^{\pm}_7,W^{\pm}_8,W^{\pm}_9,W^{\pm}_{10},W^{\pm}_{11},W^{\pm}_{12},W^{\pm}_{13})\\
&\quad=(\dot{F}^{\pm}_{11},\dot{F}^{\pm}_{21},\dot{F}^{\pm}_{31},\dot{F}^{\pm}_{12},\dot{F}^{\pm}_{22},\dot{F}^{\pm}_{32},\dot{F}^{\pm}_{13},\dot{F}^{\pm}_{23},\dot{F}^{\pm}_{33}),
\end{split}
\end{equation}
according to the transformation given in \eqref{W.def}.

\subsection{Divergence Estimates}\label{sec.div}

Inspired by the involutions in \eqref{inv2}, we introduce the linearized divergences
$\zeta^{\pm}_j$ for $j=1,2,3$ as follows:
\begin{align} \label{zeta.def}
\zeta_j^{\pm}:=
\p_{i}^{\check{\varPhi}^{\pm}}\left(
\check{\rho}_{\pm} \dot{{F}}_{i j}^{\pm}
+\check{F}_{i j}^{\pm}\dot{\rho}^{\pm}\right),
\qquad
\end{align}
where the partial derivatives
$\p_{i}^{\check{\varPhi}^{\pm}}$ (with $i=1,2,3$) are defined in \eqref{differential}.
We now present the following estimates for $\zeta_1^{\pm},$ $\zeta_2^{\pm}$ and  $\zeta_3^{\pm}$.

\begin{lemma}[Divergence estimates]
 \label{lem.div}
If the hypotheses of {\rm Theorem} {\rm \ref{thm2}} hold, then there exists a constant $\gamma_{m}\geq 1$,
independent of $T$, such that
\begin{align}
\gamma \|(\zeta_1^{\pm}, \zeta_2^{\pm},  \zeta_3^{\pm}) \|_{H^{m-1}_{\gamma}(\Omega_T)}
\lesssim
\|(W,f) \|_{H^{m}_{\gamma}(\Omega_T)}
+\|(\check{V},\check{\varPsi})\|_{H^{m+2}_{\gamma}(\Omega_T)}
\|(W,f) \|_{L^{\infty}(\Omega_T)},
\label{zeta.est}
\end{align}
for all $\gamma\geq \gamma_{m}$ and
solutions  $(W,\psi)\in H^{m+2}_{\gamma}(\Omega_T)\times H^{m+2}_{\gamma}(\omega_T)$
to the problem \eqref{reformulation}.
\end{lemma}
\begin{proof}
The equations for $\dot{F}_{ij}$ in \eqref{effective.1} can be written as
\begin{align}
\label{F.dot.eq}&
(\p_t^{\check{\varPhi}}
+\check{v}_{\ell}\p_{\ell}^{\check{\varPhi}} )
\dot{F}_{ij}
- \check{{F}}_{\ell j}\p_{\ell}^{\check{\varPhi}} \dot{v}_i
=\check{\rm c}_0   f
+\underline{\check{\rm c}}_1 \dot{V}.
\end{align}
Using equations \eqref{rho.eq} and \eqref{F.dot.eq}, we apply the operator $\p_{i }^{\check{\varPhi}}$
and use
\begin{align*}
\check{\rho} \check{{F}}_{\ell 1}
\p_{i}^{\check{\varPhi} }\p_{\ell}^{\check{\varPhi} }\dot{v}_i
-\check{\rho} \check{{F}}_{i1}
\p_{i}^{\check{\varPhi} } \p_{\ell}^{\check{\varPhi} }\dot{v}_{\ell}
=\check{\rho} \check{{F}}_{i 1} \big[\p_{\ell}^{\check{\varPhi} },\p_{i}^{\check{\varPhi} }\big]\dot{v}_{\ell}
=\check{\rm c}_2 \nabla \dot{V}
\end{align*}
to obtain that
 \begin{align} \label{zeta.eq}
(\p_t^{\check{\varPhi}}
+\check{v}_{\ell}\p_{\ell}^{\check{\varPhi}} )
\zeta_j
=\check{\rm c}_1 \nabla f
+ \check{\rm c}_1 f+\check{\rm c}_2 \nabla W + \check{\rm c}_2 W.
\end{align}
Applying operator $e^{-\gamma t}\p^{\alpha}$ with $|\alpha|\leq m-1$
to \eqref{zeta.eq}  yields
\begin{align}
\notag & (\p_t^{\check{\varPhi}}
 +\check{v}_{\ell}\p_{\ell}^{\check{\varPhi}} )
 \big({e}^{-\gamma t} \p^{\alpha} \zeta_j \big)
+\gamma {e}^{-\gamma t} \p^{\alpha} \zeta_j\\
&\quad ={e}^{-\gamma t}\p^{\alpha}(\check{\rm c}_1 \nabla f
+ \check{\rm c}_1 f+\check{\rm c}_2 \nabla W + \check{\rm c}_2 W)
- e^{-\gamma t} [ \p^{\alpha} , \p_t^{\check{\varPhi}}
+\check{v}_{\ell}\p_{\ell}^{\check{\varPhi}} ] \zeta_j.
\notag
\end{align}
Multiplying the last identity by ${e}^{-\gamma t}\p^{\alpha} \zeta_j$
and integrating over $\Omega_T$, we have
\begin{align}
 \notag
\gamma \| \p^{\alpha} \zeta_j\|_{L^2_{\gamma}(\Omega_T) }
\lesssim \;&
\| \p^{\alpha}(\check{\rm c}_1 \nabla f + \check{\rm c}_1 f
+\check{\rm c}_2 \nabla W + \check{\rm c}_2 W)\|_{L^2_{\gamma}(\Omega_T) }\\
&
+\| [ \p^{\alpha} , \p_t^{\check{\varPhi}}
+\check{v}_{\ell}\p_{\ell}^{\check{\varPhi}} ] \zeta_j\|_{L^2_{\gamma}(\Omega_T) },
\label{zeta.e1}
\end{align}
for $\gamma\geq 1$ sufficiently large, where we have used the constraints \eqref{bas.2} and
\begin{equation*}
(\p_t^{\check{\varPhi}}
+\check{v}_{\ell}\p_{\ell}^{\check{\varPhi}} )
=\p_t+\check{v}_{1}\p_1+\check{v}_{2}\p_2\qquad
\textrm{if }\ x_3\geq 0.
\end{equation*}

Using Moser-type calculus inequality \eqref{Moser3}, we obtain
\begin{align}
\notag \| \p^{\alpha}  (\check{\rm c}_1 \nabla f
+ \check{\rm c}_1  f)  \|_{L_{\gamma}^{2}(\Omega_T)}
&\lesssim
\| (\check{\rm c}_1 \p^{\alpha} \nabla f,
\check{\rm c}_1 \p^{\alpha}  f )  \|_{L_{\gamma}^{2}(\Omega_T)}
+\| ([\p^{\alpha} ,\check{\rm c}_1 ]\nabla f,[\p^{\alpha} ,\check{\rm c}_1 ] f   )  \|_{L^{2}_{\gamma}(\Omega_T)} \\
& \lesssim
\|f \|_{H_{\gamma}^{|\alpha|+1}(\Omega_T) }
+ \|(\check{V},\check{\varPsi})\|_{H_{\gamma}^{|\alpha|+2}(\Omega_T) }
\|f\|_{L^{\infty} (\Omega_T)} .
\label{zeta.e1a}
\end{align}

Notice that $\zeta_j= \check{\rm c}_1 W+\check{\rm c}_1 \nabla W$,
we apply Moser-type calculus inequalities \eqref{Moser3}--\eqref{Moser4} to deduce that
\begin{align}
&\| \p^{\alpha}(\check{\rm c}_2 \nabla W + \check{\rm c}_2 W)\|_{L^2_{\gamma}(\Omega_T) }
+\| [ \p^{\alpha} , \p_t^{\check{\varPhi}}
+\check{v}_{\ell}\p_{\ell}^{\check{\varPhi}} ] \zeta_j\|_{L^2_{\gamma}(\Omega_T) }
\notag \\
\notag
&\quad \lesssim
\left\|(\check{\rm c}_2 \p^{\alpha}\nabla W , \check{\rm c}_2\p^{\alpha} W ,
[\p^{\alpha},\check{\rm c}_2 ] W,
[\p^{\alpha},\check{\rm c}_2 ] \nabla W,
[\p^{\alpha},\check{\rm c}_1 ] \nabla^2 W )
\right\|_{L^2_{\gamma}(\Omega_T) }\\
&\quad \lesssim
\|W \|_{H_{\gamma}^{|\alpha|+1}(\Omega_T) }
+ \|(\check{V},\check{\varPsi})\|_{H_{\gamma}^{|\alpha|+3}(\Omega_T) }
\|W\|_{L^{\infty} (\Omega_T)}.
\label{zeta.e1b}
\end{align}
Substituting \eqref{zeta.e1a} and \eqref{zeta.e1b} into \eqref{zeta.e1} yields the following estimate:
\begin{align}
\gamma^{m-|\alpha|} \| \p^{\alpha} \zeta_j\|_{L^2_{\gamma}(\Omega_T) }
\lesssim
\|(W ,f )\|_{H_{\gamma}^{m}(\Omega_T) }
+ \|(\check{V},\check{\varPsi})\|_{H_{\gamma}^{m+2}(\Omega_T) }
\|(W,f)\|_{L^{\infty} (\Omega_T)},
\notag
\end{align}
from which we obtain estimate \eqref{zeta.est} and complete the proof of the lemma.
\end{proof}

\subsection{Vorticity Estimates}\label{sec.vor}

The linearized vorticities $\xi^{\pm}_j$ for the velocities $\dot{v}^{\pm},$
and the linearized vorticities $\eta_j^{\pm}$ for the columns of the deformation gradient, are defined as follows:
\begin{align}
&\xi^{\pm}_1:=\p_2 ^{\check{\varPhi}^{\pm}}\dot{v}_3^{\pm}
-\p_3^{\check{\varPhi}^{\pm}}\dot{v}_2^{\pm},\label{xi.def}\\
&\xi^{\pm}_2:=\p_3 ^{\check{\varPhi}^{\pm}}\dot{v}_1^{\pm}
-\p_1^{\check{\varPhi}^{\pm}}\dot{v}_3^{\pm},\label{xi2.def}\\
&\xi^{\pm}_3:=\p_1 ^{\check{\varPhi}^{\pm}}\dot{v}_2^{\pm}
-\p_2^{\check{\varPhi}^{\pm}}\dot{v}_1^{\pm},\label{xi3.def}\\
&\eta_{1,j}^{\pm}:=\p_2 ^{\check{\varPhi}^{\pm}}\dot{F}_{3j}^{\pm}
-\p_3^{\check{\varPhi}^{\pm}}\dot{F}_{2j}^{\pm},\label{eta.def}\\
&\eta_{2,j}^{\pm}:=\p_3 ^{\check{\varPhi}^{\pm}}\dot{F}_{1j}^{\pm}
-\p_1^{\check{\varPhi}^{\pm}}\dot{F}_{3j}^{\pm},\label{eta2.def}\\
&\eta_{3,j}^{\pm}:=\p_1 ^{\check{\varPhi}^{\pm}}\dot{F}_{2j}^{\pm}
-\p_2^{\check{\varPhi}^{\pm}}\dot{F}_{1j}^{\pm},\label{eta3.def}
\end{align}
for $j=1,2,3$.
The estimates of $\xi^{\pm}_j$, $\eta_{k,j}^{\pm}$ for $k, j=1,2,3$ are provided by the following lemma.

\begin{lemma}[Vorticity estimates]
 \label{lem.vor}
If the hypotheses of {\rm Theorem {\rm \ref{thm2}}} hold, then there exists a constant $\gamma_{m}\geq 1$,
independent of $T$, such that
\begin{align}
\gamma \Big(\|\xi^{\pm}\|_{H^{m-1}_{\gamma}(\Omega_T)}+\|\eta^{\pm}\|_{H^{m-1}_{\gamma}(\Omega_T)}\Big)
\lesssim
\|(W,f) \|_{H^{m}_{\gamma}(\Omega_T)}
+\|(\check{V},\check{\varPsi})\|_{H^{m+2}_{\gamma}(\Omega_T)}
\|(W,f) \|_{L^{\infty}(\Omega_T)}
\label{xi.est}
\end{align}
for all $\gamma\geq \gamma_{m}$ and
solutions  $(W,\psi)\in H^{m+2}_{\gamma}(\Omega_T)\times H^{m+2}_{\gamma}(\omega_T)$
of problem \eqref{reformulation}, where 
\[
\|\xi^{\pm}\|_{H^{m-1}_{\gamma}(\Omega_T)}^2:=\sum_{k=1}^3\|\xi_k^{\pm}\|_{H^{m-1}_{\gamma}(\Omega_T)}^2,
\qquad
\|\eta^{\pm}\|_{H^{m-1}_{\gamma}(\Omega_T)}^2:=\sum_{k=1}^3\sum_{j=1}^3\|\eta_{k,j}^{\pm}\|_{H^{m-1}_{\gamma}(\Omega_T)}^2.
\]
\end{lemma}
\begin{proof}
The equations for $\dot{v}_1,$ $\dot{v}_2,$ and $\dot{v}_3$ in \eqref{effective.1} are given by
\begin{align}
\label{v.dot.eq}
(\p_t^{\check{\varPhi}}
+\check{v}_{\ell}\p_{\ell}^{\check{\varPhi}} )
\dot{v}_{i}
- \check{{F}}_{\ell j}\p_{\ell}^{\check{\varPhi}}  \dot{F}_{ij}
+ \frac{c(\check{\rho})^2}{\check{\rho}} \p_{i}^{\check{\varPhi}}  \dot{\rho}
=\check{\rm c}_0   f
+\underline{\check{\rm c}}_1 \dot{V}.
\end{align}
Taking the $\check{\varPhi}$-curl of \eqref{v.dot.eq}, we obtain, for $k=1,2,3$, the transport equation
\begin{align}
\label{xi.eq}
(\p_t^{\check{\varPhi}}
+\check{v}_{\ell}\p_{\ell}^{\check{\varPhi}} )
\xi_k -\sum_{j=1}^3\check{{F}}_{\ell j} \p_{\ell}^{\check{\varPhi}}   \eta_{k,j}
=\check{\rm c}_1 \nabla f
+ \check{\rm c}_1 f+\check{\rm c}_2 \nabla W + \check{\rm c}_2 W,
\end{align}
and similarly from \eqref{F.dot.eq}, we have, for $k,j=1,2,3$,
\begin{align}
\label{eta.eq}
(\p_t^{\check{\varPhi}}
+\check{v}_{\ell}\p_{\ell}^{\check{\varPhi}} )
\eta_{k,j} -\check{{F}}_{\ell j} \p_{\ell}^{\check{\varPhi}}   \xi_k
=\check{\rm c}_1 \nabla f
+ \check{\rm c}_1 f+\check{\rm c}_2 \nabla W + \check{\rm c}_2 W.
\end{align}
Next, apply the operator ${e}^{-\gamma t}\p^{\alpha}$ with $|\alpha|\leq m-1$
to \eqref{xi.eq} ({\it resp}.~\eqref{eta.eq}) multiply the resulting identity by ${e}^{-\gamma t}\p^{\alpha} \xi_j$ and ${e}^{-\gamma t}\p^{\alpha} \eta_{k,j}$ respectively,
and sum over $k,j$. This gives
\begin{align}
\notag
&\frac{1}{2} (\p_t^{\check{\varPhi}}
+\check{v}_{\ell}\p_{\ell}^{\check{\varPhi}} )
\Big\{\sum^3_{k=1}|{e}^{-\gamma t}\p^{\alpha}\xi_k|^2
+\sum^3_{k=1}\sum^3_{j=1}|{e}^{-\gamma t}\p^{\alpha}\eta_{k,j}|^2 \Big\}\\
\notag
&\quad
- \sum_{k=1}^3\sum_{j=1}^3 \check{{F}}_{\ell j} \p_{\ell}^{\check{\varPhi}}
\left( {e}^{-2\gamma t}\p^{\alpha}\xi_k \p^{\alpha}\eta_{k,j} \right)
+\gamma \Big\{ \sum^3_{k=1}|{e}^{-\gamma t}\p^{\alpha}\xi_k|^2
+\sum^3_{k=1}\sum^3_{j=1}|{e}^{-\gamma t}\p^{\alpha}\eta_{k,j}|^2 \Big\}
\\
\notag & =
\sum_{k=1}^3 {e}^{-2\gamma t} \p^{\alpha}\xi_k
\left\{\p^{\alpha}\left(\check{\rm c}_1 \nabla f
+ \check{\rm c}_1 f+\check{\rm c}_2 \nabla W + \check{\rm c}_2 W \right)
- [ \p^{\alpha} , \p_t^{\check{\varPhi}}
+\check{v}_{\ell}\p_{\ell}^{\check{\varPhi}} ] \xi_k\right\}\\
&\quad
+\sum_{k=1}^3\sum_{j=1}^3 {e}^{-2\gamma t} \p^{\alpha}\eta_{k,j}
\left\{\p^{\alpha}\left(\check{\rm c}_1 \nabla f
+ \check{\rm c}_1 f+\check{\rm c}_2 \nabla W + \check{\rm c}_2 W \right)
- [ \p^{\alpha} , \p_t^{\check{\varPhi}}
+\check{v}_{\ell}\p_{\ell}^{\check{\varPhi}} ] \eta_{k,j} \right\}
\notag \\
&\quad
+{e}^{-2\gamma t}
\sum_{k=1}^3\sum_{j=1}^3
\left\{
\p^{\alpha}\xi_k [ \p^{\alpha} ,  \check{{F}}_{\ell j} \p_{\ell}^{\check{\varPhi}}  ] \eta_{k,j}
+\p^{\alpha}\eta_{k,j} [ \p^{\alpha} ,  \check{{F}}_{\ell j} \p_{\ell}^{\check{\varPhi}}  ] \xi_k  \right\}.
\label{xi.id2}
\end{align}
It follows from the constraints in \eqref{bas.5} that
\begin{align*}
\check{{F}}_{\ell j} \p_{\ell}^{\check{\varPhi}}
=\check{{F}}_{1 j} \p_{1}+\check{{F}}_{2 j} \p_{2},  \quad 
x_3\geq 0.
\end{align*}
We now integrate the identity \eqref{xi.id2} over $\Omega_T$ and perform
a similar analysis as for $\zeta_j$ in Lemma \ref{lem.div} to
obtain the desired estimates \eqref{xi.est}.
The proof of the lemma is thus complete.
\end{proof}

\subsection{Proof of Theorem \ref{thm2}}\label{sec.proof1}
Thanks to Lemmas \ref{lem.div} and \ref{lem.vor},
we can derive the estimates for the normal derivative
of characteristic variables defined by \eqref{W.c}.
More precisely, in view of \eqref{W.c}, \eqref{xi.def} and \eqref{xi2.def}, and \eqref{differential}, we obtain
\begin{align*}
&\xi^{\pm}_1 =-\frac{1}{\p_3\check{\varPhi}^{\pm}}
\p_3\left({\langle \p_{\rm tan}\check{\varPhi}^{\pm} \rangle^2} W_1^{\pm} \right)
+\check{\rm c}_1\p_1 W+\check{\rm c}_1\p_2 W+\check{\rm c}_2 W, \\
&\xi^{\pm}_2 =-\frac{1}{\p_3\check{\varPhi}^{\pm}}
\p_3\left({\langle \p_{\rm tan}\check{\varPhi}^{\pm} \rangle^2} W_2^{\pm} \right)
+\check{\rm c}_1\p_1 W+\check{\rm c}_1\p_2 W+\check{\rm c}_2 W,
\end{align*}
which implies that, for $r = 1, 2$, 
\begin{align}
\p_3 W_{r}^{\pm}=\check{\rm c}_1 \xi^{\pm}_r
+\check{\rm c}_1 \p_1 W+\check{\rm c}_2 \p_2 W
+\check{\rm c}_2 W.
\label{W1.id}
\end{align}
Similarly, it follows from \eqref{zeta.def}, \eqref{eta.def}, and \eqref{eta2.def} that, for each $j=1,2,3$,
\begin{align}
\p_3 \dot{F}_{1j}^{\pm}&=\check{\rm c}_1 \eta_{2,j}^{\pm}
+\check{\rm c}_1 \p_1 W+\check{\rm c}_1 \p_2 W
+\check{\rm c}_2 W,
\label{F.dot.id.1}\\
\p_3 \dot{F}_{2j}^{\pm}&=\check{\rm c}_1 \eta_{1,j}^{\pm}
+\check{\rm c}_1 \p_1 W+\check{\rm c}_1 \p_2 W
+\check{\rm c}_2 W,
\label{F.dot.id.2}\\
\p_3 \dot{F}_{3j}^{\pm}&=\check{\rm c}_1 \zeta_j^{\pm}
+\check{\rm c}_1 \p_1 W+\check{\rm c}_1 \p_2 W
+\check{\rm c}_2 W.
\label{F.dot.id.3}
\end{align}
Using identities \eqref{W1.id}--\eqref{F.dot.id.3},
we apply Moser-type calculus inequalities \eqref{Moser1}--\eqref{Moser4} and
use \eqref{W.nc.est}, \eqref{zeta.est}, and \eqref{xi.est} to obtain that
\begin{equation}\label{normal.est}
\begin{split}
\vertiii{ \p_3^k W }_{L^2(H^{m-k}_{\gamma}(\omega_T) ) }
\lesssim
\;& \vertiii{ W }_{L^2(H^{m}_{\gamma}(\omega_T) ) }
 + \gamma^{-1}\|(W,f) \|_{H^{m}_{\gamma}(\Omega_T)} \\
&+\gamma^{-1}
\|(\check{V},\check{\varPsi})\|_{H^{m+2}_{\gamma}(\Omega_T)}
\|(W,f) \|_{L^{\infty}(\Omega_T)}
\end{split}
\end{equation}
holds for $k=1$.

Using identities \eqref{W.nc.id}, \eqref{W1.id}--\eqref{F.dot.id.3},
we can combine
estimates \eqref{zeta.def} and \eqref{xi.est} to prove
\eqref{normal.est} by finite induction in $k=1,\ldots,m$.
Since
\begin{align}
\notag
\|W\|_{H^{m}_{\gamma}(\Omega_T) }
\sim
\sum_{k=0}^m \vertiii{ \p_3^k W }_{L^2(H^{m-k}_{\gamma}(\omega_T) ) },
\end{align}
we utilize \eqref{tan.est} and \eqref{normal.est} to obtain
\begin{equation}\label{tame2}
\begin{split}
&
\gamma^{1/2}\|W\|_{H^{m}_{\gamma}(\Omega_T) }
+\|W^{\mathrm{nc}}|_{x_3=0}\|_{H^{m}_{\gamma}(\omega_T)}
+\|\psi\|_{H^{m+1}_{\gamma}(\omega_T)}\\[1mm]
&\quad  \lesssim
\gamma^{-{1}/{2}}\big\|f \big\|_{H^{m}_{\gamma}(\Omega_T)}
+\gamma^{-{3}/{2}}\vertiii{ f }_{L^2(H^{m+1}_{\gamma}(\omega_T))}
+ \gamma^{-1}\|g\|_{H^{m+1}_{\gamma}(\omega_T)}
  \\[1mm]
&\quad  \quad \,
+\gamma^{-1}\|(W,f)\|_{L^{\infty}(\Omega_T)}
\big\|\big(\check{V}, \check{\varPsi}\big)\big\|_{H^{m+3}_{\gamma}(\Omega_T)}
+\gamma^{-1}
\|(W^{\mathrm{nc}} ,\psi)\|_{L^{\infty}(\omega_T)}
\big\|\big(\check{V},\check{\varPsi}\big)\big\|_{{H^{m+2}_{\gamma}(\omega_T)}} ,
\end{split}
\end{equation}
for $\gamma $ sufficiently large.

Theorem {\rm\ref{thm3}} establishes the well-posedness of the effective linear problem \eqref{effective} for the source terms $(f^{\pm},g)\in L^2(H^1(\omega_T))\times H^1(\omega_T)$
vanishing in the past. Building on the results in  \cite{RM74MR0340832,CP82MR678605},
we can use the tame estimate \eqref{tame2} to reformulate Theorem \ref{thm3} as a well-posdness statement for \eqref{effective} in $H^{m}$. Specifically, as shown in Theorem \ref{thm2}, there exists a unique solution $(\dot{V}^{\pm},\psi)\in H^{m}(\Omega_T)\times H^{m+1}(\omega_T)$,
which vanishes in the past and satisfies \eqref{tame2} for all $\gamma\geq \gamma_{m}$.

The tame estimate \eqref{thm2.est} can be derived as follows. By the Sobolev embeddings
\[
\|W\|_{L^{\infty}(\Omega_T)} \lesssim \|W\|_{H^3(\Omega_T)},
\qquad
\|\psi\|_{W^{1,\infty}(\omega_T)}\lesssim \|\psi\|_{H^3(\omega_T)},
\]
and \eqref{tame2} with $m=2$, we obtain
\begin{equation} \label{e:363}
\|W\|_{L^{\infty}(\Omega_T)}
+\|\psi\|_{W^{1,\infty}(\omega_T)}
\leq C_{T,\gamma}\left(
\big\|f \big\|_{H^{3}_{\gamma}(\Omega_T)}
+ \|g\|_{H^{3}_{\gamma}(\omega_T)}
\right).
\end{equation}
Substituting \eqref{e:363} into \eqref{tame2}
gives the tame estimate \eqref{thm2.est}, thus completing the proof of Theorem \ref{thm2}.
\qed

\section{Compatibility Conditions and Approximate Solutions}\label{sec.compa}

To apply Theorem \ref{thm2}  in the general setting, we follow the approach in \cite{CS08MR2423311}, and transform the original nonlinear problem \eqref{Phi.eq}--\eqref{EVS} into a form with zero initial data. To achieve this, we introduce approximate solutions that incorporate the initial data into the interior equations. The construction of smooth approximate solutions imposes necessary compatibility conditions on the initial data.

\subsection{Compatibility Conditions}\label{sec.compa1}

Let $m\in \mathbb{N}$ with $m\geq 3$.
Assume that  the initial data $(U_0^{\pm},\varphi_0)$ satisfy
$\widetilde{U}_0^{\pm}:=U_0^{\pm}-\bar{U}^{\pm}\in H^{m+1/2}(\mathbb{R}^3_+)$
and  $\varphi_0\in H^{m+1}(\mathbb{R}^2)$, { and
 $(\widetilde{U}_0^{\pm},\varphi_0)$ has the following compact support,}
\begin{align} \label{CA1}
\supp\,\widetilde{U}_0^{\pm}\subset \{x_3\geq 0,\, x_1^2+x_2^2+x_3^2\leq 1\},
\qquad \supp\,\varphi_0\subset [-1,1]\times [-1,1].
\end{align}
Using the trace theorem,
we can construct
$\widetilde{\varPhi}_0^+=\widetilde{\varPhi}_0^-\in H^{m+3/2}(\mathbb{R}_+^3)$ satisfying
\begin{gather} \label{CA2}
\widetilde{\varPhi}_0^{\pm} |_{x_3=0}
=\varphi_0,\quad
\supp\, \widetilde{\varPhi}_0^{\pm}\subset \left\{x_3\geq 0,\, x_1^2+x_2^2+x_3^2\leq 2\right\},\\
 \label{CA3}
\big\|\widetilde{\varPhi}_0^{\pm}\big\|_{H^{m+3/2}(\mathbb{R}_+^3)}
\leq C\|\varphi_0\|_{H^{m+1}(\mathbb{R}^2)}.
\end{gather}
{Define $\varPhi_0^{\pm}:=\widetilde{\varPhi}_0^{\pm}+\bar{\varPhi}_0^{\pm}$,
which represents the initial data for the problem \eqref{Phi.eq},
\begin{align} \label{Phi.initial}
\varPhi^{\pm}|_{t=0}=\varPhi_0^{\pm}.
\end{align}
By \eqref{CA3} and the Sobolev embedding theorem, we have
\begin{align} \label{CA4}
\pm \p_3\varPhi_0^{\pm}\geq {7}/{8}\qquad\,\, \textrm{for all }x\in\mathbb{R}_+^3,
\end{align}
for sufficiently small $\varphi_0$   in $H^{m+1}(\mathbb{R}^2)$.
}

Let the perturbation be denoted by $(\widetilde{U}^{\pm},\widetilde{\varPhi}^{\pm})
:=(U^{\pm}-\bar{U}^{\pm},\varPhi^{\pm}-\bar{\varPhi}^{\pm})$, and define the traces of the ${k}$-th order time derivatives at $\{t=0\}$  as follows:
\begin{align}
\widetilde{U}^{\pm}_{({k})}:=\p_t^{{k}}\widetilde{U}^{\pm}\Big|_{t=0},\quad
\widetilde{\varPhi}^{\pm}_{({k})}:=\p_t^{{k}}\widetilde{\varPhi}^{\pm}\Big|_{t=0}
\qquad\textrm{for  }\ {k}\in\mathbb{N}.
\label{CA5}
\end{align}
Note that $\widetilde{U}^{\pm}_{(0)}=\widetilde{U}^{\pm}_{0}$ and $\widetilde{\varPhi}^{\pm}_{(0)}=\widetilde{\varPhi}^{\pm}_{0}$.

Let $\mathcal{W}^{\pm}
:=(\widetilde{U}^{\pm},\nabla_x\widetilde{U}^{\pm},\nabla_x\widetilde{\varPhi}^{\pm})^{\top}
\in\mathbb{R}^{55}$, then
  the first equation in \eqref{Phi.eq} and  the equation \eqref{EVS.a} can be written as
\begin{align}\label{tilde.U.Phi}
\p_t \widetilde{\varPhi}^{\pm}=\mathbf{G}_1(\mathcal{W}^{\pm}),\qquad
\p_t \widetilde{U}^{\pm}=\mathbf{G}_2(\mathcal{W}^{\pm}),
\end{align}
where $\mathbf{G}_1$ and $\mathbf{G}_2$ are two $C^{\infty}$ functions vanishing at the origin. Next, we apply  $\p^{k}_t$ to \eqref{tilde.U.Phi}, take the initial traces,
and use the generalized Fa\`a di Bruno's formula (see \cite[Theorem 2.1]{M00MR1781515})
to obtain
\begin{align}  \label{tilde.Phi.0}
&\widetilde{\varPhi}^{\pm}_{({k}+1)}
=\sum_{\alpha_{i}\in\mathbb{N}^{55},|\alpha_1|+\cdots+{k} |\alpha_{{k}}|={k}}
D^{\alpha_1+\cdots+\alpha_{k}}\mathbf{G}_1(\mathcal{W}^{\pm}_{(0)})
\prod_{i=1}^{k}\frac{{k}!}{\alpha_{i}!}
\left(\frac{\mathcal{W}_{(i)}^{\pm}}{i!}\right)^{\alpha_{i}},\\
&\widetilde{U}^{\pm}_{({k}+1)}
=
\sum_{\alpha_{i}\in\mathbb{N}^{55},|\alpha_1|+\cdots+{k} |\alpha_{{k}}|={k}}
D^{\alpha_1+\cdots+\alpha_{k}}\mathbf{G}_2(\mathcal{W}^{\pm}_{(0)})
\prod_{i=1}^{k}\frac{{k}!}{\alpha_{i}!}
\left(\frac{\mathcal{W}_{(i)}^{\pm}}{i!}\right)^{\alpha_{i}},\label{tilde.U.0}
\end{align}
where $\mathcal{W}_{(i)}^{\pm}$ denote the traces
$(\widetilde{U}_{(i)}^{\pm},\nabla_x\widetilde{U}_{(i)}^{\pm},\nabla_x\widetilde{\varPhi}_{(i)}^{\pm})$.
This leads to the following lemma
({\it cf.} \cite[Lemma 4.2.1]{M01MR1842775}).

\begin{lemma}\label{lem.CA1}
If \eqref{CA1}--\eqref{CA4} hold, then relations
\eqref{tilde.Phi.0} and \eqref{tilde.U.0} determine
$\widetilde{U}^{\pm}_{({k})}\in H^{m+1/2-{k}}(\mathbb{R}_+^3)$ for ${k}=1,\ldots,m$,
and $\widetilde{\varPhi}^{\pm}_{({k})}\in H^{m+3/2-{k}}(\mathbb{R}_+^3)$ for ${k}=1,\ldots,m+1$,
which satisfy
\begin{gather}
\notag
\supp\,\widetilde{U}_{({k})}^{\pm}\subset \{x_3\geq 0,\, x_1^2+x_2^2+x_3^2\leq 1\}, \quad
\supp\, \widetilde{\varPhi}_{({k})}^{\pm}\subset \{x_3\geq 0,\, x_1^2+x_2^2+x_3^2\leq 2\},\\
\notag
\sum_{{k}=0}^{m}\big\|\widetilde{U}^{\pm}_{({k})}\big\|_{H^{m+1/2-{k}}(\mathbb{R}_+^3)}
+\sum_{{k}=0}^{m+1}\big\|\widetilde{\varPhi}^{\pm}_{({k})}\big\|_{H^{m+3/2-{k}}(\mathbb{R}_+^3)}
\phantom{\qquad \qquad \qquad }\\
 \phantom{\qquad \qquad \qquad \qquad \qquad}
\leq C\Big(\big\|\widetilde{U}^{\pm}_0\big\|_{H^{m+1/2}(\mathbb{R}_+^3)}
+\|\varphi_0\|_{H^{m+1}(\mathbb{R}^2)} \Big),
\label{CA.est}
\end{gather}
for some constant $C>0$ depending solely upon
$\|(\widetilde{U}^{\pm}_{0},\widetilde{\varPhi}^{\pm}_{0})\|_{W^{1,\infty}(\mathbb{R}_+^3)}$
and $m$.

\end{lemma}

To guarantee the smoothness of the approximate solutions, the initial data must satisfy the following compatibility conditions.

\begin{definition}
\label{def.compa}
Let $m\in\mathbb{N}$ with $m\geq 3$.
Let $\widetilde{U}^{\pm}_0:=
U_0^{\pm}-\bar{U}_0^{\pm}\in H^{m+1/2}(\mathbb{R}_+^3)$
and $\varphi_0\in H^{m+1}(\mathbb{R}^2)$ satisfy \eqref{CA1}.
The initial data $U_0^{\pm}$ and $\varphi_0$ are said to
be compatible up to order $m$
if there exist functions $\widetilde{\varPhi}_0^{\pm}\in H^{m+3/2}(\mathbb{R}_+^3)$
satisfying \eqref{CA2}--\eqref{CA4} and
\begin{align}
\label{compa3}
F_{3j,0}^{\pm}=F_{1j,0} ^{\pm}\p_1\varPhi_0^{\pm}+F_{2j,0} ^{\pm}\p_2  \varPhi_0^{\pm}
\qquad \textrm{for } j=1,2,3
\end{align}
such that functions $\widetilde{U}^{\pm}_{(0)},\ldots,\widetilde{U}^{\pm}_{(m)},
\widetilde{\varPhi}^{\pm}_{(0)},\ldots,\widetilde{\varPhi}^{\pm}_{(m+1)}$
determined by \eqref{CA5} and \eqref{tilde.Phi.0}--\eqref{tilde.U.0} satisfy
\begin{subequations} \label{compa1}\nonumber
	\begin{alignat}{2}
	& \big(\widetilde{\varPhi}^{+}_{({k})}-\widetilde{\varPhi}^{-}_{({k})}\big)\big|_{x_3=0}=0
	\qquad  && \textrm{for }\  {k}=0,\ldots,m,\\
	&\big(\tilde{\rho}^{+}_{({k})}-\tilde{\rho}^{-}_{({k})}\big)\big|_{x_3=0}=0
	\qquad\,\, && \textrm{for } \  {k}=0,\ldots,m-1,
	\end{alignat}
\end{subequations}
and
\begin{subequations}\label{compa2}\nonumber
	\begin{alignat}{2}
	&\int_{\mathbb{R}_+^3}
	\big|\widetilde{\varPhi}^{+}_{(m+1)}-\widetilde{\varPhi}^{-}_{(m+1)}\big|^2 \dd x_1 \dd x_2\frac{\dd x_3}{x_3}
	<\infty,  \\
	&\int_{\mathbb{R}_+^3}
	\big| \tilde{\rho}^{+}_{(m)}-\tilde{\rho}^{-}_{(m)}\big|^2 \dd x_1 \dd x_2\frac{\dd x_3}{x_3}
	<\infty.
	\end{alignat}
\end{subequations}
\end{definition}

\subsection{Approximate Solutions}\label{sec.compa2}
 Following the approach in \cite{CS08MR2423311}, we now introduce approximate solutions
that satisfy the problem \eqref{Phi.eq}--\eqref{EVS}
in the sense of Taylor's expansions at $t=0$.

\begin{lemma} \label{lem.app}
Let $m\in\mathbb{N}$ with $m\geq 3$.
Assume that
$\widetilde{U}^{\pm}_0
:=U_0^{\pm}-\bar{U}_0^{\pm}\in H^{m+1/2}(\mathbb{R}_+^3)$
and $\varphi_0\in H^{m+1}(\mathbb{R}^2)$ satisfy \eqref{CA1},
and that initial data $U_0^{\pm}$ and $\varphi_0$ are compatible up to order $m$.
If $\widetilde{U}^{\pm}_0$ and $\varphi_0$ are sufficiently small,
then there exist functions $U^{a\pm}$, $\varPhi^{a\pm}$, and $\varphi^a$ such that
$\widetilde{U}^{a\pm}:=U^{a\pm}-\bar{U}^{\pm}\in H^{m}(\Omega)$,
$\widetilde{\varPhi}^{a\pm}:=\varPhi^{a\pm}-\bar{\varPhi}^{\pm}\in H^{m+2}(\Omega)$,
$\varphi^a\in H^{m+3/2}(\p\Omega)$, and
\begin{subequations} \label{app}
\begin{alignat}{2}
\label{app.eq.1}&\p_t^j\mathbb{L}(U^{a\pm},\varPhi^{a\pm})|_{t=0}={\mathbf 0}
\qquad &&\textrm{for }\ j=0,\ldots,m-2,\\
\label{app.eq.2}&\p_t\varPhi^{a\pm}+v_1^{a\pm}\p_1\varPhi^{a\pm}+v_2^{a\pm}\p_2\varPhi^{a\pm}-v_3^{a\pm}=0
\qquad &&\textrm{in }\ \Omega,\\
\label{app.eq.3}&\pm\p_3\varPhi^{a\pm}\geq {3}/{4}
\qquad &&\textrm{in }\ \Omega,\\
\label{app.eq.4}&\varPhi^{a+}=\varPhi^{a-}=\varphi^a
\qquad &&\textrm{on }\ \p\Omega,\\
\label{app.eq.5}&\mathbb{B}(U^{a+},U^{a-},\varphi^a)={\mathbf 0}
\qquad &&\textrm{on }\ \p\Omega,\\
\label{app.eq.6}&
{F}_{3j}^{a\pm}={F}_{1j}^{a\pm} \p_1\varPhi^{a\pm}+{F}_{2j}^{a\pm} \p_2\varPhi^{a\pm}
\qquad &&\textrm{on }\  \bar{\Omega},\quad \textrm{for }\   j=1,2,3.
\end{alignat}
\end{subequations}
Moreover, we have
\begin{align}
\label{app2}&\supp \big(\widetilde{U}^{a\pm},\widetilde{\varPhi}^{a\pm} \big)\subset \left\{t\in[-T,T],\,x_3\geq 0,\,x_1^2+x_2^2+x_3^2\leq 3 \right\},\\
&\big\|\widetilde{U}^{a\pm}\big\|_{H^{m}(\Omega)}
+\big\|\widetilde{\varPhi}^{a\pm}\big\|_{H^{m+2}(\Omega)}+\|\varphi^a\|_{H^{m+3/2}(\p\Omega)}
\notag \\
&\qquad\qquad \qquad\quad \
 \leq\varepsilon_0\Big(\big\|\widetilde{U}^{\pm}_0\big\|_{H^{m+1/2}(\mathbb{R}_+^3)}
+\|\varphi_0\|_{H^{m+1}(\mathbb{R}^2)}\Big), \label{app3}
\end{align}
where we write $\varepsilon_0(\cdot)$ as a generic function that tends to zero as its argument tends to zero.
\end{lemma}
\begin{proof}
The proof is divided into four steps.

\vspace*{2mm}
\noindent {\it Step 1}.\ \
First we consider $\tilde{\rho}^{a-}, \tilde{v}_1^{a\pm},\tilde{v}_2^{a\pm}\in H^{m+1}(\Omega)$ and
$\widetilde{\varPhi}^{a-}\in  H^{m+2}(\Omega)$ such that the following conditions are satisfied:
\begin{alignat*}{3}
\big(\p_t^{{k}}\tilde{\rho}^{a-},\, \p_t^{{k}}\tilde{v}_1^{a\pm},\p_t^{{k}}\tilde{v}_2^{a\pm} \big)\big|_{t=0}
=\;&\big(\tilde{\rho}_{({k})}^-,\tilde{v}_{1({k}),}^{\pm}, \tilde{v}_{2({k})}^{\pm}  \big),
&&\qquad
\textrm{for } {k}=0,\ldots,m,\\
 \p_t^{{k}}\widetilde{\varPhi}^{a-}\big|_{t=0}
=\;& \widetilde{\varPhi}^-_{({k})},
&&\qquad
\textrm{for } {k}=0,\ldots,m+1,
\end{alignat*}
where $\tilde{\rho}_{({k})}^-$, $\tilde{v}_{1({k})}^{\pm}$, $\tilde{v}_{2({k})}^{\pm}$ and
$\widetilde{\varPhi}^-_{({k})}$ are constructed in Lemma \ref{lem.CA1}.
Utilizing the compatibility conditions \eqref{compa1}--\eqref{compa2},
we apply the lifting result from \cite[Theorem 2.3]{LM72MR0350178}
to select $\tilde{\rho}^{a+}\in H^{m+1}(\Omega)$
and $\widetilde{\varPhi}^{a+}\in  H^{m+2}(\Omega)$ such that
\begin{alignat*}{3}
 \p_t^{{k}} \tilde{\rho}^{a+} \big|_{t=0}
=\;& \tilde{\rho}_{({k})}^+,
&&\qquad
\textrm{for } {k}=0,\ldots,m,\\
\p_t^{{k}}\widetilde{\varPhi}^{a+}\big|_{t=0}
=\;& \widetilde{\varPhi}^+_{({k})},
&&\qquad
\textrm{for } {k}=0,\ldots,m+1,
\end{alignat*}
and
\begin{align*}
[\tilde{\rho}^{a}]=0,\qquad
[ \widetilde{\varPhi}^{a}]=0\qquad
\textrm{on }\ \p\Omega.
\end{align*}
Furthermore, $\tilde{\rho}^{a\pm},$ $\tilde{v}_1^{a\pm}$,$\tilde{v}_2^{a\pm}$ and
$\widetilde{\varPhi}^{a\pm}$ can be chosen to satisfy \eqref{app2},
 since $(\widetilde{U}^{\pm}_{({k})}, \widetilde{\varPhi}^{\pm}_{({k})} )$
have a compact support.

\vspace*{2mm}
\noindent {\it Step 2}.\ \
Now, we define
\begin{alignat*}{3}
& \varphi^{a}
=  \widetilde{\varPhi}^{a+}\big|_{x_3=0}
=\widetilde{\varPhi}^{a-}\big|_{x_3=0}
\in H^{m+3/2}(\p\Omega),\\
&\tilde{v}_3^{a\pm}
=  \p_t \widetilde{\varPhi}^{a\pm}
+(\tilde{v}_1^{a\pm} \pm \bar{v}) \p_1 \widetilde{\varPhi}^{a\pm}+ (\tilde{v}_2^{a\pm}) \p_2 \widetilde{\varPhi}^{a\pm}
\in H^{m+1}(\Omega).
\end{alignat*}
Thus, we deduce that the
functions $\tilde{v}_3^{a\pm}$ satisfy \eqref{app2},
and \eqref{app.eq.2}, \eqref{app.eq.4}, and \eqref{app.eq.5} hold.

\vspace*{2mm}
\noindent {\it Step 3}.\ \
Since $\tilde{v}^{a\pm}\in H^{m+1}(\Omega)$
and $\widetilde{\varPhi}^{a\pm}\in H^{m+2}(\Omega)$
are already specified, we take $\widetilde{F}_{ij}^{a\pm}\in H^m(\Omega)$, for $i,j=1,2,3$,
as the unique solution of the transport equation
\begin{align}\label{Fa.eq}
\big(\p_t^{\varPhi^{a\pm}}+v_{\ell}^{a\pm} \p_{\ell}^{\varPhi^{a\pm}} \big)
\widetilde{F}_{ij}^{a\pm}
-{F}_{\ell j}^{a\pm} \p_{\ell}^{\varPhi^{a\pm}}  v^{a\pm}_i=0
\qquad
\textrm{on }\ \bar{\Omega},
\end{align}
supplemented with the initial data:
\begin{align} \label{Fij.initial}
\widetilde{F}_{ij}^{a\pm}\big|_{t=0}
=\widetilde{F}_{ij(0)}^{\pm}\in H^{m+1/2}(\mathbb{R}_+^3).
\end{align}
From equations \eqref{compa3} and \eqref{Fij.initial}, it follows that the
constraints \eqref{app.eq.6} are satisfied at the initial time.
Consequently, as in the proof of Proposition \ref{pro1.1},
we deduce that \eqref{app.eq.6} holds for all $t\in\mathbb{R}$.

\vspace*{2mm}
\noindent {\it Step 4}.\ \
Equations \eqref{tilde.Phi.0}--\eqref{tilde.U.0} imply \eqref{app.eq.1}.
The estimate \eqref{app3} is derived from \eqref{CA.est} and the continuity of the lifting operator.
Using \eqref{app3} and the Sobolev embedding theorem,
we deduce \eqref{app.eq.3}, provided that the initial perturbations are sufficiently small.
This completes the proof.
\end{proof}

We define $U^a:=(U^{a+},U^{a-})^{\top}$ and
$\varPhi^a:=(\varPhi^{a+},\varPhi^{a-})^{\top}$ for simplicity.  The vector $(U^a,\varPhi^a)$ constructed in Lemma \ref{lem.app} serves as an {\it approximate solution} to   \eqref{Phi.eq}--\eqref{EVS}.
From \eqref{app.eq.4} and \eqref{app2}, it is clear that
$\varphi^a$ is supported within the region $\{-T\le t\le T,\,x_1^2+x_2^2+ x_3^2 \leq 3\}$.
Applying \eqref{app3} and the Sobolev embedding theorem yields the following estimate:
\begin{align} \notag
\big\|\widetilde{U}^{a\pm} \big\|_{W^{2,\infty}(\Omega)}
+\big\| \widetilde{\varPhi}^{a\pm}\big\|_{W^{3,\infty}(\Omega)}
\leq\varepsilon_0\left(
\big\|\tilde{U}^{\pm}_0\big\|_{H^{m+1/2}(\mathbb{R}_+^3)}+\|\varphi_0\|_{H^{m+1}(\mathbb{R}^2)}
\right)
\end{align}
for any integer $m\geq 4$.

Next, we rewrite the system \eqref{Phi.eq}--\eqref{EVS}
as a problem with zero initial data.
Define the function $f^{a}$ as follows:
$f^{a}=-\mathbb{L}(U^{a},\varPhi^{a})$ for $t>0$, and
$f^{a}={\mathbf 0}$ for $t<0$.
 Thus, $f^{a}\in H^{m-1}(\Omega)$ and
$\supp f^{a}\subset \left\{0\le t\le T,\,x_3\geq 0,\,x_1^2+x_2^2\leq 3 \right\}$, as implied by \eqref{app.eq.1}, \eqref{app2}, and
 $(\tilde{U}^{a\pm},\nabla\widetilde{\varPhi}^{a\pm})\in  H^{m}(\Omega)$.
 Using Moser-type calculus inequalities and \eqref{app3}, we obtain:
\begin{align}\label{f.a.est}
\|f^{a}\|_{ H^{m-1}(\Omega)}
\leq \varepsilon_0\left(
\big\|\tilde{U}^{\pm}_0\big\|_{H^{m+1/2}(\mathbb{R}_+^3)}
+\|\varphi_0\|_{H^{m+1}(\mathbb{R}^2)}
\right).
\end{align}

Finally, based on \eqref{app}, the solution to
 the original problem \eqref{Phi.eq}--\eqref{EVS}  on $[0,T]\times \mathbb{R}_+^3$ is expressed as $(U,\varPhi)=(U^a,\varPhi^a)+(V,\varPsi)$, where $V=(V^+,V^-)^{\top}$ and
$\varPsi=(\varPsi^+,\varPsi^-)^{\top}$ solve the following problem:
\begin{align} \label{P.new}
\left\{\begin{aligned}
&\mathcal{L}(V,\varPsi):=\mathbb{L}(U^a+V,\varPhi^a+\varPsi)-\mathbb{L}(U^a,\varPhi^a)=f^a \quad&&\textrm{in }\Omega_T,\\
&\mathcal{E}(V,\varPsi):=\p_t\varPsi+(v_1^a+v_1)\p_1\varPsi+v_1\p_1\varPhi^a+(v_2^a+v_2)\p_2\varPsi+v_2\p_2\varPhi^a-v_3=0 \quad&&\textrm{in }\Omega_T,\\
&\mathcal{B}(V,\psi):=\mathbb{B}(U^a+V,\varphi^a+\psi)={\mathbf 0}, \quad \varPsi^+=\varPsi^-=\psi \quad&&\textrm{on }\omega_T,\\
&(V,\varPsi)={\mathbf 0}\quad&& \textrm{for }t< 0.
\end{aligned}\right.
\end{align}

Thus, solving the problem \eqref{P.new} on $[0,T]\times \mathbb{R}_+^3$ completes the problem.

\section{Nash--Moser Iteration}\label{sec.Nash}

In this section, we analyze the problem \eqref{P.new} through a suitable modification of the Nash--Moser iteration scheme. First, we outline the iterative scheme for problem \eqref{P.new} and present the corresponding inductive hypothesis. We then complete the proof of Theorem \ref{thm}
by demonstrating that the inductive hypothesis holds for all integers. It is worth noting that this section follows closely the standard procedure outlined in \cite{CS08MR2423311,RChen2020}, also see \cite{AG07MR2304160,CSW19MR3925528,S16MR3524197}.

\subsection{Iterative Scheme}\label{sec.scheme}

We start by recalling the following result from \cite[Proposition 4]{CS08MR2423311}.

\begin{proposition}\ \label{pro.smooth}
 Let $T>0$, $\gamma\geq 1$, and $m\in\mathbb{N}$ with $m\geq 4$.
 Then there exists a family  of smoothing operators $\{\mathcal{S}_{\theta}\}_{\theta\geq 1}$ such that
 \begin{align*}
 \mathcal{S}_{\theta}:\ \mathcal{F}_{\gamma}^3(\Omega_T)\times\mathcal{F}_{\gamma}^3(\Omega_T)
 \longrightarrow \bigcap_{s\geq 3}\mathcal{F}_{\gamma}^{s}(\Omega_T)\times\mathcal{F}_{\gamma}^{s}(\Omega_T),
 \end{align*}
 where
  $\mathcal{F}_{\gamma}^{s}(\Omega_T):=
 \big\{u\in H^{s}_{\gamma}(\Omega_T):u=0\textrm{ if }t<0\big\}
 $ for $s\geq 0$. These operators satisfy the following estimates{\rm:}
 \begin{subequations}\label{smooth.p1}
  \begin{alignat}{2}
  \label{smooth.p1a}&\|\mathcal{S}_{\theta} u\|_{H^{k}_{\gamma}(\Omega_T)}
  \lesssim \theta^{(k-j)_+}\|u\|_{H^{j}_{\gamma}(\Omega_T)}
  &&\quad\textrm{for }j,k=1,\ldots,m,\\[1.5mm]
  \label{smooth.p1b}&\|\mathcal{S}_{\theta} u-u\|_{H^{k}_{\gamma}(\Omega_T)}
  \lesssim \theta^{k-j}\|u\|_{H^{j}_{\gamma}(\Omega_T)}
  &&\quad\textrm{for }1\leq k\leq j\leq m,\\
  \label{smooth.p1c}&\left\|\frac{\dd}{\dd \theta}\mathcal{S}_{\theta} u\right\|_{H^{k}_{\gamma}(\Omega_T)}
  \lesssim \theta^{k-j-1}\|u\|_{H^{j}_{\gamma}(\Omega_T)}
  &&\quad\textrm{for }j,k=1,\ldots,m,
  \end{alignat}
 \end{subequations}
 and
 \begin{align}
 \|\mathcal{S}_{\theta}u-\mathcal{S}_{\theta}w \|_{H^{k}_{\gamma}(\omega_T)}
 \label{smooth.p2}
 \lesssim \theta^{(k+1-j)_+}\|u-w\|_{H^{j}_{\gamma}(\omega_T)}
\quad\textrm{for }j,k=1,\ldots,m,
 \end{align}
where $j$ and $k$ are integers,
and $(k-j)_+:=\max\{0,k-j\}$.
In particular, if $u=w$ on $\omega_T$,
then $\mathcal{S}_{\theta}u=\mathcal{S}_{\theta}w$ on $\omega_T$.
 Furthermore, smoothing operators can also be constructed for functions defined on $\omega_T$ {\rm(}denoted by $\mathcal{S}_{\theta}$ for simplicity{\rm)}, which satisfy the inequalities \eqref{smooth.p1},
with norms $\|\cdot\|_{H^{\ell}_{\gamma}(\omega_T)}$.
\end{proposition}

The following lemma establishes a lifting operator that will be used in constructing the iterative scheme and the modified state
(see \cite[Chapter 5]{Francheteau2000} and \cite{CS08MR2423311} for the proof).

\begin{lemma}\  \label{lem.smooth2}
Let $T>0$, $\gamma\geq 1$, and $m\in\mathbb{N}_+$.
Then there exists a continuous operator $\mathcal{R}_T$ mapping
$\mathcal{F}_{\gamma}^s(\omega_T)$ to $\mathcal{F}_{\gamma}^{s+1/2}(\Omega_T)$
satisfying $(\mathcal{R}_T u)|_{x_3=0}=u$
when $u\in  \mathcal{F}_{\gamma}^s(\omega_T)$
for all $s\in [1,m]$.
\end{lemma}

Following \cite{CS08MR2423311,CSW19MR3925528}, we now describe the iteration scheme for problem \eqref{P.new}.
Let ${N}\ge 1$ be a given integer.

We begin by setting $ (V_0,\varPsi_0, \psi_0)=\mathbf{0}$ and assume that
$(V_{n},\varPsi_{n},\psi_{n})$ is given and satisfies
\begin{align}\label{scheme.H1}
(V_{n},\varPsi_{n},\psi_{n})\big|_{t<0}={\mathbf 0},\quad \varPsi_{n}^{+}\big|_{x_3=0}=\varPsi_{n}^{-}\big|_{x_3=0}=\psi_{n}
\quad \textit{for } n=0,\ldots,{N}.
\end{align}
Next, we define the iteration scheme:
\begin{align}\label{scheme}
V_{{N}+1}=V_{{N}}+\delta V_{{N}},\quad \varPsi_{{N}+1}=\varPsi_{{N}}+\delta \varPsi_{{N}},
\quad\,\, \psi_{{N}+1}=\psi_{{N}}+\delta \psi_{{N}},
\end{align}
where the increments
$\delta V_{{N}}$, $\delta \varPsi_{{N}}$, and $\delta \psi_{{N}}$
are determined by the problem
\begin{align} \label{effective.NM}
\left\{\begin{aligned}
&\mathbb{L}_e'(U^a+V_{{N}+1/2},\varPhi^a+\varPsi_{{N}+1/2})\delta \dot{V}_{{N}}=f_{{N}}
\qquad &&\textrm{in }\Omega_T,\\
& \mathbb{B}_e'(U^a+V_{{N}+1/2},\varPhi^a+\varPsi_{{N}+1/2})(\delta \dot{V}_{{N}},\delta\psi_{{N}})=g_{{N}}
\qquad &&\textrm{on }\omega_T,\\
& (\delta \dot{V}_{{N}},\delta\psi_{{N}})={\mathbf 0}\qquad &&\textrm{for }t<0.
\end{aligned}\right.
\end{align}
Here, the operators $\mathbb{L}_e'$ and $\mathbb{B}_e'$
are defined in \eqref{effective.1} and \eqref{effective.2}, respectively. The pair
$(V_{{N}+1/2},\varPsi_{{N}+1/2})$ represents a modified state
such that $(U^a+V_{{N}+1/2},\varPhi^a+\varPsi_{{N}+1/2})$ satisfies
constraints \eqref{bas.c2}--\eqref{bas}.
The source term $(f_{{N}},g_{{N}})$ will be determined
later on. For a detailed construction of the modified state,
see Section \ref{sec.modified}.
Following \eqref{good}, we write  
\begin{align} \label{good.NM}
\delta \dot{V}_{{N}}:=\delta V_{{N}}-\frac{\p_3 (U^a+V_{{N}+1/2})}{\p_3 (\varPhi^a+\varPsi_{{N}+1/2})}\delta\varPsi_{{N}}.
\end{align}

Then, we set $f_0:=\mathcal{S}_{\theta_0}f^a$ and $(e_0,\tilde{e}_0,g_0):={\mathbf 0}$
for sufficiently large $\theta_0\geq 1$.
Let $(f_{n},g_{n},e_{n},\tilde{e}_{n})$ be given
and vanish in the past for $n=0,\ldots,{N}-1$.
The terms $f_{{N}}$ and $g_{{N}}$ are determined by the equations
\begin{align} \label{fn.def}
\sum_{n=0}^{{N}} f_{n}+\mathcal{S}_{\theta_{{N}}}E_{{N}}=\mathcal{S}_{\theta_{{N}}}f^a,
\qquad \sum_{n=0}^{{N}}g_{n}+\mathcal{S}_{\theta_{{N}}}\widetilde{E}_{{N}}={\mathbf 0},
\end{align}
where
\begin{align}  \label{En.def}
E_{{N}}:=\sum_{n=0}^{{N}-1}e_{n}\in\mathbb{R}^{26},
\qquad \widetilde{E}_{{N}}:=\sum_{n=0}^{{N}-1}\tilde{e}_{n}\in\mathbb{R}^{3}.
\end{align}
\noindent By convention, $E_0=\widetilde{E}_0=\mathbf{0}$. Here, $\mathcal{S}_{\theta_{{N}}}$ are the smoothing operators
from Proposition \ref{pro.smooth},
with $\{\theta_{{N}}\}$ defined as
\begin{align} \label{theta.def}
\theta_0\geq 1,\qquad \theta_{{N}}=\sqrt{\theta^2_0+{N}}.
\end{align}
As a consequence, we can apply Theorem \ref{thm2} to
solve $(\delta \dot{V}_{{N}},\delta \psi_{{N}})$ for problem \eqref{effective.NM}.

Noticing \eqref{good.NM},
 we need to construct
functions $\delta\varPsi_{{N}}^{+}$
and $\delta\varPsi_{{N}}^{-}$
such that $\delta\varPsi_{{N}}^{\pm}\big|_{x_3=0}=\delta\psi_{{N}}$.
From the boundary conditions in \eqref{effective.NM}
({\it cf.}\;\eqref{b.ring}, \eqref{B.ring}, and \eqref{b.natural}),
we obtain that $\delta\psi_{{N}}$ satisfies
\begin{align}
\notag
&\p_t (\delta\psi_{{N}}) +U_{{N}+1/2,2}^{+} \p_1(\delta\psi_{{N}})+U_{{N}+1/2,3}^{+} \p_2(\delta\psi_{{N}})\\[1mm]
\notag
&\quad+ \left(
\p_1 \varPhi_{{N}+1/2}^+\frac{\p_3U_{{N}+1/2,2}^{+} }{\p_3 \varPhi_{{N}+1/2}^+}+\p_2 \varPhi_{{N}+1/2}^+\frac{\p_3U_{{N}+1/2,3}^{+} }{\p_3 \varPhi_{{N}+1/2}^+}
-\frac{\p_3U_{{N}+1/2,4}^{+} }{\p_3 \varPhi_{{N}+1/2}^+}
\right)  \delta\psi_{{N}}
\\[2mm]
\notag
&\qquad \qquad \qquad \qquad
+ \p_1 \varPhi_{{N}+1/2}^+ \delta \dot{V}_{{N},2}^+ + \p_2 \varPhi_{{N}+1/2}^+ \delta \dot{V}_{{N},3}^+ -   \delta \dot{V}_{{N},4}^+
=g_{{N},2}\qquad\qquad\quad   \textrm{on }\ \omega_T, \\[1mm]
\notag
&\p_t (\delta\psi_{{N}}) +U_{{N}+1/2,2}^{-} \p_1(\delta\psi_{{N}})+U_{{N}+1/2,3}^{-} \p_2(\delta\psi_{{N}})\\[1mm]
\notag
&\quad+ \left(
\p_1 \varPhi_{{N}+1/2}^-\frac{\p_3U_{{N}+1/2,2}^{-} }{\p_3 \varPhi_{{N}+1/2}^-}+\p_2 \varPhi_{{N}+1/2}^-\frac{\p_3U_{{N}+1/2,3}^{-} }{\p_3 \varPhi_{{N}+1/2}^-}
-\frac{\p_3U_{{N}+1/2,4}^{-} }{\p_3 \varPhi_{{N}+1/2}^-}
\right)  \delta\psi_{{N}}
\\[2mm]
\notag
&\qquad \qquad \qquad \qquad
+ \p_1 \varPhi_{{N}+1/2}^- \delta \dot{V}_{{N},2}^- +\p_2 \varPhi_{{N}+1/2}^- \delta \dot{V}_{{N},3}^-  -   \delta \dot{V}_{{N},4}^-
=g_{{N},2}-g_{{N},1}\qquad \textrm{on }\ \omega_T,
\end{align}
where we {define} $$U_{{N}+1/2}^{\pm}:=U^{a\pm}+V^{\pm}_{{N}+1/2}, \quad
\varPhi_{{N}+1/2}^{\pm}:=\varPhi^{a\pm}+\varPsi^{\pm}_{{N}+1/2}$$
for simplifying the presentation.
In accordance with the identities above, we take
$\delta\varPsi_{{N}}^+$ and $\delta\varPsi_{{N}}^-$ as the solutions to the transport equations
\begin{align}
\notag
&\p_t (\delta\varPsi_{{N}}^+) +U_{{N}+1/2,2}^{+} \p_1(\delta\varPsi_{{N}}^+)+U_{{N}+1/2,3}^{+} \p_2(\delta\varPsi_{{N}}^+)\\[1mm]
&+ \left(
\p_1 \varPhi_{{N}+1/2}^+\frac{\p_3U_{{N}+1/2,2}^{+} }{\p_3 \varPhi_{{N}+1/2}^+}+\p_2 \varPhi_{{N}+1/2}^+\frac{\p_3U_{{N}+1/2,3}^{+} }{\p_3 \varPhi_{{N}+1/2}^+}
-\frac{\p_3U_{{N}+1/2,4}^{+} }{\p_3 \varPhi_{{N}+1/2}^+}
\right)  \delta\varPsi_{{N}}^+
\nonumber\\[2mm]
\label{delta.Psi1}
&\qquad \qquad \qquad \qquad
+ \p_1 \varPhi_{{N}+1/2}^+ \delta \dot{V}_{{N},2}^+ +\p_2 \varPhi_{{N}+1/2}^+ \delta \dot{V}_{{N},3}^+  -   \delta \dot{V}_{{N},4}^+
=\mathcal{R}_T g_{{N},2} +h_{N}^+,  \\[2mm]
\notag
&\p_t (\delta\varPsi_{{N}}^-) +U_{{N}+1/2,2}^{-} \p_1(\delta\varPsi_{{N}}^-)+U_{{N}+1/2,3}^{-} \p_2(\delta\varPsi_{{N}}^-)\nonumber\\
&+ \left(
\p_1 \varPhi_{{N}+1/2}^-\frac{\p_3U_{{N}+1/2,2}^{-} }{\p_3 \varPhi_{{N}+1/2}^-}+\p_2 \varPhi_{{N}+1/2}^-\frac{\p_3U_{{N}+1/2,3}^{-} }{\p_3 \varPhi_{{N}+1/2}^-}
-\frac{\p_3U_{{N}+1/2,4}^{-} }{\p_3 \varPhi_{{N}+1/2}^-}
\right)  \delta\varPsi_{{N}}^-
\nonumber\\[2mm]
\label{delta.Psi2}
&\qquad \qquad \qquad \qquad
+ \p_1 \varPhi_{{N}+1/2}^- \delta \dot{V}_{{N},2}^- + \p_2 \varPhi_{{N}+1/2}^- \delta \dot{V}_{{N},3}^- -   \delta \dot{V}_{{N},4}^-
=\mathcal{R}_T(g_{{N},2}-g_{{N},1})+h_{N}^-.
\end{align}
Here, $\mathcal{R}_T$ is the lifting operator from Lemma \ref{lem.smooth2}, and the
source terms $h_{{N}}^{\pm}$ will be chosen via
a decomposition of the operator $\mathcal{E}$, as defined in \eqref{P.new}.

Finally, we set $(h_0^+,h_0^-,\hat{e}_0)={\mathbf 0}$,
and assume that $(h_{n}^+,h_{n}^-,\hat{e}_{n})$ are given
and vanish in the past for $n=0,\ldots,{N}-1$.
Under these conditions, 
we determine $h_{{N}}^{+}$ and $h_{{N}}^{-}$ using the equations
\begin{subequations} \label{hn.def}
 \begin{alignat}{1}
 &\mathcal{S}_{\theta_{{N}}}
 \left(\widehat{E}_{{N}}^+-\mathcal{R}_T\widetilde{E}_{{N},2}\right)
 +\sum_{n=0}^{{N}}h_{n}^+=0,\\
 &\mathcal{S}_{\theta_{{N}}}
 \left(\widehat{E}_{{N}}^--\mathcal{R}_T\widetilde{E}_{{N},2}
 +\mathcal{R}_T\widetilde{E}_{{N},1}\right)
 +\sum_{n=0}^{{N}}h_{n}^-=0,
 \end{alignat}
\end{subequations}
where 
\begin{align} \label{En.hat.def}
\widehat{E}_{{N}}=(\widehat{E}_{{N}}^+,\widehat{E}_{{N}}^-)^{\top}=\sum_{n=0}^{{N}-1}\hat{e}_{n}\in\mathbb{R}^2,
\end{align}
\noindent By convention, $\widehat{E}_0=\mathbf{0}$.
and   $h_{{N}}^{\pm}=0$ for $t<0$.
As in \cite{Francheteau2000},
we can show that the traces of $h_{{N}}^{\pm}$ on $\omega_T$ vanish.
Consequently, it follows that  $\delta\varPsi_{{N}}^{\pm}=0$, for $t<0$ and
 $\delta\varPsi_{{N}}^{\pm}|_{x_3=0}=\delta\psi_{{N}}$. These are the unique smooth solutions satisfying transport equations
\eqref{delta.Psi1}--\eqref{delta.Psi2}.
Hence, $\delta V_{{N}}$ can be derived from \eqref{good.NM}
and $(V_{{N}+1},\varPsi_{{N}+1},\psi_{{N}+1})$ can be derived from \eqref{scheme}.

From \eqref{En.def}--\eqref{fn.def} and \eqref{hn.def}--\eqref{En.hat.def},
it suffices to 
define the error terms
$e_{{N}}$, $\tilde{e}_{{N}}$, and $\hat{e}_{{N}}$. 
To this end, by an analogous argument in \cite{CS08MR2423311,CSW19MR3925528}, we decompose
\begin{align}
\notag&\mathcal{L}(V_{{N}+1},\varPsi_{{N}+1})-\mathcal{L}(V_{{N}},\varPsi_{{N}})\\
\label{decom1}&\quad = \mathbb{L}_e'(U^a+V_{{N}+1/2},\varPhi^a+\varPsi_{{N}+1/2})\delta \dot{V}_{{N}}+e_{{N}}'+e_{{N}}''+e_{{N}}'''+D_{{N}+1/2} \delta\varPsi_{{N}}
\end{align}
and
\begin{align}
\notag&\mathcal{B}(V_{{N}+1} ,\psi_{{N}+1})-\mathcal{B}(V_{{N}} ,\psi_{{N}})\\
\label{decom2}&\quad
=\mathbb{B}_e'(U^a+V_{{N}+1/2},\varPhi^a+\varPsi_{{N}+1/2})
(\delta \dot{V}_{{N}} ,\delta\psi_{{N}})+\tilde{e}_{{N}}'+\tilde{e}_{{N}}''+\tilde{e}_{{N}}''',
\end{align}
where
\begin{align}
\notag& e_{{N}}':=
\mathcal{L}(V_{{N}+1},\varPsi_{{N}+1})-\mathcal{L}(V_{{N}},\varPsi_{{N}}) -\mathbb{L}'(U^a+V_{{N}},\varPhi^a+\varPsi_{{N}})(\delta V_{{N}},\delta\varPsi_{{N}}),\\
\notag&e_{{N}}'':=
\mathbb{L}'(U^a+V_{{N}},\varPhi^a+\varPsi_{{N}})(\delta V_{{N}},\delta\varPsi_{{N}})
- \mathbb{L}'(U^a+\mathcal{S}_{\theta_{{N}}}V_{{N}},\varPhi^a+\mathcal{S}_{\theta_{{N}}}\varPsi_{{N}})(\delta V_{{N}},\delta\varPsi_{{N}}),\\
\notag&e_{{N}}''':=
\mathbb{L}'(U^a+\mathcal{S}_{\theta_{{N}}}V_{{N}},\varPhi^a+\mathcal{S}_{\theta_{{N}}}\varPsi_{{N}})(\delta V_{{N}},\delta\varPsi_{{N}})
-\mathbb{L}'(U^a+V_{{N}+1/2},\varPhi^a+\varPsi_{{N}+1/2})(\delta V_{{N}},\delta\varPsi_{{N}}),\\
\label{last.error}
&D_{{N}+1/2}:=
\left(\p_3(\varPhi^a+\varPsi_{{N}+1/2})\right)^{-1}
\p_3\mathbb{L}(U^a+V_{{N}+1/2},\varPhi^a+\varPsi_{{N}+1/2}),
\end{align}
and
\begin{align*}
 &\tilde{e}_{{N}}':=
\mathcal{B}(V_{{N}+1} ,\psi_{{N}+1})-\mathcal{B}(V_{{N}} ,\psi_{{N}})
-\mathbb{B}'(U^a+V_{{N}},\varphi^a+\psi_{{N}})
(\delta V_{{N}} ,\delta\psi_{{N}}),\\
&\tilde{e}_{{N}}'': =
\mathbb{B}'(U^a+V_{{N}},\varphi^a+\psi_{{N}})
(\delta V_{{N}} ,\delta\psi_{{N}})\\
&\qquad\ \
-\mathbb{B}'(U^a+\mathcal{S}_{\theta_{{N}}}V_{{N}},
\varphi^a+(\mathcal{S}_{\theta_{{N}}}\varPsi_{{N}})|_{x_3=0}  )
(\delta V_{{N}} ,\delta\psi_{{N}}),\\
 &\tilde{e}_{{N}}''': =
\mathbb{B}'(U^a+\mathcal{S}_{\theta_{{N}}}V_{{N}},
\varphi^a+(\mathcal{S}_{\theta_{{N}}}\varPsi_{{N}})|_{x_3=0}  )
(\delta V_{{N}} ,\delta\psi_{{N}})\\
&\qquad\ \ -\mathbb{B}_e'(U^a+V_{{N}+1/2},\varPhi^a+\varPsi_{{N}+1/2})
(\delta \dot{V}_{{N}} ,\delta\psi_{{N}}).
\end{align*}
Take
\begin{align} \label{en.def}
e_{{N}}:=e_{{N}}'+e_{{N}}''+e_{{N}}'''+D_{{N}+1/2} \delta\varPsi_{{N}},\qquad
\tilde{e}_{{N}}:=\tilde{e}_{{N}}'+\tilde{e}_{{N}}''+\tilde{e}_{{N}}'''.
\end{align}

As for error term $\hat{e}_{{N}}$,
we decompose
\begin{align}
\mathcal{E}(V_{{N}+1},\varPsi_{{N}+1})-\mathcal{E}(V_{{N}},\varPsi_{{N}})
 = \mathcal{E}'(V_{{N}+1/2},\varPsi_{{N}+1/2})(\delta V_{{N}},\delta\varPsi_{{N}})+\hat{e}_{{N}}'
+\hat{e}_{{N}}''+\hat{e}_{{N}}''', \label{decom3}
\end{align}
and set
\begin{align} \label{en.hat.def}
\hat{e}_{{N}}:=\hat{e}_{{N}}'+\hat{e}_{{N}}''+\hat{e}_{{N}}''',
\end{align}
where
\begin{align}
\notag &\hat{e}_{{N}}' :=
\mathcal{E}(V_{{N}+1},\varPsi_{{N}+1})-\mathcal{E}(V_{{N}},\varPsi_{{N}})
- \mathcal{E}'(V_{{N}},\varPsi_{{N}})(\delta V_{{N}},\delta\varPsi_{{N}}),  \\
&\hat{e}_{{N}}'' :=
\mathcal{E}'(V_{{N}},\varPsi_{{N}})(\delta V_{{N}},\delta\varPsi_{{N}})
-\mathcal{E}'(\mathcal{S}_{\theta_{{N}}}V_{{N}},\mathcal{S}_{\theta_{{N}}}\varPsi_{{N}})(\delta V_{{N}},\delta\varPsi_{{N}}),\notag\\
\notag &\hat{e}_{{N}}''' :=
\mathcal{E}'(\mathcal{S}_{\theta_{{N}}}V_{{N}},\mathcal{S}_{\theta_{{N}}}\varPsi_{{N}})(\delta V_{{N}},\delta\varPsi_{{N}})
- \mathcal{E}'(V_{{N}+1/2},\varPsi_{{N}+1/2})(\delta V_{{N}},\delta\varPsi_{{N}}).
\end{align}
It follows from \eqref{app.eq.2} that
\begin{align*}
\mathcal{E}(V,\varPsi)=\p_t (\varPhi^a+\varPsi)+(v_1^a+v_1)\p_1 (\varPhi^a+\varPsi)+(v_2^a+v_2)\p_2 (\varPhi^a+\varPsi)-(v_3^a+v_3).
\end{align*}
Then we derive from \eqref{delta.Psi1}--\eqref{delta.Psi2} and \eqref{decom3}
that
\begin{align}\notag 
\begin{bmatrix}
\mathcal{E}(V_{{N}+1}^+,\varPsi_{{N}+1}^+)-\mathcal{E}(V_{{N}}^+,\varPsi_{{N}}^+) \\[1mm]
\mathcal{E}(V_{{N}+1}^-,\varPsi_{{N}+1}^-)-\mathcal{E}(V_{{N}}^-,\varPsi_{{N}}^-)
\end{bmatrix}
=\begin{bmatrix}
\mathcal{R}_Tg_{{N},2}+h_{{N}}^++\hat{e}_{{N}}^+\\[1mm]
 \mathcal{R}_T(g_{{N},2}-g_{{N},1})+h_{{N}}^-+\hat{e}_{{N}}^-
\end{bmatrix}.
\end{align}
Thus, 
by $\mathcal{E}(V_0,\varPsi_0)=0$, one has
\begin{align}
\mathcal{E}(V_{{N}+1}^-,\varPsi_{{N}+1}^-)
=\mathcal{R}_T\left(\sum_{n=0}^{{N}}(g_{n,2}-g_{n,1})\right)
+\sum_{n=0}^{{N}}h_{n}^-+\widehat{E}_{{N}+1}^-.
\label{NM.id1}
\end{align}
Moreover, from \eqref{effective.NM} and \eqref{decom2}, we have
\begin{align}  \label{NM.id2}
g_{{N}}=
\mathcal{B}(V_{{N}+1},\psi_{{N}+1})-\mathcal{B}(V_{{N}},\psi_{{N}})-\tilde{e}_{{N}}.
\end{align}
Denote by $\mathcal{B}(V_{{N}+1},\psi_{{N}+1})_j$   the $j$th component of the vector  $\mathcal{B}(V_{{N}+1},\psi_{{N}+1})$ for $j=1,2.$
From  \eqref{P.new} and \eqref{B.def}, 
\begin{align}
\mathcal{B}(V_{{N}+1},\psi_{{N}+1})_2
&=\mathcal{E}(V_{{N}+1}^+,\varPsi_{{N}+1}^+) |_{x_3=0}     =\mathcal{E}(V_{{N}+1}^-,\varPsi_{{N}+1}^-)|_{x_3=0}
+\mathcal{B}(V_{{N}+1},\psi_{{N}+1})_1.
\label{NM.id3}
\end{align}
Using  \eqref{NM.id2}, we have
\begin{align}  \label{decom2.b}
g_{{N},2}-g_{{N},1}=
\mathcal{E}(V_{{N}+1}^-,\varPsi_{{N}+1}^-)|_{x_3=0}
-\mathcal{E}(V_{{N}}^-,\varPsi_{{N}}^-)|_{x_3=0}-\tilde{e}_{{N},2}+\tilde{e}_{{N},1}.
\end{align}
Then, \eqref{decom2.b} and \eqref{NM.id1}  yield
\begin{align}
&\mathcal{E}(V_{{N}+1}^-,\varPsi_{{N}+1}^-) 
\label{decom3.c1} 
=\mathcal{R}_T\left(\mathcal{E}\left(V_{{N}+1}^-,\varPsi_{{N}+1}^-\right)|_{x_3=0}
-\widetilde{E}_{{N}+1,2}+\widetilde{E}_{{N}+1,1}\right)
+\sum_{n=0}^{{N}}h_{n}^-+\widehat{E}_{{N}+1}^-,
\end{align}
and similarly,
\begin{align} 
&\mathcal{E}(V_{{N}+1}^+,\varPsi_{{N}+1}^+)
\label{decom3.c2}
=\mathcal{R}_T\left(\mathcal{E}\left(V_{{N}+1}^+,\varPsi_{{N}+1}^+\right)|_{x_3=0}
-\widetilde{E}_{{N}+1,2}\right)+\sum_{n=0}^{{N}}h_{n}^++\widehat{E}_{{N}+1}^+.
\end{align}

From \eqref{decom1} and \eqref{NM.id2}, together with \eqref{effective.NM} and \eqref{fn.def}, one has
\begin{align}
\label{conv.1}&\mathcal{L}(V_{N+1},\varPsi_{N+1})
= \sum_{n=0}^{N}f_{n} +E_{N+1}
=\mathcal{S}_{\theta_{N}}f^a+(I-\mathcal{S}_{\theta_{N}})E_{N}+e_{N},\\
\label{conv.2}&\mathcal{B}(V_{N+1},\psi_{N+1})
= \sum_{n=0}^{N}g_{n} +\widetilde{E}_{N+1}
=(I-\mathcal{S}_{\theta_{N}})\widetilde{E}_{N}+\tilde{e}_{N}.
\end{align}
Substituting \eqref{hn.def} into \eqref{decom3.c1}--\eqref{decom3.c2} and using \eqref{NM.id3}, we get
\begin{align}\label{conv.3}
\left\{
\begin{aligned}
\mathcal{E}(V_{N+1}^-,\varPsi_{N+1}^-)
 =\;& \mathcal{R}_T\big( \mathcal{B}(V_{{N}+1},\psi_{{N}+1})_2
-\mathcal{B}(V_{{N}+1},\psi_{{N}+1})_1\big)\\
&
+(I-\mathcal{S}_{\theta_{N}})
\Big( \widehat{E}_{N}^--\mathcal{R}_T\big(\widetilde{E}_{{N},2}-\widetilde{E}_{N,1}\big) \Big)\\
&
+\hat{e}_{N}^--\mathcal{R}_T\big(\tilde{e}_{N,2}-\tilde{e}_{N,1}\big),\\[1mm]
\mathcal{E}(V_{N+1}^+,\varPsi_{N+1}^+)
=\;&\mathcal{R}_T\big( \mathcal{B}(V_{{N}+1},\psi_{{N}+1})_2\big)\\
&
+(I-\mathcal{S}_{\theta_{N}})
\big( \widehat{E}_{N}^+-\mathcal{R}_T\widetilde{E}_{N,2}\big)
+\hat{e}_{N}^+-\mathcal{R}_T\tilde{e}_{N,2}.
\end{aligned}
\right.
\end{align}
From $\mathcal{S}_{\theta_{N}}\to Id$ as $N \to \infty$, we conclude that
if the error terms $(e_{N},\tilde{e}_{N},\hat{e}_{N})$  tend to zero,
then $$(\mathcal{L}(V_{N+1},\varPsi_{N+1}),
 \mathcal{B}(V_{N+1},\psi_{N+1}),
  \mathcal{E}(V_{N+1},\varPsi_{N+1}))\to (f^a, \mathbf{0}, 0),$$
thus, the solution to  \eqref{P.new} can be obtained formally.

In order to estimate the error terms, we need to introduce the {\it inductive hypothesis} as follows.
Let us take an integer ${\mu}\geq 4$,
a small number  $\epsilon>0$,
and another integer  $\tilde{\mu}>{\mu}$, which will be determined later.
Suppose that 
we have the   estimate
\begin{align} \label{small}
\|\widetilde{U}^a\|_{H^{\tilde{\mu}+4}_{\gamma}(\Omega_T)}
+\|\widetilde{\varPhi}^a\|_{H^{\tilde{\mu}+5}_{\gamma}(\Omega_T)}
+\|\varphi^a\|_{H^{\tilde{\mu}+9/2}_{\gamma}(\omega_T)}
+\|f^a\|_{H^{\tilde{\mu}+3}_{\gamma}(\Omega_T)}
\leq \epsilon,
\end{align}
then our inductive hypothesis $\mathrm{\bf H}_{{N}-1}$ consists of  the following four parts:

\begin{align*}
&\textrm{(i)}\;
\|(\delta V_{n},\delta \varPsi_{n})\|_{H^{m}_{\gamma}(\Omega_T)}
+\|\delta\psi_{n}\|_{H^{m+1}_{\gamma}(\omega_T)}
\leq \epsilon \theta_{n}^{m-{\mu}-1}\Delta_{n},
\quad   {n}=0,\ldots,{N}-1, \,   m=2,\ldots,\tilde{\mu}, \\
&\textrm{(ii)}\,
\|\mathcal{L}( V_{n},  \varPsi_{n})-f^a\|_{H^{m}_{\gamma}(\Omega_T)}\leq 2
\epsilon \theta_{n}^{m-{\mu}-1},
\quad   {n}=0,\ldots,{N}-1, \, m=2,\ldots,\tilde{\mu}-1,  \\
&\textrm{(iii)}\,
\|\mathcal{B}( V_{n} ,  \psi_{n})\|_{H^{m}_{\gamma}(\omega_T)}
\leq  \epsilon \theta_{n}^{m-{\mu}-1},
\quad   {n}=0,\ldots,{N}-1, \,  m=3,\ldots, {{\mu}}, \\
&\textrm{(iv)}\,
\|\mathcal{E}( V_{n},  \varPsi_{n})\|_{H^{3}_{\gamma}(\Omega_T)}
\leq  \epsilon \theta_{n}^{2-{\mu}}, \quad  {n}=0,\ldots,{N}-1,
\end{align*}
where $\theta_{n}$ is given in \eqref{theta.def}, and $\Delta_{n}:=\theta_{n+1}-\theta_{n}$ decreases to zero and satisfies 
\begin{align}\label{delta.k}
\frac{1}{3\theta_{n}}\leq \Delta_{n}:=\theta_{n+1}-\theta_{n}=
\sqrt{\theta_{n}^2+1}-\theta_{n}\leq \frac{1}{2\theta_{n}}, \quad 
{n}\in\mathbb{N}.
\end{align}
We will  show that for  sufficiently small $\epsilon$ and $f^a$,   and for sufficiently large  $\theta_0\geq 1$, $\mathrm{\bf H}_{0}$  is true and $\mathrm{\bf H}_{{N}-1}$
implies $\mathrm{\bf H}_{{N}}$, thus $\mathrm{\bf H}_{{N}}$ is true for all $n\in\mathbb{N}$, which will allow us to prove Theorem {\rm\ref{thm}} completely.

Now we  assume that   $\mathrm{\bf H}_{{N}-1}$ holds, hence   have  the following estimates as in \cite[Lemmas 6--7]{CS08MR2423311}.

\begin{lemma}  \label{lem.tri}
 If $\theta_0$ is sufficiently large, then
 \begin{align}
 \label{tri1}&
 \|( V_{n}, \varPsi_{n})\|_{H^{m}_{\gamma}(\Omega_T)}
 +\|\psi_{n}\|_{H^{m+1}_{\gamma}(\omega_T)}
 \leq
 \left\{\begin{aligned}
 &\epsilon \theta_{n}^{(m-{\mu})_+},\quad &&\textrm{if \ }m\neq {\mu},\\
 &\epsilon \log \theta_{n},\quad &&\textrm{if \ }m= {\mu},
 \end{aligned}\right.\\
 \label{tri2}&\|( (I-\mathcal{S}_{\theta_{n}})V_{n}, (I-\mathcal{S}_{\theta_{n}})\varPsi_{n})\|_{H^{m}_{\gamma}(\Omega_T)}
 \leq C\epsilon \theta_{n}^{m-{\mu}},
 \end{align}
 for $n=0,\ldots,{N}$, and $m=2,\ldots,\tilde{\mu}$.
Furthermore,
 \begin{align}
 \label{tri3}&
 \|( \mathcal{S}_{\theta_{n}}V_{n}, \mathcal{S}_{\theta_{n}}\varPsi_{n})\|_{H^{m}_{\gamma}(\Omega_T)}
 \leq
 \left\{\begin{aligned}
 &C\epsilon \theta_{n}^{(m-{\mu})_+},\quad &&\textrm{if \ }m\neq {\mu},\\
 &C\epsilon \log \theta_{n},\quad &&\textrm{if \ }m= {\mu},
 \end{aligned}\right.
 \end{align}
for $n=0,\ldots,{N}$, and $m=2,\ldots,\tilde{\mu}+5$.
\end{lemma}

\subsection{Estimates of the Quadratic and First Substitution Error Terms}\label{sec.quad}

First, we rewrite quadratic error terms $e'_{n}$, $\tilde{e}_{n}'$, and $\hat{e}_{n}'$,
  from \eqref{decom1}, \eqref{decom2}, and \eqref{decom3} respectively, as follows:
\begin{align*}
e_{n}'
=&\int_{0}^{1}
 \mathbb{L}''\big(U^a+V_{n}+\tau \delta  V_{n},
 \varPhi^a+\varPsi_{n}+\tau \delta \varPsi_{n}\big)
 \big((\delta V_{n},\delta\varPsi_{n}),(\delta V_{n},\delta\varPsi_{n})\big)
(1-\tau)\,\dd \tau,\\
\tilde{e}_{n}'
=&\int_{0}^{1}
 \mathbb{B}''\big(U^a+V_{n}+\tau \delta  V_{n},
 \varphi^a+\psi_{n}
 +\tau \delta \psi_{n}\big)
 \big((\delta V_{n},\delta\psi_{n}),(\delta V_{n},\delta\psi_{n})\big)
 (1-\tau)\,\dd \tau,
 \\
\hat{e}_{n}'
=&\int_{0}^{1}
 \mathcal{E}''\big( V_{n}+\tau \delta  V_{n},
  \varPsi_{n}+\tau \delta \varPsi_{n}\big)
 \big((\delta V_{n},\delta\varPsi_{n}),(\delta V_{n},\delta\varPsi_{n})\big)
(1-\tau)\,\dd \tau,
\end{align*}
where
$\mathbb{L}''$, $\mathbb{B}''$, and $\mathcal{E}''$
are  the second derivatives of operators
$\mathbb{L}$, $\mathbb{B}$, and $\mathcal{E}$  respectively.
More precisely, we define
\begin{align*}
&\mathbb{L}''\big(\check{U},\check{\varPhi}\big)
\big((V,\varPsi),(\widetilde{V},\widetilde{\varPsi})\big)
:=\left.\frac{\dd}{\dd \theta}
\mathbb{L}'\big(\check{U}+\theta \widetilde{V},
\check{\varPhi}+\theta \widetilde{\varPsi}\big)
\big(V,\varPsi\big)\right|_{\theta=0},\\
&\mathbb{B}''(\check{U},\check{\varphi})
\big((V,\psi),(\widetilde{V},\tilde{\psi})\big)
:=\left.\frac{\dd}{\dd \theta}
\mathbb{B}'(\check{U}+\theta \widetilde{V},
\check{\varphi}+\theta \tilde{\psi}) (V,\psi)\right|_{\theta=0},\\
&\mathcal{E}''\big(\check{V},\check{\varPsi}\big)
\big((V,\varPsi),(\widetilde{V},\widetilde{\varPsi})\big)
:=\left.\frac{\dd}{\dd \theta}
\mathcal{E}'\big(\check{V}+\theta \widetilde{V},
\check{\varPsi}+\theta \widetilde{\varPsi}\big)
\big(V,\varPsi\big)\right|_{\theta=0},
\end{align*}
where operators $\mathbb{L}'$ and $\mathbb{B}'$ are given in \eqref{L'.bb}--\eqref{B'.bb},
and $\mathcal{E}'$ is defined by
\begin{align*}
\mathcal{E}'\big(\check{V} ,\check{\varPsi} \big)(V,\varPsi )
:=\left.\frac{\mathrm{d}}{\mathrm{d}\theta}
\mathcal{E}\big(\check{V} +\theta V ,
\check{\varPsi} +\theta\varPsi \big)\right|_{\theta=0}.
\end{align*}
In fact, in our case, we have the following:
\begin{align}
\label{B''.form}
&\mathbb{B}''(\check{U},\check{\varphi})
\big((V,\psi),(\widetilde{V},\tilde{\psi})\big)
=
\begin{bmatrix}
[\tilde{v}_1]\p_1 \psi +[\tilde{v}_2]\p_2 \psi+\p_1 \tilde{\psi}[v_1]+\p_2 \tilde{\psi}[v_2]\\[1mm]
\tilde{v}_1^+|_{x_3=0} \p_1 \psi +\tilde{v}_2^+|_{x_3=0} \p_2 \psi + \p_1 \tilde{\psi} v_1^+|_{x_3=0}+ \p_2 \tilde{\psi} v_2^+|_{x_3=0}\\[1mm]
0
\end{bmatrix},\\
\label{E''.form}
&\mathcal{E}''\big(\check{V},\check{\varPsi}\big)
\big((V,\varPsi),(\widetilde{V},\widetilde{\varPsi})\big)
=
\begin{bmatrix}
\tilde{v}_1^+ \p_1 \varPsi^+ + \p_1 \widetilde{\varPsi}^+\, v_1^+ +\tilde{v}_2^+ \p_2 \varPsi^+ + \p_2 \widetilde{\varPsi}^+\, v_2^+\\[0.5mm]
\tilde{v}_1^- \p_1 \varPsi^- + \p_1 \widetilde{\varPsi}^-\, v_1^- +\tilde{v}_2^- \p_2 \varPsi^- + \p_2 \widetilde{\varPsi}^-\, v_2^-
\end{bmatrix}.
\end{align}
A straightforward computation with an application of the Moser-type calculus inequality
\eqref{Moser1} yields the next proposition (see \cite[Proposition 5]{CS08MR2423311}).

\begin{proposition} \label{pro.tame2}
 Let $T>0$ and $m\in\mathbb{N}$ with $m\geq 2$.
 If
 $(\widetilde{V},\widetilde{\varPsi})$
 belongs to $H^{m+1}_{\gamma}(\Omega_T)$ for all $\gamma \geq 1$ and satisfies
 $\|(\widetilde{V},\widetilde{\varPsi} )\|_{W^{1,\infty}(\Omega_T)} \leq \widetilde{K}$
 for some positive constant $\widetilde{K}$,
 then there exist two constants $\widetilde{K}_0>0$ and $C>0$,
 independent of $T$ and $\gamma$,
 such that,
 if $\widetilde{K} \leq \widetilde{K}_0$ and $\gamma\geq 1$,
 then
 \begin{align}
 \notag&\big\|\mathbb{L}''\big(
 \bar{U}+\widetilde{V}, \bar{\varPhi}+\widetilde{\varPsi} \big)
 \big((V_1,\varPsi_1),(V_2,\varPsi_2) \big)\big\|_{H^{m}_{\gamma}(\Omega_T)}\\
 \notag&\qquad \leq
C\|(V_1,\varPsi_1)\|_{W^{1,\infty}(\Omega_T)}\|(V_2,\varPsi_2)\|_{W^{1,\infty}(\Omega_T)}
\big\|\big(\widetilde{V},\widetilde{\varPsi} \big)\big\|_{H^{m+1}_{\gamma}(\Omega_T)}
 \\[1.5mm]
 \notag& \qquad\quad\,
 +   C \sum_{i\neq j}\|(V_i,\varPsi_i)\|_{H^{m+1}_{\gamma}(\Omega_T)}
 \|(V_j,\varPsi_j)\|_{W^{1,\infty}(\Omega_T)} ,
  \end{align}
 \begin{align}
 \notag &\big\|
 \mathcal{E}''\big(\widetilde{V},\widetilde{\varPsi} \big)\big((V_1,\varPsi_1),(V_2,\varPsi_2) \big)\big\|_{H^{m}_{\gamma}(\Omega_T)}\\
 \notag&\qquad \leq C \sum_{i\neq j}
 \left\{ \|V_i \|_{H^{m}_{\gamma}(\Omega_T)} \|\varPsi_j\|_{W^{1,\infty}(\Omega_T)}
 +\|V_i \|_{L^{\infty} (\Omega_T)} \|\varPsi_j\|_{H^{m+1}_{\gamma}(\Omega_T)} \right\},
  \end{align}
  and
 \begin{align}
 \notag&\big\|\mathbb{B}''\big(\bar{U}+\widetilde{V},
 \tilde{\psi} \big)
 \big((W_1,\psi_1),(W_2,\psi_2) \big)\big\|_{H^{m}_{\gamma}(\omega_T)}\\
 \notag&\qquad \leq C \sum_{i\neq j}
 \left\{ \|W_i \|_{H^{m}_{\gamma}(\omega_T)} \|\psi_j\|_{W^{1,\infty}(\omega_T)}
 +\|W_i \|_{L^{\infty} (\omega_T)} \|\psi_j\|_{H^{m+1}_{\gamma}(\omega_T)} \right\},
 \end{align}
where $(V_i,\varPsi_i) \in H^{m+1}_{\gamma}(\Omega_T)$
 and $(W_i,\psi_i) \in H^{m}_{\gamma}(\omega_T)\times H^{m+1}_{\gamma}(\omega_T)$
 for $i=1,2$,
symbol $\tilde{\psi}$ denotes the trace of $\widetilde{\varPsi}$ on $\omega_T$,
 and $(\bar{U},\bar{\varPhi} )$
 represents the background state defined by \eqref{background}.
\end{proposition}

In light of \eqref{small}--\eqref{tri1} and the assumption $\mathrm{\bf H}_{{N}-1}$,
as shown in \cite[Lemma 8]{CS08MR2423311} or \cite[Lemma 8.3]{CSW19MR3925528},
we can apply Proposition \ref{pro.tame2}, the Sobolev embedding theorem,
and the trace estimate to derive the following estimate.

\begin{lemma}
\label{lem.quad}
If ${\mu}\geq 4 $, then there exist suitably small  $\epsilon>0$ and sufficiently large $\theta_0\geq 1$
such that
\begin{align}\notag
\|(e_{n}', \hat{e}_{n}')\|_{H^{m}_{\gamma}(\Omega_T)}+\|\tilde{e}_{n}'\|_{H^{m}_{\gamma}(\omega_T)}
\leq C\epsilon^2 \theta_{n}^{\ell_1(m)-1}\Delta_{n},
\end{align}
 for $m=2,\ldots,\tilde{\mu}-1$, and $n=0,\ldots,{N}-1$,
 where
$\ell_1(m):=\max\{(m+1-{\mu})_++4-2{\mu},m+2-2{\mu} \}$.
\end{lemma}

For the first substitution error terms $e_{n}''$, $\tilde{e}_{n}''$, and $\hat{e}_{n}''$
defined in \eqref{decom1}, \eqref{decom2}, and \eqref{decom3},
as in \cite[Lemma 9]{CS08MR2423311} or \cite[Lemma 8.4]{CSW19MR3925528},
we can apply Proposition \ref{pro.tame2}
and use \eqref{small}, \eqref{tri2}--\eqref{tri3}, hypothesis ($H_{n-1}$),
and the trace theorem to derive the next lemma.

\begin{lemma}
\label{lem.1st}
If ${\mu}\geq 4 $, then there exist $\epsilon>0$ suitably small and $\theta_0\geq 1$  large enough such that
\begin{alignat}{3} \notag
\|(e_{n}'', \hat{e}_{n}'')\|_{H^{m}_{\gamma}(\Omega_T)} &\leq C\epsilon^2 \theta_{n}^{\ell_2(m)-1}\Delta_{n}
&&\quad\ \textrm{if }\ m=2,\ldots, \tilde{\mu}-1,\\
\notag
 \|\tilde{e}_{n}''\|_{H^{m}_{\gamma}(\omega_T)} &\leq C\epsilon^2 \theta_{n}^{\ell_2(m)-1}\Delta_{n}
&&\quad\ \textrm{if }\ m=2,\ldots, \tilde{\mu}-2,
\end{alignat}
for $n=0,\ldots,{N}-1,$
where
$$\ell_2(m):=\max\{(m+1-{\mu})_++6-2{\mu},m+5-2{\mu} \}.$$
\end{lemma}

We emphasize that Proposition \ref{pro.tame2}
reduces the estimate for $\|\tilde{e}_{n}''\|_{H^{m}_{\gamma}(\omega_T)}$
 to that for the terms involving
$\|(I-\mathcal{S}_{\theta_{n}})\varPsi_{n}\|_{H^{m+2}_{\gamma}(\Omega_T)} $,
which requires condition $m\leq \tilde{\mu}-2$ in order to apply inequality \eqref{tri2}.

\subsection{Construction and Estimates of the Modified State} \label{sec.modified}

To control the remaining error terms, we construct and analyze the modified state $(V_{{N}+1/2},\varPsi_{{N}+1/2},\psi_{{N}+1/2}),$ as described in the following lemma.

\begin{lemma}
\label{lem.modified}
If ${\mu}\geq 5$, then there exist functions $V_{{N}+1/2}$, $\varPsi_{{N}+1/2}$, and $\psi_{{N}+1/2}$,
which vanish in the past, such that
$(U^a+V_{{N}+1/2},\varPhi^a+\varPsi_{{N}+1/2}, \varphi^a+\psi_{{N}+1/2})$
satisfies \eqref{bas.2}--\eqref{bas.5},
where $(U^a, \varPhi^a)$ is the approximate solution constructed in {\rm Lemma \ref{lem.app}}.
Moreover,
\begin{gather}\label{MS.id1}
\varPsi_{{N}+1/2}^{\pm}=\mathcal{S}_{\theta_N}\varPsi_{{N}}^{\pm},
\quad
\psi_{{N}+1/2}=(\mathcal{S}_{\theta_N}\varPsi_{{N}}^{\pm})|_{x_3=0},
\\
\label{MS.id2}
v_{{N}+1/2,1}^{\pm}=\mathcal{S}_{\theta_N} v_{{N},1}^{\pm},\quad v_{{N}+1/2,2}^{\pm}=\mathcal{S}_{\theta_N} v_{{N},2}^{\pm},\\
\label{MS.e}
\|\mathcal{S}_{\theta_N}V_{{N}}-V_{{N}+1/2}\|_{H^{m}_{\gamma}(\Omega_T)}
\leq C\epsilon \theta_N^{m+2-{\mu}}
\quad \textrm{for } m=2,\ldots,\tilde{\mu}+3.
\end{gather}
\end{lemma}
\begin{proof}
We divide the proof into four steps.

\vspace*{2mm}
\noindent {\it Step 1}.\ \
It follows from \eqref{smooth.p2}--\eqref{scheme.H1} that
 $(\mathcal{S}_{\theta_N}\varPsi_{{N}}^{+})|_{x_3=0}
=(\mathcal{S}_{\theta_N}\varPsi_{{N}}^{-})|_{x_3=0}$,
and hence we can define $\varPsi_{{N}+1/2}^{\pm}$, $\psi_{{N}+1/2}$, and $v_{{N}+1/2,1}^{\pm},v_{{N}+1/2,2}^{\pm}$
by \eqref{MS.id1}--\eqref{MS.id2}.
Thanks to \eqref{app.eq.4},
constraint \eqref{bas.3} holds for
$(\varPhi^a+\varPsi_{{N}+1/2}, \varphi^a+\psi_{{N}+1/2})$.
As in \cite[Proposition 7]{CS08MR2423311}, we define
\begin{align*}
\rho_{{N}+1/2}^{\pm}&:=\mathcal{S}_{\theta_{N}}\rho_{N}^{\pm}\mp
\frac{1}{2}\mathcal{R}_{T}\left((\mathcal{S}_{\theta_{N}}\rho_{N}^{+} )|_{x_3=0}
-(\mathcal{S}_{\theta_{N}}\rho_{N}^{-} )|_{x_3=0} \right),\\
v_{{N}+1/2,3}^{\pm}&:=
\p_t\varPsi_{{N}+1/2}^{\pm}+\left(v_1^{a\pm}+v_{{N}+1/2,1}^{\pm} \right) \p_1\varPsi_{{N}+1/2}^{\pm}
+v_{{N}+1/2,1}^{\pm} \p_1\varPhi^{a\pm}\nonumber\\
&+\left(v_2^{a\pm}+v_{{N}+1/2,2}^{\pm} \right)\p_2\varPsi_{{N}+1/2}^{\pm}
+v_{{N}+1/2,2}^{\pm} \p_2\varPhi^{a\pm},
\end{align*}
so that $[\rho^a+\rho_{{N}+1/2}]=0$ on $\p\Omega$,
and constraints \eqref{bas.2}, \eqref{bas.4} hold for
$$(v^a+v_{{N}+1/2},\varPhi^a+\varPsi_{{N}+1/2},\varphi^a+\psi_{{N}+1/2}),$$
using
\eqref{app.eq.5}, Lemma \ref{lem.smooth2}, and \eqref{app.eq.2}.

\vspace*{2mm}
\noindent {\it Step 2}.\ \
Utilizing \eqref{scheme},  the trace theorem, and the hypothesis $\mathrm{\bf H}_{{N}-1}$, we obtain
\begin{align}
\|\rho_{N}^+-\rho_{N}^-\|_{H^m_{\gamma}(\omega_T)}
&\leq \|\rho_{{N}-1}^+-\rho_{{N}-1}^-\|_{H^m_{\gamma}(\omega_T)}
+\|\delta \rho_{{N}-1}^+-\delta \rho_{{N}-1}^-\|_{H^m_{\gamma}(\omega_T)}\notag \\
& \leq
\|\mathcal{B}(V_{{N}-1}, \psi_{{N}-1})\|_{H^m_{\gamma}(\omega_T)}
+C \|\delta \rho_{{N}-1}\|_{H^{m+1}_{\gamma}(\Omega_T)}\notag \\
&\leq C \epsilon \theta_{N}^{m-{\mu}-1}\qquad
\textrm{for }\ m\in [3,{\mu}].
\label{MS.p1}
\end{align}
Then we use Lemma \ref{lem.smooth2}, \eqref{smooth.p2}, and \eqref{MS.p1} to obtain
\begin{align}
\notag &\|\rho_{{N}+1/2}-\mathcal{S}_{\theta_{{N}}}\rho_{N}\|_{H^m_{\gamma}(\Omega_T)}
\leq C \|\mathcal{S}_{\theta_{{N}}}\rho_{N}^+
-\mathcal{S}_{\theta_{{N}}}\rho_{N}^-\|_{H^m_{\gamma}(\omega_T)}\\
&\quad \leq
\left\{
\begin{aligned}
& C\|\rho_{N}^+- \rho_{N}^-\|_{H^{m+1}_{\gamma}(\omega_T)}\leq C \epsilon \theta_{N}^{m-{\mu}},
  && \textrm{if }\  2\leq m\leq {\mu}-1,\\
&C\theta_{N}^{m+1-{\mu}}  \|\rho_{N}^+- \rho_{N}^-\|_{H^{{\mu}}_{\gamma}(\omega_T)}
\leq  C \epsilon \theta_{N}^{m-{\mu}},
 &\ \ & \textrm{if }\   m\geq {\mu}.
\end{aligned}
\right. \label{MS.e1}
\end{align}

\vspace*{2mm}
\noindent {\it Step 3}.\ \
Using \eqref{MS.id1}, we have
\begin{align}
\notag
v_{{N}+1/2,3}-\mathcal{S}_{\theta_{{N}}}v_{{N},3}
 =\;&\mathcal{S}_{\theta_{{N}}} \mathcal{E}(V_{{N}},\varPsi_{{N}})
+[\p_t+v_1^a\p_1,\mathcal{S}_{\theta_{{N}}}]\varPsi_{{N}}
+[\p_1\varPhi^{a}, \mathcal{S}_{\theta_{{N}}}  ]v_{{N},1} \\
& +[\p_t+v_2^a\p_2,\mathcal{S}_{\theta_{{N}}}]\varPsi_{{N}}
+[\p_2\varPhi^{a}, \mathcal{S}_{\theta_{{N}}}  ]v_{{N},2} \nonumber\\
& +\mathcal{S}_{\theta_{{N}}} v_{{N},1}\p_1\mathcal{S}_{\theta_{{N}}} \varPsi_{{N}}
-\mathcal{S}_{\theta_{{N}}} (v_{{N},1} \p_1\varPsi_{{N}} )\nonumber\\
& +\mathcal{S}_{\theta_{{N}}} v_{{N},2}\p_2\mathcal{S}_{\theta_{{N}}} \varPsi_{{N}}
-\mathcal{S}_{\theta_{{N}}} (v_{{N},2} \p_2\varPsi_{{N}} ).
\label{MS.decom}
\end{align}

Decomposing
\begin{align*}
\mathcal{E}(V_{{N}},\varPsi_{{N}})
=\;&\mathcal{E}(V_{{N}-1},\varPsi_{{N}-1})
+\p_t(\delta \varPsi_{{N}-1})+(v_1^a+v_{{N}-1,1})\p_1(\delta \varPsi_{{N}-1})\\
&+(v_2^a+v_{{N}-1,2})\p_2(\delta \varPsi_{{N}-1})+\delta v_{{N}-1,1}\p_1(\varPhi^a+\varPsi_{{N}})\\
&+\delta v_{{N}-1,2}\p_2(\varPhi^a+\varPsi_{{N}})-\delta v_{{N}-1,3},
\end{align*}
the Moser-type calculus inequality \eqref{Moser1},
hypothesis ($H_{{N}-1}$), and \eqref{tri1} implies
\begin{align*}
\|\mathcal{E}(V_{{N}},\varPsi_{{N}})\|_{H^3_{\gamma}(\Omega_T)}\leq C\epsilon\theta_{N}^{2-{\mu}},
\end{align*}
which together with \eqref{smooth.p1a} yields that
\begin{align} \label{MS.p2}
\|\mathcal{S}_{\theta_{{N}}} \mathcal{E}(V_{{N}},\varPsi_{{N}})\|_{H^m_{\gamma}(\Omega_T)}
\leq C\epsilon\theta_{N}^{m-{\mu}},
\qquad \textrm{for }m\geq 2.
\end{align}

The rest of terms on the right-hand side of \eqref{MS.decom} consist entirely of commutators.
Let us detail the estimate of $[v_1^a\p_1, \mathcal{S}_{\theta_{{N}}}]\varPsi_{{N}}$.
Using \eqref{Moser1}, the Sobolev embedding theorem,
\eqref{smooth.p1a},
\eqref{small}, and \eqref{tri3}, we obtain
\begin{align*}
\|[v_1^a\p_1, \mathcal{S}_{\theta_{{N}}}]\varPsi_{{N}}\|_{H^m_{\gamma}(\Omega_T)}
\leq\;&
\|v_1^a\p_1 (\mathcal{S}_{\theta_{{N}}}\varPsi_{{N}})\|_{H^m_{\gamma}(\Omega_T)}
+\|\mathcal{S}_{\theta_{{N}}}(v_1^a\p_1\varPsi_{{N}})\|_{H^m_{\gamma}(\Omega_T)}
\\
\leq\;&C \| \mathcal{S}_{\theta_{{N}}}\varPsi_{{N}}\|_{H^{m+1}_{\gamma}(\Omega_T)}
+C\|\tilde{v}_1^a\|_{H^m_{\gamma}(\Omega_T)}
\| \mathcal{S}_{\theta_{{N}}}\varPsi_{{N}}\|_{H^{3}_{\gamma}(\Omega_T)}\\
&+C\theta_{N}^{m-{\mu}}
\|v_1^a\p_1\varPsi_{{N}}\|_{H^{{\mu}}_{\gamma}(\Omega_T)}\\
\leq\;&C\epsilon \theta_{N}^{m-{\mu}+1}
\qquad  \textrm{for }\ {\mu}+1\leq m\leq \tilde{\mu}+4.
\end{align*}
Similarly, we obtain that
\begin{align*}
\|[v_2^a\p_2, \mathcal{S}_{\theta_{{N}}}]\varPsi_{{N}}\|_{H^m_{\gamma}(\Omega_T)}
\leq\;&C\epsilon \theta_{N}^{m-{\mu}+1}
\qquad  \textrm{for }\ {\mu}+1\leq m\leq \tilde{\mu}+4.
\end{align*}
If $2\leq m\leq {\mu}$, it follows from \eqref{smooth.p1b} and \eqref{tri1}--\eqref{tri2}
that
\begin{align*}
\|[v_1^a\p_1, \mathcal{S}_{\theta_{{N}}}]\varPsi_{{N}}\|_{H^m_{\gamma}(\Omega_T)}
&
\leq
\|v_1^a\p_1 ((\mathcal{S}_{\theta_{{N}}}-I)\varPsi_{{N}})\|_{H^m_{\gamma}(\Omega_T)}
+\|(I-\mathcal{S}_{\theta_{{N}}})(v_1^a\p_1 \varPsi_{{N}})\|_{H^m_{\gamma}(\Omega_T)}\\
&
\leq
C\|(\mathcal{S}_{\theta_{{N}}}-I)\varPsi_{{N}}\|_{H^{m+1}_{\gamma}(\Omega_T)}
+C\theta_{N}^{m-{\mu}}\|v_1^a\p_1 \varPsi_{{N}}\|_{H^{{\mu}}_{\gamma}(\Omega_T)}
\leq C\epsilon \theta_{N}^{m-{\mu}+1}.
\end{align*}
Similarly, we have
\begin{align*}
\|[v_2^a\p_2, \mathcal{S}_{\theta_{{N}}}]\varPsi_{{N}}\|_{H^m_{\gamma}(\Omega_T)}
\leq C\epsilon \theta_{N}^{m-{\mu}+1}.
\end{align*}
Applying the same analysis to the other commutators in \eqref{MS.decom}
and using \eqref{MS.p2}, we obtain
\begin{align} \label{MS.e2}
\|v_{{N}+1/2,3}-\mathcal{S}_{\theta_{{N}}}v_{{N},3}\|_{H^m_{\gamma}(\Omega_T)}
\leq C\epsilon \theta_{N}^{m-{\mu}+1}
\quad\   \textrm{for }\ m=2,\ldots, \tilde{\mu}+4.
\end{align}

\vspace*{2mm}
\noindent {\it Step 4}.\ \
Now, we construct and estimate $\bm{F}_{{N}+1/2},$
 following the approach of {\rm  Secchi--Trakhinin} \cite[Proposition 28]{ST14MR3151094}.
As outlined in Step 1, the functions $v_{{N}+1/2}$ and $\varPsi_{{N}+1/2}$ have already been specified.
Next, we define
$\bm{F}_{{N}+1/2}$ as the unique solution, vanishing in the past,
 of the linear equations
\begin{align} \label{F.modified.eq}
\mathbb{L}_{{F}_{ij}}(v^a+v_{{N}+1/2}, \bm{F}^a+\bm{F}_{{N}+1/2}, \varPhi^a+\varPsi_{{N}+1/2})=0
\quad \textrm{for }\ i,j=1,2,3,
\end{align}
where
$\mathbb{L}_{{F}_{ij}}$ represents the component of operator $\mathbb{L}$ corresponding to ${F}_{ij}$, defined as:
\begin{align} \label{L.bb.F}
\mathbb{L}_{{F}_{ij}}(v , \bm{F} , \varPhi ):=
\left(\p_t^{ \varPhi}+v_{\ell} \p_{\ell}^{ \varPhi} \right)  { {F}}_{ij} -  { {F}}_{\ell j} \p_{\ell}^{ { \varPhi} } {v_i}.
\end{align}
Since $(v^a+v_{{N}+1/2},\varPhi^a+\varPsi_{{N}+1/2})$ satisfies  \eqref{bas.2}, the
equations \eqref{F.modified.eq} do not require additional boundary condition.

To estimate $\bm{F}_{{N}+1/2}-\mathcal{S}_{\theta_{N}} \bm{F}_{N}$,
we apply standard energy method. From \eqref{F.modified.eq}, we deduce:
\begin{align}
&\mathbb{L}_{{F}_{ij}}(v^a+v_{{N}+1/2},
\bm{F}_{{N}+1/2}-\mathcal{S}_{\theta_{N}} \bm{F}_{N},
\varPhi^a+\varPsi_{{N}+1/2})
=\mathcal{H}_{1}+\mathcal{H}_{2}+\mathcal{H}_{3},
\label{MS.p3}
\end{align}
where
\begin{align*}
\mathcal{H}_{1}:=\;&
-\mathbb{L}_{{F}_{ij}}(v^a+v_{{N}+1/2},
\bm{F}^a+\mathcal{S}_{\theta_{N}} \bm{F}_{N},
\varPhi^a+\varPsi_{{N}+1/2})\\ &
+\mathbb{L}_{{F}_{ij}}(v^a+\mathcal{S}_{\theta_{N}}  v_{N},
\bm{F}^a+\mathcal{S}_{\theta_{N}} \bm{F}_{N},
\varPhi^a+\mathcal{S}_{\theta_{N}}  \varPsi_{{N}}),\\
\mathcal{H}_{2}:=\;&
-\mathbb{L}_{{F}_{ij}}(v^a+\mathcal{S}_{\theta_{N}}  v_{N},
\bm{F}^a+\mathcal{S}_{\theta_{N}} \bm{F}_{N},
\varPhi^a+\mathcal{S}_{\theta_{N}}  \varPsi_{{N}})\\&
+\mathcal{S}_{\theta_{N}}\mathbb{L}_{{F}_{ij}}(v^a+  v_{N},
\bm{F}^a+ \bm{F}_{N},  \varPhi^a+ \varPsi_{{N}}),
\end{align*}
and $\mathcal{H}_{3}:=
-\mathcal{S}_{\theta_{N}}\mathbb{L}_{{F}_{ij}}(v^a+  v_{N},
\bm{F}^a+ \bm{F}_{N},  \varPhi^a+ \varPsi_{{N}}).$
From \eqref{MS.id1}, we compute
\begin{align*}
\mathcal{H}_{1}=
\;&(\mathcal{S}_{\theta_{N}}  v_{{N},\ell}-v_{{N}+1/2,\ell})
\p_{\ell}^{ \varPhi^a+\varPsi_{{N}+1/2}}
(F_{ij}^a+\mathcal{S}_{\theta_{N}}   F_{{N},ij})\\
&-(F_{\ell j}^a+\mathcal{S}_{\theta_{N}}   F_{{N},\ell j})
\p_{\ell}^{ \varPhi^a+\varPsi_{{N}+1/2}}
(\mathcal{S}_{\theta_{N}}  v_{{N},i}-v_{{N}+1/2,i}).
\end{align*}
Applying Moser-type calculus inequality \eqref{Moser1}, the Sobolev embedding theorem,
\eqref{MS.id1}--\eqref{MS.id2}, \eqref{MS.e2}, \eqref{small}, and \eqref{tri3},
we obtain
\begin{align}
\|\mathcal{H}_{1}\|_{H^m_{\gamma}(\Omega_T)}
\leq \;&
C\|\mathcal{S}_{\theta_{N}}  v_{N}-v_{{N}+1/2}\|_{H^{3}_{\gamma}(\Omega_T)}
\|(\widetilde{\bm{F}}^a,\mathcal{S}_{\theta_{N}}   \bm{F}_{N},
\widetilde{\varPhi}^a,\mathcal{S}_{\theta_{N}} \varPsi_{N}) \|_{H^{m+1}_{\gamma}(\Omega_T)}
\notag  \\
&+ C\|\mathcal{S}_{\theta_{N}}  v_{N}-v_{{N}+1/2}\|_{H^{m+1}_{\gamma}(\Omega_T)}
\notag
\\
\leq\;& C\epsilon \theta_{N}^{m-{\mu}+2}
\qquad\quad   \textrm{for }\ m=2,\ldots, \tilde{\mu}+3.
\label{H1.cal}
\end{align}
For $\mathcal{H}_2$, we follow the same strategy used to estimate
$[v_1^a\p_1, \mathcal{S}_{\theta_{{N}}}]\varPsi_{{N}}$ in Step 3 and obtain:
\begin{align} \label{H2.cal}
\|\mathcal{H}_{2}\|_{H^m_{\gamma}(\Omega_T)}
\leq C\epsilon \theta_{N}^{m-{\mu}+2}
\qquad  \textrm{for }\ m=2,\ldots, \tilde{\mu}+3.
\end{align}
 Regarding $\mathcal{H}_3$,  from \eqref{smooth.p1a}, \eqref{Fa.eq},
and the inductive hypothesis $\mathrm{\bf H}_{{N}-1}$, we find
\begin{align*}
&\| \mathcal{S}_{\theta_{N}}\mathbb{L}_{{F}_{ij}}(v^a+  v_{N-1},
\bm{F}^a+ \bm{F}_{N-1},  \varPhi^a+ \varPsi_{{N-1}}) \|_{H^m_{\gamma}(\Omega_T)}\\
&\quad \leq C\theta_{{N}}^{m-2}
\|  \mathbb{L}_{{F}_{ij}}(v^a+  v_{N-1},
\bm{F}^a+ \bm{F}_{N-1},  \varPhi^a+ \varPsi_{{N-1}}) \|_{H^2_{\gamma}(\Omega_T)}
\leq C\epsilon\theta_{{N}}^{m-{\mu}-1}
\end{align*}
for $m\geq 2$. Using \eqref{smooth.p1a},
\eqref{Moser1}, hypothesis $\mathrm{\bf H}_{N-1}$, and \eqref{tri1} yields
\begin{align*}
&\| \mathcal{S}_{\theta_{N}}\big(\mathbb{L}_{{F}_{ij}}(v^a+  v_{N},
\bm{F}^a+ \bm{F}_{N},  \varPhi^a+ \varPsi_{{N}})  \\
&\qquad\ -  \mathbb{L}_{{F}_{ij}}(v^a+  v_{N-1},
\bm{F}^a+ \bm{F}_{N-1},  \varPhi^a+ \varPsi_{{N-1}})\big) \|_{H^m_{\gamma}(\Omega_T)}
\leq C\epsilon\theta_{{N}}^{m-{\mu}+2}
\end{align*}
for $m\geq 2$.
Combining these estimates with \eqref{H1.cal}--\eqref{H2.cal}, we have
\begin{align}
\sum_{\ell=1}^{3}\|\mathcal{H}_{\ell}\|_{H^m_{\gamma}(\Omega_T)}
\leq C\epsilon \theta_{N}^{m-{\mu}+2},
\qquad  \textrm{for }\ m=2,\ldots, \tilde{\mu}+3.
\label{MS.p4}
\end{align}
Applying a standard energy argument to equations \eqref{MS.p3} and using
estimate \eqref{MS.p4}, we conclude
\begin{align}
\|\bm{F}_{{N}+1/2}-\mathcal{S}_{\theta_{N}} \bm{F}_{N}\|_{H^m_{\gamma}(\Omega_T)}
\leq C\epsilon \theta_{N}^{m-{\mu}+2}
\qquad  \textrm{for }\ m=2,\ldots, \tilde{\mu}+3.
\label{MS.e3}
\end{align}
Finally, estimate \eqref{MS.e} follows from \eqref{MS.id2}, \eqref{MS.e1},
\eqref{MS.e2}, and \eqref{MS.e3}. This completes the proof.
\end{proof}

\begin{remark}
Using the Sobolev embedding,  \eqref{small}, \eqref{tri3}, and \eqref{MS.e}, we obtain constraint \eqref{bas.c2}.
Constraint \eqref{bas.1} can be satisfied by choosing $\epsilon$ sufficiently small. Meanwhile, constraint \eqref{bas.c1} is ensured by applying an appropriate cut-off function, so that the terms
$(V_{{N}+1/2}, \varPsi_{{N}+1/2}, \psi_{{N}+1/2})$ can be truncated.
\end{remark}

\subsection{Estimates of the Second Substitution and Last Error Terms}\label{sec.2nd}
The following lemma provides the
estimates for the second substitution error terms $e_{n}'''$, $\tilde{e}_{n}'''$, and $\hat{e}_{n}'''$, as
defined in \eqref{decom1}, \eqref{decom2}, and \eqref{decom3}, respectively.

\begin{lemma}
\label{lem.2nd}
If ${\mu}\geq 5 $, then there exist $\epsilon>0$ suitably small and $\theta_0\geq 1$  large enough such that
\begin{alignat}{3} \notag
(\tilde{e}_{n}''',\hat{e}_{n}''')={\mathbf 0},\qquad
\|e_{n}'''\|_{H^{m}_{\gamma}(\Omega_T)}  \leq C\epsilon^2 \theta_{n}^{\ell_3(m)-1}\Delta_{n}
 \quad\ \textrm{if }\ m=2,\ldots, \tilde{\mu}-1,
\end{alignat}
for $n=0,\ldots,{N}-1,$
where $\ell_3(m):=\max\{(m+1-{\mu})_++9-2{\mu},m+6-2{\mu} \}$.
\end{lemma}
\begin{proof}
From 
\eqref{B''.form} and \eqref{MS.id1}--\eqref{MS.id2},
we have
\begin{align*}
\tilde{e}_{n}'''=
\;&\mathbb{B}'(U^a+\mathcal{S}_{\theta_{n}}V_{n},
\varphi^a+(\mathcal{S}_{\theta_{n}}\varPsi_{n})|_{x_3=0}  )
(\delta V_{n} ,\delta\psi_{n})\\
&-\mathbb{B}'(U^a+ V_{n+1/2},
\varphi^a+(\mathcal{S}_{\theta_{n}}\varPsi_{n})|_{x_3=0}  )
(\delta V_{n} ,\delta\psi_{n})={\mathbf 0}.
\end{align*}
Using \eqref{E''.form}--\eqref{MS.id2}, we deduce that $\hat{e}_{n}'''={\mathbf 0}$.
Thanks to \eqref{MS.id1}, the error term ${e}_{n}'''$ can be rewritten as
\begin{align*}
&{e}_{n}'''=\int_{0}^{1}\mathbb{L}''\big( U^a+V_{n+1/2}+\tau (\mathcal{S}_{\theta_{n}}V_{n} -V_{n+1/2}),\,
\varPhi^a 
+\mathcal{S}_{\theta_{n}} \varPsi_{n} \big)
\big((\delta V_{n},\delta\varPsi_{n}),(\mathcal{S}_{\theta_{n}} V_{n}-V_{n+1/2} ,0)\big)
\,\dd \tau.
\end{align*}
Apply the Sobolev embedding theorem, \eqref{small}, \eqref{tri3}, and \eqref{MS.e}, we find that
\begin{align*}
\|(\widetilde{U}^a,\, V_{n+1/2},\, \mathcal{S}_{\theta_{n}}V_{n} -V_{n+1/2},\,
\widetilde{\varPhi}^a,\, \mathcal{S}_{\theta_{n}} \varPsi_{n} ) \|_{W^{1,\infty}(\Omega_T)} \leq
C\epsilon,
\end{align*}
allowing us to use Proposition \ref{pro.tame2} for $\epsilon$ suitably small.
Furthermore, from \eqref{small}--\eqref{tri1} and \eqref{MS.e}, we have
\begin{align*}
\|(\widetilde{U}^a,\, V_{n+1/2},\, \mathcal{S}_{\theta_{n}}V_{n} ,\,
\widetilde{\varPhi}^a,\, \mathcal{S}_{\theta_{n}} \varPsi_{n} ) \|_{H^{m+1}_{\gamma}(\Omega_T)}
\leq C\epsilon \left( \theta_{n}^{(m+1-{\mu})_++1}+
 \theta_{n}^{m+3-{\mu}}\right)
\end{align*}
for $2\leq m\leq \tilde{\mu}-1$.
Using Proposition \ref{pro.tame2}, hypothesis $\mathrm{\bf H}_{n-1}$, and
\eqref{MS.e}, we obtain the estimate for term ${e}_{n}''',$ thereby completing the proof of lemma.
\end{proof}

For the last error term \eqref{last.error},
\begin{align*}
D_{n+1/2}\delta\varPsi_{n}=\frac{\delta\varPsi_{n}}{\p_3(\varPhi^a+\varPsi_{n+1/2})}R_{n},
\quad \text{where } R_{n}:=\p_3\mathbb{L}(U^a+V_{n+1/2},\varPhi^a+\varPsi_{n+1/2}),
\end{align*}
we  first observe that
\begin{align*}
|\p_3(\varPhi^{a\pm}+\varPsi_{n+1/2}^{\pm})|\geq \frac{1}{2},
\end{align*}
as deduced  from \eqref{app.eq.3}, \eqref{MS.id1}, and \eqref{tri3}
for sufficiently small $\epsilon$.
Consequently, we arrive at the following lemma, analogous to \cite[Lemma 8.6]{CS08MR2423311} or \cite[Lemma 12]{CSW19MR3925528}. The proof expands $D_{n+1/2}\delta\varPsi_n$, uses the lower bound on $\partial_3(\varPhi^a+\varPsi_{n+1/2})$ to control the denominator, and then applies the same tame product and smoothing estimates as before; we omit the routine details.

\begin{lemma}
\label{lem.last}
If ${\mu}\geq 5$ and $\tilde{\mu}> {\mu}$,
then there exist $\epsilon>0$ suitably small and $\theta_0\geq 1$  large enough such that
\begin{align} \label{last.e0}
\|D_{n+1/2}\delta\varPsi_{n}\|_{H^{m}_{\gamma}(\Omega_T)}\leq C\epsilon^2 \theta_{n}^{\ell_4(m)-1}\Delta_{n}
 \quad\ \textrm{if }\ m=2,\ldots, \tilde{\mu}-1,
\end{align}
for $n=0,\ldots,{N}-1$,
where
$$\ell_4(m):=\max\{(m+2-{\mu})_++8-2{\mu},(m+1-{\mu})_++9-2{\mu},m+6-2{\mu}\}.$$
\end{lemma}

Lemmas \ref{lem.quad}--\ref{lem.last} provide the following estimates for $e_{n}$, $\tilde{e}_{n}$, and $\hat{e}_{n}$
defined in \eqref{en.def} and \eqref{en.hat.def}.

\begin{corollary} \label{lem.sum1}
If ${\mu}\geq 5$ and $\tilde{\mu}> {\mu}$, then there exist $\epsilon>0$ suitably small and $\theta_0\geq 1$  large enough
such that
\begin{alignat}{3}
\label{es.sum1}
\|e_{n}\|_{H^{m}_{\gamma}(\Omega_T)} &\leq C\epsilon^2 \theta_{n}^{\ell_4(m)-1}\Delta_{n}
&&\quad\ \textrm{if }\ m=2,\ldots, \tilde{\mu}-1,\\
\label{es.sum1.hat}
\|\hat{e}_{n}\|_{H^{m}_{\gamma}(\Omega_T)} &
\leq C\epsilon^2 \theta_{n}^{\ell_2(m)-1}\Delta_{n}
&&\quad\ \textrm{if }\ m=2,\ldots, \tilde{\mu}-1,\\
\label{es.sum1.tilde}
\|\tilde{e}_{n}\|_{H^{m}_{\gamma}(\omega_T)} &
\leq C\epsilon^2 \theta_{n}^{\ell_2(m)-1}\Delta_{n}
&&\quad\ \textrm{if }\ m=2,\ldots, \tilde{\mu}-2,
\end{alignat}
for $n=0,\ldots,{N}-1,$
where  $\ell_2(m)$ and $\ell_4(m)$ are defined
in {\rm Lemma  \ref{lem.1st}} and {\rm Lemma  \ref{lem.last}}, respectively.
\end{corollary}

\subsection{Proof of Theorem \ref{thm}}\label{sec.proof2}
We first prove the following lemma for
accumulated error terms $E_n$, $\widetilde{E}_n$, and $\hat{E}_n$
that are defined in \eqref{En.def} and \eqref{En.hat.def}.

\begin{lemma} \label{lem.sum2}
If ${\mu}\geq 7$ and $\tilde{\mu}={\mu}+3$, then there exist $\epsilon>0$ suitably small and $\theta_0\geq 1$  large enough such that
\begin{align}\label{es.sum2a}
&\|E_{{N}}\|_{H^{{\mu}+2}_{\gamma}(\Omega_T)}
\leq C\epsilon^2 \theta_{{N}},\\
\label{es.sum2b}
&\|\widetilde{E}_{{N}}\|_{H^{{\mu}+1}_{\gamma}(\omega_T)}
+\|\widehat{E}_{{N}}\|_{H^{{\mu}+1}_{\gamma}(\Omega_T)}
\leq C\epsilon^2 .
\end{align}
\end{lemma}
\begin{proof} Following \cite{CS08MR2423311,CSW19MR3925528}, we first note that $\ell_4({\mu}+2)\leq 1$ when ${\mu}\geq 7$.
From \eqref{es.sum1}, one has
\begin{align*}
&\|E_{{N}}\|_{H^{{\mu}+2}_{\gamma}(\Omega_T)}
\leq \sum_{n=0}^{{N}-1}\|e_{n} \|_{H^{{\mu}+2}_{\gamma}(\Omega_T)}
\leq \sum_{n=0}^{{N}-1} C\epsilon^2 \Delta_{n}
\leq C\epsilon^2\theta_{{N}},
\end{align*}
for ${\mu}\geq 7$ and ${\mu}+2\leq \tilde{\mu}-1$.
Since $\ell_2({\mu}+1)=6-{\mu}\leq -1$ for ${\mu}\geq 7$
and ${\mu}+1\leq \tilde{\mu}-2$,
from \eqref{es.sum1.hat}--\eqref{es.sum1.tilde}, we have
\begin{align*}
\|\widetilde{E}_{{N}}\|_{H^{{\mu}+1}_{\gamma}(\omega_T)}
+\|\widehat{E}_{{N}}\|_{H^{{\mu}+1}_{\gamma}(\Omega_T)}&\leq \sum_{n=0}^{{N}-1}
\left\{
\|\tilde{e}_{n} \|_{H^{{\mu}+1}_{\gamma}(\omega_T)}
+\|\hat{e}_{n} \|_{H^{{\mu}+1}_{\gamma}(\Omega_T)}
\right\}\nonumber\\
&\leq \sum_{n=0}^{{N}-1} C\epsilon^2 \theta_{n}^{-2}\Delta_{n}
\leq C\epsilon^2,
\end{align*}
where we have utilized \eqref{theta.def} and \eqref{delta.k}
 to obtain the last inequality.
The minimal possible $\tilde{\mu}$ is ${\mu}+3$.
The proof of the lemma is completed.
\end{proof}

Using the lemma above, we have the estimates for
$f_{{N}}$, $g_{{N}}$,
and $h_{{N}}^{\pm}$.
\begin{lemma}\  \label{lem.source}
If ${\mu}\geq 7$ and $\tilde{\mu}={\mu}+3$, then there exist $\epsilon>0$ suitably small and $\theta_0\geq 1$  large enough such that
 \begin{align}
 \label{es.fl}&\|f_{{N}}\|_{H^m_{\gamma}(\Omega_T)}
 \leq C \Delta_{{N}}\left\{\theta_{{N}}^{m-{\mu}-2}
 \left(\|f^a\|_{H^{{\mu}+1}_{\gamma}(\Omega_T)}
 +\epsilon^2\right)+\epsilon^2\theta_{{N}}^{\ell_4(m)-1}\right\},\\
 \label{es.gl}&\|g_{{N}}\|_{H^m_{\gamma}(\omega_T)}
 \leq C \epsilon^2 \Delta_{{N}}\big(\theta_{{N}}^{m-{\mu}-2}+\theta_{{N}}^{\ell_2(m)-1}\big),
 \end{align}
  for $m=2,\ldots,\tilde{\mu}+1$,
 and
 \begin{align}
 \label{es.hl}\|h_{{N}}^{\pm}\|_{H^m_{\gamma}(\Omega_T)}\leq C\epsilon^2 \Delta_{{N}}\big(\theta_{{N}}^{m-{\mu}-2}+\theta_{{N}}^{\ell_2(m)-1}\big)
\quad\ \textrm{for }\ m=2,\ldots, \tilde{\mu}.
 \end{align}
\end{lemma}
\begin{proof}
Using $\theta_{{N}-1}\leq \theta_{{N}}\leq \sqrt{2}\theta_{{N}-1}$
and $\Delta_{{N}-1}\leq 3\Delta_{{N}}$,
from \eqref{smooth.p1a}, \eqref{smooth.p1c}, \eqref{es.sum1}, and \eqref{es.sum2a},
we obtain
\begin{align*}
\|f_{N}\|_{H^m_{\gamma}(\Omega_T)}
&\leq
\|(\mathcal{S}_{\theta_{{N}}}-\mathcal{S}_{\theta_{{N}-1}})f^a
-(\mathcal{S}_{\theta_{{N}}}-\mathcal{S}_{\theta_{{N}-1}})E_{{N}-1}
-\mathcal{S}_{\theta_{{N}}}e_{{N}-1}\|_{H^m_{\gamma}(\Omega_T)} \\
&
\leq
C \Delta_{{N}}
 \theta_{{N}}^{m-{\mu}-2}\big(\|f^a\|_{H^{{\mu}+1}_{\gamma}(\Omega_T)}
+\theta_{{N}}^{-1}\|E_{{N}-1}\|_{H^{{\mu}+2}_{\gamma}(\Omega_T)} \big)
+\|\mathcal{S}_{\theta_{{N}}} e_{{N}-1}\|_{H^m_{\gamma}(\Omega_T)} \\
&
\leq
C \Delta_{{N}}\left\{\theta_{{N}}^{m-{\mu}-2}(\|f^a\|_{H^{{\mu}+1}_{\gamma}(\Omega_T)}
+\epsilon^2)+\epsilon^2\theta_{{N}}^{\ell_4(m)-1}\right\}.
\end{align*}
By using \eqref{es.sum1.tilde} and \eqref{es.sum2b}, we get
\begin{align*}
 \|g_{N}\|_{H^m_{\gamma}(\omega_T)}
&\leq
\| (\mathcal{S}_{\theta_{{N}}}-\mathcal{S}_{\theta_{{N}-1}})\widetilde{E}_{{N}-1}
-\mathcal{S}_{\theta_{{N}}}\tilde{e}_{{N}-1}\|_{H^m_{\gamma}(\omega_T)} \\
& \leq
C \Delta_{{N}}
\theta_{{N}}^{m-{\mu}-2} \|\widetilde{E}_{{N}-1}\|_{H^{{\mu}+1}_{\gamma}(\omega_T)}
+\|\mathcal{S}_{\theta_{{N}}} \tilde{e}_{{N}-1}\|_{H^m_{\gamma}(\omega_T)} \\
& \leq
C \epsilon^2 \Delta_{{N}}\big(\theta_{{N}}^{m-{\mu}-2}+\theta_{{N}}^{\ell_2(m)-1}\big).
\end{align*}
Similarly, we can obtain \eqref{es.hl} for $h_{N}^{\pm}$ from \eqref{es.sum1.hat}
and \eqref{es.sum2b}. The proof of Lemma \ref{lem.source} is completed.
\end{proof}

In the following lemma, we obtain the estimate of differences
$\delta V_{{N}},$ $\delta\varPsi_{{N}}$, and
$\delta\psi_{{N}},$ by using the tame estimate \eqref{thm2.est}.
See \cite[Lemma 16]{CS08MR2423311}
or \cite[Lemma 8.10]{CSW19MR3925528} for the proof.
\begin{lemma}\  \label{lem.Hl1}
 Let ${\mu}\geq 7$ and $\tilde{\mu}={\mu}+3$.
 If $\epsilon>0$ and $\|f^a\|_{H^{{\mu}+1}_{\gamma}(\Omega_T)}/\epsilon$ are suitably small
 and $\theta_0\geq1$ is  large enough, then
 \begin{align} \label{Hl.a}
 \|(\delta V_{{N}},\delta\varPsi_{{N}})\|_{H^{m}_{\gamma}(\Omega_T)}
 +\|\delta\psi_{{N}}\|_{H^{m+1}_{\gamma}(\omega_T)}
 \leq \epsilon \theta_{{N}}^{m-{\mu}-1}\Delta_{{N}}
 \quad \textrm{for  $m=2,\ldots, \tilde{\mu}$.}
 \end{align}
\end{lemma}

Lemma \ref{lem.Hl1} establishes the first part of the hypothesis $\mathrm{\bf H}_{{N}}$.
The following lemma addresses the remaining components of  $\mathrm{\bf H}_{{N}}$.

\begin{lemma}\ \label{lem.Hl2}
 Let ${\mu}\geq 7$ and $\tilde{\mu}={\mu}+3$.
 If $\epsilon>0$ and $\|f^a\|_{H^{{\mu}+1}_{\gamma}(\Omega_T)}/\epsilon$ are suitably small
 and $\theta_0\geq1$ is  large enough,
 then
 \begin{alignat}{3}\label{Hl.b}
 &\|\mathcal{L}( V_{{N}},  \varPsi_{{N}})-f^a\|_{H^{m}_{\gamma}(\Omega_T)}\leq 2 \epsilon \theta_{{N}}^{m-{\mu}-1}
 && \quad \textrm{for } m=2,\ldots, \tilde{\mu}-1,\\
 \label{Hl.c}
 &\|\mathcal{B}( V_{{N}} ,  \psi_{{N}})\|_{H^{m}_{\gamma}(\omega_T)}
 \leq  \epsilon \theta_{{N}}^{m-{\mu}-1}
 && \quad \textrm{for } m=3,\ldots,  {{\mu}},\\
 & \|\mathcal{E}( V_{{N}},  \varPsi_{{N}})\|_{H^{3}_{\gamma}(\Omega_T)}\leq
 \epsilon \theta_{{N}}^{2-{\mu}}.
 &&
 \label{Hl.d}
 \end{alignat}
\end{lemma}

We refer to \cite[Lemmas 17--18]{CS08MR2423311}
or \cite[Lemma 8.11]{CSW19MR3925528} for the proof of Lemma \ref{lem.Hl2}.
Let us assume ${\mu}\geq 7$, $\tilde{\mu}={\mu}+3$,
$\epsilon>0$ and $\|f^a\|_{H^{{\mu}+1}_{\gamma}(\Omega_T)}/\epsilon$ sufficiently small. Additionally, let $\theta_0\geq1$ be large enough to satisfy
the assumptions of Lemmas \ref{lem.Hl1}--\ref{lem.Hl2},
from which the inductive hypothesis $\mathrm{\bf H}_{{N}}$ follows.
Then, as shown in \cite[Lemma 19]{CS08MR2423311} or
\cite[Lemma 8.12]{CSW19MR3925528}, it can be proved that the hypothesis $\mathrm{\bf H}_{{0}}$ holds.

\begin{lemma}\ \label{lem.H0}
If $\|f^a\|_{H^{{\mu}+1}_{\gamma}(\Omega_T)}/\epsilon$ is small enough,
then the hypothesis $\mathrm{\bf H}_{{0}}$ holds.
\end{lemma}

With this, we are ready to complete the proof of Theorem {\rm\ref{thm}}.
Our proof closely follows the idea in \cite{CS08MR2423311, CSW19MR3925528}, but is included here for completeness.

\vspace*{3mm}
\noindent  {\bf Proof of Theorem {\rm\ref{thm}}}.\quad
Let $\tilde{\mu}:=s_0-4\geq 10$ and ${\mu}:=\tilde{\mu}-3\geq 7$.  Under the assumptions of Theorem {\rm\ref{thm}}, the initial data $(U_0^{\pm},\varphi_0)$ are compatible up to order $s_0=\tilde{\mu}+4$.
If $(\tilde{U}_0^{\pm},\varphi_0)$ is sufficiently small
in $H^{s_0+1/2}(\mathbb{R}^3_+)\times H^{s_0+1}(\mathbb{R}^2)$,
where $\widetilde{U}_0^{\pm}:=U_0^{\pm}-\bar{U}^{\pm}$,
then the  assumption \eqref{small} and all the requirements
of Lemmas \ref{lem.Hl1}--\ref{lem.H0} are satisfied,
due to \eqref{app3} and \eqref{f.a.est}. Thus,   $\mathrm{\bf H}_{{N}}$  holds
for all   ${N}\in\mathbb N$. From the estimate 
\begin{align*}
\sum_{n=0}^{\infty}\left(\|(\delta V_{n},\delta \varPsi_{n})\|_{H^{m}_{\gamma}(\Omega_T)}+\|\delta\psi_{n}\|_{H^{m+1}_{\gamma}(\omega_T)} \right)
\leq C\sum_{n=0}^{\infty}\theta_{n}^{m-{\mu}-2} <\infty, \quad 3\le  m\le {\mu}-1,
\end{align*}
 we conclude that  $(V_{n},\varPsi_{n})$ converges
to some   $(V,\varPsi)$ in $H^{{\mu}-1}_{\gamma}(\Omega_T)$,
and   $\psi_{n}$ converges to some   $\psi$ in $H^{{\mu}}_{\gamma}(\omega_T)$.
Taking the limit in \eqref{Hl.b}--\eqref{Hl.c} for $m={\mu}-1=s_0-8$,
and in \eqref{Hl.d}, we obtain that  $(V,\varPsi)$ satisfies \eqref{P.new}.
Consequently, $(U, \varPhi)=(U^a+V, \varPhi^a+\varPsi)$ is a solution to   \eqref{Phi.eq}--\eqref{EVS} in $\Omega_T^+$.
The proof of Theorem \ref{thm} is complete.
\qed

\section*{Acknowledgments}
R.M. Chen was supported in part by the NSF grant DMS-2205910. F. Huang was supported in part by National Key R\&D Program of China, grant No. 2021YFA1000800, and the National Natural Sciences Foundation of China, grant No. 12288201. D. Wang was supported in part by NSF grants DMS-2219384 and DMS-2510532. D.  Yuan was supported by EPSRC grant EP/V051121/1, Fundamental Research Funds for the Central Universities No. 2233100021 and No. 2233300008. 
 
D. Yuan would like to thank Professor Xianpeng Hu for helpful discussions during his stay at City University of Hong Kong. The authors are very grateful to the anonymous referees for the valuable comments and suggestions.

\bibliographystyle{siam}
\bibliography{VS}

\end{document}